\documentclass{EngC}
      \usepackage{soul}        
     \usepackage{rotating}
  \usepackage{floatpag}
  \rotfloatpagestyle{empty}

\usepackage{amsmath,amsfonts,empheq,latexsym,epsfig,subfigure,fancybox,verbatim,paralist,mdframed,framed,amsthm,algorithm,isomath,multirow,makeidx,cancel,enumitem,algorithmic,minitoc,url,varioref,datetime,setspace,array,fancyhdr}

\usepackage{vikram}

\usepackage{mdframed}

\usepackage[framed]{mcode}

\usepackage{fancyheadings,etoolbox}
\usepackage[title]{appendix}

\usepackage{pstricks-add,pst-sigsys,pst-tree,pst-bezier,pst-plot}
\usepackage[stable]{footmisc}

\AtBeginEnvironment{subappendices}{%
  \addtocontents{toc}{\protect\addvspace{10pt}{\em Appendix to Chapter \thechapter}}
}

\renewcommand*{\rmdefault}{ppl}

\renewcommand{\thepart}{Part \Roman{part}.}
\usepackage{titlesec}
\titleformat{\part}{\huge\bfseries}{}{0pt}{\begin{center}\thepart\end{center}}

\newcounter{rowno}
\allowdisplaybreaks

\makeindex

\newcommand{\secn}{\S}
\newcommand{\chp}{\S}

\newenvironment{myassumptions}{%
   \begin{description}[style=multiline, leftmargin = 18pt, align=left]%
}{%
   \end{description}%
}

\makeatletter
\def\nl#1#2{\begingroup
    #2%
    \def\@currentlabel{#2}%
    \phantomsection\label{#1}\endgroup
}
\makeatother

{\ensuremath{
\begin{empheq}[box=\fbox]{align}
{#1}
\end{empheq}
}}

\newtheorem{theorem}            {Theorem}[section]
\newtheorem{corollary}          [theorem]{Corollary}
\newtheorem{proposition}        [theorem]{Proposition}
\newtheorem{definition}         [theorem]{Definition}

\newtheorem{lemma}              [theorem]{Lemma}

\newcommand{\normal}{\mathbfit{N}}  

{
\begin{mdframed}
\par\noindent\textbf{#1:}\begin{rmfamily}\noindent}%
{\end{rmfamily}
\end{mdframed}
} 

\newcommand{\pwe}{Complements and Sources}

\newcommand{\I}{\Pi}

\newcommand{\pdf}{p}

\newcommand{\prob}{\mathbb{P}}
\newcommand{\E}                 {\Bbb{E}}

\renewcommand{\P}                 {\Bbb{P}}

\newcommand{\lowdim}{r}

\newcommand{\obs}{y}

\newcommand{\statem}{A}

\newcommand{\state}{x}
\newcommand{\statespace}{\mathcal{X}}
\newcommand{\obspace}{\mathcal{Y}}
\newcommand{\statedim}{X}
\newcommand{\obsdim}{{Y}}

\newcommand{\mc}{r}  

\newcommand{\fun}{\phi}

\newcommand{\oprob}{B}
\newcommand{\tp}{P}
\newcommand{\utp}{\bar{\tp}}
\newcommand{\ltp}{\underline{\tp}}

\newcommand{\btp}{\bar{\tp}}

\newcommand{\finaltime}{N}

\newcommand{\model}{\theta}
\newcommand{\modelpsi}{\psi}

\newcommand{\belief}{\pi}
\newcommand{\beliefzero}{\belief}
\newcommand{\bbelief}{\bar{\pi}}
\newcommand{\ubelief}{{q}}
\newcommand{\upbelief}{\bar{\pi}}
\newcommand{\lbelief}{\underline{\pi}}

\newcommand{\lmean}{\underline{\state}}
\newcommand{\mean}{{\hat{\state}}}
\newcommand{\umean}{\bar{\state}}

\newcommand{\Belief}{\Pi(\statedim)}

\newcommand{\tbelief}{\pi^0}

\newcommand{\sigs}{\sigma}
\newcommand{\Bs}{R^\pi} 

\newcommand{\priv}{\eta}
\newcommand{\etaregion}{\kappa}

   \newcommand{\ca}{\cost_\action}
   
   \newcommand{\ta}{\tilde{\action}}

\newcommand{\Delayseti}{\{0 \text{ (announce change)} ,1,2,\ldots, L\}}

\newcommand{\ole}{\stackrel{\text{defn}}{=}}

\newcommand{\gtp}{\underset{\text{\tiny TP2}}{\geq}}
\newcommand{\gr}{\geq_r}
\newcommand{\lr}{\leq_r}
\newcommand{\gs}{\geq_s}
\newcommand{\ls}{\leq_s}

\newcommand{\filterd}{\sigma}

\newcommand{\filter}{T}

\newcommand{\argmin}{\operatornamewithlimits{argmin}}
\newcommand{\argmax}{\operatornamewithlimits{argmax}}

\newcommand{\reals}{{\rm I\hspace{-.07cm}R}}

\newcommand{\F}{\mathcal{F}}   

\newcommand{\beq}{\begin{equation}}
\newcommand{\eeq}{\end{equation}}
\newcommand{\nn}{\nonumber}

\renewcommand{\(}		{\left(}
\renewcommand{\)}		{\right)}

\newcommand{\Y}{\mathbf{Y}}

\renewcommand{\th}{\theta}

\newcommand{\p}{\prime}

\newcommand{\one}{\mathbf{1}}
\newcommand{\ones}{\mathbf{1}}
\newcommand{\zero}{\mathbf{0}}

\newcommand{\f}{f} 

\newcommand{\diag}{\textnormal{diag}}

\newcommand{\bq}{\bar{\ubelief}}

\newcommand{\lR}{\preceq}
\newcommand{\gR}{\succeq}
\newcommand{\cons}{{\text{\bf Cons}}}
\newcommand{\consb}{\overline{\text{\bf Cons}}}
\newcommand{\conv}{{\text{conv}}}

\newcommand{\levels}{g}

\newcommand{\map}{\hat{\state}^{\text{MAP}}}
\newcommand{\lmap}{\underline{\state}^{\text{MAP}}}
\newcommand{\umap}{\bar{\state}^{\text{MAP}}}

\newcommand{\bmodel}{\bar{\model}}

\newcommand{\cost}{c}

\newcommand{\reward}{r}
\newcommand{\terminalcost}{\cost_\finaltime}
\newcommand{\Cost}{C}

\newcommand{\yi}{y^{(1)}}
\newcommand{\yii}{y^{(2)}}

\newcommand{\bmc}{\bar{m}}

\newcommand{\bQ}{\bar{Q}}
\newcommand{\bvalueb}{\bar{\valueb}}

\newcommand{\action}{u}
\newcommand{\baction}{\bar{\action}}
\newcommand{\actionspace}{\,\mathcal{U}}
\newcommand{\actiondim}{U}

\newcommand{\exptotalcost}{J_{\bpolicy}}
\newcommand{\opttotalcost}{J_{\bpolicy^*}}
\newcommand{\discount}{\rho}

\newcommand{\region}{\mathcal{R}}

 \newcommand{\Ep}{\E_{\policy}}

\newcommand{\policy}{\mu}
\newcommand{\optpolicyv}{\bpolicy^*}
\newcommand{\optpolicy}{\policy^*}
\newcommand{\bpolicy}{{\boldsymbol{\mu}}}

\newcommand{\valuef}{V}

\newcommand{\valueb}{J}
\newcommand{\info}{\mathcal{I}}

\newcommand{\bvalvec}{\bar{\gamma}}
\newcommand{\valvec}{\gamma}
\newcommand{\Valvec}{\Gamma}

\newcommand{\overlook}{\beta}
\newcommand{\blockprob}{q}
\newcommand{\augss}{\bar{\statespace}}
\newcommand{\terminal}{T}
\newcommand{\schRwd}{h}

\newcommand{\schHst}{\info}
\newcommand{\policySpc}{\mathcal{U}}
\newcommand{\schRwdFn}{J}

\newcommand{\beliefmon}{\belief_{\mathcal{M}}}

\newcommand{\iter}{n}
\newcommand{\iterfinal}{N}

\renewcommand{\time}{k}
\newcommand{\timet}{t}

\newcommand{\Hyperplane}{\mathcal{H}}

\newcommand{\poly} {S}

\renewcommand{\l}{\mathcal{L}}

\newcommand{\eye}{\textit{I}}
\newcommand{\bd}{\succeq_{\mathcal{B}}}
\newcommand{\aB}{R}
\newcommand{\tpone}{{\tp}}
\newcommand{\tptwo}{\bar{\tp}}

\newcommand{\boprob}{\bar{\oprob}}

\newcommand{\noise}{n}

\newcommand{\statelvl} {h}
\newcommand{\sqg} {H}
\newcommand{\expcost} {C}

\newcommand{\cop}{C_o}
\newcommand{\thr}{\mathbf{\Gamma}}
\newcommand{\copomat}{\Gamma}

\newcommand{\uvalue}{\underline{V}}
\newcommand{\bvalue}{\overline{V}}
\newcommand{\valueaction}{Q}

\newcommand{\bvalueaction}{\overline{Q}}
\newcommand{\optvalue}{V}
\newcommand{\trid}{\Upsilon}
\newcommand{\filternorm}{\sigma}

\newcommand{\gc} {\succeq}
\newcommand{\lc} {\preceq}

\newcommand{\error}  {\eta}

\newcommand{\lcost}{\underline{\textit{C}}}
\newcommand{\ucost}{\overline{\textit{C}}}
\newcommand {\uf} {\textit{f}}
\newcommand {\of} {\textit{g}}

\newcommand {\setoverlap} {{\Pi_{O}}}
\newcommand {\policyu} {\overline{\mu}}
\newcommand {\policyl} {\underline{\mu}}

\newcommand{\basisvec} {\textit{e}}

\newcommand {\percentloss} {\epsilon}
\newcommand {\approxpolicy} {\tilde{\mu}}

\newcommand{\pomSt}{\ensuremath{s}}
\newcommand{\pomStSpc}{\ensuremath{S}}
\newcommand{\pomRwd}{\ensuremath{g}}
\newcommand{\pomTranMat}{\ensuremath{P}}

\newcommand{\gl}{\geq_{L_i}}
\newcommand{\glp}{\geq_{L_{i+1}}}

\newcommand{\glX}{\geq_{L_X}}
\newcommand{\glone}{\geq_{L_1}}

\newcommand{\bp}{{\bar{\pi}}}

\newcommand{\hphi}{\hat{\phi}}
\newcommand{\nablat}{\widehat{\nabla}_{\phi}}
\newcommand{\direction}{\omega}

\newcommand{\epoch}{\tau}
\newcommand{\D}{D}

\newcommand{\delayseti}{\{1,2,\ldots,L\}}
\newcommand{\kstar}{k^*}
\newcommand{\changetime}{\tau^0}
\newcommand{\falsealarm}{{f}}
\newcommand{\Vb}{\bar{V}}
\newcommand{\Pp}{\P_\mu}
\newcommand{\Cb}{\bar{C}}
\newcommand{\risk}{\epsilon}

\newcommand{\A}{\mathbb{A}}

\newcommand{\numtarget}{L}
\newcommand{\target}{l}
\newcommand{\delt}{\delta}

\newcommand{\barray}{\begin{array}{ll}}
\newcommand{\earray}{\end{array}}

\newcommand{\e}{\varepsilon}

\newcommand{\blocked}{b}

\newcommand{\eprob}{\gamma}

\begin{document}

\title[A Tutorial]{Structural Results for Partially Observed Markov Decision Processes}
\author{Vikram Krishnamurthy\\ University of British Columbia, \\ Vancouver, Canada. V6T   1Z4. Vancouver. \\ {\tt vikramk@ece.ubc.ca}, \\
December 2015.
}

\maketitle

\dominitoc

\tableofcontents

\frontmatter

\mainmatter
\chapter{Introduction}

This  article  provides an introductory tutorial on structural results in
partially observed Markov decision
processes  (POMDPs).  Typically, computing the optimal policy of a POMDP is computationally intractable.
We use  lattice programming methods to characterize the structure
of the optimal policy of a POMDP  without brute force computations. This article is a very short and somewhat incomplete treatment.
Details, substantially more tutorial material, further examples and proofs can be found in the forthcoming book \cite{Kri16}.

 Contributions to POMDPs have been made by the several communities: operations
research, robotics, machine learning, speech recognition, artificial intelligence, control systems theory, and economics.
POMDPs have numerous examples in controlled sensing, wireless communications,
machine learning, control systems, social learning and sequential detection.

We start with some terminology.
\begin{compactitem}
\item A Markov Decision Process (MDP)  is obtained by controlling the transition probabilities of a Markov chain as
it evolves over time.
\item
A Hidden Markov Model (HMM) is a noisily observed Markov chain.
\item A partially observed Markov decision process (POMDP) is  obtained by controlling
the transition probabilities and/or observation probabilities of an HMM. 
\end{compactitem}
 These relationships are illustrated in Figure \ref{fig:terminology}.

A POMDP specializes to a MDP if
the observations are  noiseless and equal to the state of the Markov chain.
A POMDP specializes to an HMM  if the control is removed.
Finally,  an HMM specializes to a Markov chain if the observations are  noiseless and equal to the state of the Markov chain.

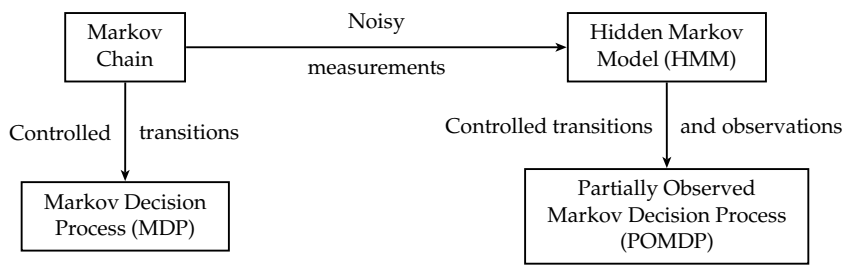
\begin{figure}[h] \centering
\scalebox{0.9}{
\begin{pspicture}[showgrid=false](0.5,-2.2)(14,1.2)


\dotnode[dotstyle=square*,dotscale=0.001](15.5,0.75){dot}

\psblock(2,.75){H1}{\begin{tabular}{c} Markov  \\ Chain \end{tabular}}

\psblock(10,.75){B2}{\begin{tabular}{c}  Hidden Markov  \\ Model (HMM) \end{tabular}}

\psblock(2,-1.75){B1}{\begin{tabular}{c} Markov Decision \\  Process  (MDP)\end{tabular}}

\psblock(10,-1.75){B3}{\begin{tabular}{c} Partially Observed \\  Markov Decision Process \\ (POMDP) \end{tabular}}




\psset{style=Arrow}

\ncline{H1}{B2} \naput[npos=.5]{Noisy}
\nbput[npos=.5]{measurements}

\ncline{H1}{B1} \nbput[npos=.5]{Controlled}
\naput[npos=.5]{transitions}

\ncline{B2}{B3} \nbput[npos=.5]{Controlled transitions}
\naput[npos=.5]{and observations}

\end{pspicture}}
\caption{Terminology of HMMs, MDPs and POMDPs} \label{fig:terminology}
\end{figure}

Suppose a sensor provides 
noisy observations $\obs_k$ of the evolving state $\state_k$ of a Markov stochastic system. The Markov system together with the noisy
sensor constitute a partially observed Markov model (also called a stochastic state space model or  Hidden Markov Model.
The aim is to estimate the state $\state_k$ at each time instant $k$ given the observations $\obs_1,\ldots,\obs_k$. 

In classical statistical signal processing,
the   optimal filter computes the posterior distribution $\belief_{k}$  of the state at time $k$ via the recursive algorithm
\begin{empheq}[box=\fbox]{equation}  \belief_{k} = \filter(\belief_{k-1}, \obs_{k})  \label{eq:introfilter}  \end{empheq}
where the operator $T$ is essentially Bayes' rule. Once the posterior $\belief_k$ is evaluated,
  the  optimal estimate (in the minimum mean square
sense)
of the state $\state_k$  given the noisy observations $\obs_1,\ldots,\obs_k$  can be computed by integration.

Statistical signal processing deals with extracting signals from noisy
measurements.   Motivated by physical, communication
and social constraints, we address the deeper issue of how to
dynamically schedule and optimize signal processing resources to
extract signals from noisy measurements.
Such problems are formulated as POMDPs.
 Figure \ref{fig:sensadapt2} displays the schematic setup.

\begin{figure}
\scalebox{0.8}{
\begin{pspicture}[showgrid=false](0.5,-2)(20,2.3)
\dotnode[dotstyle=square*,dotscale=0.001](13,0.75){dot}
\psblock(2,.75){H1}{\begin{tabular}{c} Stochastic  \\ System \\ (Markov)\end{tabular}}
\psblock(5,.75){B2}{\begin{tabular}{c} Noisy \\ Sensor \end{tabular}}
\psblock(10,.75){B1}{\begin{tabular}{c} Bayesian  \\ Filter   \end{tabular}}
\psblock(8.3,-1.75){B3}{\begin{tabular}{c} POMDP Controller \\ (Decision Maker) \end{tabular}}

\psset{style=Arrow}
\ncline{H1}{B2} \naput[npos=.5]{state}
\nbput[npos=.5]{$\state_k$}
\ncline{B2}{B1}  \nbput[npos=.5]{$\obs_k$}
\naput[npos=.5]{observation}
\ncline{B1}{dot} \naput[npos=.5]{belief}
\nbput[npos=.5]{$\belief_k$}
\ncangle[angleA=0,angleB=180]{B3}{dot}
\ncangle[angleA=180,angleB=-90]{B3}{B2}
\nbput[npos=.5]{action}
\naput[npos=.5]{$\action_k$}
\ncangle[angleA=180,angleB=-90]{B3}{H1}
\psframe[linestyle=dashed,style=Dash](0.5,-.5)(8.5,2.2)
\rput(3.5,2){Hidden Markov Model}
\end{pspicture}}
\caption{Schematic of Partially Observed Markov Decision Process (POMDP). 
This article  deals with determining the structure of a POMDP to ensure that the optimal action taken is a monotone function
of the belief; this can result in numerically efficient algorithms to  compute the optimal action (policy).}
\label{fig:sensadapt2}
\end{figure}
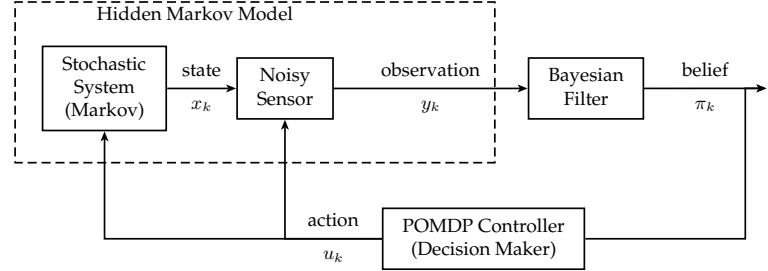

 As in the filtering problem, at each time $k$, a  decision maker   has access to the noisy
observations $\obs_k$ of the state $\state_k$ of a Markov process.
Given these noisy observations,  the aim  is to control the trajectory of the state and observation process by choosing actions  $\action_k$ at each time  $k$. The decision maker knows ahead of time that
if it chooses action  $\action_k$ when the system is in state $\state_k$, then a cost $\cost(\state_k,\action_k)$ will be incurred at time $k$. (Of course the decision maker does not
know state $\state_k$ at time $k$ but can estimate the cost based on the observations $\obs_k$.) 
The goal of the decision maker is to choose the sequence of actions $\action_0,\ldots,\action_{\finaltime-1}$ to minimize the expected {\em cumulative}  cost 
$\E\{ \sum_{k=0}^N \cost(\state_k,\action_k) \}$
where $\E$ denotes mathematical expectation.

The optimal choice of actions
is determined by a {\em policy} (strategy) as $\action_k=  \optpolicy_k(\belief_k)$ where the optimal policy $\optpolicy_k$ satisfies {\em Bellman's
stochastic dynamic programming equation}:
\begin{empheq}[box=\fbox]{align}
  \optpolicy_{k}(\belief)  = \argmin_\action Q_k(\belief,\action)  ,\quad \valueb_k(\belief) = \min_\action Q_k(\belief,\action) ,\nn \\
Q_k(\belief,\action) = \sum_\state \cost(\state,\action) \belief(\state) + \sum_{\obs} \valueb_{k+1}( \filter(\belief,\obs, \action) ) \filterd(\belief,\obs,\action).
\label{eq:introdp}
 \end{empheq}
Here $\filter$ is the optimal filter (\ref{eq:introfilter}), and $\filterd$ is a normalization term for the filter. Also $\belief$ is the posterior computed via the optimal filter (\ref{eq:introfilter}).

  Chapter \ref{ch:pomdpbasic} starts our formal presentation of POMDPs.
  The POMDP model and stochastic dynamic
  programming recursion are formulated in terms of the belief state computed   by the 
  Bayesian  filter.   Several algorithms for solving POMDPs over a finite horizon  are then presented.   Optimal search theory for a moving target is used as an illustrative example of a POMDP.

In general, solving Bellman's dynamic programming equation (\ref{eq:introdp}) for a POMDP is computationally intractable. The main aim of this article is to show
that 
by introducing assumptions on the POMDP model,  important  structural properties of the optimal policy  can be
determined without brute-force computations. These structural properties can then be exploited to compute  the optimal policy.

The main idea behind    is to  give conditions
on the POMDP model so that the
 optimal policy $ \optpolicy_{k}(\belief)$ is monotone\footnote{By monotone, we mean either increasing for all $\belief$ or decreasing for all $\belief$. ``Increasing"  is used here in  the weak sense, it means ``non-decreasing". Similarly for decreasing.}  in belief $\belief$. In simple terms, $ \optpolicy_{k}(\belief)$ is shown to be
 increasing in belief $\belief$ by showing that  \index{submodularity}
$Q_k(\belief,\action) $ in Bellman's equation (\ref{eq:introdp}) is {\em submodular}. The main result is:
\begin{empheq}[box=\fbox]{equation} \underbrace{Q_k(\belief,\action+1) - Q_k(\belief,\action)  \downarrow\belief }_{\text{submodular}}   \implies \underbrace{ \optpolicy_k(\belief) \uparrow 
\belief.}_{\text{increasing policy } } \label{eq:introsubmod} \end{empheq}
Obtaining conditions for $Q_k(\belief,\action) $ to be submodular involves powerful ideas in stochastic dominance and lattice programming.

 Once the optimal policy of a POMDP is shown
to be monotone, this structure can be exploited to devise efficient algorithms. Figure \ref{fig:pomdp2stateintro} 
illustrates  an  increasing optimal
policy $ \optpolicy_k(\belief) $ in the belief $\belief$ with two actions $u_k \in \{1,2\}$. Note that any increasing function which takes on 
two possible values has to be 
 a step function. So
computing $ \optpolicy_k(\belief) $  boils down to determining the single belief  $\belief_1^*$ at which the step function jumps.
Computing (estimating)  $\belief_1^*$   can be substantially easier than directly solving
 Bellman's equation (\ref{eq:introdp}) for  $ \optpolicy_k(\belief) $ for all beliefs $\belief$, especially when
$ \optpolicy_k(\belief) $ has no special structure.

\begin{figure}[h] 
\scalebox{0.65}{
\begin{pspicture}[showgrid=false](-2.75,-0.75)(6,2.5)
\psset{
  xunit = 8,
  yunit = 1
}
     \psaxes{->}(0,0)(0,-0.4)(1.05,2.2)[{\large $\belief$ (belief)},0][{\large $\action_k = \optpolicy_k(\belief)$},90]
    \psset{algebraic,linewidth=1.5pt}
     \pnode(0.4,1){A}
    \pnode(0.4,0){B}
    \pnode(0.8,2){C}
    \pnode(0.8,0){D}
    \pscustom
    {
        \psplot{0}{0.4}{1}
	
        \psplot{0.4}{1}{2}
    }
   
      \ncline[linestyle=dotted]{A}{B}
             \rput(0.4,-0.5){\psframebox*{\large $\belief^*_1$}}
\end{pspicture}}
\caption{Example of   optimal policy $\optpolicy(\belief)$ that is monotone (increasing) in the belief $\belief$. The policy is a step function and   completely characterized by the threshold  state $\belief_1^*$.}
\label{fig:pomdp2stateintro}
\end{figure}
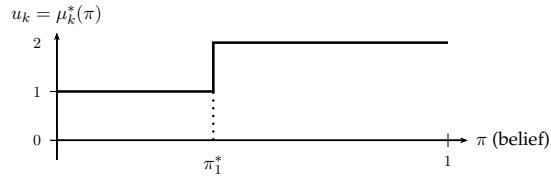

Chapter \ref{chp:monotonemdp} gives sufficient conditions for a MDP to have a monotone (increasing) optimal policy. The explicit dependence of the
MDPs optimal cumulative cost on transition probability is also discussed. 

In order to give conditions for the  optimal policy of a POMDP to be monotone,
one first needs to show  monotonicity of the underlying Hidden Markov Model filter.  To this end, Chapter \ref{chp:filterstructure} discusses the monotonicity of
Bayesian (Hidden Markov Model) filters.  This monotonicity of the optimal filter
is used  to construct reduced complexity filtering algorithms that
 provably lower and upper bound the optimal filter. 
 
Chapters \ref{chp:monotonevalue} to \ref{chp:myopicul} give conditions on the POMDP model  for the dynamic programming recursion to have a monotone  solution.  Chapter \ref{chp:monotonevalue} discusses conditions for the value function in dynamic programming
to be monotone. This is used to characterize the structure of 2-state POMDPs and POMDP multi-armed bandits.

Chapter \ref{ch:pomdpstop} gives conditions under which stopping time POMDPs have monotone optimal policies. As examples,
Chapter \ref{chp:stopapply}  covers
quickest change detection, controlled social learning and a variety of other applications.  The structural results  provide a unifying theme and insight to
what might otherwise simply be a collection of examples.

Finally Chapter  \ref{chp:myopicul} gives conditions under which the optimal policy of a general POMDP can be lower and upper
bounded by judiciously chosen myopic policies. Bounds on the sensitivity of the optimal cumulative cost of POMDPs to 
the parameters are also discussed.

This article is a butchered version (and incomplete version) of the forthcoming book \cite{Kri16} which contains a thorough treatment of structural results, dynamic programming algorithms for POMDPs,
and reinforcement learning algorithms.

\section{Examples of Controlled (Active) Sensing} \index{active sensing} \label{sec:csexamples}
This section outlines some applications of controlled sensing formulated as a POMDP. Controlled sensing also
known as ``sensor adaptive signal processing" or ``active sensing" is a special case of a POMDP where the decision maker (controller) controls 
the observation
noise distribution but  not the dynamics of the stochastic system.  
The setup is as in Figure \ref{fig:sensadapt2} with the link between the controller and stochastic system omitted. 

In controlled sensing,
the decision maker controls the  observation
noise distribution by switching between various sensors or sensing modes. An accurate sensor yields less noisy measurements
by is expensive to use. An inaccurate sensor yields more noisy measurements by is cheap to use. How should the decision
maker decide at each time which sensor or sensing mode to use?
Equivalently, 
how can a sensor be made ``smart" to adapt its behavior to its environment in real time?
Such an active sensor  uses {\em feedback} control.  As shown in Figure \ref{fig:sensadapt2}, the estimates of the signal are fed to a controller/scheduler that decides the sensor should
adapt so as to obtain improved  measurements; or alternatively minimize a measurement cost. Design and analysis of such  closed loop systems
which  deploy stochastic control
is non-trivial.  The estimates from the signal processing algorithm are uncertain (they are posterior probability distribution functions).
So controlled sensing requires decision making under uncertainty. 


We now
highlight some  examples in controlled sensing.

\paragraph{Example 1.  Adaptive Radars}

Adaptive
multifunction radars are capable of  switching between
various measurement modes, e.g., radar transmit waveforms, beam pointing
directions, etc, so that the tracking system is able to tell the radar
which  mode to use at the next measurement epoch.  Instead of the operator continually changing
the radar from mode to mode depending on the environment, the aim is to construct feedback control 
algorithms that dynamically adapt where the radar radiates its pulses to achieve the command operator objectives.
This results in radars that autonomously switch beams, transmitted waveforms, target dwell and re-visit times.

\paragraph{Example 2.  Social Learning and Data Incest}
\index{social learning} 

A {\em social sensor}
(human-based sensor)  denotes an agent that provides information about its environment  (state of nature) to  a social network. Examples of  such social sensors include Twitter posts, Facebook status updates, and ratings on online reputation systems like Yelp and Tripadvisor.
Social sensors  present unique challenges from a statistical estimation point of view.
since they  interact with and influence other social sensors. Also, due to privacy
concerns, they reveal their decisions 
(ratings, recommendations, votes) which can be viewed as a low resolution (quantized) function of their raw measurements.

\paragraph{Example 3. Quickest Detection and Optimal Sampling} \index{quickest detection}  \index{optimal sampling}

Suppose
a decision maker records measurements of a finite-state  Markov chain  corrupted by  noise. The goal is to decide when   the Markov chain hits  a specific target state. The decision maker can choose from a finite set of sampling intervals to pick the next time
to look at the Markov chain. 
The aim is to
 optimize an objective comprising of false alarm, delay cost and cumulative measurement sampling cost.
Taking more frequent measurements yields accurate estimates but incurs a higher measurement cost. Making an erroneous decision too soon incurs a false alarm penalty. Waiting too long to declare the target state
incurs a delay penalty.
What is the optimal sequential strategy for the decision maker?
It is shown in \secn \ref{chp:qdos} that  the optimal sampling problem  results in a POMDP that has a monotone optimal strategy in the belief state.

\paragraph{Example 4. Interaction of Local and Global Decision Makers} \index{social learning}
 In a multi-agent network, how can agents  use their noisy observations and decisions made by previous agents to estimate an underlying randomly
evolving state?
How do decisions made by previous agents affect decisions made by subsequent agents?
In \secn \ref{sec:change}, these
 questions will be   formulated as a multi-agent sequential detection problem involving social learning.
 Individual agents record  noisy observations of an underlying state process,  and  perform social learning
to estimate the underlying state. They make local decisions about whether a change has occurred that optimize their 
individual utilities. Agents then broadcast
their local decisions  to subsequent agents. As  these local decisions accumulate over
time, a global decision maker needs to decide (based on these local decisions) whether or not to declare a change has occurred.
How can the global decision maker achieve such change detection to minimize a cost function comprised of false alarm rate and delay penalty? The local and global decision makers
interact, since
the local decisions determine the posterior distribution of subsequent agents which determines the global decision (stop or continue) which determines subsequent  local decisions.
\paragraph{Other Applications of POMDPs}

POMDP are used in numerous other domains. Some applications include:

\begin{compactitem}
\item Optimal Search: see \secn \ref{chp:optimal search}.


\item Quickest Detection and other Sequential Detection Problems: see Chapter \ref{ch:pomdpstop}.

\item Dialog Systems:  see \cite{YGT13}  and references therein.

\item Robot navigation and planning: see   \cite{KHL08} and references therein.

\item Cognitive Radio dynamic spectrum sensing:   see \cite{ZTSC07} and references therein.
\end{compactitem}

\chapter{Partially Observed Markov Decision Processes (POMDPs)}
\label{ch:pomdpbasic}

A POMDP is a {\em controlled} HMM.  An HMM consists of an $\statedim$-state  Markov chain $\{\state_k\}$  observed via a noisy observation process $\{\obs_k\}$.
 Figure  \ref{fig:part3fig} displays the schematic setup of a POMDP where the action $\action_k$ affects the state and/or observation (sensing) process
 of the HMM.
The HMM filter  computes the posterior distribution $\belief_k$ of the state. The posterior $\belief_k$ is called the {\em belief state}. In a POMDP, the stochastic controller depicted in Figure  \ref{fig:part3fig} uses the belief state to choose the next action. 

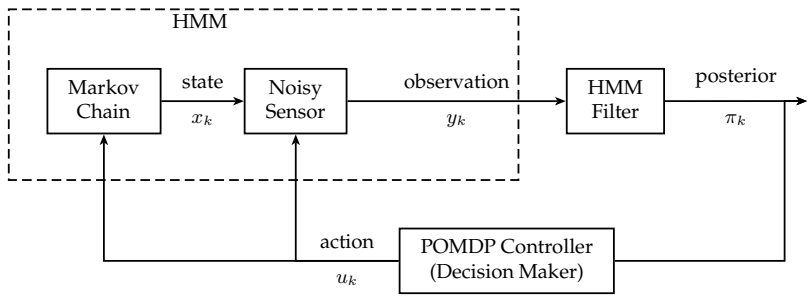
\begin{figure}[h]
\centering
\scalebox{0.85}{\begin{pspicture}[showgrid=false](0.5,-2)(20,2.55)


\dotnode[dotstyle=square*,dotscale=0.001](13,0.75){dot}

\psblock(2,.75){H1}{\begin{tabular}{c} Markov  \\ Chain \end{tabular}}

\psblock(5,.75){B2}{\begin{tabular}{c} Noisy \\ Sensor \end{tabular}}

\psblock(10,.75){B1}{\begin{tabular}{c} HMM  \\ Filter   \end{tabular}}

\psblock(8.3,-1.75){B3}{\begin{tabular}{c} POMDP Controller \\ (Decision Maker) \end{tabular}}




\psset{style=Arrow}

\ncline{H1}{B2} \naput[npos=.5]{state}
\nbput[npos=.5]{$\state_k$}

\ncline{B2}{B1} \nbput[npos=.5]{$\obs_k$}
\naput[npos=.5]{observation}

\ncline{B1}{dot} \naput[npos=.5]{posterior}
\nbput[npos=.5]{$\belief_k$}

\ncangle[angleA=0,angleB=180]{B3}{dot}

\ncangle[angleA=180,angleB=-90]{B3}{B2}
\nbput[npos=.5]{action}
\naput[npos=.5]{$\action_k$}

\ncangle[angleA=180,angleB=-90]{B3}{H1}

\psframe[linestyle=dashed,style=Dash](0.5,-.5)(8.5,2.2)
\rput(3.5,2){HMM}
\end{pspicture}}
\caption{Partially Observed Markov Decision Process (POMDP) schematic setup. The Markov system together with noisy sensor constitute
a Hidden Markov Model (HMM). The HMM filter computes the posterior (belief  state) $\belief_k$ of the state of the Markov chain. The controller 
(decision maker) then chooses the action $\action_k$
at time $k$
based on  $\belief_k$.}
\label{fig:part3fig}
\end{figure}

This chapter is organized as follows.
\secn \ref{sec:fhpomdp} describes  the POMDP model. Then \secn \ref{sec:beliefstate} gives the belief state formulation and the Bellman's dynamic programming equation for the optimal policy of a POMDP.  It is shown that a POMDP is equivalent to a {\em continuous-state} MDP where the states are
belief states (posteriors).  Bellman's equation for continuous-state MDP was discussed  in \chp \ref{sec:contstatemdp}.
\secn \ref{sec:mrpomdp} gives a toy example of a POMDP.
Despite being a continuous-state MDP, \secn \ref{sec:fdcontroller} shows that 
 for finite horizon POMDPs, Bellman's equation  has a finite dimensional characterization. 
\secn \ref{sec:pomdpexactalgorithms}  discusses several algorithms that exploit this finite dimensional characterization to  compute the optimal policy.
\secn \ref{chp:discpomdp}  considers discounted cost infinite horizon POMDPs.
As an example of a  POMDP,  optimal search of a moving target is discussed in \secn \ref{chp:optimal search}.

\section{Finite Horizon POMDP} \label{sec:fhpomdp}
A  POMDP model with finite horizon $\finaltime$ is a 7-tuple
\beq \label{eq:pomdpmodel}
(\statespace, \actionspace, \obspace, \tp(\action),  \oprob(\action), \cost(\action), \cost_\finaltime).
\eeq
\begin{compactenum}
\item   $\statespace= \{1,2,\ldots,\statedim\}$ denotes the  state space and $\state_k \in \statespace$ denotes the state of 
a controlled Markov chain at time $k=0,1,\ldots,\finaltime$.
\item 
$\actionspace = \{1,2,\ldots,\actiondim\}$ denotes the action space with $\action_k \in \actionspace$ denoting the action chosen at time $k$
by the controller.
\item 
 $\obspace$ denotes the observation space which can either be finite or a subset of $\reals$. $
 \obs_k \in \obspace$ denotes the observation recorded at each time $k \in \{1,2,\ldots, \finaltime\}$.
\item
For each action $\action  \in \actionspace$, $\tp(\action)$ denotes a $\statedim \times \statedim$ transition probability matrix with
elements
\beq  \tp_{ij}(\action)  = \prob(\state_{k+1} = j | \state_k = i, \action_k = \action), \quad i,j \in \statespace.  \label{eq:pomdptp}\eeq
\item
For each action $\action  \in \actionspace$, $\oprob(\action)$ denotes the observation distribution with 
\beq \oprob_{i \obs}(\action) =  \prob(\obs_{k+1} = \obs | \state_{k+1}=i,\action_k=\action) , \quad i \in \statespace, \obs \in \obspace. 
\label{eq:pomdpoprob} \eeq
\item 
For state $\state_k$ and action $\action_k$, the decision-maker incurs a cost $\cost(\state_k,\action_k)$. 
\item Finally,  at terminal
time $\finaltime$, a terminal cost $\cost_N(\state_N)$ is incurred.
\end{compactenum}

The  POMDP model  (\ref{eq:pomdpmodel}) is a {\em partially observed} model since  the decision-maker does not observe the state $\state_k$. It only observes noisy observations $\obs_k$ that depend on the action and the state specified by the probabilities in (\ref{eq:pomdpoprob}). 
Recall that an HMM    is characterized by
 $(\statespace, \obspace, \tp,  \oprob)$; so a POMDP is a controlled HMM with the additional ingredients of  action space $\actionspace$,
 action dependent transition probabilities, action dependent observation probabilities and costs. 
In general, the transition matrix, observation distribution and cost can be explicit functions of time; however
to simplify notation, we have omitted this time dependency.

Given the  model (\ref{eq:pomdpmodel}),
the dynamics of a POMDP proceed  according to Algorithm~\ref{alg:pomdp}.
This  involves at each time $k$ choosing an action $\action_k$, accruing an instantaneous cost $\cost(\state_k,\action_k)$,
evolution of the  state from $\state_{k}$ to $\state_{k+1}$, and observing $\state_{k+1}$ in noise as $\obs_{k+1}$.

\begin{algorithm}
\caption{Dynamics of Partially Observed Markov Decision Process}
\label{alg:pomdp}
At time $k=0$, the state $\state_0$ is simulated from initial distribution $\belief_0$. \\
For time $k=0,1,\ldots,\finaltime -1 $:
\begin{compactenum}
\item Based on available information 
\beq \info_0 = \{\belief_0\},\;  \info_k= \{\belief_0,\action_0,\obs_1,\ldots, \action_{k-1},\obs_k\} , \label{eq:infopomdp}\eeq the decision-maker chooses action 
\beq \action_k = \policy_k(\info_k)  \in \actionspace, \quad 
 k=0,1,\ldots, \finaltime - 1.  \label{eq:chooseaction}\eeq
Here,  $\policy_k$ denotes a  policy  that the decision maker uses at time $k$.

\item The decision-maker incurs a cost $\cost(\state_k,\action_k)$ for choosing  action $\action_k$.
\item The state evolves randomly with transition probability  $\tp_{\state_k \state_{k+1}}(\action_k)$ to the next state $\state_{k+1}$ at time $k+1$. Here
$$\tp_{ij}(\action) = \prob(\state_{k+1} = j | \state_k = i, \action_k = \action).$$
\item The decision-maker records a  noisy observation $\obs_{k+1} \in \obspace$ of the state $\state_{k+1}$ according to
  $$ \prob(\obs_{k+1} = \obs | \state_{k+1}=i,\action_k=\action) = \oprob_{i \obs}(\action). $$
\item The decision-maker updates its  available  information  as $$\info_{k+1} = \info_k \cup  \{\action_k, \obs_{k+1}\}.$$
If $ k < \finaltime$, then set $k$ to $k+1$  and go back to Step 1.\\
If $k = \finaltime$, then the decision-maker pays a terminal cost $\terminalcost(\state_\finaltime)$ and the process terminates.
\end{compactenum}
\end{algorithm}

 As depicted in (\ref{eq:chooseaction}), at each time $k$, the decision maker uses all the information available until time $k$ (namely, $\info_k$) to
 choose   action $\action_k = \policy_k(\info_k) $  using policy
$\policy_k$.  
With the dynamics specified by Algorithm \ref{alg:pomdp}, 
denote the sequence of policies that the decision-maker uses from time $0$ to $\finaltime-1$ as
 $\bpolicy = (\policy_0,\policy_1,\ldots,\policy_{\finaltime-1})$. 
 
 \subsection*{Objective}
 To specify a POMDP completely, in addition to the model (\ref{eq:pomdpmodel}), dynamics in Algorithm \ref{alg:pomdp}  and policy sequence\footnote{\secn
\ref{sec:beliefstate} shows  that a POMDP is equivalent to a continuous-state MDP. So  it suffices to consider non-randomized policies to achieve the minimum in (\ref{eq:optpolicy}).}
  $\bpolicy$, we need to specify a performance criterion
or objective function.  This section  considers the finite horizon objective
\beq \exptotalcost(\belief_0)  = \E_{\bpolicy} \left\{  \sum_{k=0}^{\finaltime-1} \cost(\state_k,\action_k) +  \terminalcost(\state_\finaltime)  \mid \belief_0\right\} .
\label{eq:objectivepomdpfirst}
\eeq
which is the 
expected  cumulative  cost incurred by  the decision-maker when using policy $\bpolicy$ up to time $\finaltime$ 
given the initial  distribution $\belief_0$ of the Markov chain.
Here, $ \E_{\bpolicy}$ denotes expectation with respect to the joint probability distribution of
$(\state_0,\obs_0,\state_1,\obs_1,\ldots, \state_{\finaltime-1},\obs_{\finaltime-1},\state_\finaltime, \obs_\finaltime)$.
The goal of the decision-maker is to determine the optimal policy  sequence
\beq \optpolicyv = \argmin_\bpolicy \exptotalcost(\belief_0) , \quad \text{ for any initial prior } \belief_0  \label{eq:optpolicy} \eeq
that minimizes the expected  cumulative cost. 
Of course, the optimal policy sequence $ \optpolicyv$ may not be unique.

\subsubsection{Remarks}
\begin{compactenum}
\item  The decision-maker does not observe the state $\state_k$. It only observes noisy observations $\obs_k$ that depend on the action and the state via Step~4. 
Also, the decision-maker knows the cost matrix $\cost(\state,\action)$  for all possible states and actions in $\statespace, \actionspace$.
But since the decision-maker does not know the state $\state_k$ at time $k$, it does not know the cost accrued at time $k$ in Step 2 or terminal cost in Step  5.
Of course, the decision-maker can estimate the cost  by using the noisy observations of the state.
\item 
The term  POMDP is usually reserved for the case when the observation space $\obspace$ is finite. However, we consider both finite and continuous valued observations.
\item  The action $\action_k$ affects the evolution of the state (Step 3) and observation distribution (Step 4).  In controlled sensing applications such as radars and sensor networks,
the action only affects the observation distribution and not the evolution of the target.
\item More generally, the cost can be of the form  
$\bar{\cost}(\state_k=i,\state_{k+1}=j,\obs_k = \obs, \obs_{k+1}= \bar{\obs},\action_k=\action)$. This is equivalent to the cost  (see (\ref{eq:infocosta}) below)
\beq \cost(i,\action) = \sum_{y\in \obspace} \sum_{\bar{y} \in \obspace} \sum_{j\in \statespace} 
\bar{\cost}(i,j,,y,\bar{y},\action) \tp_{ij}(\action) \oprob_{j\bar{\obs}}(\action) \oprob_{i\obs}(\action). \label{eq:gencoste} \eeq
\end{compactenum}

\section{Belief State Formulation and  Dynamic Programming} \label{sec:beliefstate} \index{belief state|(}
This section details a crucial step in the formulation and solution of a POMDP, namely, the belief state formulation. In this formulation,
a POMDP is equivalent to a continuous state MDP with states being the belief states.
We then formulate the optimal policy as the solution to Bellman's dynamic programming recursion written in terms of the belief state.
Finally, we state and prove the main result: the solution of the dynamic programming recursion for a POMDP has an explicit
piecewise linear and concave solution.

\subsection{Belief State Formulation of POMDP}
Recall from  \secn \ref{sec:dpalgmdp}
that for  a fully observed MDP,  the optimal policy is Markovian and  the optimal action $u_k = \optpolicy_k(\state_k)$.
In comparison, for a POMDP the optimal action chosen by the decision maker is in general  \beq u_k = \optpolicy_k(\info_k), \quad
\text{  where } \quad \info_k =  (\belief_0,\action_0,\obs_1,\ldots, \action_{k-1},\obs_k) .
\label{eq:infopatternpomdp} \eeq
 Since $\info_k$ is increasing in dimension with $k$, to implement a controller, it is useful  to obtain a sufficient statistic  that does not grow in dimension.
 The posterior distribution $\belief_k$ computed via the HMM filter  is
a sufficient  statistic for $\info_k$. 
Define the posterior distribution of the Markov chain given $\info_k$ as
\beq \belief_k(i)  = \prob(\state_k = i | \info_k) ,\; i \in \statespace \quad \text{ where }  \info_k = \{\belief_0,\action_0,\obs_1,\ldots, \action_{k-1},\obs_k\} .  \label{eq:beliefdef} \eeq
We will call $\belief_k$ as the {\em  belief state} or  {\em information state}  at time $k$. It is computed via 
 the HMM filter namely
$\belief_{k} = \filter(\belief_{k-1},\obs_k,\action_{k-1})$ where
\begin{align}  \filter\left(\belief,\obs,\action\right) &= \cfrac{\oprob_{\obs}(\action)\tp'(\action)\belief}{\filternorm\left(\belief,\obs,\action\right)} , \quad \text{ where } 
\filternorm\left(\belief,\obs,\action\right) = \one_{\statedim}'\oprob_{\obs}(\action) \tp'(\action)\belief,  \label{eq:hmmc}  \\
\oprob_{\obs}(\action) &= \diag\big(\oprob_{1\obs}(\action),\cdots,\oprob_{\statedim\obs}(\action) \big). \nn 
\end{align}
The main point established below in Theorem \ref{thm:pomdpdpeq} is that (\ref{eq:infopatternpomdp}) is equivalent to
\beq u_k = \optpolicy_k(\belief_k). \label{eq:mainmap}
  \eeq
In other words, the optimal controller operates on the belief state $\belief_k$ (HMM filter posterior) to determine the action $\action_k$.

In light of (\ref{eq:mainmap}), let us first define the space where $\belief_k$ lives in.
The beliefs $\belief_k$$, k=0,1,\ldots$  defined  in (\ref{eq:beliefdef}) are $\statedim$-dimensional probability vectors. Therefore
they lie in the $\statedim-1$ dimensional unit simplex denoted as \index{belief space unit simplex $\Belief$}
\begin{align*}
\Belief \ole \left\{\pi \in \reals^{\statedim }: \one^{\p} \pi = 1,
\quad 
0 \leq \pi(i) \leq 1 \text{ for all } i \in \statespace = \{1,2,\ldots,\statedim\}\right\} .
\end{align*}
 $\Belief$ is called  the {\em belief space}.
  $\Pi(2)$ is a one dimensional simplex (unit line segment), As shown in Figure  \ref{fig:pyramid},
$\Pi(3)$ is a two-dimensional  simplex (equilateral triangle); $\Pi(4)$
is a tetrahedron, etc.
Note that the unit vector states  $e_1,e_2,\ldots, e_X$ of the underlying Markov chain $\state$ are the vertices of $\Belief$.

\begin{figure}
\centering
\subfigure[$\statedim=3$]{\includegraphics[scale=0.3]{part2/figures/simplexbig.eps}} \qquad
\subfigure[$\statedim=4$]{\scalebox{0.65}{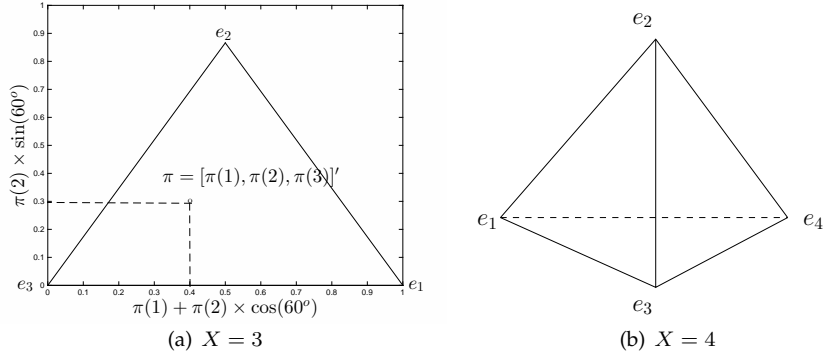}}
\caption{For an $\statedim$ state Markov chain,
the belief space $\Belief $
in a $\statedim-1$ dimensional unit simplex $\Belief$. The figure shows $\Belief $ for $\statedim=3,4$.}
\label{fig:pyramid}
\end{figure}



We now formulate the POMDP objective (\ref{eq:objectivepomdpfirst}) in terms of the belief state.
Consider  the objective (\ref{eq:objectivepomdpfirst}). Then
\begin{align}
\exptotalcost(\belief_0)  & = \E_{\bpolicy} \bigl\{  \sum_{k=0}^{\finaltime-1} \cost(\state_k,\action_k) +  \terminalcost(\state_\finaltime)  \mid \belief_0\bigr\}  \nonumber \\
&\stackrel{(a)}{=}  \E_{\bpolicy} \bigl\{  \sum_{k=0}^{\finaltime-1} \E\{ \cost(\state_k,\action_k) \mid \info_k\} +  \E\{ \terminalcost(\state_\finaltime) \mid \info_\finaltime\} \mid \belief_0\bigr\} \nonumber \\
&=  \E_{\bpolicy}  \bigl\{  \sum_{k=0}^{\finaltime-1} \sum_{i=1}^\statedim \cost(i,\action_k) \, \belief_k(i) + \sum_{i=1}^\statedim \terminalcost(i) \, \belief_\finaltime(i)  \mid \belief_0\bigr\} \nonumber \\
&=  \E_{\bpolicy}  \bigl\{  \sum_{k=0}^{\finaltime-1} \cost^\p _{\action_k} \, \belief_k  + \terminalcost^\p\, \belief_\finaltime  \mid \belief_0\bigr\}  \label{eq:infocosta} \end{align}
where (a) uses the smoothing property of conditional expectations. \index{smoothing property of conditional expectation}
In (\ref{eq:infocosta}),  the $\statedim$-dimensional cost vectors $\cost_{\action}(k)$ and terminal cost vector $\terminalcost$ are defined
as
\begin{align} 
\cost_{\action} &= \begin{bmatrix} \cost(1,\action) & \cdots & \cost(\statedim,\action) \end{bmatrix}^\p , \;
\terminalcost = \begin{bmatrix}  \terminalcost(1) & \cdots & \terminalcost(\statedim)  \end{bmatrix}^\p.
\end{align}

\noindent {\em Summary}: The  POMDP has been expressed as a continuous-state (fully observed) MDP  with dynamics (\ref{eq:hmmc}) given
by the HMM filter and objective function (\ref{eq:infocosta}). This continuous-state MDP   has  belief state $\belief_k$
which lies in unit simplex belief space $\Belief$.
Thus we have the following useful decomposition  illustrated in Figure~\ref{fig:part3fig}:
 \begin{compactitem}
 \item An HMM filter  uses the noisy observations $\obs_k$ to compute the belief state $\belief_k$
 \item The POMDP controller then  maps the belief state $\belief_k$  to the action $\action_k$.
 \end{compactitem}
 Determining the optimal policy for a POMDP is equivalent to partitioning $\Belief$ into regions where  a particular action $\action \in \{1,2,\ldots,\actiondim\}$ is optimal.
  \index{belief state|)}

\subsection{Stochastic Dynamic Programming for POMDP}
Since a POMDP is a continuous-state MDP with state space being the unit simplex,
we can straightforwardly write down the dynamic programming equation for the optimal policy as we did in \chp \ref{sec:contstatemdp} for continuous-state MDPs.

\begin{theorem} \label{thm:pomdpdpeq}For a  finite horizon POMDP with model (\ref{eq:pomdpmodel}) and dynamics  given by
Algorithm \ref{alg:pomdp}:
\begin{compactenum}
\item The minimum expected  cumulative cost $\opttotalcost(\belief)$ is achieved by 
  deterministic policies 
$$ \bpolicy^* =  (\policy^*_0,\policy^*_1,\ldots,\policy^*_{\finaltime-1}), \quad  \text{ where }  \action_k = \policy^*_k(\belief_k) .$$
\item 
 The optimal policy $\optpolicyv=(\policy_0,\policy_1,\ldots,\policy_{\finaltime -1})$ for a POMDP is the solution of the following
 Bellman's dynamic programming    \index{Bellman's equation! finite horizon POMDP}
backward recursion:   Initialize $\valueb_\finaltime(\belief) = \terminalcost^\p \belief$ and then for $k=\finaltime -1, \ldots, 0 $
\begin{align}
\valueb_k(\belief) &=  \min_{\action \in \actionspace} \{ \cost^\p_{\action} \belief + \sum_{\obs \in \obspace} \valueb_{k+1}\( \filter(\belief,\obs,\action) \) \filterd(\belief,\obs,\action) \}   \nonumber \\
\optpolicy_k(\belief) &= \argmin_{\action \in \actionspace}  \{ \cost^\p_{\action} \belief + \sum_{\obs \in \obspace} \valueb_{k+1}\( \filter(\belief,\obs,\action) \) \filterd(\belief,\obs,\action) \} .  \label{eq:bdppomdp} 
\end{align}
The   expected cumulative cost  $J_{\bpolicy^*}(\belief)$ (\ref{eq:infocosta}) of the optimal policy $\optpolicyv$ is given by the value function $\valueb_0(\belief)$ for any initial belief $\belief \in \Belief$. \label{thm:dpfhpomdp}
\end{compactenum}
\end{theorem}

Since the  belief  space $\Belief$ is uncountable, the
above dynamic programming  recursion  does not
translate into practical solution methodologies.
$\valueb_k(\belief) $ needs to be evaluated at each $\pi \in \Belief$, an
uncountable set.

\section{Machine Replacement POMDP. Toy Example}  \index{machine replacement! POMDP}  \label{sec:mrpomdp}
To illustrate the POMDP model and dynamic programming recursion described above, consider a toy example involving the machine replacement problem.
 Here we describe the 2-state version of the problem with noisy observations. 

The state space is $\statespace = \{1,2\}$ where state 1 corresponds to a poorly performing machine  while state 2 corresponds to a 
 brand new 
machine.  The  action space is $\actionspace \in \{1,2\}$ where action $2$ denotes keep using the machine, while action 1 denotes replace the machine with a brand new one which starts in state 2.
The transition probabilities of the machine state are
$$ \tp(1) = \begin{bmatrix} 0 & 1  \\ 0 & 1  \end{bmatrix}, \quad \tp(2) = \begin{bmatrix}
1 & 0 \\
\theta & 1-\theta  \end{bmatrix}.
$$
where $\theta \in [0,1]$ denotes the probability that machine deteriorates.

Assume that the state of the machine $\state_k$ is indirectly observed via the quality of the product $\obs_k \in \obspace = \{1,2\}$ generated by the machine.  
Let  $p$ denote the probability that the machine operating in the good state  produces a high  quality product, and $q$ denote the probability that a deteriorated machine
produces a poor quality product. Then the observation probability matrix is
$$ \oprob = \begin{bmatrix}  p  & 1-p  \\ 1-q  & q \end{bmatrix}.
$$

Operating the machine in state $\state$ incurs an operating cost
 $\cost(\state,u=2)$.  On the other hand, replacing the machine at any state $\state$, costs $R$, that is, $\cost(\state,\action=1) = R$.
The aim is to minimize the  cumulative  expected  cumulative cost
$\E_{\bpolicy}\{\sum_{k=0}^{\finaltime-1} \cost(\state_k,\action_k) | \belief_0\}$ for some specified horizon $\finaltime$. Here  $\belief_0$ denotes
the initial distribution of the state of the machine at time $0$.

Bellman's equation (\ref{eq:bdppomdp})  reads: Initialize $\valueb_\finaltime(\belief) = 0$ (since there is no terminal cost) and for $k=N-1,\ldots,0$:
\begin{align*}  \valueb_k(\belief) &= \min\bigl\{ c_1^\p \belief +  \valueb_{k+1}(e_1), \;\; c_2 ^\p \belief +  \sum_{\obs \in \{1,2\}} \valueb_{k+1}( \filter(\belief, \obs,
2)) \filterd(\belief,\obs,2)  \bigr\} \\ \text{ where } &
\filter(\belief, \obs, 2) = \frac{\oprob_\obs \tp^\p(2) \belief}{\filterd(\belief,\obs,2)}, \; \filterd(\belief,\obs,2) = \one^\p \oprob_\obs \tp^\p(2) \belief , \; \obs \in \{1,2\}
,\\
\oprob_1 &= \begin{bmatrix} p & 0 \\ 0 & 1-q \end{bmatrix}, \; \;  \oprob_2 = \begin{bmatrix}  1-p & 0 \\ 0 & q \end{bmatrix}.
\end{align*}
Since the number of states is $\statedim=2$, the belief space $\Belief$ is a one dimensional simplex, namely the interval $[0,1]$. So $ \valueb_k(\belief) $ can be expressed
in terms of $\belief_2\in [0,1]$, because $\belief_1 = 1 - \belief_2$. Denote this as $\valueb_k(\belief_2) $.

One can then implement the dynamic programming recursion numerically by discretizing $\belief_2$ in the  interval $ [0,1]$ over a finite grid and
running the Bellman's equation   over this finite grid. Although this numerical implementation is somewhat naive, the reader should do this  to visualize the value function
and optimal policy. The reader would notice that the value function $\valueb_k(\belief_2) $ is piecewise linear and concave in $\belief_2$. The main
result in the next section is that for a finite horizon POMDP, the value function is always piecewise linear and concave,  and  the value function and optimal policy
can be determined exactly  (therefore a grid approximation is not required).

\section{Finite Dimensional Controller for finite horizon POMDP} \label{sec:fdcontroller}
Despite the belief space $\Belief$ being continuum, the following remarkable result due to Sondik \cite{Son71,SS73} shows that 
Bellman's equation (\ref{eq:bdppomdp}) for a finite horizon POMDP  has a  finite dimensional characterization
when the observation space $\obspace$ is finite.
\index{value function for POMDP! concavity}
\index{value function for POMDP! finite dimensional characterization}

\begin{theorem} \label{thm:pwlc} Consider the POMDP model (\ref{eq:pomdpmodel}) with finite action space $\actionspace=\{1,2,\ldots,\actiondim\}$ and 
 finite observation space $\obspace= \{1,2,\ldots,\obsdim\}$. At  each time $k$ the value function $\valueb_k(\belief)$ of Bellman's equation
(\ref{eq:bdppomdp})
 and associated optimal policy $\optpolicy_k(\belief)$ have the following finite dimensional characterization:
\begin{compactenum}
\item  $\valueb_k(\belief)$ is piecewise linear and concave with respect to $\belief \in \Belief$. That is,
\beq
\valueb_k(\belief) = \min_{\valvec \in \Valvec_k}  \valvec^\p \belief. \label{eq:pwlc} \eeq
 Here,  $\Valvec_k$ at iteration $k$ is a finite set of $\statedim$-dimensional vectors.\\
 Note $\valueb_\finaltime(\belief) = \terminalcost^\p\, \belief$ and $\Valvec_\finaltime = \{\terminalcost\}$ where
 $\terminalcost$ denotes the terminal cost vector.
\item  The optimal policy $\optpolicy_k(\belief)$  has the following finite dimensional
characterization: The belief space $\Belief$ can be partitioned into at most $|\Valvec_k|$ convex polytopes.
In each such polytope $\region_l =  \{\belief: \valueb_k(\belief)  =  \valvec_l^\p \belief\}$,
the optimal policy
 $\optpolicy_k(\belief)$ is a constant corresponding to a single action.
 That is for belief $\belief \in \region_l$ the optimal policy is
$$\optpolicy_k(\belief) = \action( \argmin_{\valvec_l \in \Valvec_k}  \valvec_l^\p \belief) $$
where the right hand side is the action associated with polytope  $\region_l$.

\end{compactenum}
\end{theorem}

The above theorem says that at each time $k$, the belief space $\Belief$ can be partitioned into at most $|\Valvec_k|$ convex polytopes
and the optimal action within each such polytope is a constant. 
Also each
vector $\valvec \in \Valvec_k$ will be called a {\em gradient vector} since if $\valueb_k(\belief) = \valvec^\p \belief$, then
$\valvec$ is the sub-gradient \cite{Ber00} of the concave function $\valueb_k$ at belief state $\belief$.

Figure  \ref{fig:pwlc} illustrates the piecewise linear concave structure of the value function $\valueb_k(\belief)$ for the case of a two-state Markov chain ($\statedim = 2$). In this case, 
the belief state $\belief =\begin{bmatrix}  1- \belief(2) \\  \belief(2) \end{bmatrix} $ is parametrized by the scalar
$\belief(2) \in [0,1]$ and the belief space is the 
  one-dimensional simplex $\Pi(2) = [0,1]$.
 Figure  \ref{fig:pwlc} also  illustrates the finite dimensional structure of the optimal policy $\optpolicy_k(\belief)$ asserted
 by the above theorem.
 In each region of belief space  where $\gamma_l^\p \belief $ is active, the optimal policy takes on a single action.

\begin{figure} \centering
 \scalebox{0.85}{
 \begin{pspicture*}(-1.5,-0.7)(10.5,7)

   \rput(5,-0.5){\psframebox*{$\belief(2)$}}
   
  \rput(0,-0.5){\psframebox*{$0$}}
   \rput(10,-0.5){\psframebox*{$1$}}

\rput(-0.6, 5){\psframebox*{$\valueb_k(\belief)$}}


\pnode(0,0){A}
\pnode(0,2){B}
\pnode(2,4){C}
\pnode(3,6){D}
\pnode(2,7){F}
\pnode(4,6){G}
\pnode(6,8){H}
\pnode(6,6){I}
\pnode(8,4){J}
\pnode(10,3){K}
\pnode(10,2){L}
\pnode(0,6){aa}
\pnode(10,6.9){bb}

\psline[arrows=->,linecolor=black,linewidth=0.5pt]
  (0,0)(10,0)%
\psline[arrows=->,linecolor=black,linewidth=0.5pt]
  (0,0)(0,8)%
  
  \ncline{A}{C}\nbput[npos=.5]{$\gamma_1$}

 \ncline[linestyle=dashed]{A}{D}
\ncline[linestyle=dashed]{B}{H} 
 \ncline{C}{G}  \nbput[npos=.5]{$\gamma_2$}
\ncline[linestyle=dashed]{F}{K} 
\ncline{G}{J}\nbput[npos=.5]{$\gamma_3$}
\ncline[linestyle=dashed]{I}{L} 
\ncline{J}{L}\nbput[npos=.5]{$\gamma_4$}

\ncline[linestyle=dashed]{aa}{bb}\nbput[npos=.8]{$\gamma_5$}

\pnode(2,0){M}
\pnode(4,0){N}
\pnode(8,0){O}
\pnode(10,0){P}

\ncline[linestyle=dotted]{C}{M}
\ncline[linestyle=dotted]{G}{N}
\ncline[linestyle=dotted]{J}{O}

\ncline[linewidth=0pt]{A}{M}\naput[npos=.5]{$\action=1$}
\ncline[linewidth=0pt]{M}{N}\naput[npos=.5]{$\action=2$}
\ncline[linewidth=0pt]{N}{O}\naput[npos=.5]{$\action=1$}
\ncline[linewidth=0pt]{O}{P}\naput[npos=.5]{$\action=2$}


\end{pspicture*}}

\caption{Example of piecewise linear concave value function $\valueb_k(\belief)$ of a POMDP with a 2-state
underlying Markov chain ($\statedim = 2$). 
Here $\valueb_k(\belief) = \min \{\valvec_1^\p \belief, \valvec_2^\p \belief,
\valvec_3^\p \belief, \valvec_4^\p \belief \} $ is depicted
by solid lines. The figure also shows that the belief space can be partitioned into 4 regions.
Each region where line segment $\valvec_l^\p \belief$ is active (i.e., is equal to the solid line) corresponds to a single action,  $\action=1$ or $\action=2$. Note that $\gamma_5$ is never active.}\label{fig:pwlc}
\end{figure}

\subsection{Proof of Theorem \ref{thm:pwlc}}
The proof of Theorem \ref{thm:pwlc} is important since it gives an explicit construction
of the value function. Exact algorithms for solving POMDPs that will be described in \secn  \ref{sec:pomdpexactalgorithms}
are based on this construction.

 Theorem  \ref{thm:pwlc}  is proved by backward  induction for $k=\finaltime,\ldots, 0$.
Obviously, $ \valueb_{\finaltime}(\belief) = \terminalcost^\p \belief$ is linear  in $\belief$. 
Next assume  $ \valueb_{k+1}(\belief) $ is piecewise linear and concave in  $\belief$: so  
$\valueb_{k+1}(\belief) =  \min_{\bvalvec \in \Valvec_{k+1}}  \bvalvec^\p \belief $. Substituting this in (\ref{eq:bdppomdp}) yields
\begin{align}  \valueb_k(\belief) &= 
 \min_{\action \in \actionspace} \{ \cost^\p_{\action} \belief + \sum_{\obs \in \obspace}   \min_{\bvalvec \in \Valvec_{k+1}}
 \frac{  \bvalvec^\p
 \oprob_\obs(\action) \tp^\p(\action) \belief }{\cancel{ \filterd(\belief,\action,\obs)}} \cancel{\filterd(\belief,\action,\obs)}  \}  \nn
\\
&= \min_{\action \in \actionspace}  \biggl\{ \sum_{\obs \in \obspace}  \min_{\bvalvec \in \Valvec_{k+1}} \bigl\{
\left[\frac{\cost_{\action}}{\obsdim} +  \tp(\action)  \oprob_\obs(\action) \bvalvec \right]^\p \belief \bigr\}\biggr\}.
\label{eq:valvecstep}
\end{align}
The right hand side is the minimum (over $\action)$ of the sum (over $\obs$) of piecewise linear concave functions. Both these operations
preserve the piecewise linear concave property.
This implies   $\valueb_k(\belief) $ is piecewise linear and concave of the form
\beq  \valueb_{k}(\belief) = \min_{\valvec \in \Valvec_{k}}  \valvec^\p \belief , \quad \text{ where  }\;
\Valvec_k = \cup_{\action \in \actionspace} \oplus_{\obs \in \obspace}
 \left\{ \frac{ \cost_{\action} }{\obsdim} + \tp(u) \oprob_\obs(u) \,\bvalvec \mid \bvalvec \in \Valvec_{k+1}
\right\} .  \label{eq:valvecconstruct}
\eeq
Here $\oplus$ denotes the cross-sum operator: given two sets of vectors $A$ and $B$, $A\oplus B$ consists of all pairwise
additions of vectors from these two sets. Recall $\actionspace=\{1,2,\ldots,\actiondim\}$ and $\obspace = \{1,2,\ldots,\obsdim\}$ are 
finite sets. A more detailed explanation of going from (\ref{eq:valvecstep}) to (\ref{eq:valvecconstruct})
is given in (\ref{eq:incprune}),  (\ref{eq:valvecdetail}).

\section{Algorithms for Finite Horizon POMDPs with Finite Observation Space} \label{sec:pomdpexactalgorithms}
This section discusses algorithms for solving\footnote{By ``solving" we mean  solving Bellman's dynamic programming equation (\ref{eq:bdppomdp}) for the optimal policy $\optpolicy_k(\belief)$, $k=0,\ldots, \finaltime-1$. Once the optimal policy is obtained, then the real time controller is implemented according to Algorithm
\ref{alg:pomdp}.}
a finite horizon POMDP when the observation set $\obspace$ is finite. These algorithms
exploit the  finite dimensional characterization of the value function and optimal policy  given in Theorem \ref{thm:pwlc}.

Consider the POMDP model (\ref{eq:pomdpmodel}) with finite action space $\actionspace = \{1,2,\ldots,\actiondim\}$ and finite observation
set $\obspace = \{1,2,\ldots,\obsdim\}$.
Given the finite dimensional characterization  in Theorem \ref{thm:pwlc}, the next step is to 
compute the set of  gradients $\Valvec_k$ that determine the piecewise linear segments of the value function 
$\valueb_k(\belief)$ at each time $k$. 
Unfortunately, the number of piecewise linear segments can increase exponentially with the action space dimension
$\actiondim$ and double exponentially with time $k$.  This is  seen from the fact that
given the set of vectors $\Valvec_{k+1}$ that characterizes the value function at time $k+1$, a single step of the dynamic programming recursion yields that the set of
all vectors at time $k$ are $\actiondim | \Valvec_{k+1}|^\obsdim$.  (Of these it is possible that many vectors are never active
such as $\valvec_5$ in Figure  \ref{fig:pwlc}.)
 Therefore, exact computation
of the optimal policy is only computationally tractable for 
 small  state dimension $\statedim$, small action  space dimension $\actiondim$ and small observation
space dimension $\obsdim$.
Computational complexity theory gives worst case bounds for solving a problem.
 It is shown  in \cite{PT87} that solving a POMDP is a PSPACE complete problem. 
 \cite{Lit96} gives  examples of POMDPs that exhibit 
this worst case behavior.

\subsection{Exact Algorithms: Incremental Pruning, Monahan and Witness}

Exact\footnote{Exact here means that there is no approximation involved in the dynamic programming algorithm. 
However, the  algorithm is still subject to numerical round-off and finite precision effects.} algorithms for  solving finite horizon POMDPs are based on the finite dimensional characterization of the value function provided by Theorem
\ref{thm:pwlc}. The first exact algorithm for solving finite horizon POMDPs was proposed by Sondik \cite{Son71}; see \cite{Lov91,Cas98,Cas98b,Lit09} for several
algorithms.
 Bellman's dynamic programming recursion (\ref{eq:bdppomdp}) can be expressed as the following three steps:
\begin{align}
Q_k(\belief,\action,\obs) &= \frac{ \cost^\p_{\action} \belief}{\obsdim} + \valueb_{k+1}\( \filter(\belief,\obs,\action) \) \filterd(\belief,\obs,\action) \nonumber \\
Q_k(\belief,\action) &= \sum_{\obs \in \obspace} Q_k(\belief,\action,\obs) \nonumber\\
\valueb_k(\belief) &= \min_\action Q_k(\belief,\action). \label{eq:incprune}
\end{align}
Based on the above three steps,
  the set of vectors
$\Valvec_k$ that form the piecewise linear  value function in Theorem \ref{thm:pwlc}, can be constructed as
\begin{align}
\Valvec_{k}(\action,\obs) &=  \left\{ \frac{ \cost_{\action} }{\obsdim} + \tp(u) \oprob_\obs(u) \,\valvec \mid \valvec \in \Valvec^{(k+1)}
\right\}\nn \\
\Valvec_k(\action) &= \oplus_{\obs} \Valvec_{k}(\action,\obs)\nn  \\
\Valvec_k &= \cup_{\action \in \actionspace} \Valvec_k(\action).  \label{eq:valvecdetail}
\end{align}
Here $\oplus$ denotes the cross-sum operator: given two sets of vectors $A$ and $B$, $A\oplus B$ consists of all pairwise
additions of vectors from these two sets.

In general, the set  $\Valvec_k $ constructed according to (\ref{eq:valvecdetail}) may contain superfluous vectors (we call them ``inactive vectors" below)  that never arise in the value function
$\valueb_k(\belief) = \min_{\valvec_l \in \Valvec_k}  \valvec_l^\p \belief$. 
The algorithms  listed below seek to eliminate such useless vectors by pruning $\Valvec_k$
 to maintain a parsimonious set of vectors.

{\bf Incremental Pruning Algorithm}:   \index{POMDP algorithm! Incremental Pruning}
We start with  the incremental pruning algorithm described in Algorithm \ref{alg:incprune}. The code is freely downloadable from
\cite{POMDP}.

\begin{algorithm}
\caption{Incremental Pruning Algorithm for solving POMDP} \label{alg:incprune}
Given set $\Valvec_{k+1}$ generate $\Valvec_k$ as follows:\\
Initialize $\Valvec_k(\action,\obs) , \Valvec_k(\action), \Valvec_k$ as   empty sets
\\
For each $\action \in \actionspace$
\\
\mbox{} \hspace{0.3cm} For each $\obs \in \obspace$
 \begin{align*}
&  \Valvec_{k}(\action,\obs) \leftarrow \text{prune}\(\left\{ \frac{ \cost_{\action} }{\obsdim} + \tp(\action) \oprob_\obs(\action) \valvec \mid \valvec \in \Valvec^{(k+1)}
\right\}  \)\\
& \Valvec_k(\action) \leftarrow \text{prune}\(   \Valvec_k(\action) \oplus \Valvec_{k}(\action,\obs) \) 
\end{align*}
\mbox{} \hspace{0.3cm} $ \Valvec_k \leftarrow  \text{prune} \(  \Valvec_k  \cup \Valvec_k(\action) \) $
\end{algorithm}

Let us  explain the ``prune'' function in Algorithm \ref{alg:incprune}. Recall the piecewise linear concave characterization of the value function $\valueb_k(\belief)= \min_{\valvec \in \Valvec_k} \valvec^\p \belief$ 
with set of vectors $\Valvec_k$.
Suppose there is a vector $\valvec \in \Valvec_k$ such that for all $\belief \in \Belief$, it holds that $\valvec^\p \belief \geq \bar{\valvec}^\p \belief$ for all vectors $\bar{\valvec} \in
\Valvec_k - \{\valvec\}$. Then $\valvec$ dominates every other vector in $\Valvec_k$ and is never active. 
For example, in Figure  \ref{fig:pwlc}, $\valvec_5$ is never active.
The prune function in Algorithm \ref{alg:incprune} eliminates such inactive vectors $\valvec$ and so reduces the computational
cost of the algorithm. 

Given a set of vectors $\Valvec$, how can an inactive vector be identified and therefore pruned (eliminated)?
The  following linear programming
dominance test can be used to identify inactive vectors:    \index{linear programming! incremental pruning algorithm}
\begin{align}
 & \min\; x \\
 \text{subject to:} \;\; &  (\valvec - \bar{\valvec} )^\p \belief  \geq x , \quad  \forall \bar{\valvec} \in \Valvec - \{\valvec\}
 \nonumber
 \\
& \belief(i) \geq 0, i\in\statespace, \quad \one^\p \belief = 1 ,\quad 
\text{i.e. }  \belief  \in \Belief. \nonumber
\end{align}
Clearly, if the above linear program yields a solution $x \geq 0$, then
$\valvec$ dominates all other vectors in $\Valvec - \{\valvec\}$. Then vector $\valvec$ is inactive and can be
eliminated from~$\Valvec$.
In the worst case, it is possible that all vectors are active and none can be pruned.

{\bf Monahan's Algorithm}:
 \index{POMDP algorithm! Monahan's algorithm}
Mohahan \cite{Mon82} proposed an algorithm
that  is identical to Algorithm \ref{alg:incprune} except that the prune steps
in computing  $ \Valvec_{k}(\action,\obs)$ and $ \Valvec_{k}(\action)$ are omitted. So $\Valvec_k(\action)$ comprises of
$\actiondim |\Valvec_{k+1}|^\obsdim$ vectors and these are then pruned according to the last step of Algorithm \ref{alg:incprune}.

{\bf Witness Algorithm}:   \index{POMDP algorithm! Witness algorithm}
The Witness algorithm \cite{CKL94}, constructs $\Valvec_k(\action)$ associated with $Q_k(\belief,\action)$ (\ref{eq:incprune})
 in polynomial time with
respect to $\statedim$, $\actiondim$, $\obsdim$ and $|\Valvec_{k+1}|$.
\cite{CLZ97} shows that the incremental pruning Algorithm  \ref{alg:incprune}  has the same computational cost as the Witness algorithm and can outperform it by a constant factor.


\subsection{Lovejoy's Suboptimal Algorithm} \label{sec:lovejoy}
 \index{POMDP algorithm! Lovejoy's algorithm}
 Computing  the value function and therefore optimal policy of a POMDP via the  exact algorithms given above
 is intractable apart from
small toy examples.
Lovejoy \cite{Lov91b} proposed an ingenious suboptimal algorithm that computes upper and lower bounds to the value function
of a POMDP.
The intuition behind this algorithm is depicted in Figure  \ref{fig:lovejoy} and is as follows: Let $\bar{\valueb}_k$ and $\underline{\valueb}_k$,
respectively, denote upper and lower bounds to $\valueb_k$.
It is obvious that
by considering only a subset of the piecewise linear
segments in $\Valvec_k$
and discarding the other segments, one gets an upper bound  $\bar{\valueb}_k$.
That is, for any  $\bar{\Valvec}_k \subset \Valvec_k$,
$$\bar{\valueb}_k(\belief) = \min_{\valvec_l \in \bar{\Valvec}_k}  \valvec_l^\p \belief \geq   \min_{\valvec_l \in \Valvec_k}  \valvec_l^\p \belief = \valueb_k(\belief). $$
In Figure  \ref{fig:lovejoy},  $\valueb_k$ is characterized by line segments in
$ \Valvec_k = \{\gamma_1,\gamma_2,\gamma_3,\gamma_4\}$ and 
the upper bound $\bar{\valueb}_k$ is constructed from line segments in $\bar{\Valvec}_k = \{\gamma_2,\gamma_3\}$, i.e.  discarding segments  $\gamma_1$ and $\gamma_4$.
This upper bound is
displayed in dashed lines.

\begin{figure} \centering
 \scalebox{0.7}{ \begin{pspicture*}(-1.5,-1.5)(10.5,6.5)

   \rput(5,-0.5){\psframebox*{\large $\belief(2)$}}
   
  \rput(0,-0.5){\psframebox*{\large $0$}}
   \rput(10,-0.5){\psframebox*{\large $1$}}

\rput(-0.6, 5){\psframebox*{\large $\valueb_k(\belief)$}}


\pnode(0,0){A}
\pnode(0,2){B}
\pnode(2,4){C}
\pnode(3,6){D}
\pnode(2,7){F}
\pnode(4,6){G}
\pnode(6,8){H}
\pnode(6,6){I}
\pnode(8,4){J}
\pnode(10,3){K}
\pnode(10,2){L}
\pnode(5,5.5){Z}

\psline[arrows=->,linecolor=black,linewidth=0.5pt]
  (0,0)(10,0)%
\psline[arrows=->,linecolor=black,linewidth=0.5pt]
  (0,0)(0,8)%
  
  \ncline{A}{C}\naput[npos=.4]{$\gamma_1$}
 
\ncline[linestyle=dashed]{B}{G} 
 \ncline{C}{G}  \nbput[npos=.5]{\large$\gamma_2$}
\ncline[linestyle=dashed]{G}{K} 
\ncline{G}{J}\naput[npos=.5]{\large $\gamma_3$}
\ncline{J}{L}\naput[npos=.7]{\large $\gamma_4$}

\ncline[linestyle=dotted,linewidth=1pt]{A}{Z}
\ncline[linestyle=dotted,linewidth=1pt]{Z}{L}


\end{pspicture*}}

\caption{Intuition behind Lovejoy's
 suboptimal algorithm for solving a POMDP for $\statedim=2$. The piecewise linear concave value function $\valueb_k$ is denoted by unbroken lines.
 Interpolation (dotted lines) yields a lower bound to the value function. Omitting  any of the piecewise linear segments leads
 to an upper bound (dashed lines). The main point is that the dotted and dashed lines sandwich the value function (unbroken line).} \label{fig:lovejoy}
\end{figure}
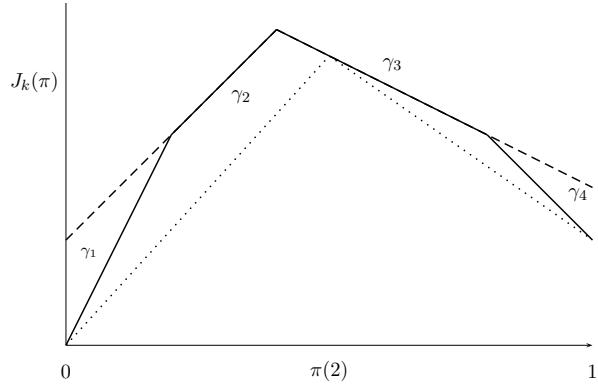

By choosing $ \bar{\Valvec}_k$ with small cardinality at each iteration $k$,
one can reduce the computational cost of computing $\bar{\valueb}_k$.
This is the basis of Lovejoy's \cite{Lov91b} lower  bound approximation.
Lovejoy's algorithm \cite{Lov91b}  operates as follows: \\
{\bf Initialize}: $\bar{\Valvec}_\finaltime =
\Valvec_\finaltime = \{\terminalcost\}$. Recall $\terminalcost$ is the terminal cost vector.
\\
{\bf Step 1}: Given a set of vectors $\Valvec_{k}$,
construct the set $\bar{\Valvec}_{k}$  by pruning 
$\Valvec_{k}$ as follows:
 Pick any $R$ belief states  $\pi_1,\pi_2,\ldots,\pi_{R}$
 in the belief  simplex $\Belief$. (Typically, one often picks  the $R$ points based on
a uniform Freudenthal triangulization of $\Belief$, see \cite{Lov91b} for details). Then set
$$ \bar{\Valvec}_{k} = \{ \arg \min_{\valvec \in \Valvec_{k}}
 \valvec^{\p} \pi_r, \quad
r=1,2,\ldots,{R}\}.$$
{\bf Step 2}: Given $\bar{\Valvec}_{k}$,
 compute the set of vectors
 $\Valvec_{k-1}$ using
a standard POMDP algorithm.
\\
{\bf Step 3}: $k \rightarrow  k -1 $ and go to Step 1.

Notice that
$\bar{\valueb}_k(\belief)  = \min_{\valvec \in \bar{\Valvec}_k} \valvec^{\p} \belief $
is  represented
completely by ${R}$ piecewise linear segments. 
Lovejoy \cite{Lov91b} shows that for all $k$,
  $\bar{\valueb}_k$ is an upper bound to the optimal value function
$ \valueb_k$,
 Thus Lovejoy's algorithm gives a suboptimal policy that yields an upper bound to the value function
 at a computational cost of no more than ${R}$ evaluations per
iteration $k$. 

So far we have discussed Lovejoy's upper bound.
Lovejoy \cite{Lov91b} also provides a constructive
procedure for computing a lower bound to the optimal value function. The intuition behind the lower bound is displayed in Figure  \ref{fig:lovejoy}
and is as follows: a linear interpolation to a concave function lies below the concave function.
Choose any $R$ belief states $\belief_1,\belief_2,\ldots,\belief_R$. Then construct
$\underline{\valueb}_k(\belief)$ depicted in dotted lines in Figure  \ref{fig:lovejoy} as  the  linear interpolation between the points
$(\belief_i, \valueb_k(\belief_i))$, $i=1,2,\ldots,R$. Clearly due to concavity of $\valueb_k$, it follows that 
$\underline{\valueb}_k(\belief) \leq{\valueb}_k(\belief)$ for all $\belief \in \Belief$.

\subsection{Point-Based Value Iteration Methods} \label{sec:pointbased}   \index{value iteration! point-based algorithms for POMDPs}
\index{POMDP algorithm! point-based value iteration}
Point-based  value iteration methods seek to compute an approximation of the value function at special points in the belief space. 
  The main idea is to compute solutions only for those belief states that have been visited by running the POMDP.
This motivates the development of approximate solution techniques that use a sampled set of belief states on which the POMDP is solved
\cite{Hau00,SBS07,KHL08}. 

As mentioned above, Lovejoy \cite{Lov91b} uses Freudenthal triangulation to form a grid on the belief space 
and then computes the approximate policy at these belief states.
 Another possibility is to take all extreme points of the belief
simplex or to use a random grid. Yet another option is to include belief states that are encountered when simulating
the POMDP. Trajectories can be generated in the belief simplex by sampling random actions and observations at each time
\cite{PGS03,SV05}.  More sophisticated schemes for belief sampling have been proposed in \cite{SBS07}.
The SARSOP approach of \cite{KHL08} performs successive approximations of the reachable belief space from following the optimal policy.

\subsection{Belief Compression POMDPs}  \index{POMDP algorithm! belief compression}  \label{sec:compress}
An interesting class of suboptimal POMDP algorithms  \cite{RGT05}  involves reducing the dimension of the belief space
$\Belief$. 
 The   $\statedim$-dimension belief states $\belief$  are projected to 
 $\bar{\statedim}$-dimension belief states $\bar{\belief}$ that live in the reduced dimension simplex $\I(\bar{\statedim})$ where
$\bar{\statedim} \ll \statedim$. The principal component analysis (PCA) algorithm  \index{principal component analysis for belief compression} is used to achieve this belief
compression as follows:
 Suppose a sequence of beliefs
$\belief_{1:\finaltime} = (\belief_1,\ldots,\belief_\finaltime)$ is generated; this is  a $\statedim \times \finaltime$ matrix.
Then perform a singular value decomposition $\belief_{1:\finaltime} = U D V^\p$ and  choose the largest $\bar{\statedim}$ singular values.
Then the original belief $\belief $ and compressed belief $\bbelief$ are related by
$\belief = U_{\bar{\statedim}} \bbelief$  or $\bbelief =  U^\p _{\bar{\statedim}}  \belief$ where 
$U_{\bar{\statedim}} $ and
  $ D_{\bar{\statedim}} $ denote the truncated matrices corresponding to the  largest $\bar{\statedim}$ singular values.
  (PCA is suited to dimensionality reduction when the data lies near a linear manifold. However, POMDP belief manifolds are rarely linear and so
 \cite{RGT05} proposes  an exponential family PCA.) The next step is to quantize the low dimensional belief state space $\I(\bar{\statedim})$ into a finite state space $\mathcal{\bar{S}}= \{\bq_1,\ldots,\bq_L\}$. The corresponding full dimensional beliefs are
 $\mathcal{S} = \{U_{\bar{\statedim}} \bq_1, \ldots,U_{\bar{\statedim}}\bq_L\}$.
The reduced dimension dynamic programming recursion then is identical to that of a finite state MDP. It reads
\begin{align*} \valuef_{k+1}(\bq_i) &= \min_{\action \in \actionspace} \bigl\{ \tilde{\cost}(\bq_i,\action) +  \sum_{j=1}^L \tilde{\filter}(\bq_i,\bq_j,\action) \valuef_{k}(\bq_j) \bigr\} \\
\text{ where } \tilde{\cost}(\bq_i,\action) &= c_u^\p q_i, \;\; \text{ and }
 \tilde{\filter}(\bq_i,\bq_j,\action) = U_{\bar{\statedim}} \frac{\oprob_\obs(\action) \tp^\p(\action) q_i}{\one^\p \oprob_y(\action) \tp^\p(\action) q_i} 
 \end{align*}
rounded off to the nearest belief in $\mathcal{\bar{S}} $.

Chapter \ref{chp:filterstructure} presents  algorithms for approximating the belief with provable bounds. In \cite{YKI04,YIK09}, we present stochastic gradient
algorithms for estimating the underlying state of an HMM directly. Either of these algorithms can be used in the above reduced dimensional dynamic programming
algorithm instead of PCA type compression.

\section{Discounted Infinite Horizon POMDPs}\label{chp:discpomdp}
So far we have considered finite horizon POMDPs.
This section
consider   infinite horizon discounted cost POMDPs.  The discounted POMDP model is a 
7-tuple   $(\statespace, \actionspace, \obspace, \tp(\action),  \oprob(\action), \cost(\action),\discount) $
where $ \tp(\action)$, $ \oprob(\action)$ and $ \cost$ are no longer explicit functions of time and $\discount  \in [0,1)$ in an economic
discount factor.
Also, compared to  (\ref{eq:pomdpmodel}).
there is no terminal cost $\cost_\finaltime$.

%
Define a stationary policy sequence as  $\bpolicy = (\policy,\policy,\policy, \cdots)$ where $\policy$ is not
an explicit function of time $k$.  We will
use $\policy$ instead of $\bpolicy$ to simply notation.
For stationary policy  $\policy: \Belief \rightarrow \actionspace$,
 initial belief  $\belief_0\in \Belief$,  discount factor $\discount \in [0,1)$, define the  objective function as the discounted expected  cost:
 $$ J_{\policy}(\belief_0)  =  \Ep\left\{\sum_{\time=0}^{\infty} \discount ^{\time} \cost(\state_\time,\action_\time)\right\}, \quad \text{ where } u_k  = \mu(\belief_k)$$
 As in \secn \ref{sec:beliefstate} we can re-express this objective in terms of the belief state as
\begin{align}\label{eq:discountedcostaa}
J_{\policy}(\belief_0) = \Ep\left\{\sum_{\time=0}^{\infty} \discount ^{\time} \cost_{\policy(\belief_\time)}^\p\belief_\time\right\} ,
\end{align}
where  $\cost_\action = [\cost(1,\action),\ldots,\cost(\statedim,\action)]^\p$, $\action \in \actionspace$ is the cost vector for each action, and the belief state evolves according to the HMM filter 
$\belief_{k} = \filter(\belief_{k-1},\obs_k,\action_{k-1})$ (\ref{eq:hmmc}). 

The aim is to compute the optimal  stationary policy $\optpolicy:\Belief \rightarrow \actionspace$ such that
$J_{\optpolicy}(\belief_0) \leq J_{\policy}(\belief_0)$ for all $\belief_0 \in \Belief$.
From the   dynamic programming recursion, we have for any finite horizon $\finaltime$ that
$$
\valueb_k(\belief) =  \min_{\action \in \actionspace} \{ \discount^k \cost^\p_{\action} \belief + \sum_{\obs \in \obspace} \valueb_{k+1}\( \filter(\belief,\obs,\action) \) \filterd(\belief,\obs,\action) \} $$
initialized by $\valueb_\finaltime(\belief) = 0$.
For discounted cost  problems, it is more convenient to work with a forward iteration of indices. Accordingly,
define the following value function $\valuef_n(\belief)$:
$$ \valuef_n(\belief) =  \discount^{n - \finaltime}  \, \valueb_{\finaltime - n}(\belief)  , \quad 0 \leq n \leq \finaltime, \; \belief  \in \Belief.
$$
Then it is easily seen that  $\valuef_n(\belief)$ satisfies the dynamic programming equation
\beq  \label{eq:fhpomdpvaluef}
 \optvalue_n(\belief) =  ~\cost_\action^\prime\belief + \discount\sum_{\obs \in \obsdim} \optvalue_{n-1}\left(\filter\left(\belief,\obs,\action\right)\right)\filternorm \left(\belief,\obs,\action\right), \quad \valuef_0(\belief) = 0.
\eeq

\subsection{Bellman's Equation for discounted  infinite horizon POMDP}
The main result for infinite horizon discounted cost POMDPs  is as follows:  \index{Bellman's equation! discounted cost POMDP}
\begin{theorem}  \label{thm:discountpomdp}
Consider an infinite horizon discounted cost POMDP with discount factor $\discount \in [0,1)$. Then 
\begin{compactenum}
\item The optimal expected  cumulative cost  is achieved by a stationary deterministic Markovian policy $\optpolicy$.
\item  The optimal  policy $\optpolicy(\belief)$ and value function $\optvalue(\belief)$ satisfy  Bellman's dynamic programming equation
\begin{align}
 \optpolicy(\belief) & =  \underset{\action \in \actionspace}\argmin ~\valueaction(\belief,\action),  \quad J_{\optpolicy}(\belief_0) = \optvalue(\belief_0)  \label{eq:bellmaninfpomdp}\\
\optvalue(\belief)  &= \underset{\action \in \actionspace}\min ~\valueaction(\belief,\action), \quad
  \valueaction(\belief,\action) =  ~\cost_\action^\prime\belief + \discount\sum_{\obs \in \obsdim} \optvalue\left(\filter\left(\belief,\obs,\action\right)\right)\filternorm \left(\belief,\obs,\action\right).\nonumber
\end{align}
where $\filter(\belief,\obs,\action)$ and $\filterd(\belief,\obs,\action)$ are the HMM filter and normalization
(\ref{eq:hmmc}).\\
The expected  cumulative cost incurred by the optimal policy is $J_{\optpolicy}(\belief) = \valuef(\belief)$.
\item The value function $V(\belief)$ is continuous and concave in $\belief \in \Belief$. \index{value function for POMDP! concavity! discounted cost}
\end{compactenum}
\end{theorem}


\subsection{Value  Iteration Algorithm for discounted cost POMDPs} \label{sec:vipomdpfirst}
 Let $\iter=1,2,\ldots,\iterfinal$ denote iteration number. The   value iteration algorithm for a discounted cost POMDP is 
 a successive approximation algorithm for computing
 the value function $V(\belief)$ of 
Bellman's equation  (\ref{eq:bellmaninfpomdp}) and proceeds as follows: Initialize $\optvalue_0(\belief) = 0$. For iterations $\iter =1,2,\ldots,\finaltime$,
evaluate  \index{value iteration! discounted cost POMDP}
\beq
\begin{split}
\optvalue_\iter(\belief)  &= \underset{\action \in \actionspace}\min ~\valueaction_\iter(\belief,\action), \quad \optpolicy_\iter(\belief) = \argmin_{\action \in \actionspace} \valueaction_\iter(\belief,\action), \\
  \valueaction_\iter(\belief,\action) &=  \cost_\action^\prime\belief + \discount\sum_{\obs \in \obsdim} \optvalue_{\iter-1}\left(\filter\left(\belief,\obs,\action\right)\right)\filternorm \left(\belief,\obs,\action\right). \end{split}  \label{eq:vipomdp}
\eeq
Finally, the stationary policy $\optpolicy_\finaltime$ is used at each time instant $k$ in the real time controller of Algorithm \ref{alg:pomdp}.
The obvious advantage of the stationary policy  is that only 
the policy 
$ \optpolicy_\iterfinal(\belief)$  needs to be  stored for real time implementation of the controller in Algorithm \ref{alg:pomdp}.

{\em Summary}: The POMDP value iteration algorithm (\ref{eq:vipomdp}) is identical to the finite horizon dynamic programming recursion (\ref{eq:bdppomdp}).
 So at each iteration $\iter$, $\optvalue_\iter(\belief)$ is piecewise linear and concave in $\belief$ (by Theorem \ref{thm:pomdpdpeq}) and can be computed using any of the POMDP algorithms discussed
in \secn \ref{sec:pomdpexactalgorithms}.  
The number of piecewise linear segments that characterize $\optvalue_\iter(\belief)$ can grow exponentially with  iteration $n$. Therefore, except for small state, action and observation spaces, suboptimal
algorithms (such as those discussed in  \secn \ref{sec:pomdpexactalgorithms}) need to be used.

How are the number of iterations   $\finaltime$ chosen in the value iteration algorithm (\ref{eq:vipomdp})?
The value iteration algorithm (\ref{eq:vipomdp})  generates a sequence of value functions
$\{\optvalue_\iter\} $  that will converge uniformly (sup-norm metric) as $\iterfinal \rightarrow \infty$ to $\optvalue(\pi) $, the optimal value
 function of Bellman's equation.
  The number of iterations $\iterfinal$ in (\ref{eq:vipomdp})  can be chosen as follows: Let $\epsilon>0$ denote a specified tolerance.

 \begin{theorem} \label{thm:vipomdpgeom} Consider the value iteration algorithm with discount factor $\discount$ and $\iterfinal$ iterations. Then:
 \begin{compactenum}
\item
 $\sup_\belief | \optvalue_\iterfinal (\belief)- \optvalue_{\iterfinal-1}(\belief)| \leq \epsilon$ implies that
 $\sup_\belief | \optvalue_\iterfinal (\belief)- \optvalue(\belief)| \leq  \frac{\epsilon \discount}{1-\discount}$.
\item $|\optvalue_\iterfinal(\belief) - \optvalue(\belief)| \leq \frac{\discount^{\iterfinal+1}}{1-\discount} \max_{\state,\action} |\cost(\state,\action)|$.
\end{compactenum}
\end{theorem}
Actually, similar to the value iteration algorithm, one  can  evaluate the
expected discounted  cumulative cost of an arbitrary stationary policy (not necessarily the optimal policy) as follows:

\begin{corollary}[Policy Evaluation] For any stationary policy $\policy$, the associated expected
 discounted  cumulative cost $J_\policy(\belief)$ defined in (\ref{eq:discountedcostaa}) for a POMDP satisfies
\beq J_\policy(\belief)  = c^\p_{\policy(\belief)} \belief + \discount \sum_{\obs \in \obspace} J_\policy\bigl( \filter(\belief,\obs,\policy(\belief)) \bigr) \,
\filterd(\belief,\obs,\policy(\belief)) .  \label{eq:policyeval} \eeq
Similar to the value iteration algorithm (\ref{eq:vipomdp}), $ J_\policy(\belief)$ can be obtained
as  $J_\policy(\belief)= \lim_{\iter \rightarrow \infty} V_{\policy,\iter}(\belief)$. Here $V_{\policy,\iter}(\belief)$, $n=1,2,\ldots$  satisfies the recursion
$$ V_{\policy,\iter}(\belief) = c^\p_{\policy(\belief)} \belief + \discount \sum_{\obs \in \obspace} V_{\policy,\iter-1}\bigl( \filter(\belief,\obs,\policy(\belief)) \bigr) \,
\filterd(\belief,\obs,\policy(\belief)) , \quad V_{\policy,0}(\belief) = 0. $$
Also Theorem \ref{thm:vipomdpgeom} holds for    $J_\policy(\belief) $.   \qed  \label{pomdppoleval}\end{corollary}

The proof is omitted since
it is similar to  the corresponding
theorems for the optimal policy given above.
Note that  (\ref{eq:policyeval}) can be written as
\beq \label{eq:policyeval2}
J_\mu(\belief) = \gamma^\p_{\policy(\belief)} \belief, \quad \text{ where } \;
\gamma_{\policy(\belief)} = \cost_{u} + \discount \sum_{\obs \in \obspace} 
\gamma_{\policy(\filter(\belief,\obs,u))} \tp(u) \oprob_\obs(u), \eeq
where  $u = \policy(\belief)$ on the right hand side. We will use this representation below.

\section{Example: Optimal Search for a Markovian Moving  Target} \label{chp:optimal search}
\index{optimal search|(}   \index{search|see {optimal search}}  \index{search|seealso{random search}} 
Optimal search of a moving target is a useful illustrative example of a POMDP.
From an abstract point of view, many resource allocation problems involving controlled sensing and communication
with noisy information can be formulated
as an optimal search problem; see for example \cite{JK06} where opportunistic transmission over a fading channel is formulated
as an optimal search problem.

A target moves among $\statedim$ cells according to a  Markov chain with transition matrix $\tp$. At  time instants $k \in \{0,1,2,\ldots\}$, the
searcher must choose an action from the action space $\actionspace$. 
The set $\actionspace$ contains actions that search a particular cell or a group
of cells simultaneously. Assuming action $\action$ is selected by the searcher at time $k$, 
it is executed with probability $1 - \blockprob(\action)$. If the action cannot be executed,
the searcher is said to be \emph{blocked} for time $k$. 
This blocking event with probability $\blockprob(\action)$  models the scenario when the search sensors are a shared resource and
not enough resources are available to carry out the search at time $k$.
If the searcher is not blocked and action $\action$ searches the cell that  the target is in, the target is detected 
with probability  $1 - \overlook(\action)$; failure to detect the target
when it is in the cell searched is called an \emph{overlook}. 
So the overlook probability in cell $\action$ is $\overlook(\action)$.

If the decision maker knows $\tp,\overlook,\blockprob$, in which order should it search the cells to find the moving target
with minimum expected effort?

\subsection{Formulation of Finite-Horizon Search Problem} \label{sec:searchformulation}
It is assumed here that the searcher has  a total of $\finaltime$ attempts to find the moving target.
Given a target that moves between $X$ cells, 
let $\augss=\{1,2,\ldots,\statedim,\terminal\}$ 
denote the  augmented state space. Here $\terminal$ corresponds to a fictitious \emph{terminal state} 
 that is added as a means of terminating search
if the target is detected prior to exhausting the $\finaltime$ search horizon.

Denote the observation space as $\obspace = \{F, \bar{F}, \blocked\}$.
Here $F$ denotes ``target found'',  $\bar{F}$ denotes ``target not found'' and $\blocked$ denotes ``search blocked''.

An optimal search problem consists of the following ingredients:

\noindent 
\textbf{1. Markov State Dynamics}: The location of the target is modelled as a finite
state Markov chain. The target moves amongst the $\statedim$ cells
according to transition probability matrix  $\tp$.
Let $\state_k \in \augss$ denote the state (location) of the target at the start
of search epoch $k$ where $k=0,1,2,\ldots,\finaltime-1$. 
To model termination of the search process after the target is found, we model the target dynamics 
by
 the {\em observation} dependent transition probability matrices $\tp^\obs$, $\obs \in \obspace=\{F, \bar{F}, \blocked\}$:
\begin{eqnarray}
\tp^F = \left [ 
\begin{array}{cccc}
0 & 0 & \cdots & 1 \\
0 & 0 & \cdots & 1 \\
\vdots & \vdots & \ddots & \vdots \\
0 & 0 & \cdots & 1 
\end{array} \right ], \quad
\tp^{\bar{F}} = \tp^\blocked =  \left [
\begin{array}{cc}
\tp & \zero \\
\zero^\p & 1 
\end{array} \right ]. \label{eq:schTranMatrices}
\end{eqnarray}
That is, $\prob(\state_{k+1} = j|\state_k=i,\obs_k=y) = \tp_{ij}^{y}$.
As can be seen from the transition matrices above, the terminal state $\terminal$ is designed to be absorbing.
 A
transition to  $\terminal$ occurs only when the target is detected.
The initial state of the target is  sampled from 
$\belief_0(i)$, $i \in \{1,\ldots,\statedim\}$.

\noindent 
\textbf{2. Action}:  At each time $k$, the decision maker chooses  action $\action_k$ 
 from the finite set of search actions $\actionspace$. 
In addition to searching the target in one of the $\statedim$ cells, 
$\actionspace$ may contain actions that specify the simultaneous search in a number of cells.

\noindent 
\textbf{3. Observation}: Let $\obs_k \in \obspace =\{F, \bar{F}, \blocked\}$ denote the 
observation received at time  $k$
upon choosing action $\action_k$. Here
\begin{eqnarray}
\obs_k &=& \left \{ 
\begin{array}{lll}
F & \text{target is found}, \\
\bar{F} & \text{target is not found}, \\
\blocked & \text{search action is blocked due to insufficient available resources}.
\end{array} \right. \nonumber 
\end{eqnarray}
Define the \emph{blocking} probabilities $\blockprob(\action)$ and 
\emph{overlook} probabilities $\overlook(\action)$, $\action \in \actionspace$ as: \index{optimal search! blocking probability}
 \index{optimal search! overlook probability}
\begin{align}
\blockprob(\action) &= \prob(\text{insufficient resources to perform action $\action$ at 
epoch $k$}), \nonumber
\\
\overlook(\action) &= \prob(\text{target not found$|$target is in the cell $\action$}). \label{eq:schOvlkPr} \end{align}
Then, the  observation $\obs_k$ received is characterized probabilistically as follows.
For all $\action \in \actionspace$ and
$j=1,\ldots,\statedim$,
\begin{eqnarray}
\prob(\obs_k = F|\state_k=j,\action_k=\action) & = & \left \{ 
\begin{array}{ll}
(1 -  \blockprob(\action))(1-\overlook(\action))
& \text{if action $\action$ searches cell $j$}, \\
0 & \text{otherwise}, 
\end{array} \right. \nonumber \\
\prob(\obs_k = \bar{F}|\state_k=j,\action_k=\action)& = & \left \{ 
\begin{array}{ll}
1 - \blockprob(\action) &  \text{if action $\action$ does not search cell $j$}, \\
\overlook(\action)(1-\blockprob(\action)) & \text{otherwise}, 
\end{array} \right.  \nonumber \\
\prob(\obs_k = \blocked|\state_k= j,\action_k=\action) & = & \blockprob(\action). \label{eq:obsProb} \end{eqnarray}
Finally, for the fictitious terminal state $\terminal$, the observation $F$ is always received regardless of the action taken, so that
$$ \prob(\obs_k = F|\state_k=\terminal,\action_k=\action)  =  1. $$

\noindent 
\textbf{4. Cost}: Let $\cost(\state_k,\action_k)$ denote the instantaneous cost
for choosing action $\action_k$ when the  target's state is $\state_k$.
Three types of  instantaneous costs are of interest.
\\ \indent {\em 1. Maximize Probability of Detection} \cite{Pol70,Eag84}. The instantaneous reward is the probability
of detecting the target (obtaining observation $F$) for the current state and action. This constitutes a negative
cost. So
\begin{align}
\cost(\state_k=j,\action_k =\action) & =  - \prob(\obs_k = F | 
\state_k=j,\action_k = \action) 
\quad \text{for }j=1,...,\statedim, \nonumber \\
\cost(\state_k=\terminal,\action_k=\action) & =  0.   \label{eq:rewardMaxProb}
\end{align}
\indent {\em 2. Minimize Search Delay} \cite{Pol70}. An instantaneous cost
of 1 unit is accrued for every action taken until the target is found, i.e., until the target reaches  the terminal state $\terminal$:
\begin{eqnarray}
\cost(\state_k=j,\action_k=\action) & = &  1 \quad \text{for }j=1,...,\statedim, \nonumber \\
\cost(\state_k=\terminal,\action_k=\action) & = & 0. \label{eq:rewardMinSearches}
\end{eqnarray}
\indent {\em 3. Minimize Search Cost}. The instantaneous cost depends only on the action taken.
  Let $\cost(\action)$ denote the 
positive cost incurred for action $\action$, then
\begin{eqnarray}
\cost(\state_k=j,\action_k=\action) & = & \cost(\action)
 \quad \text{for }j=1,...,\statedim, \nonumber \\
\cost(\state_k=\terminal,\action_k=\action) & = & 0. \label{eq:rewardMinCost}
\end{eqnarray}

\noindent
\textbf{5. Performance criterion}: 
Let $\schHst_k$ denote the information (history) available 
at the start of search epoch $k$:
\begin{equation}
\schHst_0= \{\belief_0\}, \quad
\schHst_k = \{\belief_0,\action_0, \obs_0,\ldots,\action_{k-1},\obs_{k-1}\} \quad \text{for~}k=1,\ldots,\finaltime. \label{eq:schInfoVec}
\end{equation}
$\schHst_k$ contains the initial probability distribution $\belief_0$, the actions taken 
and observations received prior to search time $k$.
A
\emph{search policy} $\bpolicy $ is 
a sequence of \emph{decision rules} $\bpolicy = \{ \policy_0,\ldots,\policy_{\finaltime-1} \}$ 
where each decision rule $\policy_k:\schHst_k \rightarrow \actionspace$.
The performance criterion considered is
\begin{equation}
\schRwdFn_{\bpolicy}(\belief_0) =  \E_\bpolicy
\left\{ 
\sum_{k=0}^{\finaltime-1} \cost(\state_k,\policy_k(\schHst_k)) \big|  \belief_0
\right\}.
\label{eq:schExptReward}
\end{equation}
This is the expected cost accrued after  $\finaltime$ time points using
search policy
 $\policy$ when the initial distribution of the target is $\belief_0$.
The \emph{optimal search problem} is to find the policy $\bpolicy^{\ast} $
that  minimizes (\ref{eq:schExptReward}) for all initial distributions,
i.e., 
\begin{equation}
\bpolicy^{\ast} =  
{\operatornamewithlimits{argmin}_{\bpolicy \in \policySpc}}~
\schRwdFn_{\bpolicy}(\belief_0), \quad  ~ \forall ~ \belief_0 \in \Belief.
\label{eq:optPagePol} 
\end{equation}

Similar to (\ref{eq:infocosta}), we can express the objective in terms of the belief state as
\begin{align}
\schRwdFn_{\bpolicy}(\belief_0) &=  \E_\bpolicy
\left\{ 
\sum_{k=0}^{\finaltime} \cost(\state_k,\policy_k(\schHst_k)) \big|  \belief_0
\right\}, \nn  \\
&=  \E_\bpolicy
\left\{ 
\sum_{k=0}^{\finaltime} \E\{ \cost(\state_k,\policy_k(\schHst_k))| \info_k \} \big|  \belief_0\right\} =
\sum_{k=0}^{\finaltime}   \E\{ \cost^\p_{\action_k} \belief_k\} \label{eq:beliefcostsearch}
\end{align}
where belief state $\belief_k=[\belief_k(1),\ldots,\belief_k(\statedim+1)]^\p$ is
defined as $\belief_k(i) = \prob(\state_k=i | \schHst_k)$. The belief state is
 updated by the HMM predictor\footnote{The reader should note the difference
between the information pattern of a standard POMDP (\ref{eq:infopatternpomdp}), namely, $\info_k =(\belief_0,\action_0,\obs_1,\ldots, \action_{k-1},\obs_k)$ and the information pattern $\info_k $ for the search problem in (\ref{eq:schInfoVec}). In the search
problem, $\info_k$ has observations until time $k-1$, thereby requiring the HMM predictor (\ref{eq:coSchSysEqn}) to evaluate
the inner conditional expectation of the cost  in (\ref{eq:beliefcostsearch}).}  as follows: 
\begin{align}
\belief_{k+1} &=
\filter(\belief_k,\obs_k,\action_k) =
        \frac{ {\tp^{\obs_k}}^\p  \tilde{\oprob}_{\obs_k}(\action_k)  \belief_k}  {\filterd(\belief_k,\obs_k,\action_k)}
        ,\quad 
\filterd(\belief,\obs,\action) =  \one^\p {\tilde{\oprob}}_\obs(\action)   \belief\label{eq:coSchSysEqn} 
\\ 
\tilde{\oprob}_\obs(\action) &= \text{diag}(
\prob(\obs_k = \obs | \state_k =1, \action_k = \action), \ldots, 
 \prob(\obs_k = \obs | \state_k =\statedim, \action_k = \action), \nonumber \\ 
 & \hspace{5cm} \prob(\obs_k = \obs | \state_k =\terminal, \action_k = \action)).
\nonumber
 \end{align}
Recall the observation dependent transition probabilities $\tp^\obs$ are defined in (\ref{eq:schTranMatrices}).

The optimal policy $\bpolicy^*$  can be computed via the dynamic programming recursion (\ref{eq:bdppomdp}) where
$\filter$ and $\filterd$ are defined in (\ref{eq:coSchSysEqn}).

\subsection{Formulation of Optimal Search as a POMDP} \index{optimal search! POMDP}
In order to use POMDP software to solve the  search problem
(\ref{eq:optPagePol}) via dynamic programming,
 it is necessary to express the search problem as a POMDP.
The search problem in \secn \ref{sec:searchformulation}  differs from a standard  POMDP in two ways:
\newline {\em Timing of the events}: In a POMDP, the observation $\obs_{k+1}$ received
upon adopting action $\action_k$ is with regards 
to the new state of the system, $\state_{k+1}$. 
In the search problem, the observation $\obs_k$ received for action  $\action_k$ 
conveys information about the
current state of the target  (prior to transition),  $\state_k$. 
\newline {\em Transition to the new state}: In a POMDP the  
probability distribution that characterizes the system's new
state,  $\state_{k+1}$, is a function of its current state $\state_k$
and the action $\action_k$ adopted. However, in the search problem, the distribution of the
new state, $\state_{k+1}$, is a function of 
 $\state_k$ and observation $\obs_k$.

The search problem of  \secn \ref{sec:searchformulation} can  be reformulated as a POMDP  (\ref{eq:pomdpmodel}) 
as follows: Define the augmented state process 
$\pomSt_{k+1} = (\obs_k,\state_{k+1})$. 

Then consider the following POMDP with $2\statedim+1$ underlying states.

\noindent
{\bf State space}:  $\pomStSpc= \{(\bar{F},1),(\bar{F},2),\ldots,(\bar{F},\statedim),
(\blocked,1),(\blocked,2),\ldots,(\blocked,\statedim), (F,T)\}$.
\\
{\bf Action space}:  $\actionspace$ (same as search problem in  \secn \ref{sec:searchformulation})  \\ 
{\bf Observation space}: 
$\obspace = \{F,\bar{F},\blocked\}$ (same as search problem in  \secn \ref{sec:searchformulation})

\noindent
{\bf Transition probabilities}: For  each action $\action \in \actionspace$, define
\begin{align}
\pomTranMat(\action) &= 
\left[ \begin{array}{ccccc} 
\oprob_{\bar{F}}(\action)\, \tp & \oprob_{\blocked}(\action)\, \tp 
&  \oprob_{F}(\action) \one \\
\oprob_{\bar{F}}(\action)\,\tp & \oprob_{\blocked}(\action)\, \tp 
&  \oprob_{F}(\action)\,\one \\
\zero^\p &  \zero^\p &  1  \end{array} \right]
  \label{eq:searchDPrecursions-3}
 \\ \intertext{ where $\tp$ is the transition matrix  of the moving target and for $\obs \in \{F,\bar{F},\blocked\}$}
 \oprob_\obs(\action) &=  \diag\big( \prob(\obs_k = \obs|\state_k=1,\action_k=\action), \ldots,  \prob(\obs_k = \obs|\state_k=\statedim,
\action_k=\action)\big).  \label{eq:obssearch} \end{align}
Recall these are computed in terms of the blocking and overlook probabilities 
using (\ref{eq:obsProb}).

\noindent
{\bf Observation probabilities}: For each action $\action \in \actionspace$, define the observation probabilities
$R_{s \obs}(\action) = \prob(\obs_k = \obs| s_{k+1} = s, \action_k = \action)$ where
for $s = (\bar{y},x)$, 
$$ R_{\bar{y}x, \obs}(\action) =  \begin{cases} 1 & \bar{y} = y, \\
							 0 & \text{ otherwise } \end{cases} , \quad \obs, \bar{y} \in \obspace,\; x \in \{1,2,\ldots,\statedim\}.$$

\noindent
{\bf Costs}: For  each action $\action \in \actionspace$, define the POMDP instantaneous costs as
\begin{equation}
\pomRwd(i,\action) = \left \{  
\begin{array}{ll}
\cost(i,\action) & \text{for}~ i=1,\ldots,\statedim, \\
\cost(i-\statedim,\action) & \text{for}~ i=\statedim+ 1,\ldots,2 \statedim, \\
0 & \text{for}~ i=2 \statedim+ 1. \end{array} \right.
\label{eq:searchDPrecursions-1}
\end{equation}
To summarize, optimal search over a finite horizon is  equivalent to the finite horizon POMDP 
$(S,\actionspace, \obspace, \tp(\action), R(\action), \pomRwd(\action))$.

\section{\pwe}

The following websites are
repositories of papers and software for solving POMDPs: \\
\url{http://www.pomdp.org} \\
 \url{http://bigbird.comp.nus.edu.sg/pmwiki/farm/appl/} 

The belief state formulation in partially observed stochastic control goes back to the early 1960s; see the seminal works of
Stratonovich \cite{Str60}, Astrom \cite{Ast65} and Dynkin \cite{Dyn65}.
Sondik \cite{Son71} first showed that Bellman's equation for a finite horizon POMDP has a finite dimensional piecewise linear concave solution. This led
to the influential papers \cite{SS73,Son78}; see \cite{Mon82,Lov91,Cas98b,SPK13} for   surveys.
POMDPs have been applied in dynamic spectrum management for cognitive radio \cite{Hay05,ZTSC07}, adaptive radars
\cite{MSH06,KD09,KBGM12}.

Optimal search theory is a well studied problem \cite{Sto89,Ii90}.  Early papers in the area include
 \cite{Eag84,Pol70}. The paper
\cite{MJ95} shows that  in many cases
optimal search for a target moving between two cells has a threshold type optimal policy - this verifies a conjecture
made by Ross in \cite{Ros83}. The proof that the search problem for a moving target is a stochastic shortest path problem is
given in \cite{Pat01,SK03}. 

Radar scheduling using POMDPs has been studied in \cite{Kri02,EKN05,Kri05,KD07,KD09}.

%

\chapter{Structural Results for Markov Decision Processes} \label{chp:monotonemdp}

For finite state MDPs with large dimensional state spaces,  computing the optimal policy by solving Bellman's dynamic programming recursion or the associated linear programming problem can be prohibitively expensive.
Structural results  give sufficient conditions on the MDP model to ensure that the optimal policy $\optpolicy(\state)$ is  increasing (or decreasing)
in the state $\state$. Such policies will be called  {\em monotone policies}. To see why monotone policies are important, consider an MDP with  two actions $\actionspace=\{1,2\}$ and a large state space $\statespace = \{1,2,\ldots,\statedim\}$.
If the optimal policy  $\optpolicy(\state)$ is increasing\footnote{Throughout this book, increasing is used in the weak sense to mean increasing.} in $\state$, then it has to be a step function of the form
\beq  \optpolicy(\state) =  \begin{cases}  1  &  \state <  \state^*  \\
  							  2  & \state \geq \state^* .\end{cases}  \label{eq:sparse1}
\eeq
  Here  $\state^* \in \statespace$ is some fixed state at which the optimal policy switches from  action 1 to action 2.
 A policy of the form (\ref{eq:sparse1})  will be called a {\em threshold policy} \index{threshold policy! MDP} and $\state^*$ will be called the threshold state. Figure  \ref{fig:thresholdmdp} illustrates 
a threshold policy.

Note that $\state^*$ completely characterizes the threshold policy (\ref{eq:sparse1}). Therefore,
if  one can prove that the optimal policy $\optpolicy(\state)$ is increasing in $\state$,
then one only needs to compute the  threshold state $\state^*$.
Computing (estimating)   $\state^*$ is often more efficient (from a computational point of view) than solving Bellman's equation when nothing is known about
the structure of the optimal policy. Also real time implementation of a controller with monotone policy (\ref{eq:sparse1}) is simple.

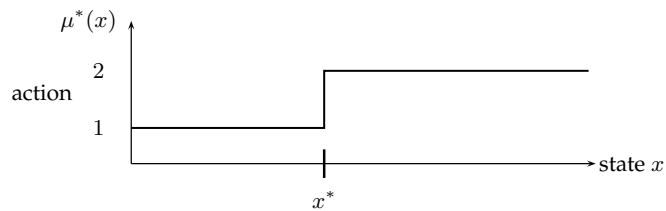
\begin{figure}
\centering
 \scalebox{0.95} 
{\begin{pspicture}(0,-1.5)(7,2)



\psline[linecolor=black]
  (0,0)(2.7,0)(2.7,0.8)(6.4,0.8)


 \rput(-0.5,0){\psframebox*{ $1$}}
 \rput(-0.5,0.8){\psframebox*{ $2$}}
 
    \rput(-1.25,0.5){\psframebox*{action}}

 \psline[linecolor=black,linewidth=1pt]
 (2.7,-0.7)(2.7,-0.3)%
   \rput(2.7,-1){\psframebox*{$\state^*$}}
   



   \rput(7,-0.5){\psframebox*{state $\state$}}

\rput(-0.6, 1.5){\psframebox*{$\optpolicy(\state)$}}


\psline[arrows=->,linecolor=black,linewidth=0.5pt]
  (0,-0.5)(6.5,-0.5)%
\psline[arrows=->,linecolor=black,linewidth=0.5pt]
  (0,-0.5)(0,1.5)%

\end{pspicture}}

\caption{Monotone increasing threshold policy $\optpolicy(\state)$.  Here, $\state^*$ is the threshold state at which the policy switches from 1 to 2.}
\label{fig:thresholdmdp}
\end{figure}

For a finite horizon MDP, Bellman's equation for the optimal policy $\optpolicy_k(\state)$ reads:
\begin{align}
Q_k(\state,\action) & \ole \cost(\state,\action ,k ) + 
\valueb_{k+1}^\p \tp_\state(\action) \label{eq:bellmanmdps}
\\
\valueb_k(\state) &= \min_{\action \in \actionspace} Q_k(\state,\action) , \quad
\optpolicy_k(\state) = \argmin_{\action \in \actionspace}  Q_k(\state,\action)  
\nn \end{align}
where $\valueb_{k+1} = \begin{bmatrix} \valueb_{k+1}(1) , \ldots, \valueb_{k+1}(\statedim) \end{bmatrix}^\p$ denotes the value
function.
The key point is that the optimal policy is $\optpolicy_k(\state) = \argmin_\action Q_k(\state,\action)$. 
What are sufficient conditions on the MDP model to ensure that the optimal policy  $\optpolicy_k(\state)$ is increasing in $\state$ (as shown in
Figure \ref{fig:thresholdmdp})? 
The answer to this question lies in the area of {\em monotone comparative statics} - which studies how the argmin or argmax of a function behaves as one of the variables changes.
The main result of this chapter is to
 show that $Q_k(\state,\action)$ in (\ref{eq:bellmanmdps}) being {\em submodular} in $(\state,\action)$ is a sufficient condition for 
$\optpolicy(\state)$ to increase in $\state$.
Since  $Q_k(\state,\action)$ is the 
conditional expectation
of the cost to go given the current state, giving conditions on the MDP model to ensure that
 $Q_k(\state,\action)$ is submodular requires characterizing how  expectations vary as the state varies. For this
we will use {\em stochastic dominance}.

In the next two sections we introduce these  two important tools, namely, {\em submodularity/supermodularity}  and {\em stochastic dominance}.
They will be used to give conditions under which an MDP has monotone optimal policies.

\section{Submodularity and Supermodularity} \label{chp:fullsupermod}
    \index{submodularity|(}
Throughout this chapter we assume that the state space $\statespace= \{1,2,\ldots,\statedim\}$ and action space $\actionspace= \{1,2,\ldots,\actiondim\}$ are finite.

\subsection{Definition and Examples}
A real valued function $\fun(\state,\action)$ is {\em submodular} in $(\state,\action)$ if 
\beq  \fun(\state, \action+1 ) - \fun(\state,\action)  \geq \fun(\state+1, \action+1 ) - \fun(\state+1,\action) . \label{eq:submodfinite} \eeq
In other words, $\fun(\state,\action+1) - \fun(\state,\action)$ has decreasing differences
 with respect to $\state$. A function
$\fun(\state,\action)$ is {\em supermodular}   if $-\fun(\state,\action)$ is submodular.

Note that submodularity and supermodularity treat $\state$ and $\action$ symmetrically. That is, an equivalent definition is $\fun(\state,\action)$ is submodular
if $\fun(\state+1,\action) - \fun(\state,\action)$ has decreasing differences with respect to $\action$. 

\noindent {\em Examples}: The following
are submodular in $(\state,\action)$ \\ (i) $\fun(\state,\action) = - \state \action $. \hspace{1cm}  (ii) $\fun(\state,\action) = \max(\state,\action)$\\
(iii) Any function of one variable  such as $\fun(\state)$ or $\fun(\action)$ is trivially submodular.\\
(iv) The sum of submodular functions is submodular. \\

 \subsection{Topkis' Monotonicity Theorem} \index{Topkis' monotonicity theorem}
 \index{submodularity! Topkis' monotonicity theorem}
Let $\action^*(\state)$ denote the set of possible minimizers of
 $\fun(\state,\action)$ with respect to $\action$:
 $$\action^*(\state) = \{ \argmin_{\action \in \actionspace} \fun(\state,\action) \}, \quad \state \in \statespace. $$
 In general there might not be a unique  minimizer and then $\action^*(\state)$ has multiple elements.
Let $\bar{\action}^*(\state)$ and  $\underline{ \action}^*(\state)$ denote the maximum and minimum elements of this set.
We call these, respectively, the {\em maximum and minimum selection}  of $\action^*(\state)$.

The key result is the following Topkis' Monotonicity theorem.\footnote{Supermodularity   for structural results in MDPs and game theory was introduced by Topkis in the seminal paper \cite{Top78}. Chapter
\ref{ch:pomdpstop} gives a more general statement in terms of lattices.}

\begin{theorem}  \label{thm:mon} Consider a function $\fun:  \statespace \times \actionspace \rightarrow \reals$.
\begin{compactenum}
\item
If $\fun(\state,\action)$ is submodular, then the maximum and minimal selections  $\bar{\action}^*(\state)$ and  $\underline{ \action}^*(\state)$ are
 increasing in $\state$.
\item If $\fun(\state,\action)$ is supermodular, then  the maximum and minimal selections  $\bar{\action}^*(\state)$ and  $\underline{ \action}^*(\state)$ are
 decreasing in $\state$.
\end{compactenum}
\end{theorem}

\begin{figure}
\centering
\scalebox{0.8}{\begin{pspicture}(9,4) 

\psbcurve(0,1.6)(1,2)(2,0.9)(4,1.7)(5,2)(6,3)(7,2)

\psbcurve[linestyle=dotted](0,2)(1,3)(3,3)(4,1)(5,2)(6,2)(7,3)

\psaxes[linewidth=0.7pt,labels=none, ticks=none]{->}(7,3)[$\action$,-90][${\fun(\state,\action)}$,180]

\psxTick[labelsep=2pt ]{0}(2){\action^*(\state)}

\psxTick[labelsep=2pt ]{0}(4.1){\action^*(\state+1)}

\psline[linestyle=dashed](2,0)(2,0.9)

\psline[linestyle=dashed](4.1,0)(4.1,1.7)

\end{pspicture}}
\caption{Visual illustration of main idea of proof of Theorem \ref{thm:mon}. The solid (respectively, dotted) curve represents $\fun(\state,\action)$ (respectively $\fun(\state+1,\action)$) plotted versus 
$\action$.  Also, $\action^*(\state+1) = \argmin_\action \fun(\state+1,\action)$ . If for all values of $\action$, $\fun(\state,\action)$ to the right  of  ${\action}^*(\state+1)$ is larger  than $\fun(\state,\action)$ to left  of  ${\action}^*(\state+1)$,  \index{Topkis' monotonicity theorem! vizualization}
then clearly the $\argmin_\action \fun(\state,\action)$  must lie to the left of $\action^*(\state+1)$. That is, $\action^*(\state) \leq \action^*(\state+1)$.}
\label{fig:topkis}
\end{figure}

\begin{proof} We prove statement 1. The proof is illustrated visually in Figure \ref{fig:topkis}.
Fix $\state$ and consider $ \fun(\state,\action)$ as a function of $\action$.
Clearly,
 if  $\fun(\state,\action)$ is larger than $\fun(\state,\baction^*(\state+1)) $ for $\action$ to the `right' of $\baction^*(\state+1)$, then
$\argmin_\action \fun(\state,\action)$  must lie to the `left' of $\baction^*(\state+1)$ (see  Figure \ref{fig:topkis}). Therefore,
$$ \fun(\state,\baction^*(\state+1)) \leq \fun(\state,\action) \text{   for } \action \geq \baction^*(\state+1) \implies 
\action^*(\state) \leq \baction^*(\state+1). $$
Let us write this sufficient condition as
 $$ \fun(\state,\action) - \fun(\state,\baction^*(\state+1))  \geq 0 \text{  for }  \action \geq \baction^*(\state+1). $$
Also by definition  $\baction^*(\state+1) \in \argmin_\action \fun(\state+1,\action)$. So clearly $$ \fun(\state+1,\action) - \fun(\state+1,\baction^*(\state+1))  \geq 0 \text{  for all } \action \in \actionspace. $$
From the above two inequalities,  it is sufficient that 
$$ \fun(\state,\action) - \fun(\state,\baction^*(\state+1)) \geq  \fun(\state+1,\action) - \fun(\state+1,\baction^*(\state+1)) \text{  for }  \action \geq \baction^*(\state+1).  $$ 
A sufficient condition for this is that for any $\baction \in \actionspace$,
$$ \fun(\state,\action) - \fun(\state,\baction) \geq  \fun(\state+1,\action) - \fun(\state+1,\baction) \text{  for }  \action \geq \baction.  $$ 
\end{proof}
\section{First Order Stochastic Dominance}
Stochastic dominance is the next tool that will be used to develop MDP structural results. 

\begin{definition}[First order stochastic dominance]  \label{def:gs}
 Let  $\belief_1,\belief_2$ denote  two 
pmfs or pdfs\footnote{Recall the acronyms
pmf (probability mass function) and pdf (probability density function).} with distribution functions  $F_1, F_2$, respectively.
Then  $\belief_1$ is said to first order stochastically dominate $\belief_2$  (written as $\belief_1 \gs \belief_2$ or $\belief_2 \ls \belief_1$)
 if   $$1 - F_1(x) \geq 1-F_2(x), \quad \text{  for all } x \in \reals .$$
 Equivalently,  $\belief_1 \gs \belief_2$ if
 $$ F_1(x) \leq F_2(x), \quad \text{  for all } x\in \reals. $$
 \index{stochastic dominance! first order}
\end{definition}

In this chapter, we  consider pmfs. 
For  pmfs   with support on $\statespace = \{1,2,\ldots,\statedim\}$, Definition \ref{def:gs} is equivalent to the following property of the tail sums:
$$ \belief_1 \gs \belief_2 \; \text{ if }  \;
\sum_{i=j}^\statedim \belief_1(i) \geq \sum_{i=j}^\statedim \belief_2(i), \quad  j \in \statespace. $$

For state space dimension $\statedim =2$,  first order stochastic dominance is a complete order since
$\belief = [1-\belief(2), \;\belief(2)]^\p$ and  so $\belief_1 \gs \belief_2$ if $\belief_1(2) \geq \belief_2(2)$.
Therefore, for $\statedim=2$, any two pmfs are first order stochastic orderable.

For $\statedim \geq 3$, 
first order stochastic dominance is a {\em partial order} since it is not always
possible to order any two belief pmfs  $\belief_1$ and $\belief_2$.

The following is an equivalent characterization of first order  dominance: 
\begin{theorem}[\cite{MS02}] \label{res1}
 Let $\mathcal{V}$ denote the set of all $\statedim$ dimensional vectors
$v$ with 
 increasing components, i.e., $v_1 \leq v_2 \leq \cdots
v_\statedim$.
Then $\belief_1 \gs \belief_2$ if and only if for all $v \in \mathcal{V}$,
 $v^\p \belief_1 \geq v^\p \belief_2$. Similarly, pdf $\belief_1 \gs \belief_2$ if and only if $\int \fun(x) \belief_1(x) dx \geq \int \fun(x) \belief_2(x) dx$ for any increasing
 function $\fun(\cdot)$.
\end{theorem}
In other words, $\belief_1 \gs \belief_2$ if and only if, for any increasing function $\fun(\cdot )$, 
$\E_{\belief_1}\{\fun(x) \} \geq \E_{\belief_2}\{\fun(x) \}$, where $\E_\belief$ denotes expectation with respect to the pmf
(pdf)  $\belief$.
As a trivial consequence, choosing $\fun(x) = x$, it follows that $\belief_1 \gs \belief_2$ implies that the mean of pmf $\belief_1 $ is larger than that of pmf $\belief_2$.

Finally, we need the following concept that combines  supermodularity with stochastic dominance.
We say that a transition probabilities $\tp_{ij}(\action)$ are {\em tail-sum supermodular}  in $(i,\action)$ if $\sum_{j \geq l} \tp_{ij}(\action)$ is supermodular
in $(i,\action)$, i.e., 
\beq  \sum_{j= l}^\statedim \(\tp_{i j}(\action+1) - \tp_{i j}(\action) \) \text{  is increasing in } i, \quad i \in \statespace, \action \in \actionspace  . \label{eq:supermodtp} \eeq
In terms of first order stochastic dominance, (\ref{eq:supermodtp})  can be re-written as
\beq \label{eq:supermodsd}
\frac{1}{2} \bigl( \tp_{i+1}(\action+1) + \tp_i(\action) \bigr) \gs \frac{1}{2} \bigl( \tp_i(\action+1) + \tp_{i+1}(\action) \bigr) , \quad  i \in \statespace, \action \in \actionspace,
\eeq
where $\tp_{i}(\action)$ denotes the $i$-th row of the matrix $\tp(\action)$.  Due to the term $1/2$ both sides are valid probability mass functions.
Thus we have the following result.

\begin{theorem} \label{thm:supermodtp}
Let $\bar{\mathcal{V}}$ denote the set of $\statedim$ dimensional vectors
$v$ with 
  decreasing components, i.e., $v_1 \geq v_2 \geq \cdots \geq
v_\statedim$.
Then  $\tp_{ij}(\action)$ is tail-sum supermodular in $(i,\action)$
iff for all $v \in \mathcal{V}$,
$ v^\p  \tp_i(\action)$ is submodular in $(i,\action)$, that is:
$$v^\p\bigl( \tp_{i+1}(\action+1) - \tp_{i+1}(\action) \bigr) \leq v^\p  \bigl( \tp_{i}(\action+1) - \tp_{i}(\action) \bigr), \quad  i \in \statespace, \action \in \actionspace .$$
\end{theorem}
The proof  follows immediately from Theorem \ref{res1} and (\ref{eq:supermodsd}).

\section{Monotone Optimal Policies for MDPs}\label{sec:monotonecond}
With the above two  tools, we now give sufficient conditions for an MDP to have a monotone  optimal policy.

For finite horizon MDPs,  the model is the 
5-tuple
\beq (\statespace, \actionspace, \tp_{ij}(\action,k), \cost(i,\action,k), \cost_\finaltime(i)) , \qquad  i,j \in \statespace, \action \in \actionspace . \eeq
Assume the  MDP  model 
 satisfies the following 4 conditions:
\begin{description} 
\item[(A1)]  Costs $\cost(\state,\action,k)$ are decreasing in $\state$.  \\ 
The terminal cost $\cost_\finaltime(\state)$ is decreasing in $\state$.
\item[(A2)] $\tp_i(\action,k) \ls \tp_{i+1}(\action,k)$ for each $i$ and $\action$.
Here $\tp_i(\action,k)$ denotes the $i$-th row of the transition matrix for action $\action$ at time $k$.
\item[(A3)] $\cost(\state,\action,k)$ is submodular in $(\state,\action)$ at each time $k$. That is:\\   \index{submodularity! MDP monotone policy}
$ \cost(\state,\action+1,k) -  \cost(\state,\action,k)$ is decreasing in $\state$.
\item[(A4)] $\tp_{ij}(\action,k)$ is tail-sum supermodular in $(i,\action)$ in the sense of (\ref{eq:supermodtp}). That is,\\
$\sum_{j\geq l} \(\tp_{i j}(\action+1,k) - \tp_{i j}(\action,k) \)$ is increasing in $i$.
\end{description}


For  infinite horizon discounted cost and average cost MDPs, identical conditions will be used
except that   the instantaneous costs   $\cost(\state,\action)$ and transition matrix $\tp(\action)$ are time invariant, and there is 
no terminal cost.

Note that (A1) and (A2) deal with different states for a fixed action $u$, while (A3) and (A4) involve different actions and states.

The following is the main structural result for an MDP.

\begin{theorem} \label{thm:mdpmonotone}
\begin{compactenum}
\item Assume that a finite horizon MDP satisfies conditions (A1), (A2), (A3) and (A4).
Then at each time $k=0,1,\ldots,\finaltime-1$, there exists an optimal policy $\optpolicy_k(\state)$ that is  increasing in state $\state \in \statespace$.

\item  Assume that a  discounted infinite horizon cost problem or unichain average cost problem satisfies (A1), (A2), (A3) and (A4).
Then there exists an
 optimal stationary policy $\optpolicy(\state)$ that is  increasing in state $\state \in \statespace$.
\end{compactenum}
\end{theorem}

\begin{proof}
{\bf Statement 1}: To prove statement 1, write   Bellman's equation as
\begin{align}
Q_k(i,\action) & \ole \cost(i,\action ,k ) + 
\valueb_{k+1}^\p \tp_i(\action,k) \label{eq:qfnsm}
\\
\valueb_k(i) &= \min_{\action \in \actionspace} Q_k(i,\action) , \quad
\optpolicy_k(i) = \argmin_{\action \in \actionspace}  Q_k(i,\action) \nn  \end{align}
where $\valueb_{k+1} = \begin{bmatrix} \valueb_{k+1}(1) , \ldots, \valueb_{k+1}(\statedim) \end{bmatrix}^\p$ denotes the value
function.
The proof proceeds in two steps.
\begin{compactenum}
\item[{\em Step 1. Monotone Value Function.}]  Assuming (A1) and (A2), we show via mathematical induction that   \index{monotone value function! MDP} 
$Q_k(i,\action)$ is   decreasing in $i$ for each $\action \in \actionspace$. So the value function $\valueb_k(i)$ is decreasing in $i$ for $k = N,N-1,\ldots,0$.

Clearly $Q_\finaltime(i,\action) = \cost_\finaltime(i)$ is decreasing in $i$ by (A1).
Now for the induction step. Suppose $Q_{k+1}(i,\action)$ is decreasing in $i \in \statespace$ for each $\action$.
 Then $\valueb_{k+1}(i) = \min_{\action}  Q_{k+1}(i,\action)$ is decreasing in $i$, since the minimum of decreasing functions
 is decreasing. So the $\statedim$-dimensional vector $\valueb_{k+1}$ has decreasing elements.
 
 Next $\tp_i(\action,k)  \ls  \tp_{i+1}(\action,k)$ by (A2).  Hence from Theorem \ref{res1}, $\valueb_{k+1}^\p \tp_i(\action,k) 
 \geq \valueb_{k+1}^\p \tp_{i+1}(\action,k) $.
Finally since 
 $\cost(i,\action ,k )$ is decreasing in $i$ by (A1), it follows that 
$$\cost(i,\action ,k ) + \valueb_{k+1}^\p \tp_i(\action,k) 
\geq \cost(i+1,\action ,k ) +   \valueb_{k+1}^\p \tp_{i+1}(\action,k)  
$$
Therefore,
$ Q_k(i,\action) \geq Q_{k}(i+1,\action)$ implying that $Q_k(i,\action)$ is decreasing in $i$ for each $u \in \actionspace$.
(This in turn implies that $J_k(i) = \min_u Q_k(i,u)$ is decreasing in $i$.) Hence the induction step is complete.

\item[{\em Step 2. Monotone Policy.}]  Assuming (A3) and (A4) and using the fact that $\valueb_k(i)$ is decreasing in $i$ (Step 1), we
show that $Q_k(i,\action)$ is submodular in $(i,u)$. \index{monotone policy! MDP}

 \index{submodularity! MDP monotone policy}

By (A3), $ \cost(i,\action ,k )$ is submodular in $(i,u)$. 
By assumption  (A4), 
since  $\valueb_{k+1}$ is a vector with decreasing elements (by Step~1),
it follows from Theorem \ref{thm:supermodtp} that $\valueb^\p_{k+1} \tp_i(u,k)$ is submodular in $(i,u)$.
Since the sum of submodular functions is submodular, it follows that 
 $Q_k(i,\action)  =  \cost(i,\action ,k ) +  \valueb_{k+1}^\p \tp_i(\action,k)  $
 is submodular in $(i,u)$.

Since $ Q_k(i,\action)$ is submodular, it then  follows from Theorem \ref{thm:mon} that $\optpolicy_k(i)=\argmin_{\action \in \actionspace}  Q_k(i,\action) $ is increasing in $i$.
(More precisely, if $\optpolicy$ is not unique, then there exists a version of $\optpolicy_k(i)$ that is increasing in $i$.)
\end{compactenum}

\end{proof}

\section{How does the optimal  cost depend on the transition matrix?} \label{sec:mdptpvar}
\index{structural result! effect of transition matrix}

How does the  optimal expected   cumulative cost $J_{\optpolicy}$ of an MDP vary
  with
 transition matrix? Can the  transition
 matrices  be ordered
so that the larger they are (with respect to some ordering), the larger the optimal  cost?  Such a result would allow us to compare the optimal
performance of different MDPs,  even though computing these via dynamic programming can be numerically expensive.
 
 Consider two distinct MDP models  with transition matrices  $\tp(\action)$ and $\btp(\action)$, $\action \in \actionspace$, respectively.
Let $\optpolicy(\tp)$ and $\optpolicy(\btp)$ denote the optimal policies for these two different MDP models.
 Let $J_{\optpolicy(\tp)}(\state;\tp) $ and 
$J_{\optpolicy(\btp)}(\state;\btp) $ denote the   optimal value functions  corresponding to applying the respective optimal policies.

Introduce the following ordering on the transition matrices of the two MDPs.
 \begin{description}
 \item[(A5)] Each row of  $\tp(\action)$ first order stochastic dominates the corresponding row of $\btp(\action)$ for $\action \in \actionspace$.
 That is, 
 $\tp_i(\action) \gs \btp_i(\action)$ for $i\in \statespace$, $\action \in \actionspace$.
 \end{description}

\begin{theorem} \label{thm:tmdpmove}
Consider two distinct  MDPs with transition matrices $\tp(\action)$ and $\btp(\action)$, $\action \in \actionspace$.
If  (A1), (A2), (A5) hold, then   the expected  cumulative  costs incurred  by the optimal policies 
 satisfy $J_{\mu^*(\tp)}(\state;\tp) \leq J_{\mu^*(\btp)}(\state;\btp)$. 
 \end{theorem}

The theorem says that controlling an MDP with transition matrices $\tp(\action)$, $\action \in \actionspace$ is always cheaper
than an MDP with transition matrices $\btp(\action)$, $\action \in \actionspace$ if (A1), (A2) and (A5) hold. Note that Theorem \ref{thm:tmdpmove}  does need require  numerical evaluation of the optimal policies or value functions.

\begin{proof}

Suppose 
\begin{align*}  Q_k(i,\action) =  \cost(i,\action ,k ) + 
\valueb_{k+1}^\p \tp_i(\action) , \quad
 \bQ_k(i,\action) =  \cost(i,\action ,k ) + 
\bvalueb_{k+1}^\p \btp_i(\action)
\end{align*}
The proof is by induction.
Clearly $\valueb_\finaltime(i)  = \bvalueb_\finaltime(i)= \cost_\finaltime(i)$ for all $ i \in \statespace$.
Now for the inductive step.
Suppose $\valueb_{k+1}(i) \leq \bvalueb_{k+1}(i)$ for all $i \in \statespace$.
 Therefore
$\valueb_{k+1}^\p \tp_i(\action)  \leq  \bvalueb_{k+1}^\p \tp_i(\action)$.
By (A1), (A2), $\bvalueb_{k+1}(i)$  is decreasing in $i$. By (A5), $\tp_i \gs \btp_i$.
Therefore $\bvalueb_{k+1}^\p \tp_i  \leq \bvalueb_{k+1}^\p \btp_i$.
So $ \cost(i,\action ,k ) + \valueb_{k+1}^\p \tp_i(\action) \leq \cost(i,\action ,k ) + \bvalueb_{k+1}^\p \btp_i(\action) $ or equivalently,
$Q_k(i,\action) \leq \bQ_k(i,\action)$. Thus $\min_\action Q_k(i,\action) \leq \min_\action \bQ_k(i,\action)$, or equivalently,
$\valueb_k(i) \leq \bvalueb_k(i)$ thereby completing the induction step.

\end{proof}

\section{Algorithms for  Monotone Policies - Exploiting Sparsity} \label{sec:algsparse} 
Consider an   average cost MDP.
Assume that  the costs and transition matrices satisfy  (A1)-(A4). Then
by   Theorem \ref{thm:mdpmonotone}  the optimal stationary policy 
$\optpolicy(\state)$ is  increasing in $\state$. How can this monotonicity property be exploited to  compute (estimate) the optimal policy?
This section discusses several approaches. (These approaches also apply to discounted cost MDPs.)

\subsection{Policy Search and Q-learning with Submodular Constraints}  Suppose, for example, \index{monotone policy! sigmoidal approximation}
$\actionspace=\{1,2\}$ so that  the monotone
 optimal stationary policy is a step function of the form (\ref{eq:sparse1})   (see Figure \ref{fig:thresholdmdp}) and is completely defined by the threshold state $\state^*$.
 We need an algorithm to search for  $\state^*$ over the finite state space $\statespace$. In \cite{YKI04}, discrete-valued stochastic optimization algorithms for implementing this.
Another possibility is to solve a continuous-valued relaxation as follows: Define the parametrized policy $ \policy_\modelpsi(\state)$ 
where $\modelpsi \in \reals_+^2$ denotes the parameter vector. Also define the sample path  cumulative cost estimate $ \hat{C}_\finaltime(\modelpsi)$ over some fixed time horizon $\finaltime$ as
$$ \policy_\modelpsi(\state) =  \begin{cases} 1 &  \frac{1}{1 + \exp(-\modelpsi_1(\state- \modelpsi_2))} < 0.5 \\
2	& \text{ otherwise} \end{cases}, \quad  \hat{C}_\finaltime(\modelpsi) = \frac{1}{\finaltime+1} \sum_{k=0}^\finaltime \cost(\state_k,\policy_\modelpsi(\state_k) ) .
$$ Note  $\policy_\modelpsi$
 is a sigmoidal approximation to the step function (\ref{eq:sparse1}).
Consider the stochastic  optimization problem: Compute $\modelpsi^* = \argmin_{\modelpsi }
 \E\{\hat{C}_\finaltime(\modelpsi)\} $.
This can be solved readily via simulation based  gradient algorithms    such as the  SPSA Algorithm.

Alternatively, instead of exploiting the monotone structure in policy space, the submodular structure of the value function  can be
exploited.   \index{submodularity! Q-learning}
 From Theorem \ref{thm:mdpmonotone}, the
Q-function $Q(\state,\action)$ in (\ref{eq:qfnsm})
 is submodular. This submodularity can be exploited in Q-learning algorithms as described in \cite{DK07,DK07b}.
 
The above methods operate without requiring explicit knowledge of  transition matrices.  Such methods  are useful in 
transmission scheduling in wireless communication where often by modeling assumptions (A1)-(A4) hold, but the actual values
of the transition matrices are not known.

\subsection{Sparsity Exploiting Linear Programming}
Here  we describe \index{sparsity! monotone policy}
 how the linear programming  formulation  for the optimal policy can exploit the monotone structure 
 of the optimal policy.
Suppose the number of actions $\actiondim$ is small but the number of states $\statedim$ is large.
Then the monotone   optimal policy  $\optpolicy(\state)$ is  {\em sparse} in the sense that it is a piecewise constant function of the state 
$\state$ that jumps
upwards at most at $\actiondim-1$ values (where by assumption $\actiondim$ is small). In other words,  $\optpolicy(\state+1) - \optpolicy(\state)$ is non-zero for at most $\actiondim-1$ values of $\state$.
In comparison,  an unstructured policy can jump between states at arbitrary values and is therefore not sparse.

 How can  this sparsity property of a monotone optimal policy  \index{monotone policy! sparsity fused lasso}
 be exploited to compute the optimal policy?
A convenient way of parametrizing sparsity in a  monotone  policy is in terms of the {\em conditional} probabilities $\theta_{\state,\action}$ defined in (\ref{eq:randpol}). Indeed $\theta_{\state,\action}- \theta_{\state-1,\action}$ as a function of $\state$ for fixed $\action$  is non zero for up to only two values of $\state \in \statespace$.
A natural idea, arising from  sparse estimation and compressed sensing \cite{HTF09}, is to add a Lagrangian 
sum-of-norms    \index{sparsity! monotone policy! sum-of-norms minimization}
term\footnote{Each term in the summation is  a $l_2$ norm, the overall expression  is the sum of norms.
This is similar to the $l_1$ norm which is the sum of absolute values.}
 \beq \lambda\sum_{\state \geq 2} \|\theta_\state  - \theta_{\state-1} \|_2,\quad \lambda\geq 0 , 
   \label{eq:l1lagrangian}
 \eeq
 to a cost function whose minimum yields the optimal policy. Here
$ \th_\state = [\th_{\state,1},\ldots,\th_{\state,\actiondim}]^\p$.  
The term (\ref{eq:l1lagrangian}) is a variant of the \emph{fused lasso}\footnote{
The lasso (least absolute shrinkage and selection operator) estimator was originally proposed in \cite{Tib96}. This is one of the most influential papers 
 in statistics since the 1990s. It seeks to determine $\th^*=\argmax_\th   \|\obs-\statem \th\|_2^2 + \lambda \|\th\|_1 $ given 
 an observation vector $\obs$, input matrix $A$ and scalar $\lambda> 0$.}
  or \emph{total variation} penalty, and can be interpreted as a convex 
relaxation of a penalty on the number of changes of  conditional probability $\theta$ (as a function of state $\state$).  We refer to \cite{KRW13} for details.

\section{Example: Transmission Scheduling over Wireless Channel} \label{sec:minh}
\index{monotone policy! transmission scheduling|(}
The conditions given in \secn \ref{sec:monotonecond} are {\em sufficient} for the optimal policy to have a monotone structure.
We conclude this chapter by describing an
 example where the sufficient conditions in \secn \ref{sec:monotonecond} do not hold. However,
a somewhat more  sophisticated proof shows that  the optimal policy is monotone.
The formulation  below generalizes the classical result of  \cite{DLR76},\cite{Ros83} to the case of Markovian dynamics.

Consider the transmission of time sensitive video (multimedia) packets in  a wireless communication system with the use of an ARQ protocol for retransmission. Suppose  $L$ such packets stored in a buffer need to transmitted over $N \geq L$ time slots.
At each time slot, assuming the channel state is known, the transmission controller
decides whether to attempt a transmission. 
The quality of  the wireless channel (which evolves due to fading) is represented abstractly
by a finite state Markov chain. The channel quality affects the  error probability of successfully transmitting a packet.
If a transmission is attempted, the result (an ACK or NACK of whether successful transmission was achieved)
is received.
If a packet is transmitted but not successfully received, it remains in the buffer and may be retransmitted. At the end of all $N$ time slots, no more transmission is allowed and a penalty  cost is incurred  for packets that remain in the buffer. 

How should a transmission controller  decide at which time slots to transmit the packets?
 It is shown below that the optimal transmission scheduling policy is a monotone (threshold) function of time and buffer size. The framework is  applicable to  any delay-sensitive real time packet transmission system. 

\subsection{MDP Model for Transmission Control}
We formulate the above transmission scheduling problem as a finite horizon MDP with a penalty terminal cost.  
The wireless fading  channel is modeled as a  finite state Markov chain.
Let $s_{k} \in S=\{\gamma_{1},\ldots,\gamma_{K} \}$  denote the  channel state at time slot $k$. Assume  $s_{k}$ evolves as a Markov chain according to transition probability matrix  $\tp = (\tp_{ss'}: s,s' = 1,2,\ldots,K)$, where $\tp_{ss'} = \mathbb{P}(s_{k+1}=\gamma_{s'}|s_{k} =\gamma_{s})$.
Here the higher the state $s$, the better the quality of the channel.

Let  $\actionspace=\{u_{0}=0 \text{ (do not transmit)}, u_{1}=1 \text{ (transmit)}\}$ denote the action space.
In a time slot, if action $u$ is selected,  a cost $\cost(u)$ is accrued, where $\cost(\cdot)$ is an increasing function. The probability that a transmission is successful is an increasing function of the action $\action$ and channel state  $s$:
\begin{align}
\eprob(u,s) = \left \{ \begin{array}{clcl} 0 & \textup{If } u = 0 \\
1 - P_{e}(s) & \textup{If } u = 1.
\end{array} \right. 
\label{eq:ps_tx}
\end{align}
 Here $P_{e}(s)$ denotes  the error probability for channel state $s$ and is a decreasing function of $s$.

 Let $n$ denotes the residual transmission time: $n = N,N-1,\ldots,0$. At the end of all $N$ time slots, i.e. when $n=0$, 
 a terminal penalty cost $\cost_\finaltime(i)$ is paid if $i$  untransmitted packets remain in  the buffer. It is assumed that  $\cost_\finaltime(i)$ is increasing in $i$ and $\cost_\finaltime(0) = 0$. 
 
 The optimal scheduling policy $\policy_{n}^{*}(i,s)$ is the solution of Bellman's equation:\footnote{It is notationally convenient here to use Bellman's equation
 with forward indices. So we use $\valuef_n= \valueb_{\finaltime - n}$ for the value function. This notation was used previously for MDPs  in (\ref{eq:fowardvaluef}).}
\begin{align}
&V_{n}(i,s)  = \min \limits_{u\in U}~ Q_{n}(i,s,u), \quad 
  \policy^{*}_{n}(i,s) = \arg \min \limits_{u\in U}~ Q_{n}(i,s,u),  \label{eq:bellman_policy} \\
 Q_{n}(i,s,u)  & =  \biggl\{ \cost(u) + \sum \limits_{s' \in S}  \tp_{ss'} \biggl[ \eprob(u,s) V_{n-1}(i-1,s') +  \left(1-\eprob(u,s) \right) V_{n-1}(i,s') \biggr] \biggr\}  \nonumber 
\end{align}
initialized with   $ V_{n}(0,s) = 0, V_{0}(i,s)  =\cost_\finaltime(i)$.
 A larger terminal  cost $\cost_\finaltime(i)$  emphasizes  delay sensitivity while a larger action cost $\cost(u)$ emphasizes  energy consumption.
If $\tp_{ss} = 1$ then the problem reduces to that considered in \cite{DLR76,Ros83}.

\subsection{Monotone Structure of Optimal Transmission Policy}  \index{threshold policy! transmission control} 


\begin{theorem} \label{thm:minh}The optimal transmission policy  $\policy^{*}_{n}(i,s)$ in (\ref{eq:bellman_policy}) has the following monotone
structure:
\begin{compactenum}
\item If the terminal cost $\cost_\finaltime(i)$ is  increasing in the buffer state $i$, then  $\policy^{*}_{n}(i,s)$  is  decreasing in  the number of transmission time slots remaining $n$.
\item If  $\cost_\finaltime(i)$ is  increasing in the  buffer state $i$ and is integer convex, i.e.,
 \begin{equation} \cost_\finaltime(i+2) - \cost_\finaltime(i+1) \geq  \cost_\finaltime(i+1) - \cost_\finaltime(i) ~ \forall i \geq 0, \label{eq:conv_cond} \end{equation}
 then $\policy^{*}_{n}(i,s)$ is  a threshold policy of the form:
 $$\policy^{*}_{n}(i,s) = \begin{cases} 0 &  i < i^{*}_{{n},s} \\
 							1 &  i \geq  i^{*}_{{n},s} \end{cases} $$
Here the  threshold buffer state $ i^{*}_{{n},s}$  depends on $n$ (time remaining) and $s$ (channel state).
Furthermore, the threshold $i^{*}_{n,s}$ is increasing in $n$.
\end{compactenum} \label{thm:opt_n}
\end{theorem}

The theorem says that  the optimal transmission policy is  {\em aggressive}  since it is optimal to transmit more often when the residual transmission time is less or the buffer occupancy is larger.
The threshold structure of the optimal transmission scheduling policy can be used to reduce the computational cost in solving the dynamic programming problem or the memory required to store the solutions. 
For example, the total number of transmission policies given $L$ packets, $N$ time slots and $K$ channel states is
$2^{NLK}$. In comparison, the number of transmission policies that are monotone in the number of transmission time slots remaining
and the buffer state is $N L^K$, which can be substantially smaller.
\cite{HS84} gives several algorithms (e.g., (modified) value iteration, policy iteration) that exploit monotone results to efficiently compute the optimal policies.

 To prove  Theorem \ref{thm:minh},  the following results can be established.

\begin{lemma} 
The value function $V_{n}(i,s)$ defined by (\ref{eq:bellman_policy}) is  increasing in the number of remaining packets $i$ and decreasing in the number of remaining time slots $n$.
\label{lm:vmono}
\end{lemma}

\begin{lemma} 
If $\cost_\finaltime(\cdot)$ is an increasing function (of the terminal buffer state) then the value function $V_{n}(i,s)$ satisfies the following submodularity condition:  \index{submodularity! transmission control}
\begin{align}
 & V_{n}(i+1,s) - V_{n}(i,s) \geq  V_{n+1}(i+1,s) - V_{n+1}(i,s),
\label{eq:V_submod_in}
\end{align}
for all  $i \geq 0,~s \in S$. Hence,  $Q_{n}(i,s,u)$ in (\ref{eq:bellman_policy}) is supermodular in $(n,u)$.\\
Furthermore, if the penalty cost $\cost_\finaltime(\cdot)$ is an increasing function and satisfies (\ref{eq:conv_cond}) then  $V_{n}(i,s)$ has increasing differences in the number of remaining packets:
\begin{align}
 & V_{n}(i+2,s) - V_{n}(i+1,s) \geq  V_{n}(i+1,s) - V_{n}(i,s),  \label{eq:convex_i}
\end{align}
for all $i \geq 0,~ s \in S$. Hence $Q_{n}(i,s,u)$  is submodular in $(i,u)$.
\label{lm:v_spcv}
\end{lemma}

With the above two lemmas, the proof of  Theorem \ref{thm:minh} is as follows:

First statement:
If $\cost_\finaltime(i)$ is  increasing in $i$ then $V_{n}(i,s)$ satisfies (\ref{eq:V_submod_in}) and $Q_{n}(i,s,u)$ given by (\ref{eq:bellman_policy}) is supermodular in $(u,n)$ (provided that  $\eprob(u,s)$ defined in (\ref{eq:ps_tx}) is an increasing function of the action $u$). Therefore, $\policy^{*}_{n}(i,s)$ is decreasing in $n$. 

Second statement:
Due to Lemma~\ref{lm:v_spcv}, if $\cost_\finaltime(i)$ is  increasing in $i$ and satisfies (\ref{eq:conv_cond}) then $V_{n}(i,h)$ satisfies (\ref{eq:convex_i}) and $Q_{n}(i,h,u)$  is submodular in $(u,i)$ (provided that $\eprob(u,h)$ increases in $u$). The submodularity of  $Q_{n}(i,h,u)$ in $(u,i)$ implies that $\policy^{*}_{n}(i,h)$ is  increasing in $i$.
. The result that  $i^{*}_{\bar{n},\bar{h}}$ is increasing in $\bar{n}$ follows since   $\policy^{*}_{n}(i,h)$  is  decreasing in $n$ (first statement).
\index{monotone policy! transmission scheduling|)}

\section{\pwe}
The use of 
supermodularity   for structural results in MDPs and game theory was pioneered by Topkis in the seminal paper \cite{Top78}
culminating in the book \cite{Top98}.
We refer the reader to \cite{Ami05} for a tutorial description of supermodularity with applications in economics.
\cite[Chapter 8]{HS84} has an insightful  treatment of submodularity in MDPs.
Excellent books in stochastic dominance include  \cite{MS02,SS07}. \cite{SM02} covers several cases of monotone MDPs. The paper \cite{MS94} is  highly influential in the area of monotone comparative statics (determining how the $ \argmax$ 
or $\argmin$ behaves as a parameter varies) and discusses the single crossing condition.
\cite{QS12} has some recent results on conditions where the single crossing property is closed under addition.
 A more general version of Theorem \ref{thm:tmdpmove} is proved in \cite{Mul97}.
The example in \secn \ref{sec:minh} is expanded in \cite{NK09} with detailed numerical examples. Also \cite{NK10} considers the average cost version of the transmission scheduling problem
on a countable state space (to model an infinite buffer).  \cite{HK10} gives structural results for Markov decision games.  \cite{YG94b} studies supermodularity and monotone policies in discrete event systems.
\cite{AGH04} covers deeper results in multimodularity, supermodularity and extensions of convexity to discrete spaces for  discrete event
systems. In \cite{VK03}, gradient based stochastic approximation algorithms are presented to estimate the optimal policy of a constrained MDP. It would be of interest to generalize these results
by adding constraints for a monotone policy.


\chapter{Structural Results for Optimal Filters} \label{chp:filterstructure}

\index{HMM filter structural results|(}
\index{stochastic dominance!optimal filtering|(}
\index{filter! structural results|(}

This chapter and the following four
chapters  develop structural results
for the optimal policy of a POMDP. 
  In Chapter \ref{chp:monotonemdp}, we used first order stochastic dominance to characterize the structure of optimal policies for finite state
 MDPs.  
However, first order stochastic dominance is not preserved under Bayes rule for the belief state update. So 
 a stronger stochastic order is required to obtain structural results
for POMDPs. We will use the {\em monotone likelihood ratio (MLR) stochastic order}  to order belief states.
This chapter  develops important structural
results for the  HMM filter using the MLR  order. 
These results  form
a crucial step in  formulating the structural results for POMDPs.



\subsection*{Outline of this Chapter}
Recall from Chapter \ref{ch:pomdpbasic} that for a POMDP, with controlled transition matrix $\tp(\action)$ and observation probabilities
$\oprob_{\state\obs}(\action) = \pdf(\obs|\state,\action)$, given the observation $\obs_{k+1}$,
the 
HMM filter recursion  (\ref{eq:hmmc}) for the belief state $\belief_{k+1}$ in terms of $\belief_k$ reads
\begin{align}
& \belief_{k+1} = \filter(\belief_k,\obs_{k+1},\action_k)  = \frac{\oprob_{\obs_{k+1}}(\action_k) \tp^\p(\action_k) \belief_k}{\filterd(\belief_k,\obs_{k+1} ,\action_k)}, 
\;  \filterd(\belief,\obs,\action) = \one^\p \oprob_{\obs}(\action) \tp^\p(\action) \belief,   \nn \\
& \oprob_\obs(\action)   =  \diag(\oprob_{1\obs}(\action),\ldots,\oprob_{\statedim y}(\action)), \quad \action \in \actionspace= \{1,2, \ldots,\actiondim\}, \obs \in \obspace. \label{eq:hmmstruc}
\end{align}

The two main questions addressed in this chapter are:
\begin{compactenum}
\item How can beliefs  (posterior distributions) $\belief $ computed by the HMM filter  be ordered within the belief space (unit simplex) $\Belief$? 
\item Under what conditions does the HMM filter $\filter(\belief,\obs,\action) $ increase with belief $\belief$,  observation $\obs$ and  action~$\action$?
\end{compactenum}

Answering the first question is crucial to define  what it means for a POMDP
to have an optimal policy $\optpolicy(\belief)$ increasing with 
$\belief$. Recall from Chapter \ref{chp:monotonemdp} that for the fully observed MDP case, we gave conditions under which 
the optimal policy is  increasing
with scalar state $\state$ -- ordering the states  was trivial since they were scalars.
However, for a POMDP,   we  need to order the belief
states $\belief$ which are probability mass functions (vectors) in the unit simplex $\Belief$.

The first question above will be answered by using the
  monotone likelihood ratio (MLR) stochastic order 
to order belief states.  The MLR  is a {\em partial order} that is ideally suited for Bayesian estimation
(optimal filtering)  since it is preserved under conditional expectations. \secn \ref{sec:mlrorder} discusses the MLR order.

The answer to the second  question is essential for giving conditions on a POMDP so that  the optimal policy $\optpolicy(\belief)$  increases
in $\belief$. Recall that Bellman's  equation (\ref{eq:bellmaninfpomdp}) for a POMDP involves the term $\sum_\obs V(\filter(\belief,\obs,\action) ) \filterd(\belief,\obs,\action)$.
So to show that the optimal policy is monotone,  it is necessary to characterize the behavior of  the HMM filter $ \filter(\belief,\obs,\action)$
and normalization term $\filterd(\belief,\obs,\action)$. \secn \ref{sec:tpcp} to \secn \ref{sec:assumpdiscussion} discuss structural properties of the HMM filter.

Besides their importance in establishing structural results for POMDPs, the 
 structural results for HMM filters developed in this chapter are also useful for  constructing reduced complexity HMM filtering algorithms that provably
upper and lower bound the optimal posterior (with respect to the MLR  order).   We shall describe the construction
of such reduced complexity HMM filters in \secn \ref{sec:filterlowbound}.
 Sparse rank transition matrices
for the low complexity filters will be constructed via nuclear norm minimization (convex optimization).
\section{Monotone Likelihood Ratio (MLR)  Stochastic Order} \label{sec:mlrorder}
For dealing with POMDPs,
the MLR stochastic order is the main concept that will be used to
 order belief states in the unit $\statedim-1$ dimensional unit simplex 
\begin{align*}
\Belief = \left\{\belief \in \reals^{\statedim}: \one^{\p} \belief = 1,
\quad 
0 \leq \belief(i) \leq 1, \; i \in \statespace =\{1,2,\ldots,\statedim\}\right\} .
\end{align*}

 
 \subsection{Definition}
 
\begin{definition}[Monotone Likelihood Ratio (MLR) ordering]  \index{stochastic dominance! monotone likelihood ratio} \label{def:mlr}
Let $\belief_1, \belief_2 \in \Belief$ denote  two belief state vectors.
Then $\belief_1$ dominates $\belief_2$ with respect to the MLR order, denoted as
$\belief_1 \gr \belief_2$,
 if 
\beq \belief_1(i) \belief_2(j) \leq \belief_2(i) \belief_1(j), \quad i < j, i,j\in \{1,\ldots,\statedim\}. 
\label{eq:mlrorder}\eeq
\end{definition}
Similarly $\belief_1 \lr \belief_2$ if  $\leq$ in (\ref{eq:mlrorder}) is replaced
by a $\geq$.

Equivalently, $\belief_1 \gr \belief_2$  if the likelihood ratio  $\belief_1(i)/\belief_2(i)$ is increasing.   Similarly, for the case of pdfs, define  $\belief_1 \gr \belief_2$ if their ratio
 $\belief_1(x)/\belief_2(x)$ is increasing in $x \in \reals$.

\begin{definition}
A function $\fun:\Belief\rightarrow \reals$ is said to be MLR increasing if $\belief_1 \gr \belief_2$ implies $\fun(\belief_1) \geq \fun(\belief_2)$. Also $\fun$ is MLR decreasing if $-\fun$ is MLR increasing.
\end{definition}

Recall  the definition of first order stochastic dominance  (Definition \ref{def:gs} in Chapter \ref{chp:monotonemdp}). 
MLR dominance is a stronger condition than first order dominance.

\begin{theorem} \label{lem:mlrfo} For  pmfs or pdfs $\belief_1$ and $\belief_2$,
$\belief_1 \gr \belief_2$ implies
 $\belief_1\gs \belief_2$.
\end{theorem}
\begin{proof}  $\belief_1 \gr \belief_2$ implies  $\belief_1(x)/\belief_2(x)$ is increasing in $x$. Denote the corresponding cdfs as $F_1,F_2$.
Define $t = \{\sup x: \belief_1(x) \leq \belief_2(x)\}$. Then $\belief_1 \gr \belief_2$ implies that for $x \leq t$, $\belief_1(x) \leq \belief_2(x)$ and
for  $x \geq t$,  $\belief_1(x) \geq \belief_2(x)$.
So for $x \leq t$,
 $F_1(x) \leq F_2(x)$.  Also  for $x > t$,   $\belief_1(x) \geq \belief_2(x)$ implies $1 - \int_x^\infty \belief_1(x) dx \leq 
 1 - \int_x^\infty \belief_2(x) dx$ or equivalently,  $F_1(x) \leq F_2(x)$. Therefore $\belief_1 \gs \belief_2$. \end{proof}

For state space dimension $\statedim =2$, MLR is a {\em complete} order and coincides with
first order stochastic dominance. The reason is that for $\statedim=2$, since $\belief(1) + \belief(2) = 1$, it suffices to choose
the second component $\belief(2)$ to order $\belief$.
Indeed, for
$$\statedim = 2, \quad \belief_1 \gr \belief_2\;  \iff \; \belief_1 \gs \belief_2 \; \iff\;   \belief_1(2) \geq \belief_2(2). $$

For state space dimension $\statedim \geq 3$, 
MLR and first order dominance are   {\em partial orders} on the belief space $\Belief$. Indeed,  $[\Belief,\gr]$ is a partially ordered set (poset) since it is not always
possible to order any two belief states in $ \Belief$. \\
Example (i): $ [0.2, 0.3, 0.5]^\p \gr [0.4, 0.5, 0.1]^\p$  \\
Example (ii):  $[0.3,0.2,0.5]^\p$ and $[0.4. 0.5. 0.1]^\p$ are not MLR comparable.\\
Figure \ref{simplexmlr} gives a geometric interpretation of first order and MLR dominance for
 $\statedim=3$.

\subsection{Why MLR ordering?}
 
The MLR stochastic order is   useful in  filtering and POMDPs since it is preserved under conditional expectations (or more naively, application of Bayes rule). 

\begin{theorem} \label{thm:mlrbayes} MLR dominance is preserved under Bayes rule: For continuous or discrete-valued observation 
$\obs$ with observation likelihoods $\oprob_y = \diag( \oprob_{1\obs},\ldots,\oprob_{\statedim y})$, $\oprob_{\state\obs} = \pdf(\obs|\state)$, given two beliefs $\belief_1,\belief_2\in \Belief$, then
$$\belief_1 \gr \belief_2 \iff \frac{\oprob_y \belief_1 }{\one^\p \oprob_y \belief_1} \gr  \frac{\oprob_y \belief_2 }{\one^\p \oprob_y \belief_2} $$
providing $\one^\p \oprob_y \belief_1$ and  $\one^\p \oprob_y \belief_2$ are non-zero.\footnote{A notationally elegant  way of saying this is: Given two random variables $X$
and $Y$, then 
$X\lr Y$ iff $X| X\in A \lr Y| Y \in A$ for all events $A$ providing $P(X \in A) >0 $ and $P(Y \in A) > 0$.
Requiring $\one^\p \oprob_y \belief > 0$  avoids pathological cases such as $\belief = [1, 0]^\p$ and $\oprob_y = \diag(0,1)$, i.e., prior says state 1 with
certainty, while observation says state 2 with certainty. \label{foot:mlrdef}
}
\end{theorem}

\begin{proof}
By definition of MLR dominance, the right hand side is 
$$ \cancel{\oprob_{i \obs} \oprob_{i+1, \obs} }\belief_1(i) \belief_2(i+1) \leq \cancel{\oprob_{i \obs} \oprob_{i+1, \obs}} \belief_1(i+1) \belief_2(i) $$ which is
equivalent to  $\belief_1 \gr \belief_2$.
\end{proof}

Notice that in the fully observed MDP of Chapter \ref{chp:monotonemdp}, we used first order stochastic dominance.
 However,  first order stochastic dominance  is not preserved under conditional expectations and so is not useful for POMDPs.
 
 \subsection{Examples}   \label{sec:fsdmlr}
 1. {\em First order stochastic dominance is not closed under Bayes rule}: Consider beliefs $\belief_1 = (\frac{1}{3}, \frac{1}{3}, \frac{1}{3})^\p$, $\belief_2 = (0, \frac{2}{3}, \frac{1}{3})^\p$.  Then clearly $\belief_1 \ls \belief_2$. 
 Suppose $\tp = I$ and the observation likelihoods  $\oprob_{\state\obs}$ have values:
 $\prob(\obs|\state=1) = 0$, 
 $\prob(\obs|\state=2) = 0.5$, $\prob(\obs|\state=3) = 0.5$. Then the filtered updates are
 $\filter(\belief_1,\obs,\action) = (0,\frac{1}{2},\frac{1}{2})^\p$ and $\filter(\belief_2,\obs,\action) = (0,\frac{2}{3},\frac{1}{3})^\p$. Thus
 $\filter(\belief_1,\obs,\action) \gs \filter(\belief_2,\obs.\action)$  showing that first order stochastic dominance is not preserved after a Bayesian update.
 
 The MLR order is stronger than first order dominance. For example choosing $\belief_3 = (0, 1/3, 2/3)^\p$, then $\belief_1 \lr \belief_3$. The filtered
 updates are  $\filter(\belief_1,\obs,\action) = (0,\frac{1}{2},\frac{1}{2})^\p$ and $\filter(\belief_3,\obs,\action) = (0,\frac{1}{3},\frac{2}{3})^\p$ and it is seen that
  $\filter(\belief_1,\obs,\action) \lr \filter(\belief_3,\obs,\action) $.
 \\
 2. Examples of pmfs that satisfy MLR dominance are (\cite{MS02} has a detailed list):
 \begin{align*}
 \text{ Poisson:} & \quad  \frac{\lambda_1^k}{k!} \exp(\-\lambda_1) \lr  \frac{\lambda_2^k}{k!} \exp(\-\lambda_2), \quad  \lambda_1 \leq \lambda_2 \\
 \text{ Binomial:} & \quad \binom{n_1}{ k} p_1^k (1-p_1)^{n_1-k}  \lr  \binom{n_2}{ k} p_1^k (1-p_2)^{n_2-k}, \quad n_1 \leq n_2, p_1\leq p_2 \\
 \text{ Geometric:} & \quad (1-p_1)\, p_1^k  \lr (1-p_2) \,p_2^k,  \quad p_1 \leq p_2.
 \end{align*}
 \noindent
3. The MLR order is also defined for probability density functions (pdfs): $p \gr q$ if $p(x)/q(x)$ is increasing in $x$. If the pdfs are differentiable, this
 is equivalent to saying $\frac{d}{dx} \frac{p(x)}{q(x)} \geq 0$. Examples include:
 \begin{align*}
 \text{ Normal:} & \quad \normal(x;\mu_1,\sigma^2) \lr \normal(x;\mu_2,\sigma^2), \quad \mu_1 \leq \mu_2 \\
 \text{Exponential:} & \quad \lambda_1 \exp(\-\lambda_1(x-a_1)) \lr  \lambda_2 \exp(\-\lambda_2(x-a_2)), \quad a_1 \leq a_2, \lambda_1 \geq \lambda_2.
 \end{align*}
Uniform pdfs $U[a,b] = I(x\in[a,b])/(b-a)$ are not MLR comparable with respect to $a$ or $b$.

\section{Total Positivity and Copositivity} \label{sec:tpcp}

This section defines the concepts of total positivity and copositivity. These are crucial concepts in 
obtaining monotone properties of the optimal filter with respect to the MLR order.

\begin{definition}[Totally Positive of Order 2 (TP2)]  \label{def:tp2}
A stochastic matrix $M$ is TP2 if  all its second order minors are non-negative.  That is,
determinants  \index{TP2 matrix}
\beq
 \begin{vmatrix}   M_{i_1j_1} & M_{i_1j_2} \\ M_{i_2j_1} & M_{i_2j_2} \end{vmatrix} \geq 0 
\;\text{ for } i_2 \geq i_1, j_2 \geq j_1. \label{eq:tp2} \eeq
Equivalently, a   transition or observation
kernel\footnote{We use the term ``kernel" to allow for continuous and discrete valued observation spaces $\obspace$.
If $\obspace$ is discrete, then $\oprob(\action)$ is a $\statedim \times \obsdim$ stochastic matrix. If $\obspace \subseteq \reals$, then
$\oprob_{\state\obs}(\action) = \pdf(\obs|\state,\action)$ is a probability density function.} denoted $M$ is TP2 if the $i+1$-th row MLR dominates the $i$-th row: that is,
 $M_{i,:}  \gr M_{j,:}$
 for every $i> j$.
\end{definition}

Next, we define a copositive ordering. Recall from \chp  \ref{sec:hmmfilter} that an HMM is parameterized \index{copositive matrix}
by the transition and observation probabilities $(\tp,\oprob)$.  Start with the following notation.
Given an HMM $(\tp(\action), \oprob(\action))$  and another HMM\footnote{The notation
 $\action$ and $\action+1$ is used to distinguish between the two HMMs. Recall that in POMDPs,  $\action$  denotes the actions taken by the controller.}
 $(\tp(\action+1),\oprob(\action+1))$, define the sequence of $\statedim\times \statedim$ dimensional symmetric matrices $ \copomat^{j,\action,\obs}$, $j = 1,\ldots, \statedim-1$,
$\obs \in \obspace$ as
\begin{align}
&\copomat^{j,\action,\obs} = \cfrac{1}{2}\left[\gamma^{j,\action,y}_{mn} + \gamma^{j,\action,y}_{nm}\right]_{\statedim\times\statedim}
, \quad \text{ where } 
\\ 
\gamma^{j,\action,y}_{mn} &= \oprob_{j,\obs}(\action)\oprob_{j+1,\obs}(\action+1)\tp_{m,j}(\action) \tp_{n,j+1}(\action+1) \nn \\
& \hspace{1cm}  - \oprob_{j+1,\obs}
(\action)
\oprob_{j,\obs}(\action+1) \tp_{m,j+1}(\action)\tp_{n,j}(\action+1). \nonumber
\end{align}

\begin{definition}[Copositive Ordering $\lR$ of Transition and Observations Probabilities] \label{def:copositiveall}
\index{copositive ordering}
Given $(\tp(\action), \oprob(\action))$  and $(\tp(\action+1),\oprob(\action+1))$, we say that 
$$ (\tp(\action), \oprob(\action)) \lR  (\tp(\action+1),\oprob(\action+1))  $$
if the sequence of $\statedim\times \statedim$  matrices $ \copomat^{j,\action,\obs}$, $j = 1,\ldots, \statedim-1$,
$\obs \in \obspace$ are copositive.\footnote{A symmetric matrix $M$ is positive semidefinite if
$x^\p M x \geq 0$ for any vector $x$. In comparison, $M$ is copositive if $\belief^\p M \belief \geq 0$ for any probability vector 
$\belief$. 
 Clearly if a symmetric matrix is positive  definite then it is copositive. Thus copositivity is a weaker condition than positive definiteness. }
  That is, 
\beq  \belief^\prime\copomat^{j,\action,\obs}\belief  \ge 0, \forall \belief \in \Belief, \text{ for each }  j,  \obs. 
\label{eq:copositivestr}
\eeq
\end{definition}
The above notation  $(\tp(\action), \oprob(\action)) \lR  (\tp(\action+1),\oprob(\action+1))$ is intuitive since it will be shown below
that the copositive condition (\ref{eq:copositivestr})
is a necessary and sufficient condition for the HMM filter update to satisfy $\filter(\belief,\obs,\action ) \lr \filter(\belief,\obs,\action+1)$
for any posterior $\belief \in \Belief$ and observation $\obs \in \obspace$. This will be denoted as Assumption \ref{A4f} below.

We are also interested in the special
case of the optimal HMM predictors instead of optimal filters. Recall that optimal prediction is a special
case of filtering obtained by choosing  non-informative observation probabilities, i.e., all elements
of $\oprob_{\state,\obs}$ versus $\state$ are  identical. In analogy to Definition \ref{def:copositiveall} 
we make the following definition.

\begin{definition}[Copositive Ordering of Transition Matrices]  \label{def:lR}
Given $\tp(\action)$ and $\tp(\action+1)$, we say that   \index{copositive matrix}
$$ \tp(\action) \lR \tp(\action+1)  $$
if the sequence of $\statedim \times \statedim$ matrices $\copomat^{j,\action}$, $j=1\,\ldots,\statedim-1$ are copositive, i.e., 
\begin{align}
\belief^\p  \copomat^{j,\action} \belief & \geq 0 ,  \quad \forall \belief \in \Belief, \quad \text{ for each } j,
\text{ where }   \label{eq:tpdominance} \\ & \hspace{-1.2cm}
 \copomat^{j,\action} = \cfrac{1}{2}\left[\gamma^{j,\action}_{mn} + \gamma^{j,\action}_{nm}\right]_{\statedim\times\statedim},
\;
\gamma^{j,\action}_{mn} =  \tp_{m,j}(\action)\tp_{n,j+1}(\action+1) - \tp_{m,j+1}(\action)\tp_{n,j}(\action+1) . \nn
\end{align}
\end{definition}

\section{Monotone Properties of Optimal Filter}  \index{HMM filter structural results! monotone properties|(}
With the above definitions, we can now give sufficient conditions for the optimal filtering recursion $\filter(\belief,\obs,\action)$ to be monotone
with respect to the MLR order.
The following are the main assumptions. 

\begin{myassumptions}
\item[\nl{A2}{(F1)}]  $\oprob(\action)$ with elements $\oprob_{\state,\obs}(\action)$ is  TP2 for each $\action \in \actionspace$ (see Definition \ref{def:tp2}).  \index{TP2 matrix}
\item[\nl{A3}{(F2)}] $\tp(\action)$ is TP2 for each action $\action \in \actionspace$.
\item[\nl{A4f}{(F3)}]  $ (\tp(\action), \oprob(\action)) \lR  (\tp(\action+1),\oprob(\action+1)) $  (copositivity condition in Definition \ref{def:copositiveall}).
\item[\nl{A4}{(F3')}]  All elements of the  matrices $\copomat^{j,\action,\obs}$, are non-negative. (This is  sufficient\footnote{Any square matrix $M$ with 
 non-negative elements is copositive since $\belief^\p M \belief$ is  always non-negative
for any belief $\belief$. So a sufficient  
condition for the copositivity (\ref{eq:copositivestr}) is that the individual elements $\frac{1}{2} \left(\gamma^{j,\action,y}_{mn} + 
\gamma^{j,\action,y}_{nm}\right) \geq 0$.}
 for \ref{A4f}.)

\item[\nl{A4pf}{(\underline{F3})}]  $ \tp(\action) \lR \tp(\action+1) $  (copositivity condition in Definition \ref{def:lR}).
\item[\nl{A4p}{(\underline{F3}')}]  All the elements of $\copomat^{j,\action}$ are non-negative. (This is sufficient 
 for \ref{A4pf}).


\item [\nl{A5}{(F4)}] $\sum_{\obs \le \bar{\obs}}\sum_{j \in \statespace}\left[\tp_{i,j}(\action)\oprob_{j,\obs}(\action)
 - \tp_{i,j}(\action+1)\oprob_{j,\obs}(\action+1)\right] \le 0 $ for all $i \in \statespace$ and $\bar{\obs}\in \obspace$.
\end{myassumptions}

Assumptions \ref{A2}, \ref{A3} deal with the transition and observation probabilities for a fixed action $\action$.
In comparison, \ref{A4f}, \ref{A4}, \ref{A4pf}, \ref{A4p}, \ref{A5} are conditions on the transition and observation
probabilities of two different HMMs corresponding to the  actions $\action$ and $\action+1$.

{\bf Main Result}:
 The following theorem  is the main result of this chapter. The theorem
characterizes  how the HMM filter $\filter(\belief,\obs,\action)$ and normalization measure
$\filterd(\belief,\obs,\action)$ behave with increasing $\belief$, $\obs$ and $\action$.
The theorem  forms the basis of all the structural results for POMDPs presented in subsequent
chapters. \index{filter! monotonicity w.r.t.\ observation, posterior and action}

\begin{theorem}[Structural result for filtering] \label{thm:filterstructure}
Consider the HMM filter $\filter(\belief,\obs,\action)$ and normalization measure
$\filterd(\belief,\obs,\action)$ defined as
\beq\begin{split}
&\filter(\belief,\obs,\action)  = \frac{\oprob_{\obs}(\action) \tp^\p(\action) \belief}{\filterd(\belief,\obs ,\action)}, 
\;  \filterd(\belief,\obs,\action) = \one^\p \oprob_{\obs}(\action) \tp^\p(\action) \belief,  \text{ where }  \\
& \oprob_\obs(\action)   =  \diag(\oprob_{1\obs}(\action),\ldots,\oprob_{\statedim y}(\action)), \quad \action \in \actionspace= \{1,2, \ldots,\actiondim\}, \obs \in \obspace.
\end{split} \label{eq:hmmstruc2}\eeq
%
%
Suppose $\belief_1,\belief_2 \in \Belief$ are arbitrary belief states. Then
\begin{compactenum}
\item   \label{1} \begin{compactenum}
\item \label{1a} For  $\belief_1 \gr \belief_2$,   the HMM predictor satisfies
$\tp^\p(\action) \belief_1 \gr \tp^\p(\action) \belief_2$
iff \ref{A3} holds.
 \item   \label{1b}  Therefore, for $\belief_1 \gr \belief_2$,   the HMM filter satisfies
 $\filter(\belief_1,y,u)\gr \filter(\belief_2,y,u)$ for any observation $\obs$  iff \ref{A3} holds
(since MLR dominance is preserved by Bayes rule, Theorem \ref{thm:mlrbayes}).
\end{compactenum}
 \item  \label{2}  Under \ref{A2}, \ref{A3},  $\belief_1 \gr \belief_2$ implies  
 $\filterd(\belief_1,u) \gs \filterd(\belief_2,u)$ where  $$\filterd \left(\belief,\action \right) \equiv \left[\filterd \left(\belief, 1, \action \right), \cdots,\filterd \left(\belief,\obsdim,\action \right)\right].$$
 \item  \label{3} 
For $y,\bar{y} \in \obspace$,  $y > \bar{y}$ implies $\filter(\belief_1,y,u)\gr \filter(\belief_1,\bar{y},u)$ iff 
\ref{A2} holds.
\item \label{4} Consider two HMMs  $(\tp(\action),\oprob(\action) )$ and  $(\tp(\action+1),\oprob(\action+1) )$.  Then
\begin{compactenum} \item   \label{4a} $\filter(\belief,\obs,\action+1) \gr \filter(\belief,\obs,\action)$ iff \ref{A4f} holds.
\item  \label{4b} 
Under \ref{A4}, $\filter(\belief,\obs,\action+1) \gr \filter(\belief,\obs,\action)$.
\end{compactenum}
\item \label{5} Consider two HMMs $(\tp(\action),\oprob)$ and  $(\tp(\action+1),\oprob )$. Then
 \begin{compactenum}
\item  \label{5a}  \ref{A4pf} is  necessary and sufficient 
 for $\tp^\p(\action+1) \belief \gr \tp^\p(\action) \belief$. 
 \item  \label{5b} \ref{A4p} is sufficient  for $\tp^\p(\action+1) \belief \gr \tp^\p(\action) \belief$.
 \item  \label{5c} Either \ref{A4pf} or \ref{A4p} are sufficient for the optimal filter to satisfy  $\filter(\belief,\obs,\action+1) \gr \filter(\belief,\obs, \action)$ for any $\obs \in \obspace$. 
\end{compactenum}
\item \label{6} Under \ref{A5}, $ \filterd(\belief,\action+1) \gs \filterd(\belief,\action)$.
\item Statement \ref{4} holds for discrete valued observation space $\obspace$. All the other statements hold for discrete and continuous-valued $\obspace$.
\end{compactenum}

\end{theorem}
The proof is the appendix \secn  \ref{sec:prooffilterstructure}.

Statement \ref{1a}  asserts that a TP2 transition matrix is sufficient for a one-step ahead HMM predictor  to preserve MLR stochastic dominance with respect to $\belief$. \index{TP2 matrix}
As a consequence Statement \ref{1b} holds since applying Bayes rule to $\tp^\p(\action) \belief_1$ and $\tp^\p (\action)\belief_2$,
respectively,  yields
the filtered updates, and Bayes rule 
preserves MLR dominance (recall Theorem \ref{thm:mlrbayes}).

Statement \ref{2} asserts that  the normalization measure  \index{TP2 matrix} of the HMM filter
$\filterd(\belief,\obs,\action)$  is  monotone increasing in $\belief$ (with respect to first order dominance) if the   observation kernel is TP2  \ref{A2} and  transition matrix is TP2  \ref{A3}.
  \index{TP2 matrix}

Statement \ref{3} asserts that the HMM filter $\filter(\belief,\obs,\action)$ is monotone increasing in the observation $\obs$ iff \ref{A2} holds.
That is,  a larger observation yields a larger belief if and only if the observation kernel is TP2.

Finally Statements \ref{4}, \ref{5} and \ref{6} compare the filter update and normalization measures for two different HMMs index by actions $\action$ and $\action+1$.
Statements \ref{4} and \ref{6} say that if  $(\tp(\action),\oprob(\action))$ and $(\tp(\action+1),\oprob(\action+1))$
satisfy the specified conditions, then the belief update and normalization term with parameters 
$(\tp(\action+1),\oprob(\action+1))$  dominate those with parameters $(\tp(\action),\oprob(\action))$. Statement \ref{5} gives a similar result
for predictors  and HMMs with identical observation probabilities.
 \index{HMM filter structural results! monotone properties|)}

\section{Illustrative Example}
This section gives simple examples to illustrate Theorem \ref{thm:filterstructure}. Suppose
$$ \tp(1) = \begin{bmatrix} 0.6 & 0.3 & 0.1 \\ 0.2 & 0.5 & 0.3 \\ 0.1 & 0.3 & 0.6 \end{bmatrix}, 
\tp(2) = \begin{bmatrix} 0 & 1 & 0 \\ 1 & 0 & 0 \\ 0 & 0 & 1 \end{bmatrix}, 
\belief_1 =  \begin{bmatrix} 0.2 \\ 0.2 \\ 0.6 \end{bmatrix},
\belief_2 = \begin{bmatrix} 0.3 \\ 0.2 \\ 0.5 \end{bmatrix}. 
$$
It can be checked that the transition matrix $\tp(1)$ is TP2 (and so \ref{A3} holds).
$\tp(2)$ is not TP2 since the second order minor comprised of the (1,1), (1,2) (2,1) (2,2) elements is
$-1$.
Finally,  $\belief_1 \gr \belief_2$ since the ratio of their elements  $[2/3,1, 6/5]$ is increasing.

\noindent {\em Example (i)}: Statement \ref{1a}  says that $\tp^\p(1) \belief_1 \gr \tp^\p(1) \belief_2$ which can be verified since the ratio of elements  $[0.8148, 1, 1.1282]^\p$
is increasing.
On the other hand since $\tp(2)$ is not TP2, $\tp^\p(2) \belief_1 $ is not MLR smaller than $\tp^\p(2) \belief_2$. The ratio
of elements of $\tp^\p(2) \belief_1 $
with   $\tp^\p(2) \belief_2$ is $[1, 0.6667, 1.2]$  implying that they are not
MLR orderable (since the ratio is neither increasing or decreasing).

Statement \ref{1b} (MLR order is preserved by Bayes rule) was illustrated numerically in \secn \ref{sec:fsdmlr}.

\noindent{\em Example (ii)}:  To illustrate Statement \ref{2},  suppose $\oprob(1) = \tp(1)$ so that \ref{A2}, \ref{A3} hold.  Then 
$$\filterd(\belief_1,1) = [0.2440, 0.3680, 0.3880]^\p, \quad \filterd(\belief_2,1)
=[0.2690, 0.3680, 0.3630]^\p .$$ Clearly $\filterd(\belief_1,1) \gs \filterd(\belief_2,1)$.

\noindent{\em Example (iii)}: 
Regarding Statement \ref{3}, if $\oprob = \tp(1)$ then $\oprob_{1} = \diag(0.6,0.2.0.1)$,  $\oprob_{2} = \diag(0.3,0.5.0.3)$.
Then writing $\filter(\belief,\obs,\action)$ as $\filter(\belief,\obs)$, $$\filter(\belief_1,\obs=1) = [0.5410, 0.2787, 0.1803]^\p, \;
\filter(\belief_1,\obs=2) = [0.1793, 0.4620, 0.3587]^\p$$  implying that $\filter(\belief_1,\obs=1) \lr \filter(\belief_1,\obs=2)$.

\noindent {\em Example (iv)}: 
 Consider a cost vector 
$$ \cost = [\cost(x=1),\cost(x=2),\cost(x=3)]^\p  = \begin{bmatrix} 3 &   2 &   1 \end{bmatrix}^\p .$$
Suppose a random variable $x$ has a prior $\belief$ and is observed via noisy observations $y$ with observation matrix $\oprob(1)$ above. Then the expected cost
after observing $y$ is $\cost^\p \filter(\belief,\obs) $. It is intuitive that a larger  observation $\obs$ corresponds to a larger state and therefore a smaller
expected cost (since the cost vector is decreasing in the state). From Theorem \ref{thm:filterstructure}(\ref{3}) this indeed is the case since
$\filter(\belief,\obs) \gr \filter(\belief,\bar{\obs})$ for $\obs > \bar{\obs}$ which implies that $\filter(\belief,\obs) \gs \filter(\belief,\bar{\obs})$ and
therefore $\cost^\p \filter(\belief,\obs)  $ is decreasing in $y$.



\section{Discussion and Examples of Assumptions \ref{A2}-\ref{A5}} \label{sec:assumpdiscussion}
Since Assumptions \ref{A2}-\ref{A5} will be used a lot in subsequent chapters, we now discuss their motivation with examples.
\paragraph{Assumption \ref{A2}}
\ref{A2} is required
for preserving the MLR ordering with respect to observation $\obs$ of the Bayesian filter update.
\ref{A2}
is satisfied by numerous continuous and discrete distributions, see any  classical
detection theory book such as \cite{Poo93}. 
Since \ref{A2} is equivalent to each row of $\oprob$ being MLR dominated by subsequent rows, any of the examples
in \secn \ref{sec:fsdmlr} yield TP2 observation kernels. For example, if the $i$-th row of $\oprob$  is $\normal(\obs-\state; \mu_i, \sigma^2)$
with $\mu_i < \mu_{i+1}$ 
then $\oprob$ is TP2. The same logic applies to 
 Exponential, Binomial, Poisson, Geometric, etc.
For a discrete distribution example, suppose each sensor obtains  measurements
$y$ of  the state $x$ in  quantized Gaussian noise. Define 
\beq \prob(\obs|\state=i) = \frac{\bar{b}_{iy}}{\sum_{y=1}^Y \bar{b}_{iy}} \text{ where } 
 \bar{b}_{iy} = \frac{1}{\sqrt{2 \belief \Sigma}}
\exp \biggl( - \frac{1}{2} \frac{(y - g_i)^2}{ \Sigma} \biggr) 
\label{eq:gaussnoise}\eeq
Assume  the state levels $g_i$ are increasing in $i$.
Also,  $\Sigma \geq 0$ denotes the noise variance  and 
reflects the quality of the measurements.
It is easily verified that  
 (A2) holds.
As another example, consider equal dimensional observation and state spaces ($X=Y$)
and suppose $P(y=i|x=i) = p_i$, $P(y=i-1|x=i) = P(y=i+1|x=i) = (1 - p_i)/2$. Then for $1/(\sqrt{2}+1)
\leq  p_i \leq 1$, (A2) holds.

\paragraph{Assumption \ref{A3}}
\ref{A3} is essential for  the Bayesian update $T(\belief,y,\action)$ 
 preserving monotonicity with respect to $\belief$. TP2 stochastic orders and kernels  have been
 studied in great detail in \cite{KR80}. \index{TP2 matrix}
\ref{A3} is
 satisfied by several classes of  transition matrices; see  \cite{Kij97,KK77}.
 
 The left-to-right Bakis  HMM used in speech recognition \cite{Rab89}  has an upper triangular transition matrix which has a 
TP2 structure under mild conditions, e.g.,  if the upper triangular elements in row $i$ are $(1- \tp_{ii} )/(\statedim-i)$ then $\tp$ is TP2 if
$\tp_{ii} \leq 1/(\statedim-i)$.
 
As another example, consider  a  tridiagonal transition probability matrix $\tp$ with
 $\tp_{ij} = 0$ for $ j\geq i+2$ and $j \leq i-2$. As shown in
 \cite[pp.99--100]{Gan60}, a necessary and sufficient condition for
 tridiagonal $\tp$ to be TP2 is that $\tp_{ii} \tp_{i+1,i+1} \geq \tp_{i,i+1}
 \tp_{i+1,i} $. 

 Karlin's classic book  \cite[pp.154]{KT81} shows that
the matrix exponential of any tridiagonal generator matrix is TP2.
That is,  $\tp = \exp(Q t)$ is TP2 if  $Q$ is a tridiagonal generator 
matrix (nonnegative off-diagonal entries and each  row adds to $0$) and $t>0$.

The following lemmas give  useful properties  of TP2 transition matrices.  \index{TP2 matrix}
\begin{lemma} If $P$ is TP2, i.e., \ref{A3} holds, then  $P_{11} \geq P_{21} \geq P_{31} \geq \cdots \geq P_{X1}$.  \label{lem:tp21}
\end{lemma}
\begin{proof} We prove the 
contrapositive, that is, $P_{i1}<P_{i+1,1}$ implies $P$ is not TP2.
Recall from (A3-Ex1), TP2 means that $P_{i1} P_{i+1,j} \geq P_{i+1,1} P_{ij}$ for all $j$. So assuming  $P_{i1}<P_{i+1,1}$, to show
that $P$ is not TP2, we need to show that there is at least one $j$ such that
$P_{i+1,j} < P_{ij}$. But $P_{i1}<P_{i+1,1}$ implies $\sum_{k\neq 1} P_{i+1,k} < \sum_{k\neq 1} P_{ik}$, which in turn implies that 
at least for one $j$, $P_{i+1,j} < P_{ij}$. \end{proof}

\begin{lemma} The product of two TP2 matrices is TP2. \label{thm:tp2product}
\end{lemma}

Lemma \ref{thm:tp2product} lets us construct TP2 matrices by multiplying other TP2 matrices. The lemma is also used in the proof
of the main Theorem \ref{thm:filterstructure}.

\paragraph{Assumption \ref{A4f}, \ref{A4} and  Optimal Prediction} Assumption \ref{A4} is sufficient condition  for  the belief due to action $\action+1$ to  MLR dominate the  belief due to action $\action$, i.e., in the terminology of
  \cite{Mil81}, $\action+1$ yields a more ''favorable outcome"  than~$\action$. 
  In general, the problem of verifying   copositivity of a matrix is NP-complete  \cite{BD09}. Assumption \ref{A4} is a simpler but more restrictive sufficient condition than \ref{A4f}  to ensure that $\copomat^{j,\action,\obs}$ in \eqref{eq:copositivestr} is copositive.
Here is an example of $ (\tp(1), \oprob(1)) \lR  (\tp(2),\oprob(2)) $ which satisfies  \ref{A4}:
\begin{align*}
\tp(1)& = \begin{bmatrix} 0.8000  &   0.1000  &   0.1000 \\
    0.2823  &   0.1804   &  0.5373 \\ 
    0.1256  &   0.1968  &   0.6776\end{bmatrix}, 
    \oprob(1) = \begin{bmatrix}
      0.8000  &   0.1000 &    0.1000 \\
    0.0341   &  0.3665  &   0.5994 \\
    0.0101  &   0.2841&     0.7058
\end{bmatrix}  \\
\tp(2) & =
\begin{bmatrix}
    0.0188    & 0.1981  &   0.7831 \\
    0.0051   &  0.1102 &    0.8847 \\
    0.0016   &  0.0626   &  0.9358
\end{bmatrix},
 \oprob(2) = \begin{bmatrix}
 0.0041  &   0.1777 &   0.8182 \\
    0.0025    &  0.1750    &  0.8225 \\
    0.0008   &   0.1290  &    0.8701
 \end{bmatrix}.
\end{align*}

\ref{A4pf} is necessary and sufficient for   the optimal predictor with transition matrix $\tp(\action+1)$ to MLR dominate the optimal predictor
with transition matrix $\tp(\action)$.  \ref{A4p} is a sufficient condition for \ref{A4p} since it requires all the elements of the matrix to be non-negative
which trivially implies copositivity. We require MLR dominance of the predictor since then by
 Theorem \ref{thm:mlrbayes} MLR dominance of the filter is assured for any observation distribution.  (First order dominance is not closed under Bayesian
 updates).
  
  A straightforward sufficient  condition for (\ref{eq:tpdominance}) to hold is 
  if all rows of $\tp(\action+1)$ MLR dominate the last row of $\tp(\action)$.

\paragraph{Assumption \ref{A5}} This ensures that the normalized measure $\filterd(\belief,\action+1)$  first order stochastically
dominates $\filterd(\belief,\action)$.

 Assumptions \ref{A4} and \ref{A5} are 
relaxed versions of 
Assumptions (c), (e), (f) of \cite[Proposition 2]{Lov87} and Assumption (i) of \cite[Theorem 5.6]{Rie91} 
in the stochastic control literature.
The assumptions (c), (e), (f) of \cite{Lov87} require that $\tp(\action+1)$ $ \gtp \tp(\action)$ and $\oprob(\action+1) \gtp \oprob(\action)$ (where $\gtp$ denotes TP2 stochastic order)
which is impossible for stochastic matrices, unless $\tp(\action) =  \tp(\action+1)$, 
  $\oprob(\action) =\oprob(\action+1)$ or the matrices $\tp(\action), \oprob(\action)$ are rank 1 for all $\action$ meaning that the observations are non-informative. 


\section{Example: Reduced Complexity HMM Filtering with Stochastic Dominance Bounds} \index{filter! reduced complexity using stochastic dominance|(}  \label{sec:filterlowbound}
The main result Theorem \ref{thm:filterstructure} 
can be exploited to design reduced complexity HMM filtering algorithms with provable  sample path bounds.
In this section we derive such reduced-complexity  algorithms by using Assumptions \ref{A3} and \ref{A4p} with statement
\ref{5} of Theorem \ref{thm:filterstructure}
 for the transition matrix. 

\subsection{Upper and Lower Sample Path  Bounds for Optimal Filter}
 \index{HMM filter structural results! reduced complexity filter|(}

Consider an HMM with $\statedim \times \statedim$ transition matrix $\tp$ and observation matrix $\oprob$ with elements
$\oprob_{\state\obs }= \pdf(\obs_k = \obs|\state_k=\state)$. The  observation space $\obspace$ can be discrete or continuous-valued;
so that $\oprob$ is either a pdf or a pmf.
The HMM filter computes the posterior 
\beq
\belief_{k+1} = \filter(\belief_k,\obs_{k+1};\tp ),   \text{ where } 
\filter(\belief,\obs;\tp) =  \frac{\oprob_{\obs} \tp^\p \belief} { \one^\p  \oprob_{\obs} \tp^\p \belief}, \;
 \oprob_{\obs} = \diag(\oprob_{1\obs},\ldots,\oprob_{\statedim y}).
\label{eq:hmmfilterexplicit}
\eeq
The above notation  explicitly shows the dependence on the transition matrix $\tp$.
Due to the  matrix-vector multiplication $\tp^\p \belief$, the HMM filter involves
$O(\statedim^2)$ multiplications and  can be excessive for large $\statedim$.

The main idea of this section is to construct low rank transition matrices $\ltp$ and $\utp$ such that the above filtering recursion using these matrices form lower and upper bounds to 
 $\belief_k$ in the MLR stochastic dominance sense. Since  $\ltp$ and $\utp$ are low rank (say $\lowdim$),
the  cost involved in computing these lower and upper bounds to $\belief_k$ at each time $k$ will be   $O(X\lowdim)$ 
 where $\lowdim \ll \statedim$. 

Since that plan is  to compute filtered estimates using $\ltp$ and $\utp$ instead of the original transition matrix $\tp$, we need  additional notation to distinguish between the posteriors and estimates computed using $\tp$, $\ltp$ and $\utp$.  Let 
\begin{align*}
\underbrace{ \belief_{k+1} = \filter(\belief_k,\obs_{k+1};\tp)}_{\text{optimal}}, \quad
\underbrace{\upbelief_{k+1} = \filter(\upbelief_k,y_{k+1};\utp)}_{\text{upper bound}}, \quad \underbrace{\lbelief_{k+1} =  \filter(\lbelief_k,y_{k+1};\ltp)}_{\text{lower bound}} \end{align*}
denote the posterior  updated using the optimal filter  (\ref{eq:hmmfilterexplicit}) with transition matrices $\tp$, $\utp$ and $\ltp$, respectively.
Assuming that the state levels of the Markov chain are $\levels = (1,2,\ldots,\statedim)^\p$, the conditional mean estimates of the underlying state computed using 
$\tp$, $\ltp$ and $\utp$, respectively, will be denoted as
\begin{multline}  \label{eq:lmean}
\mean_k =  \E\{\state_k|\obs_{0:k};\tp\} =  \levels^\p \belief_k, \quad \lmean_k \ole \E\{\state_k|\obs_{0:k};\ltp\} =  \levels^\p \lbelief_k, \\
\umean_k \ole \E\{\state_k|\obs_{0:k};\utp\} =  \levels^\p \upbelief_k.
\end{multline}
Also denote the maximum aposteriori  (MAP) state estimates computed using $\ltp$ and $\utp$  as
\beq \label{eq:lmap}
\map \ole\argmax_i \belief_k(i), \quad  \lmap_k \ole \argmax_i \lbelief_k(i), \quad \umap_k \ole \argmax_i \upbelief_k(i).
\eeq

The following is the main result of this section (proof in \cite{Kri16}).

\begin{theorem}[Stochastic Dominance Sample-Path Bounds] \label{thm:mainsdf}
Consider  the HMM  filtering updates $\filter(\belief,\obs;\tp)$, $\filter(\belief,y;\utp)$ and $\filter(\belief,y;\ltp)$ where $\filter(\cdot)$ is defined in (\ref{eq:hmmfilterexplicit}) and $\tp$ denotes the transition
matrix of the HMM.
\begin{compactenum}
\item For any transition matrix $\tp$, 
there exist transition  matrices 
$\ltp$ and $\utp$ such that $\ltp \lR \tp \lR \utp$ (recall $\lR$ is the copositive ordering defined in Definition \ref{def:lR}).

\item Suppose transition matrices $\ltp$ and $\utp$ are constructed such that  $\ltp \lR \tp \lR \utp$.
Then  for any $\obs $ and $\belief \in \Belief$, the filtering updates satisfy the sandwich	 result
$$\filter(\belief,y;\ltp) \lr \filter(\belief,\obs;\tp) \lr \filter(\belief,y;\utp).$$
\item Suppose $\tp$  is TP2 (Assumption \ref{A3}).  Assume the  filters $\filter(\belief,\obs;\tp)$, $\filter(\belief,y;\utp)$ and $\filter(\belief,y;\ltp)$
 are initialized with common prior $\belief_0$. Then the posteriors satisfy   \index{TP2 matrix}
 $$ \lbelief_k \lr \belief_k \lr \upbelief_k, \quad \text{ for all time } k = 1,2,\ldots $$
 As a consequence for all time $k=1,2,\ldots$, 
 \begin{compactenum}
\item The conditional mean state estimates  defined in  (\ref{eq:lmean}) satisfy $\lmean_k \leq \mean_k \leq \umean_k$.
\item The  MAP state estimates defined in  (\ref{eq:lmap}) satisfy $\lmap_k \leq \map_k \leq \umap_k$. 
\end{compactenum}
 \end{compactenum} \qed
\end{theorem}

Statement 1 says that for any transition matrix $\tp$, there always exist  transition matrices $\ltp$ and $\utp$ such that $\ltp \lR \tp \lR \utp$ (copositivity dominance). 
Actually if $\tp$ is TP2, then one can trivially construct the tightest rank 1 bounds $\ltp$ and $\utp$ as shown below.

Given  existence of $\ltp$ and $\utp$, the next step is to  optimize the choice of $\ltp$ and $\utp$. This is discussed in
\secn \ref{sec:convex} where nuclear norm minimization is used to construct sparse eigenvalue matrices $\ltp$ and $\utp$.

Statement 2 says that  for any prior $\belief$ and observation $\obs$, the one step filtering updates using $\ltp$ and $\utp$ constitute lower and upper bounds
to the original filtering problem. This is simply  a consequence of \ref{A4p} and Statement \ref{5} of Theorem \ref{thm:filterstructure}.

Statement 3 globalizes Statement 2 and asserts  that with the additional assumption that the transition matrix $\tp$ of the original filtering problem is TP2, then the upper and lower bounds hold for all time.
Since MLR dominance implies first order stochastic dominance (see Theorem \ref{lem:mlrfo}), the conditional mean estimates satisfy $\lmean_k \leq \mean_k \leq \umean_k$.

\subsection{Convex Optimization to Compute Low Rank Transition Matrices} \label{sec:convex}

It only remains to 
 give algorithms for constructing  low rank transition matrices $\ltp$ and $\utp$
 that yield the lower and upper bounds $\lbelief_k$ and $\upbelief_k$ for the optimal filter posterior $\belief_k$.
 These involve
  convex optimization \cite{FHB03} for minimizing the nuclear norm.
{\em The computation of 
 $\ltp$ and $\utp$ is independent of the observation sample path and  so the associated   computational cost
 is irrelevant to the real time  filtering}. 
  Recall that the motivation is as follows:
 If $\ltp$ and $\utp$ have rank~$\lowdim$, then the computational
cost of the filtering recursion is $O(\lowdim\statedim)$ instead of $O(\statedim^2)$ at each time~$k$.

\subsubsection{Construction of $\ltp,\utp$ without rank constraint} 
Given a TP2 matrix $\tp$, the transition matrices $\ltp$ and $\utp$ such that
$\ltp \lR \tp \lR  \utp$ can be constructed straightforwardly via an LP solver.
With $\ltp_1,\ltp_2,\ldots,\ltp_\statedim$ denoting the rows of $\ltp$, a sufficient condition for $\ltp \lR \tp$ is that 
$\ltp_i \lr \tp_1$ for any row $i$.  
Hence, the rows $\ltp_i$ satisfy linear constraints with respect to $\tp_1$ and can be straightforwardly constructed via an LP solver.
A similar construction holds for the upper bound $\utp$, where it is sufficient to construct $\utp_i \gr \tp_X$.

{\em Rank 1 bounds}:
If $\tp$ is TP2, 
an obvious  construction is to construct $\ltp$ and $\utp$ as follows: Choose rows $\ltp_i = \tp_1$ and $\utp_i = \tp_\statedim$ for $i=1,2,\ldots,\statedim$. These yield rank 1 matrices $\ltp$ and $\utp$.
It is clear from Theorem~\ref{thm:mainsdf} that $\ltp$ and $\utp$ constructed in this manner are the tightest rank 1 lower and upper bounds.

\subsubsection{Nuclear Norm Minimization Algorithms to Compute Low Rank Transition Matrices $\ltp$, $\utp$} \label{sec:alg}
This subsection
constructs $\ltp$ and $\utp$    as  low rank transition matrices  subject to the condition
$\ltp \lR \tp \lR \utp$. 
To save space we consider  the lower bound  transition matrix $\ltp$; construction of $\utp$ is similar.
Consider  the following   optimization problem for $\ltp$:
\begin{align}
 \text{  Minimize rank of  }  \statedim\times \statedim \text{ matrix }  \ltp   
\label{eq:obj0} 
\end{align}

subject to the constraints $\cons(\Belief,\ltp,m)$ for $ m = 1,2,\ldots, \statedim-1 $, where for $\epsilon>0$,  
{ \begin{subequations}
\begin{empheq}[left={  \cons(\Belief,\ltp,m) \equiv } \empheqlbrace]{align}
 & 
  \copomat^{(m)}   \text{ is  copositive on $\Belief$ } \label{eq:con1} \\
&  \| \tp^\p \belief - \ltp^\p \belief \|_1 \leq \epsilon \text{ for all } \belief \in \Belief \label{eq:con2} \\
&  \ltp \geq 0 , \quad \ltp \one = \one.  \label{eq:con3} \end{empheq} \\
\end{subequations}}
Recall  $\copomat$ is defined in (\ref{eq:tpdominance}) and  (\ref{eq:con1}) is equivalent to $ \ltp \lR \tp$.
The constraints $\cons(\Belief,\ltp,m)$ are convex in matrix $\ltp$, since  (\ref{eq:con1})  and (\ref{eq:con3})  are linear in the elements of $\ltp$, and  (\ref{eq:con2}) is convex (because
 norms are convex).
The constraints (\ref{eq:con1}), (\ref{eq:con3}) are exactly the conditions of Theorem \ref{thm:mainsdf}, namely that $\ltp$ is a stochastic matrix 
satisfying $\ltp \lR \tp$.

The  convex constraint (\ref{eq:con2}) is equivalent to
$\|\ltp-\tp\|_1 \leq \epsilon$, where $\|\cdot\|_1$ denotes the induced 1-norm for matrices.\footnote{
The three  statements  $\| \tp^\p \belief - \ltp^\p \belief \|_1 \leq \epsilon$,  $\|\ltp-\tp\|_1 \leq \epsilon$ and $\sum_{i=1}^\statedim \| (\tp^\p - \ltp^\p)_{:,i} \|_1 \belief(i)  \leq \epsilon$
are all equivalent since $\|\belief\|_1 = 1$ because $\belief$ is a probability vector (pmf)}

To solve the above problem, we proceed in two steps:
\begin{compactenum}
\item The objective (\ref{eq:obj0}) is replaced with the reweighted nuclear norm (see \secn \ref{sec:nuclear} below). 
\item
Optimization over the copositive cone (\ref{eq:con1}) is achieved via a sequence of simplicial decompositions (see remark
at end of \secn \ref{sec:nuclear}. 
\end{compactenum}

\subsubsection{Reweighted Nuclear Norm} \label{sec:nuclear}  \index{sparsity!  transition matrix rank (nuclear norm)}
Since 
  the	 rank	   is a non-convex function of a matrix,  direct minimization of the rank (\ref{eq:obj0}) is  computationally intractable.
Instead,  we follow the approach developed by Boyd and coworkers \cite{FHB03} to minimize the  iteratively reweighted  nuclear norm.
Inspired by  Cand{\`e}s and Tao \cite{CT09}, there has been much recent interest in minimizing nuclear norms for constructing matrices with sparse eigenvalue sets
or equivalently low rank. Here we compute $\ltp, \utp$ by
minimizing their nuclear norms subject to copositivity conditions that ensure $\ltp \lR \tp \lR \utp$. 

Let  $\| \cdot \|_*$ denote the nuclear norm, \index{nuclear norm} which corresponds to the sum of the singular values of a matrix,
The re-weighted nuclear norm minimization
 proceeds as a {\em sequence} of convex optimization problems indexed by $n=0,1,\ldots$.
Initialize $\ltp^{(0)} = I$. For $n=0,1,\ldots$, compute $\statedim \times \statedim$ matrix
\begin{align}
\ltp^{(n+1)} &= \argmin_{\ltp} \| \underline{W}_1^{(n)} \ltp \; \underline{W}_2^{(n)} \|_*   \label{eq:obj} \\
\text{ subject to: } & \text{ constraints  $\cons(\Belief,\ltp,m)$, $m=1,\ldots,\statedim-1$ }  \nn \\ & \text{ namely,  (\ref{eq:con1}), (\ref{eq:con2}), (\ref{eq:con3}). } \nonumber
\end{align}
Notice that at iteration $n+1$, the previous estimate, $\ltp^{(n)}$ appears in  the cost function of \eqref{eq:obj} in terms of
weighting matrices
 $\underline{W}_1^{(n)}$, $\underline{W}_2^{(n)}$. These 
weighting matrices  are evaluated iteratively as
\begin{align}
\underline{W}_1^{(n + 1)} &= ([\underline{W}_1^{(n)}]^{-1} U \Sigma U^T [\underline{W}_1^{(n)}]^{-1} + \delta I)^{-1/2}, \nn
\\
\underline{W}_2^{(n + 1)} &= ([\underline{W}_2^{(n)}]^{-1} V \Sigma V^T [\underline{W}_2^{(n)}]^{-1} + \delta I)^{-1/2}. \label{eq:w1w2}\end{align}
Here $\underline{W}_1^{(n)} \ltp^{(n)} \underline{W}_2^{(n)} = U \Sigma V^T$ is a reduced singular value decomposition, starting with $\underline{W}_1^{(0)} = \underline{W}_2^{(0)} = I$ and $\ltp^{0} = \tp$. Also $\delta$ is a small positive constant  in the regularization term $\delta I$.  
In numerical examples of \secn \ref{sec:example}, we used {\tt YALMIP} with {\tt MOSEK} and {\tt CVX}  to solve the above convex optimization problem.

 The intuition behind the reweighting iterations is that as the estimates $\ltp^{(n)}$ converge to the limit $\ltp^{(\infty)}$, the cost function becomes approximately equal to the rank of $\ltp^{(\infty)}$.


\noindent {\em Remark}:
Problem \eqref{eq:obj} is  a convex optimization problem in $\ltp$.
 However, one additional issue needs to be resolved: the constraints \eqref{eq:con1} involve a copositive cone and cannot be solved directly by standard interior point methods.  
To deal with the copositive constraints  \eqref{eq:con1}, one can  use the  state-of-the-art simplicial decomposition
method detailed in~\cite{BD09}, see \cite{KR14} for details.

\subsection{Discussion. Reduced Complexity Predictors}
If one were interested in constructing reduced complexity HMM predictors (instead of filters), the results in this section are straightforwardly relaxed using first order dominance $\ls$
instead of MLR dominance $\lr$  as follows:
Construct $\ltp$  by nuclear norm minimization as in (\ref{eq:obj}), where (\ref{eq:con1}) is replaced by the linear constraints 
$\ltp_i \ls \tp_i$, on the rows $i=1,\ldots,\statedim$, and (\ref{eq:con2}), (\ref{eq:con3}) hold.
Thus the construction of $\ltp$ is a standard convex optimization problem and the bound $\ltp^\p \belief \ls \tp^\p \belief$ holds for the optimal predictor for all $\belief \in \Belief$.

Further, if  $\ltp$ is chosen so that its rows satisfy
the linear constraints $\ltp_i \ls \ltp_{i+1}$, $i=1,\ldots, \statedim-1$, then the following global bound holds for the optimal predictor:
$(\ltp^\p)^k \belief  \ls  (\tp^\p)^k  \belief$ for all time $k$ and $\belief\in \Belief$. 
A similar result holds for the upper bounds in terms of $\utp$.

It is instructive to compare this with the 
filtering case, where we imposed a TP2 condition on $\tp$ for the global bounds of
Theorem \ref{thm:mainsdf}(3) 
  to hold
wrt $\lr$.
We could have  equivalently imposed a TP2 constraint  on $\ltp$ and allow $\tp$ to be arbitrary for the global filtering bounds   to hold, however the TP2
 constraint is non-convex and so it is difficult to then optimize $\ltp$.

Finally, note that 
the  predictor bounds in terms of $\ls$ do not hold if  a filtering update
is performed since $\ls$ is not closed w.r.t.\ conditional expectations.
 \index{HMM filter structural results! reduced complexity filter|)}
\index{filter! reduced complexity using stochastic dominance|)} 
\section{\pwe}
The books \cite{MS02,SS07} give comprehensive accounts of stochastic dominance.  \cite{Whi82} discusses the TP2 stochastic order which is a multivariate generalization of the MLR order. Karlin's book \cite{Kar68} is a classic on totally positive matrices. The classic paper \cite{KR80} studies
multivariate TP2 orders; see also \cite{Rie91}.

The material in \secn  \ref{sec:filterlowbound} is based on \cite{KR14}; where additional numerical results are given.
Also it is shown in \cite{KR14} how the reduced complexity bounds on the posterior can be exploited  by using
a Monte-Carlo importance sampling filter.
The approach in \secn  \ref{sec:filterlowbound}  of
optimizing the nuclear norm as a surrogate for rank has been studied  as a convex optimization problem in several papers \cite{ZV09}.
Inspired by the seminal work of Cand{\`e}s and Tao  \cite{CT09},
there has been much recent interest in minimizing nuclear norms in the context of  sparse matrix completion problems.
  Algorithms for testing for copositive matrices and copositive programming have been studied recently in \cite{BD08,BD09}.
  
There has been extensive work in  signal processing on posterior Cram\'er-Rao bounds for nonlinear filtering  \cite{TMN98}; see also \cite{RAG04} for a textbook treatment.
These yield  lower bounds to the achievable variance of the conditional mean estimate of the optimal filter.
However,  such posterior Cram\'er-Rao bounds do not give  constructive algorithms for computing upper and lower bounds for the
sample path of the filtered distribution. The sample path bounds proposed in this chapter have the attractive feature that 
they are guaranteed to yield lower and upper bounds to both hard and soft estimates of the optimal filter. It would be of interest to develop similar results for jump Markov linear systems,
in particular to use such constraints for particle filtering algorithms \cite{DGK01}.

\begin{subappendices}
\section{Proofs}

\subsection{Proof of Theorem \ref{thm:filterstructure}}  \label{sec:prooffilterstructure}
First recall  from Definition \ref{def:tp2} that $\belief_1 \gr \belief_2 $ is equivalent to saying that the matrix
$\begin{bmatrix} \belief^\p_2 \\ \belief^\p_1\end{bmatrix} $ is TP2. This TP2 notation is  more convenient for proofs.

{\bf Statement  \ref{1a}}  {\em If  \ref{A3} holds, then $\belief_1 \gr \belief_2$ implies $\tp^\p \belief_1 \gr \tp^\p \belief_2$}:
Showing that $\tp^\p \belief_1 \gr \tp^\p \belief_2$ is equivalent to showing that 
$\begin{bmatrix} \belief_2^\p \tp \\ \belief_1^\p \tp \end{bmatrix} $ is TP2. But
$\begin{bmatrix} \belief_2^\p \tp \\ \belief_1^\p \tp \end{bmatrix} =\begin{bmatrix} \belief_2^\p  \\ \belief_1^\p  \end{bmatrix} \tp $.
Also  since $\belief_1 \gr \belief_2$, the matrix  $\begin{bmatrix} \belief_2^\p  \\ \belief_1^\p  \end{bmatrix}$
 is TP2. By \ref{A3}, $\tp$ is TP2.
Since the product of TP2 matrices is TP2 (see Lemma \ref{thm:tp2product}), the result holds.

{\em If  $\belief_1 \gr \belief_2$ implies $\tp^\p \belief_1 \gr \tp^\p \belief_2$ then \ref{A3} holds}: Choose $\belief_1 = e_j$ and $\belief_2 = e_i$ where  $j > i$,
and  as usual $e_i$ is the unit vector with 1 in the $i$-th position. Clearly then $\belief_1 \gr \belief_2$. Also $\tp^\p e_i$ is the $i$-th row of $\tp$.
So $\tp^\p e_j \gr \tp^\p e_i$ implies the $j$-th row is MLR larger than the $i$-th row of $\tp$. This implies that $\tp$ is TP2 by definition \ref{def:tp2}.

{\bf Statement \ref{1b}} follows by applying Theorem \ref{thm:mlrbayes} to \ref{1a}.

{\bf Statement \ref{2}}. Since MLR dominance implies first order dominance,
by \ref{A2},  $\sum_{\obs\geq \bar{\obs}} \oprob_{\state,\obs}(\action) $ is increasing in $\state$.
By \ref{A3}, $(\tp_{i,1},\ldots \tp_{i,\statedim} ) \ls  ( \tp_{j,1},\ldots, \tp_{j,\statedim})$ for $i \leq j$.  Therefore
$\sum_{j} \tp_{ij}(\action) \sum_{\obs \geq \bar{\obs}} \oprob_{j,\obs}(\action)$ is increasing in $i \in \statespace$.
Therefore $\belief_1 \gr \belief_2$ implies $\filterd(\belief_1,\action) \gs \filterd(\belief_2,\action)$.

{\bf Statement \ref{3}}. Denote $\tp^\p(\action) \belief_1 = \bbelief$. Then   $\filter(\belief_1,y,u)\gr \filter(\belief_1,\bar{y},u)$ is equivalent
to 
$$\left( \oprob_{i, \obs} \oprob_{i+1, \bar{\obs}} -  \oprob_{i+1,y} \oprob_{i,\bar{y}} \right) \cancel{\bbelief(i) \bbelief(i+1) }\leq 0, \quad \obs > \bar{\obs}.$$
This is equivalent to $\oprob$ being TP2, namely condition \ref{A2}.

{\bf Statement \ref{4a}}.  By defintion of MLR dominance, $\filter(\belief,\obs,\action) \lr \filter(\belief,\obs,\action+1)$ is equivalent to
$$ \sum_m \sum_n  \oprob_{jy}(\action+1) \oprob_{j+1,y}(\action) \tp_{nj}(\action+1) \belief_n \belief_m
\leq  \sum_{m} \sum_{n} \oprob_{jy}(\action) \oprob_{j+1,y}(\action+1) \tp_{mj}(\action) \tp_{n,j+1}(\action+1) \belief_m \belief_n $$
and also
$$  \sum_m \sum_n  \oprob_{jy}(\action+1) \oprob_{j+1,y}(\action) \tp_{nj}(\action+1) \belief_n \belief_m
\leq  \sum_{m} \sum_{n} \oprob_{jy}(\action) \oprob_{j+1,y}(\action+1) \tp_{nj}(\action) \tp_{m,j+1}(\action+1) \belief_m \belief_n
$$
This is equivalent to \ref{A4}.

{\bf Statement \ref{4b}} follows since \ref{A4} is sufficient for \ref{A4f} .

The proofs of Statement \ref{5} and \ref{6} are very similar to Statement \ref{4} and omitted.

 \end{subappendices}


\chapter{Monotonicity of Value Function for POMDPs}  \label{chp:monotonevalue}
\minitoc

\index{value function for POMDP! monotone|(}
This chapter gives sufficient conditions on the POMDP model so that 
the value function in dynamic programming is  decreasing with respect to the monotone likelihood ratio( MLR)  stochastic order. 
 That is, $\belief_1 \gr  \belief_2 $ (in terms MLR dominance)
implies  $\optvalue(\belief_1) \leq \optvalue(\belief_2)$. To prove this result, we will use the structural properties of the optimal
filter  established in Chapter \ref{chp:filterstructure}.

\index{monotone value function! POMDP|(}
Giving conditions for a POMDP  to have a  monotone value function is useful for several reasons: It serves as an essential step in establishing sufficient conditions for a stopping time
POMDPs to have a monotone optimal policy -- this is discussed in Chapter \ref{ch:pomdpstop}. For more general POMDPs (discussed in 
Chapter \ref{chp:myopicul}), it allows us to upper and lower bound the optimal policy by judiciously constructed myopic policies.
Please see Figure \ref{fig:organization} for the sequence of chapters on POMDP structural results.

After giving sufficient conditions for a monotone value function, this chapter also
 gives two examples of POMDPs  to illustrate the usefulness of this result:
\begin{compactitem}
\item {\em Example 1. Monotone Optimal Policy for 2-state POMDP}: \index{monotone policy! two state POMDP}
 \secn \ref{sec:2statepomdp} gives sufficient conditions  for a 2 state POMDP to have a monotone optimal policy.
The optimal policy is  characterized by at most  $\actiondim-1$ threshold belief states (where $\actiondim$ denotes the number
of possible actions). One only needs to compute (estimate) these
 $\actiondim-1$ threshold belief states in order to determine the optimal policy. 
 This is easier than solving Bellman's equation when nothing is known about the structure of the optimal policy. 
 Also real time implementation of a controller with a monotone policy is simple;  only the threshold belief states need to be stored
 in a lookup table.
 Figure \ref{fig:figpomdp2state} illustrates a  monotone policy 
 for  a two state POMDP with  $\actiondim = 3$. 
 
\item  {\em Example 2. POMDP Multi-armed Bandits and Opportunistic Scheduling}:
 \secn \ref{sec:POMDPbandit} discusses how monotone value functions can be used to solve POMDP multi-armed bandit problems efficiently. It
is shown that for such problems,
the optimal strategy is ``opportunistic'':  choose  the bandit with the largest belief state in terms of MLR order. 
\end{compactitem}

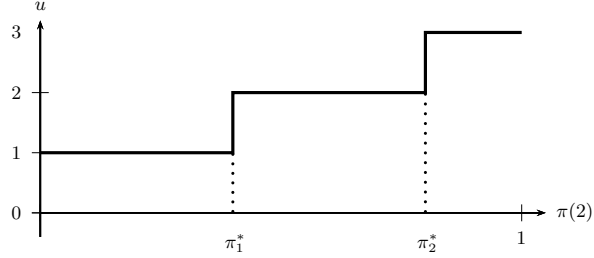
\begin{figure} 
\scalebox{0.8}{
\begin{pspicture}[showgrid=false](-2.75,-0.75)(6,4)
\psset{
  xunit = 8,
  yunit = 1
}
     \psaxes{->}(0,0)(0,-0.4)(1.05,3.2)[$\belief(2)$,0][$\action$,90]
    \psset{algebraic,linewidth=1.5pt}
     \pnode(0.4,1){A}
    \pnode(0.4,0){B}
    \pnode(0.8,2){C}
    \pnode(0.8,0){D}
    \pscustom
    {
        \psplot{0}{0.4}{1}
	
        \psplot{0.4}{0.8}{2}
        \psplot{0.8}{1}{3}
    }
   
      \ncline[linestyle=dotted]{A}{B}
           \ncline[linestyle=dotted]{C}{D}
             \rput(0.4,-0.5){\psframebox*{ $\belief^*_1$}}
              \rput(0.8,-0.5){\psframebox*{ $\belief^*_2$}}
\end{pspicture}}
\caption{\secn \ref{sec:2statepomdp} gives sufficient conditions for a  2-state  POMDP to have a monotone optimal policy. The figure illustrates such a
monotone policy for $\actionspace = 3$. The optimal policy is
 completely determined by the threshold belief states $\belief_1^*$ and $\belief_2^*$.}
\label{fig:figpomdp2state}
\end{figure}

\section{Model and Assumptions}

Consider  a discrete time, infinite horizon discounted cost\footnote{This chapter considers discounted cost POMDPs for notational convenience
to avoid denoting the time dependencies of parameters and policies. The main result of  this chapter, namely
Theorem \ref{thm:pomdpmonotoneval}  also holds for finite horizon POMDPs providing conditions
\ref{A1}, \ref{A2} and \ref{A3} hold at each time instant for the time dependent cost, observation matrix and transition matrix.}
  POMDP which was formulated in \chp  \ref{chp:discpomdp}. 
  The state space for the underlying Markov chain is $\statespace = \{1,2,\ldots,\statedim\}$,  the action space
  is $\actionspace = \{1,2,\ldots,\actiondim\}$ and the belief space is the unit $\statedim-1$ dimensional unit simplex 
  $$ \Belief =  \left\{\belief \in \reals^{\statedim}: \one^{\p} \belief = 1,
\quad 
0 \leq \belief(i) \leq 1, \; i \in \statespace =\{1,2,\ldots,\statedim\}\right\} .
$$
 
     For stationary policy  $\policy: \Belief \rightarrow \actionspace$,
 initial belief  $\belief_0\in \Belief$,  discount factor $\discount \in [0,1)$,  the  discounted cost is
\begin{align}\label{eq:discountedcost}
J_{\policy}(\belief_0) = \Ep\left\{\sum_{\time=1}^{\infty} \discount ^{\time-1} \Cost(\belief_\time,\policy(\belief_\time))\right\}.
\end{align}
Here $\Cost(\belief,\action)$ is the cost accrued at each stage and is not necessarily linear in $\belief$.
 
%
The belief evolves according to the HMM filter
$\belief_{k} = \filter(\belief_{k-1},\obs_k,\action_k)$ where
\begin{align}  \filter\left(\belief,\obs,\action\right) &= \cfrac{\oprob_{\obs} (\action)\tp'(\action)\belief}{\filternorm\left(\belief,\obs,\action\right)} , \quad
\filternorm\left(\belief,\obs,\action\right) = \one_{\statedim}'\oprob_{\obs}(\action) \tp'({\action})\belief, \nonumber \\ 
\oprob_{\obs}(\action) &= \diag(\oprob_{1\obs}(\action),\cdots,\oprob_{\statedim\obs}(\action)), \quad \text{ where } 
\oprob_{\state\obs}(\action) = \pdf(\obs|\state,\action). \label{eq:information_state}
\end{align}
Throughout this chapter  $\obs \in \obspace$ can be discrete-valued in which case $\pdf $ in (\ref{eq:information_state})
is a pmf
 or continuous-valued in which case $\pdf$ is a pdf.

The 
 optimal  stationary policy $\optpolicy:\Belief \rightarrow \actionspace$ such that
$J_{\optpolicy}(\belief_0) \leq J_{\policy}(\belief_0)$ for all $\belief_0 \in \Belief$ satisfies
 Bellman's  dynamic programming equation (\ref{eq:bellmaninfpomdp})
\begin{align}   \optpolicy(\belief) &=  \underset{\action \in \actionspace}\argmin ~\valueaction(\belief,\action), \quad J_{\optpolicy}(\belief_0) = \optvalue(\belief_0)   \label{eq:bellman2state} \\
\optvalue(\belief)  &= \underset{\action \in \actionspace}\min ~\valueaction(\belief,\action), \quad
  \valueaction(\belief,\action) =  \Cost(\belief,\action) + \discount\sum_{\obs \in \obsdim} \optvalue\left(\filter\left(\belief,\obs,\action\right)\right)\filternorm \left(\belief,\obs,\action\right). \nn
\end{align}

\subsubsection*{Assumptions}
\begin{myassumptions}
\item[\nl{A1}{(C)}] 
The  cost  $\Cost(\belief,\action)$ is first order stochastically decreasing with respect to $\belief$ for each action
$\action \in  \{1,2,\ldots,\actiondim\}$.
That is $\belief_1 \gs \belief_2$ implies 
$\Cost(\belief_1,\action) \leq \Cost(\belief_2,\action)$. 

For linear costs $C(\belief,\action) = \cost_\action^\p \belief$,  \ref{A1} is equivalent to the  condition: \\ The instantaneous cost
$\cost(\state,\action) $ is decreasing in $\state$ for each $\action$.
\item[\ref{A2}]  The observation probability kernel $\oprob(\action)$ is  TP2 for  each action $\action \in \{1,2,\ldots,\actiondim\}$.
\item[\ref{A3}] The transition matrix $\tp(\action)$ is TP2 for each action $\action \in  \{1,2,\ldots,\actiondim\}$.
\end{myassumptions}
Recall that assumptions \ref{A2} and \ref{A3} were discussed  in Chapter \ref{chp:filterstructure}; see \ref{A2}, \vref{A3}.
  \ref{A2} and \ref{A3} are required
for the Bayesian filter 
$T(\belief,y.\action)$ to be monotone increasing with  observation $y$ and $\belief$ with respect to the MLR order.
This is a key step in showing $V(\belief)$ is MLR decreasing in $\belief$. 


\subsubsection*{Sufficient Conditions for \ref{A1}}  We  pause briefly to discuss Assumption  \ref{A1}, particularly in the context
of nonlinear costs that arise in controlled sensing (discussed in Chapter \ref{sec:nonlinearpomdpmotivation}).

For linear costs  $C(\belief,\action) = c_\action^\p \belief$, obviously the elements of $c_\action$ decreasing is necessary and sufficient for $C(\belief,\action)$ to be
decreasing with respect to $\gs$.

For nonlinear costs, we can give the following sufficient condition for $C(\belief,\action)$ to be decreasing in $\belief$ with respect to first order stochastic 
dominance,
Consider the subset of  $\reals_+^\statedim$ defined as $\Delta = \left\{\delta: 1 = \delta(1)\ge\delta(2)\cdots\ge\delta(\statedim)\right\}$. Define the $\statedim \times \statedim$ matrix
\begin{align}\label{eq:transformationmatrix}
\Psi =
 \begin{bmatrix}
  1 & -1 & 0 & \cdots & 0 \\
  0 & 1 & -1 & \cdots & 0 \\
  \vdots  & \vdots  & \vdots & \ddots & \vdots \\
  0 & 0 & 0 & \cdots & 1
 \end{bmatrix}.
\end{align}
Clearly every $\belief \in \Belief$ can be expressed as $\belief = \Psi\,\delta$ where  $\delta \in \Delta$. Consider two beliefs $\belief_1 = \Psi\delta_1$ and $\belief_2 = \Psi\delta_2$ such that $\belief_1 \gs~\belief_2$. The equivalent partial order induced on $\delta_1$ and $\delta_2$ is: $\delta_1 \gc \delta_1$ where $\gc$ is the componentwise partial order on $\reals^\statedim$.

\begin{lemma}\label{lm:monotone_nl_cost} Consider a nonlinear cost $C(\belief_2,\action)$ that is differentiable 
in $\belief$. 
\begin{compactenum}
\item For $\belief_1 \gs \belief_2$, a sufficient condition for 
 $C(\belief_1,\action)\leq  C(\belief_2,\action)$ is $\frac{d}{d\delta} C(\Psi \delta) \leq 0$ element wise.  
 \item Consider the  special case of a  quadratic cost
 $ C(\belief,\action) = \phi_\action^\p \belief - \alpha (h^\p \belief)^2 $
where $\alpha$ is a non-negative constant, $\phi_\action,h \in \reals^\statedim$ with elements $\phi_{i\action}$, $h_i$, $i=1,\ldots,\statedim$.
Assume  $h$ a vector of non-negative elements that are either monotone increasing or decreasing. Then  a sufficient condition for $C(\belief,\action)$ to be first
order decreasing in $\belief$ is 
  \beq   \phi_i - \phi_{i+1} \geq 2 \alpha h_1  (h_i - h_{i+1}) .   \label{eq:suffmon}\eeq  
  \end{compactenum}
\end{lemma}
Recall from \chp  \ref{sec:nonlinearpomdpmotivation} that POMDPs with quadratic costs arise in controlled sensing.
Lemma \ref{lm:monotone_nl_cost} gives sufficient conditions for such costs to be decreasing with respect to first order stochastic dominance.\footnote{\label{foot:firstmlr} Note that $C(\belief,\action)$ first order increasing in $\belief $ implies
that $C(\belief,\action)$ is MLR increasing in $\belief$, since MLR dominance implies first order dominance.}
\begin{proof}
It is sufficient to show under \ref{A1} that  $C(\belief,\action)$ is $\gs$ decreasing in $\belief \in \Belief, \forall \action \in \actionspace$.  Since $C(\belief, \action)$ $\gs$ decreasing on $\belief \in \Belief$ is equivalent to $C(\delta,\action)$ $\gc$ decreasing on $\delta \in \Delta$, a sufficient condition is $\cfrac{\partial C(\belief, \action)}{\partial\delta(i)} \le 0 , i = 2, \cdots, \statedim$. Evaluating this  yields
\beq 
 \phi_{i\action} - \phi_{i+1,\action} \geq 2 \alpha h^\p (h_i - h_{i+1})  \label{eq:quadcostcondition} \eeq
If $h_i$ is either monotone increasing or decreasing in $i$, then a sufficient condition for (\ref{eq:quadcostcondition}) is
$ \phi_{iu} - \phi_{i+1,u} \geq 2 \alpha h_1 ( h_i - h_{i+1})$.
\end{proof}

\section{Main Result: Monotone Value Function} The following is the main result of this chapter.
\index{monotone value function}
\index{value iteration! monotone structure}

\begin{theorem} \label{thm:pomdpmonotoneval}
Consider an infinite horizon  discounted cost
 POMDP with continuous or discrete-valued observations.
Then 
%
%
under \ref{A1}, \ref{A2}, \ref{A3}, $Q(\belief,u)$ is MLR decreasing in $\belief$.
As a result, the value function
$\optvalue(\belief)$ in Bellman's equation
(\ref{eq:bellman2state})
is MLR decreasing in $\belief$. That is, $\belief_1 \gr \belief_2$ implies that $\optvalue(\belief_1) \leq \optvalue(\belief_2)$.
\end{theorem}

\begin{proof}  
The  proof is by mathematical induction on the value iteration algorithm and makes extensive use of the structural properties of the HMM
filter developed in Theorem \ref{thm:filterstructure}. Recall from (\ref{eq:vipomdp})
that the   value iteration
algorithm  proceeds as follows:  Initialize $\optvalue_0(\belief) = 0$ and for iterations  $\iter=1,2,\ldots,$
$$
\optvalue_\iter(\belief)  = \underset{\action \in \actionspace}\min ~\valueaction_\iter(\belief,\action), \quad
  \valueaction_\iter(\belief,\action) =  \Cost(\belief,\action) + \discount\sum_{\obs \in \obsdim} \optvalue_{\iter-1}\left(\filter\left(\belief,\obs,\action\right)\right)\filternorm \left(\belief,\obs,\action\right).
$$
Assume
that $\optvalue_{n-1}(\belief)$ is MLR decreasing in $\belief$ by the induction hypothesis. 
Under \ref{A2}, Theorem \ref{thm:filterstructure}(\ref{3}) says that $\filter(\belief,\obs,\action)$ is MLR increasing in $\obs$.
As a result, $\optvalue_{\iter-1}\left(\filter\left(\belief,\obs,\action\right)\right) $ is decreasing in $\obs$.
Under \ref{A2}, \ref{A3}, Theorem \ref{thm:filterstructure}(\ref{2}) says 
\beq  \belief \gr \bbelief   \implies \filterd(\belief,\action) \gs \filterd(\bbelief,\action) . 
\label{eq:filterdfdom}
\eeq
 $\optvalue_{\iter-1}\left(\filter\left(\belief,\obs,\action\right)\right) $  decreasing in $\obs$ and  the first order dominance
(\ref{eq:filterdfdom})
 implies using Theorem \ref{res1}  that
\beq   \belief \gr \bbelief   \implies \sum_\obs  \optvalue_{\iter-1}\left(\filter\left(\belief,\obs,\action\right)\right)  \filterd(\belief,y,\action)
\leq   \sum_\obs  \optvalue_{\iter-1}\left(\filter\left(\belief,\obs,\action\right)\right)  \filterd(\bbelief,y ,\action)  \label{eq:domvaly} \eeq
Next, from Theorem \ref{thm:filterstructure}(\ref{1}), it follows that 
under \ref{A3},  
$$\belief \gr \bbelief   \implies \filter(\belief,\obs,\action) \gr \filter(\bbelief,\obs,\action) $$
Using the induction hypothesis that $\optvalue_{n-1}(\belief)$ is MLR decreasing in $\belief$ implies
$$  \belief \gr \bbelief   \implies \optvalue_{\iter-1}\left(\filter\left(\belief,\obs,\action\right)\right) \leq 
\optvalue_{\iter-1}\left(\filter\left(\bbelief,\obs,\action\right)\right). 
$$
which in turn implies 
\beq
\belief \gr \bbelief   \implies  \sum_\obs \optvalue_{\iter-1}\left(\filter\left(\belief,\obs,\action\right)\right)  \filterd(\bbelief,y ,\action)\leq 
\sum_\obs \optvalue_{\iter-1}\left(\filter\left(\bbelief,\obs,\action\right)\right)  \filterd(\bbelief,y ,\action). 
\label{eq:domvalb}
\eeq
Combining  (\ref{eq:domvaly}), (\ref{eq:domvalb}), it follows that
\beq   \belief \gr \bbelief   \implies \sum_\obs  \optvalue_{\iter-1}\left(\filter\left(\belief,\obs,\action\right)\right)  \filterd(\belief,y,\action)
\leq 
\sum_\obs \optvalue_{\iter-1}\left(\filter\left(\bbelief,\obs,\action\right)\right)  \filterd(\bbelief,y,\action).  \eeq
Finally, under \ref{A1}, $C(\belief,\action)$ is MLR decreasing (see Footnote \ref{foot:firstmlr})
\beq \belief \gr \bbelief  \implies C(\belief,\action) \leq C(\bbelief,\action) . \label{eq:domcost}\eeq
Since the sum of decreasing functions is decreasing, it follows that
\begin{multline*}  \belief \gr \bbelief  \implies 
 C(\belief,\action)  +   \sum_\obs  \optvalue_{\iter-1}\left(\filter\left(\belief,\obs,\action\right)\right)  \filterd(\belief,y ,\action) \\
\leq  C(\bbelief,\action) +
\sum_\obs \optvalue_{\iter-1}\left(\filter\left(\bbelief,\obs,\action\right)\right)  \filterd(\bbelief,y,\action)
\end{multline*}
which is equivalent to $\valueaction_{\iter}(\belief,\action) \leq \valueaction_{\iter}(\bbelief,\action)$. 
Therefore $\valueaction_{\iter}(\belief,\action)$ is MLR decreasing in $\belief$.
Since the minimum of decreasing functions is decreasing,   $\optvalue_\iter(\belief) = \min_\action
\valueaction_{\iter}(\belief,\action)
$ is MLR decreasing in $\belief$. 
Finally, since $\optvalue_\iter$ converges uniformly
to $\optvalue$, it follows that $\optvalue(\belief)$ is also MLR decreasing.
\end{proof}

To summarize, although value iteration is not useful from a computational point of view for POMDPs,  we have exploited its
structure of prove the monotonicity of the value function. 
In the next two chapters, several examples will be given that exploit the monotone structure of the value function of a POMDP.

\section{Example 1: Monotone Policies for 2-state POMDPs} \label{sec:2statepomdp} \index{monotone policy! two state POMDP|(}
This section gives sufficient conditions for the optimal policy $\optpolicy(\belief)$ to be monotone 
increasing in $\belief$ when the underlying Markov chain
has  $\statedim=2$ states (see 
Figure \ref{fig:figpomdp2state}).
For $\statedim=2$,
since $\belief$ is a two-dimensional probability mass function
with $\belief(1) + \belief(2) = 1$,   it suffices to order the beliefs  in terms of the second component $\belief(2)$ which lies in
the interval $[0,1]$.

Consider a discounted cost
POMDP $(\statespace, \actionspace, \obspace, \tp(\action),  \oprob(\action), \cost(\action),\discount) $
 where  state space $\statespace = \{1,2\}$, action space $\actionspace = \{1,2,\ldots,\actiondim\}$,  observation space $\obspace$ can be continuous or discrete, and $\discount \in [0,1)$.
The main assumptions are as follows:  \index{submodularity! 2-state POMDP monotone policy}

\begin{myassumptions}
\item[\ref{A1}]    $\cost(\state,\action) $ is decreasing in $\state\in \{1,2\}$ for each~$\action \in \actionspace$.
\item[\ref{A2}]  $\oprob$ is    totally positive of order 2 (TP2).
\item[\ref{A3}]  $\tp(\action)$ is    totally positive of order 2 (TP2).
\item[\ref{A4p}] $\tp_{12}(\action+1) - \tp_{12}(\action) \leq \tp_{22}(\action+1) - \tp_{22}(\action)$ (tail-sum supermodularity).
\item[\nl{S2}{(S)}]   The costs are submodular: $\cost(1,\action+1) - \cost(1,\action) \geq \cost(2,\action+1) - \cost(2,\action) $.
\end{myassumptions}

Recall  \vref{A1} and \ref{A2}, \vref{A3}. The main additional assumption above  is the submodularity assumption (S).
 Apart from \ref{A2}, the above conditions are identical to the fully observed MDP case
considered in Theorem~\vref{thm:mdpmonotone}.  Indeed (A2) and (A4)  in Theorem~\ref{thm:mdpmonotone} are equivalent
to \ref{A3}  and \ref{A4p}, respectively, for $\statedim=2$.

\begin{theorem} \label{thm:pomdp2state}
Consider a POMDP with an underlying $\statedim =2$ state Markov chain. Under \ref{A1}, \ref{A2}, \ref{A3}, \ref{A4p}, \ref{S2}, the optimal policy $\optpolicy(\belief)$ 
is increasing in $\belief$. 
 Thus  $\optpolicy(\belief(2))$ has the following  finite dimensional
characterization: There exist $\actiondim+1$ thresholds (real numbers) $0 = \belief^*_0 \leq \belief^*_1 \leq \cdots \leq \belief^*_{\actiondim} \leq 1$ such that
$$  \optpolicy(\belief) =  \sum_{\action \in \actionspace} \action \,I\left(\belief(2) \in (\belief_{\action-1}^*, \belief_{\action}^*]\right).$$
\end{theorem}
 The theorem also applies to finite horizon problems. Then the optimal policy $\optpolicy_k$ at each
time $k$ has the above structure.

The proof is in \cite{Kri16}.
It exploits the fact that the value function $V(\belief)$ is decreasing in $\belief$  (Theorem \vref{thm:pomdpmonotoneval}) and is concave to show that $Q(\belief,\action)$ is submodular. 
That is
\beq  Q(\belief,u) - Q(\belief,\baction) - Q(\bbelief,u)+ Q(\bbelief,\baction) \leq 0, \quad u > \baction, \; \belief \gr \bbelief .
 \label{eq:pomdpsubmodgen}\eeq
 where $Q(\belief,\action)$ is defined in Bellman's equation (\ref{eq:bellman2state}).
Recall that for $\statedim=2$,  $\gs$ (first order dominance) and $\gr$ (MLR dominance) coincide, implying that
$\belief \gr \bbelief\iff \; \belief \gs \bbelief \iff  \belief(2) \geq \bbelief(2)$.
As a result, for $\statedim=2$, submodularity of $Q(\belief,\action)$ needs to be established with respect to $(\belief(2),\action)$
where $\belief(2)$ is a scalar in the interval $[0,1]$. Hence the same simplified definition  of submodularity  used for a fully observed MDP 
(Theorem \vref{thm:mon}) can be used.   \index{threshold policy! 2-state POMDP} 

{\em Summary}: We have given sufficient conditions for a 2 state POMDP to have a monotone (threshold) optimal policy as
illustrated in 
Figure \ref{fig:figpomdp2state}.
 The threshold  \index{policy gradient algorithm!  POMDP with monotone optimal policy}
values can then be estimated via simulation based policy gradient algorithm such as the SPSA Algorithm, see also \secn \ref{sec:spsastoppomdp}.
Theorem \ref{thm:pomdp2state} only holds for $\statedim=2$ and does not generalize to $\statedim\geq 3$. 
For $\statedim\geq 3$,  determining sufficient conditions for submodularity (\ref{eq:pomdpsubmodgen}) to hold is an open problem.
In Chapter \ref{chp:myopicul}, we will  instead  construct judicious myopic bounds for $\statedim\geq 3$.
\index{monotone policy! two state POMDP|)}

\section{Example 2:  POMDP Multi-armed Bandits Structural Results}
 \index{multi-armed bandit}  \label{sec:POMDPbandit} \index{multi-armed bandit! POMDP}
 \index{sensor scheduling! multi-armed bandits}
In this section, we  show how the monotone value function result of Theorem
\ref{thm:pomdpmonotoneval}
facilitates solving
POMDP multi-armed bandit problems efficiently.

The multi-armed bandit problem is a dynamic stochastic
scheduling problem for optimizing in a sequential manner the allocation effort
between a number of competing projects. Numerous applications
of finite state Markov chain multi-armed bandit problems appear in the
operations research and stochastic control literature, see
\cite{Git89},  \cite{Whi80} for examples in job
scheduling and resource allocation for manufacturing systems.
The reason why multi-armed bandit problems are interesting is because their structure implies that the optimal 
policy can be found by a so-called Gittins index rule \cite{Git89,Ros83}: At each time instant, the optimal action
is to choose the process 
with the highest Gittins index, where the Gittins index of each
process is a function of the state of that process.
So the problem  decouples into solving individual control problems
for each process.

This section considers multi-armed bandit problems where the  finite state 
Markov chain is not directly observed -- instead the observations 
noisy measurements 
of the unobserved Markov chain.
Such POMDP  multi-armed bandits  are a useful model in stochastic scheduling.

\subsection{POMDP Multi-armed Bandit Model}
 \label{sec:banditmdp}
The POMDP multi-armed bandit has the following model:
Consider $\numtarget$ independent projects  $\target = 1,\ldots,\numtarget$. 
Assume for convenience   each project $\target$ has the same  finite state space $\statespace = \{1,2,\ldots,\statedim\}$.
Let $\state_k^{(\target)}$ denote the state of project $\target$ at discrete time 
$k=0,1\ldots,$. At each time
instant $k$ only one of these projects  can be worked on. The setup is as follows:
\begin{compactitem}
\item If project $\target$ is worked on at time $k$:
\begin{compactenum} \item An instantaneous non-negative reward
$\discount^{k}\,  \reward(\state^{(\target)}_{k})$ is accrued
where 
  $0\leq  \discount < 1$ denotes the discount factor.
  \item The
state $\state_k^{(\target)}$ evolves according to an $\statedim$-state
homogeneous Markov chain with transition probability matrix $\tp$.
\item The state of the active project $\target$ is observed via noisy
measurements  $\obs^{(\target)}_{k+1} \in \obspace = \{1,2,\ldots,\obsdim\}$ of the active project
state $\state_{k+1}^{(\target)}$ with observation probability $\oprob_{\state\obs} = \prob(\obs^{(\target)} = \obs| \state^{(\target)}=\state)$.
\end{compactenum}
\item
The states of all the other $(\numtarget-1)$ idle projects are unaffected, i.e.,
$\state_{k+1}^{(\target)} = \state_{k}^{(\target)}$,  if  project
$\target$ is idle at time $k$. No observations are obtained for idle projects.
\end{compactitem}
For notational convenience we assume all the projects have the same reward functions, transition and   observation probabilities and state spaces. 
So   the reward $\reward(\state^{(\target)},\target)$ is denoted as $\reward(\state^{(\target)}) $, etc.
All projects are  initialized
with $\state_0^{(\target)} \sim \belief_0^{(\target)}$ where $\belief_0^{(\target)}$ are specified
initial distributions for $\target=1,\ldots,\numtarget$. Denote $\belief_0 = (\belief_0^{(1)},\ldots,\belief_0^{(L)})$.

Let 
$\action_k \in \{1,\ldots, \numtarget\}$ denote which project is worked on at time $k$.
So $\state_{k+1}^{(\action_k)}$ is  the state of the active
project at time $k+1$.
Denote the  history
  at time $k$ as $$ \info_0 = \belief_0, \quad \info_k = \{\belief_0, y^{(u_0)}_{1},\ldots,y^{(u_{k-1})}_{k}, u_{0},\ldots,u_{k-1}\}.$$
Then the project at time $k$ is chosen according to
$u_{k} = \mu(\info_k)$, where the policy denoted as $\mu$ belongs to the class
of stationary policies. 
The  cumulative  expected discounted reward  over an 
infinite time horizon is given by
\begin{equation}
        J_{\mu}(\belief) =  \E_\policy\big\{\sum_{k=0}^{\infty}\discount^{k} \reward\left(\state^{(\action_{k})}_{k}\right)| \belief_0 = \belief\big\} , \quad \action_{k} = \mu(\info_{k}).
    \label{eq:banditstatecost}
\end{equation}
{\em The aim is to determine the optimal stationary policy $\mu^{*}(\belief) = \arg\max_{\mu }  J_{\mu}(\belief)$ which
yields the maximum reward in (\ref{eq:banditstatecost})}.

Note that we have formulated the problem in terms of rewards rather than costs since typically the formulation involves maximizing rewards
of active projects. Of course the formulation is equivalent to minimizing a cost.

At first sight (\ref{eq:banditstatecost}) seems intractable since the equivalent state space dimension is $\statedim^L$. The multi-armed bandit structure
yields a remarkable simplification - the problem can be solved by considering $L$ individual POMDPs each of dimension $\statedim$. Actually
with the structural result below, one only needs to evaluate the belief state for the $L$ individual HMMs and choose the largest belief at each time
(with respect to the MLR order).

\subsection{Belief State Formulation}
\label{sec:banditinfo_state}
We  now formulate the POMDP multi-armed bandit in terms of the belief state.
For each project $\target$, denote by $\belief_k^{(\target)}$
the belief at time $k$ where
$$
   \belief_k^{(\target)}(i) = \prob (\state^{(\target)}_{k} = i |   \info_k)
$$
The  POMDP multi-armed bandit problem can  be viewed
as the following scheduling problem:
Consider $P$ parallel HMM state estimation filters, one for each project.
The project $\target$ is active, an observation $\obs_{k+1}^{(\target)}$ is obtained
and the belief
$\belief_{k+1}^{(\target)}$ is computed  recursively
by the HMM state filter 
\begin{align}
\belief_{k+1}^{(\target)} &= \filter(\belief_k^{(\target)},\obs_{k+1}^{(\target)}) 
      \qquad
\text{ if  project $\target$ is worked on at time $k$}  
    \label{eq:filterbandit}  \\
 \text{where } &  \filter(\belief^{(\target)},\obs^{(\target)})   =
\frac{\displaystyle 
\oprob_{y^{(\target)}} \tp^\prime \belief^{(\target)}} {\filterd(\belief^{(\target)},\obs^{(\target)})},\quad
 \filterd(  x^{(\target)},y^{(\target)}) =  {\mathbf{1}}^{\p} \oprob_{y^{(\target)}} \tp^{\prime} \belief^{(\target)} \nn
    \end{align}
In (\ref{eq:filterbandit})  
$\oprob_{\obs ^{(\target)}}= \diag\big(\prob(\obs^{(\target)}| \state^{(\target)}=1), \ldots, \prob(\obs^{(\target)}| \state^{(\target)}=\statedim)\big)$.

The beliefs of the other $\numtarget-1$ projects remain unaffected,
 i.e.
\begin{equation}
 \belief_{k+1}^{(q)} = \belief_{k}^{(q)} \qquad \text{ if project } q 
\text{ is not worked on}, \quad q\in \{1,\ldots,\numtarget\},\;q\neq \target
\label{eq:unaffected} \end{equation}
Note that each  belief
 $\belief^{(\target)}$ lives in the unit simplex $\Belief$.
 
 Let $\reward$ denote the
$\statedim$ dimensional
reward  vector $[\reward(\state_k^{(\target)}=1),\ldots,\reward(\state_k^{(\target)}=\statedim)]^{\p}$.
In terms of the belief state, 
 the 
 reward functional (\ref{eq:banditstatecost})
can be re-written  as
\begin{equation}
        J_{\mu}(\belief) =  \E\big\{ \sum^{\infty}_{k=0}
        \discount^k \,\reward^{\prime}\belief^{(u_{k})}_{k} \mid (\belief_0^{(1)}, \ldots,  \belief_0^{(L)}) = \belief\big\},
        \quad u_{k} = \mu(\belief_{k}^{(1)},\ldots,\belief_{k}^{(\numtarget)}).
    \label{eq:infocost}
\end{equation}
 The aim
is to compute the optimal policy $\optpolicy(\belief) = 
\arg \max_{\mu} J_{\mu}(\belief)$.

\subsection{Gittins Index Rule} \index{Gittins index} \index{multi-armed bandit! Gittins index}

Define $\bar{M} \ole \max_{i } \reward(i)/(1-\discount) $ and 
let $M$ denote a scalar  in the interval
$[0, \bar{M}]$.

It is  known that
the optimal 
policy of a multi-armed bandit has an {\em indexable rule} \cite{Whi80}. Translated to the POMDP multi-armed bandit the result
reads:  

\begin{theorem}[Gittins index]  \label{thm:banditstandard} Consider the POMDP multi-armed bandit problem comprising $\numtarget$ projects.
For each project  $\target$ there is a  
function $\gamma(\belief^{(\target)}_k)$ called the 
{\em Gittins index}, \index{Gittins index}
which is only a function of the  parameters of project $\target$ and the information 
state $\belief^{(\target)}_{k}$, whereby the optimal scheduling policy at time 
$k$ is to work on
the project  with 
the largest Gittins index: 
\beq
\mu^*(\belief^{(1)}_k,\belief^{(2)}_k,\ldots,\belief^{(\numtarget)}_k) 
= \max_{\target \in \{1,\ldots,\numtarget\}} \left\{\gamma(\belief^{(\target)}_{k})\right\}
\label{eq:gitpolicy}\eeq
The Gittins index  
of project $\target$ with belief $\belief^{(\target)}$
is
\begin{align}
 \gamma(\belief^{(\target)}) 
&= \min\{M: V(\belief^{(\target)},M) = M \} \label{eq:gittindef}\\
\intertext{where $V(\belief^{(\target)},M)$ satisfies   Bellman's  equation}
V(\belief^{(\target)},M) &= \max\biggl\{\reward^{\p} \belief^{(\target)} + \discount \sum_{\obs =1}^{\obsdim}
V\left(\filter(\belief^{(\target)},\obs),M \right)
\filterd(  \belief^{(\target)},\obs) 
%
,\: M\biggr\}  \label{eq:bellmanbandit}
\end{align} 
\end{theorem}
  Theorem \ref{thm:banditstandard}  says that the optimal policy is ``greedy":  at each time choose 
the project  with the largest Gittins index; and the Gittins index for each project can be obtained by solving a dynamic programming equation
for that project.  Theorem \ref{thm:banditstandard}
is well known in the multi-armed bandit literature \cite{Git89}
and will not be proved here.

 Bellman's equation (\ref{eq:bellmanbandit}) can be approximated over  any finite horizon $\finaltime$ via the value iteration algorithm.
Since as described in Chapter \ref{ch:pomdpbasic}, the value function of a POMDP at each iteration has a finite dimensional characterization, the Gittins index can be computed explicitly for any
finite horizon using any of the exact POMDP algorithms in Chapter \ref{ch:pomdpbasic}.
 Moreover, the error bounds for
value iteration for horizon $\finaltime$ (compared to infinite horizon) of Theorem \ref{thm:vipomdpgeom} directly translate to error bounds in determining the
Gittins index.  However, for large dimensions, solving each individual POMDP to compute
the Gittins index  is
computationally intractable.

\subsection{Structural Result: Characterization of Monotone Gittins Index} \index{multi-armed bandit! monotone Gittins index}
\label{sec:structural}

Our focus below is to
 show how the monotone value function result of Theorem
\ref{thm:pomdpmonotoneval}
facilitates solving  (\ref{eq:gitpolicy}), (\ref{eq:gittindef})
efficiently.
We show that under reasonable conditions on the 
rewards, transition matrix and observation probabilities, the Gittins index is monotone increasing
in the belief (with respect to the MLR order). This means that if the information
states of the $\numtarget$ processes at a given time instant are MLR comparable, the optimal policy is
to pick the process with the largest belief. This is straightforward  to implement
and makes the solution practically useful.

Since we are dealing rewards rather than costs, we say that assumption \ref{A1} holds if 
$\reward(i)$ is increasing in $i \in \statespace$. (This corresponds to the cost decreasing in~$i$).

\begin{theorem} \label{thm:structure} Consider the POMDP multi-armed bandit where all the $L$ projects have identical
transition and observation matrices and reward vectors.
Suppose assumptions \ref{A1}, \ref{A2} and \vref{A3} hold for each project.
Then the Gittins index $\gamma(\belief)$ is MLR increasing in $\belief$.
Therefore, if the  beliefs $\belief_k^{(\target)}$ of
the $\numtarget$ projects are MLR comparable,
then the optimal policy $\mu^*$ defined in (\ref{eq:gitpolicy})  is opportunistic:
\beq 
u_k = \mu^*(\belief^{(1)}_k,\ldots, \belief^{(\numtarget)}_k) = \argmax_{\target \in \{1,\ldots,\numtarget\}}  \belief_k^{(\target)} 
\label{eq:opportunistic}
\eeq
\end{theorem}

\begin{proof}
First using exactly the same proof as Theorem \ref{thm:pomdpmonotoneval}, it follows that $V(\belief,M)$ is MLR increasing in $\belief$.

Given that the value function $V(\belief,M)$ is MLR increasing in $\belief$, we can now characterize the Gittins
index.
Recall from (\ref{eq:gittindef}) that $\gamma(\belief) =  \min\{M: V(\belief,M) - M =0\}$.
Suppose  $\belief^{(1)} \gr \belief^{(2)}$. This  implies $V(\belief^{(1)},M) \geq V(\belief^{(2)},M)$ for all $M$.
So
$ V(\belief^{(1)},\gamma(\belief^{(2)})) - \gamma(\belief^{(2)}) \geq V(\belief^{(2)},\gamma(\belief^{(2)}) ) - \gamma(\belief^{(2)}) = 0$.
Since $V(\belief,M) - M$ is decreasing in $M$  (this is seen by subtracting $M$ from both sides of (\ref{eq:bellmanbandit})), it follows from the previous
inequality that the point $\min\{M: V(\belief^{(1)},M) - M = 0\} >
\min\{M: V(\belief^{(2)},M) - M =0\}$. So $\gamma(\belief^{(1)}) \geq \gamma(\belief^{(2)})$.
\end{proof}

\paragraph{Discussion}
It is instructive to compare (\ref{eq:gitpolicy}) with  (\ref{eq:opportunistic}). Theorem \ref{thm:structure} says that  instead of choosing
the project with the largest Gittins index, it suffices to choose the project with the largest MLR belief (providing
the beliefs of the $\numtarget$ projects are MLR comparable).  In other words, the optimal policy is {\em opportunistic} (other terms
used are ``greedy" or ``myopic")  with respect to the beliefs ranked by MLR   order.
The resulting optimal policy  is trivial to implement
and makes the solution practically useful. There is no need to compute the Gittins index.

The following examples yield trajectories of belief states which  are MLR comparable across the $\numtarget$ projects.
As a result,  under \ref{A1}, \ref{A2}, \ref{A3},  the optimal policy is opportunistic and completely specified by Theorem \ref{thm:structure}.

{\em Example 1}: If $\statedim=2$, then
all beliefs are MLR
comparable.

{\em Example 2}: Suppose $\oprob$ is a bi-diagonal matrix.
Then if $\belief_0^{(\target)}$ is
a unit indicator vector, then all subsequent beliefs are MLR comparable
(since all beliefs comprise of two consecutive non-zero elements
and the rest are zero elements).

{\em Example 3}: Suppose $\oprob_{i\obs} = 1/\obsdim$ for all $i,\obs$. Suppose
all processes have same initial belief $\belief_0$ and pick
$\tp$ such that either 
$\tp^\p \belief_0 \gr \belief_0$  or $\tp^\p \belief_0 \lr \belief_0$. Then from 
Theorem \ref{thm:filterstructure}\ref{1a},
 if $\tp$ is TP2, all beliefs are MLR comparable. \index{TP2 matrix}

When the trajectories of beliefs for the individual bandit processes are not 
MLR comparable, they can be projected to MLR comparable beliefs, and a suboptimal policy implemented as follows:

 Assume at time instant $k$, 
the beliefs of all $\numtarget$ processes
are MLR comparable. Let $\sigma(1),\ldots,\sigma(\numtarget)$
denote the permutation of $(1,\ldots,\numtarget)$ so
that 
$$\belief_k^{\sigma(1)}\gr \belief_k^{\sigma(2)}\gr  \ldots \gr \belief_k^{\sigma(\numtarget)}.$$
From Theorem \ref{thm:structure}, the optimal action is $u_k = \sigma(1)$.
But  the updated belief $\belief_{k+1}^{\sigma(1)}$ 
may not be MLR comparable with the other $\numtarget-1$ information
states. So we project 
$\belief_{k+1}^{\sigma(1)}$ to the nearest belief denoted $\bar{\belief}$ 
in the simplex
 $\Belief$ that is MLR comparable with the other $\numtarget-1$ information
 states. That is, at time $k+1$ solve the following $\numtarget$
 optimization problems: Compute the projection distances
 \begin{align*}
 \mathcal{P}(\bar{\belief}^{(1)}) &= \min_{\bar{\belief} \in  \Belief} \|\bar{\belief} - \belief_{k+1}^{\sigma(1)}\|
 \text{ subject to } \bar{\belief} \gr \belief_k^{\sigma(2)} \\ 
 \mathcal{P}(\bar{\belief}^{(\target)}) &= \min_{\bar{\belief} \in  \Belief} \|\bar{\belief} - \belief_{k+1}^{\sigma(1)}\|
 \text{ subject to } \belief_k^{\sigma(\target)} \gr \bar{\belief} \gr \belief_k^{\sigma(p+1)} , \; p = 2,\ldots,\numtarget-1\\
\mathcal{P}(\bar{\belief}^{(\numtarget)}) &= \min_{\bar{\belief} \in  \Belief} \|\bar{\belief} - \belief_{k+1}^{\sigma(1)}\|
 \text{ subject to }  \belief_k^{\sigma(\numtarget)} \gr\bar{\belief}. \end{align*}
 Here $\|\cdot\|$ denotes some norm, and  $\mathcal{P}$, $\bar{\belief}^{p}$ denote, respectively, the minimizing 
 value and minimizing solution
 of each of the problems.
Finally  set $\belief_{k+1}^{\sigma(1)} = 
\argmin_{\bar{\belief}_p}  \mathcal{P}(\bar{\belief}_p)$.
 The above $\numtarget$ problems are  convex optimization problems and
 can  be solved efficiently in real time. Thus  all the   beliefs at
 time $k+1$ are  MLR comparable,  the  action $u_{k+1}$ is chosen as the index of the largest
 belief.  

{\em Summary}: For POMDP multi-armed bandits that satisfy the conditions of Theorem \ref{thm:structure}, the optimal policy is to choose the project
with the largest belief. 

\section{\pwe}
The proof that under suitable conditions the POMDP value function is monotone with respect to the MLR order goes back to Lovejoy \cite{Lov87}. The MLR order is the natural setting
for POMDPs due to the Bayesian nature of the problem.  More generally,  a similar  proof holds for multivariate observation distributions - in this case the TP2 stochastic order (which is a multivariate version of the MLR order) is used - to establish sufficient conditions for monotone value function
for a multivariate POMDP; see \cite{Rie91}. 

The result in  \secn \ref{sec:2statepomdp} that establishes a threshold policy for 2-state POMDPs is from
\cite{Alb79}. However, the proof does not work for  action dependent (controlled)  observation probabilities. In \chp  \ref{sec:blackwelldom} we will use  Blackwell ordering to deal with action dependent observation probabilities. For optimal search problems with 2 states, \cite{MJ95} proves the optimality of threshold policies under certain conditions. Other types of structural result for 2-state POMDPs
are in \cite{Gro07,BG13}.

 The POMDP multi-armed bandit structural result in
\secn \ref{sec:POMDPbandit} is from \cite{KW09} where several numerical examples are presented; see also \cite{Kri05} for applications in radar.  \cite{Whi80} and  \cite{Git89} are  classic works in 
Bayesian multi-armed bandits. More generally, \cite{LZ10} establishes the optimality of indexable policies for restless bandit POMDPs
with 2 states. (In a restless bandit, the state of idle projects also evolves.)   \secn \ref{sec:banditnon} gives  a short discussion on non-Bayesian bandits.


\chapter{Structural Results for  Stopping Time POMDPs} \label{ch:pomdpstop}

\section{Introduction}
The previous chapter established conditions under which the value function of a POMDP is monotone with respect to the MLR
order. Also conditions were given for the optimal policy for a two-state POMDP to be monotone (threshold).
 This  and the next chapter develop structural results
for the optimal policy of multi-state POMDPs.
 To establish the structural results, we will
use  submodularity,
and stochastic dominance on the lattice\footnote{A lattice is a partially ordered set (in our case belief space $\Belief$) in which every two elements have a supremum and infimum 
(in our case with respect to the monotone likelihood ratio ordering).
The appendix gives definitions  of supermodularity on lattices}
 of belief states to analyze Bellman's dynamic programming equation -- such analysis falls under the area of  ``Lattice Programming" \cite{HS84}.
 Lattice programming and  ``monotone comparative statics''  pioneered by Topkis \cite{Top98} (see also \cite{Ami05,Ath02})  provide a  general set of sufficient conditions for the existence
of monotone strategies. Once a POMDP is shown to have a monotone policy, then gradient based algorithms
that exploit this structure can be designed to estimate this policy. This and the next two chapters rely heavily on the structural results for filtering (Chapter \ref{chp:filterstructure})
and monotone value function (Chapter \ref{chp:monotonevalue}). Please see Figure \vref{fig:organization} for the context of this chapter.

\subsection{Main Results}
This chapter deals with structural results for {\em stopping time POMDPs}. Stopping time POMDPs have action space $\actionspace = \{ 1 \text{ (stop)}, 2 \text{ (continue) }\}$. They arise
in  sequential detection  such as quickest  change detection and  machine replacement.
Establishing structural results for stopping time POMDPs are easier than that for general POMDPs (which is considered in the next chapter).
The main  structural results in this chapter  regarding stopping time POMDPs  are:
\begin{compactenum}
\item  {\em Convexity of stopping region}: \secn \ref{sec:convexstop} shows that the set of beliefs where it is optimal to apply action 1 (stop) is a convex subset of the belief
space. This result unifies several well known results about the convexity of the stopping set for sequential detection problems.
\item  {\em Monotonicity of the optimal policy}: \secn \ref{sec:monotonestop} gives conditions under which the 
optimal policy of a stopping time POMDP is monotone with respect to the monotone likelihood ratio (MLR) order. 
The MLR order is naturally  suited for POMDPs since it is preserved under conditional expectations.
\end{compactenum}

\begin{figure}[h] \centering
\mbox{\subfigure[$\statedim =2 $]
{ \scalebox{0.75} 
{\begin{pspicture}(0,-1.5)(7,2)
\psline[linecolor=black]
  (0,0)(2.7,0)(2.7,0.8)(6,0.8)
  \rput(-0.5,0){\psframebox*{ $1$}}
 \rput(-0.5,0.8){\psframebox*{ $2$}}
 \psline[linecolor=black,linewidth=1pt]
 (2.7,-0.7)(2.7,-0.3)%
   \rput(2.7,-1){\psframebox*{ $\belief^*$}}
   \rput(5,-0.3){\psframebox*{$\belief(2)$}}
\rput(-0.6, 1.5){\psframebox*{$\optpolicy(\belief)$}}
\rput(0, -0.9){\psframebox*{$0$}}
\rput(6, -0.9){\psframebox*{$1$}}
\psline[arrows=->,linecolor=black,linewidth=0.5pt]
  (0,-0.5)(6,-0.5)%
\psline[arrows=->,linecolor=black,linewidth=0.5pt]
  (0,-0.5)(0,1.5)%
\end{pspicture}}}
\subfigure[$\statedim=3$]{\includegraphics[scale=0.25]{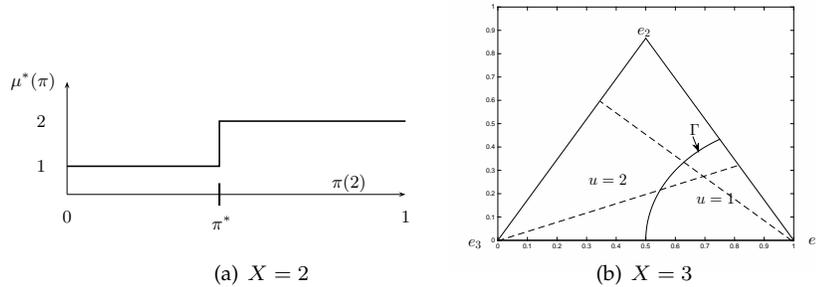}}}
\caption{Illustration of the two main structural results for stopping time POMDP established in this chapter. Theorem \ref{thm:pomdpconvex} shows that the 
stopping set (where action $\action=1$ is optimal) is convex.
Theorem \ref{thm:pomdpstop1} shows that the optimal policy is increasing on any line from $e_1$  to $(e_2,e_3)$ and decreasing on any line
 $e_3$ to $(e_1,e_2)$. Therefore, the stopping set includes state 1 ($e_1$).
Also the boundary of the stopping set $\Gamma$ intersects  any line from $e_1$ to $(e_2,e_3)$ at most once; similarly for any
line from $e_3$ to $(e_1,e_2)$. Thus the set of beliefs  where $\action =2$ is optimal is a connected set.
Figure \ref{fig:invalid} shows several types of stopping sets that are excluded by  Theorem \ref{thm:pomdpstop1}.} \label{fig:pomdpstopstructure}
\end{figure}

Figure \ref{fig:pomdpstopstructure} displays these structural results. For  $\statedim=2$, we will show that stopping set is the interval $[\belief^*,1]$ and the optimal policy $\optpolicy(\belief)$ is a step function; see Figure  \ref{fig:pomdpstopstructure}(a)).
 So it is only necessary to compute
 the threshold state $\belief^*$. 

Most of  this chapter is  devoted to characterizing the optimal policy for stopping time POMDPs when
 $\statedim \geq 3$. The main result  shown is that under suitable conditions, the optimal policy for  a stopping time POMDP is MLR
 increasing and therefore has a threshold switching "curve" (denoted by
$\Gamma$ in Figure  \ref{fig:pomdpstopstructure}(b)). So  one only needs to estimate this  curve $\Gamma$, rather than solve
a dynamic programming equation. 
We will show that the threshold  curve $\Gamma$ has useful properties: it can intersect any  line from $e_1$ to the edge $(e_2,e_3)$
only once. Similarly, it can intersect any line from $e_3$ to the edge $(e_1,e_2)$ only once (These are the dashed lines in
Figure  \ref{fig:pomdpstopstructure}(b)).  It will be shown that the optimal MLR  linear threshold policy, which approximates the curve $\Gamma$,
can then be estimated via simulation based stochastic approximation algorithms.
Such a linear threshold policy is straightforward to implement in a  real time POMDP controller.

 \secn \ref{sec:multivarpomdp} discusses structural results for POMDPs with multivariate observations. 
The multivariate TP2 stochastic order is used.

The structural results  presented in this chapter provide a unifying theme and insight into
what might otherwise simply be a collection of techniques and results in sequential detection.
In Chapter \ref{chp:stopapply} we will present several examples of stopping time POMDPs in
sequential   quickest change detection, multi-agent social learning and controlled measurement sampling.

\section{Stopping Time POMDP and Convexity of Stopping Set}\label{sec:convexstop}

\index{stopping time POMDP! convexity of stopping set|(}
A stopping time  POMDP has  action space $\actionspace = \{1 \text{ (stop)},2 \text{ (continue)}\}$.

For continue action $\action= 2$, the  state $\state \in \statespace = \{1,2,\ldots,\statedim\}$  evolves with transition matrix $\tp$ and is observed
 via observations
$\obs$ with
observation probabilities $\oprob_{\state\obs} = \prob(\obs_k=\obs|\state_k=\state)$.   An instantaneous  cost $\cost(\state,\action=2)$ is incurred.Thus for $\action=2$, 
the belief state evolves according to the HMM filter
$\belief_{k} = \filter(\belief_{k-1},\obs_k)$ defined in (\ref{eq:hmmc}). Since  action 1 is  a stop action and has no dynamics, to simplify notation,
we write $\filter(\belief,\obs,2)$ as $\filter(\belief,\obs)$ and $\filterd(\belief,\obs,2)$ as $\filterd(\belief,\obs)$  in this chapter.

The  action 1 incurs a terminal cost
of $\cost(\state,\action=1)$ and the problem terminates. 

We consider the class of stationary policies
\beq \label{eq:actionpolicyaa}
 \action_k = \policy(\belief_k)  \in \actionspace = \{1 \text{ (stop)} ,2 \text{ (continue) }\}. \eeq

 Let $\tau$ denote  a stopping time adapted to $\{\belief_k\}$, $k\geq 0$.
 That is, with $\action_k$ determined by decision policy (\ref{eq:actionpolicyaa}),  
\beq
\tau = \{ \inf k:  \action_k = 1\} .   \label{eq:tauu}\eeq
Let $\Belief = \left\{\belief \in \reals^{\statedim}: \one^{\p} \belief = 1,
\quad 
0 \leq \belief(i) \leq 1 \text{ for all } i \in \statespace \right\} $ denote the belief space.
  For stationary policy  $\policy: \Belief \rightarrow \actionspace$,
 initial belief  $\belief_0\in \Belief$,  discount factor\footnote{In stopping time POMDPs we allow for $\discount=1$ as well.} $\discount \in [0,1]$,  the  discounted cost objective  is
\beq \label{eq:stopdiscountedcost}
J_{\policy}(\belief_0) = \Ep\left\{\sum_{\time=0}^{\tau-1} \discount ^{\time} \cost(\state_k,2) + \discount^\tau   \cost(\state_\tau,1)
\right\} =
\Ep\left\{\sum_{\time=0}^{\tau-1} \discount ^{\time} \cost_{2}^\p\belief_\time+ \discount^\tau   \cost_{1}^\p\belief_\tau  \right\} , 
\eeq
 where  $ \cost_\action = [\cost(1,\action),\ldots,\cost(\statedim,\action)]^\p$.
The aim is to determine the 
 optimal  stationary policy $\optpolicy:\Belief \rightarrow \actionspace$ such that
$J_{\optpolicy}(\belief_0) \leq J_{\policy}(\belief_0)$ for all $\belief_0 \in \Belief$.

For the above stopping time POMDP, $\optpolicy$ is the solution of
 Bellman's equation which is of the form\footnote{The stopping time POMDP can be expressed as an infinite horizon POMDP. Augment $\Belief$ to include the  fictitious stopping state $e_{\statedim+1}$ which is cost free, i.e.,
  $\cost(e_{\statedim+1},\action) = 0$ for all $\action\in \actionspace$. When decision $\action_k=1$ is chosen, the belief state $\belief_{k+1}$ transitions to $e_{\statedim+1}$ and remains there indefinitely.
   Then (\ref{eq:stopdiscountedcost}) is equivalent to
$ J_\mu(\belief) = \Ep\{\sum_{k=0}^{\tau-1} \discount^{k} \cost_2^\p \belief_{k}
+  \rho^{\tau} \cost_1^\p \belief  + \sum_{k=\tau+1}^\infty \rho^{k} \cost(e_{\statedim+1},\action_{k}
 \} $, where the last summation is zero.}  (where  $\optvalue(\belief)$ below denotes the value function):
\begin{align} \label{eq:bellmanstop}
 \optpolicy(\belief) &=  \underset{\action \in \actionspace}\argmin ~\valueaction(\belief,\action), \quad 
\optvalue(\belief) = \underset{\action \in \actionspace}\min ~\valueaction(\belief,\action),     \\
  \valueaction(\belief,1) &= \cost_1^\prime\belief , \quad
  \valueaction(\belief,2) =  ~\cost_2^\prime\belief + \discount\sum_{\obs \in \obsdim} \optvalue\left(\filter\left(\belief,\obs\right)\right)\filternorm \left(\belief,\obs\right).
\nonumber s
\end{align}
where $\filter(\belief,\obs)$ and $\filterd(\belief,\obs)$ are the HMM filter and normalization (\ref{eq:information_state}).

\subsection{Convexity of Stopping Region}
We now present the first  structural result for stopping time POMDPs:  the stopping region for the optimal policy is convex.
Define the stopping set $\region_1$  as the set of belief states for which stopping ($\action=1$)  is the optimal action.
Define $\region_2$ as the set of belief states for which continuing ($\action=2$) is the optimal action. That is
\beq
\region_1 = \{\belief:  \optpolicy(\belief) = 1  \text{ (stop) }\} , \quad \region_2 =  \{\belief:  \optpolicy(\belief) = 2 \} = \Belief - \region_1.\eeq

  The theorem below shows that the stopping set $\region_1$  is convex (and therefore a connected set). 
Recall that the value function $\valuef(\belief)$ is concave on $\Belief$. (This essential property of POMDPs was proved in
Theorem  \ref{thm:pwlc}.)

\begin{theorem}[\cite{Lov87a}]  \label{thm:pomdpconvex} Consider the stopping-time POMDP with value function given by (\ref{eq:bellmanstop}).
Then the stopping set $\region_1$ is a convex subset of the belief space $\Belief$. \index{convexity of stopping region}
\end{theorem}

\begin{proof}
Pick any two belief states $\belief_1,\belief_2 \in \region_1$. To demonstrate convexity of $\region_1$,
we need to show for any $\lambda \in [0,1]$,  $\lambda \belief_1 + (1-\lambda) \belief_2 \in \region_1$.
Since $V(\belief)$ is concave,
\begin{align*}
V(\lambda \belief_1 + (1-\lambda) \belief_2) &\geq \lambda V(\belief_1) + (1-\lambda) V(\belief_2) \nonumber\\
&= \lambda Q(\belief_1,1) + (1-\lambda) Q(\belief_2,1)  \text{ (since $\belief_1,\belief_2 \in \region_1$) } \nonumber\\
&= Q(\lambda \belief_1 + (1-\lambda) \belief_2,1 ) \text{ (since $Q_{1}(\belief,1)$ is linear in $\belief$) }\nonumber \\
& \hspace{-1cm} \geq V(\lambda \belief_1 + (1-\lambda) \belief_2) \text{ (since $V(\belief)$ is the optimal value function) }
\end{align*}
Thus all the inequalities above are equalities, and $\lambda \belief_1 + (1-\lambda) \belief_2 \in 
\region_1$.
\end{proof}
Note that the theorem says nothing about the ``continue" region  $\region_2$. In Theorem \ref{thm:pomdpstop1} below we will characterize both $\region_1$ and $\region_2$.
Figure \ref{fig:pomdpstopstructure} illustrates the assertion 
 of  Theorem \ref{thm:pomdpconvex} for $X=2$ and $X=3$.

\subsection{Example 1. Classical Quickest Change Detection}  \label{sec:classicalqd}  \index{quickest detection! classical|(}

Quickest  detection is a useful example of a stopping time POMDP that  has applications in
biomedical signal processing, machine monitoring and finance   \cite{PH08,BN93}. 
The classical Bayesian quickest detection problem is as follows: 
An underlying discrete-time state process $\state$ jump changes at a geometrically distributed random time $\tau^0$.
Consider a sequence of random measurements $\{\obs_k,k \geq 1\}$, such that 
 conditioned on the event $\{\tau^0 = t\}$, $\obs_k$, $\{k \leq t\}$  are i.i.d. random variables with distribution 
$\oprob_{1\obs}$ and $\{\obs_k, k >t\}$ are i.i.d. random variables with distribution $\oprob_{2\obs}$.
The quickest  detection problem involves detecting the change time $\tau^0$ with minimal cost. That is,
at each time $k=1,2,\ldots$, a decision $u_k \in \{\text{continue}, \text{stop and announce change}\}$ needs to be made to optimize a tradeoff
between false alarm frequency and linear delay penalty.\footnote{There are two general formulations for quickest time
detection.  In the first
formulation, the change point $\tau^0$ is an unknown deterministic time,
and the goal is to determine a
stopping rule such that a  worst case delay penalty is
minimized subject to a constraint on the false alarm frequency
(see, e.g., \cite{Mou86,Poo98,YKP99,PH08}). 
The second formulation, which is the formulation considered in this book (this chapter and also Chapter \ref{chp:stopapply}),  is  the  Bayesian approach where
the change time $\tau^0$ is specified by a prior distribution.}

A geometrically distributed change time $\tau^0$ is realized by a  two state ($\statedim=2$) Markov chain  
with absorbing transition matrix $\tp$ and prior $\belief_0$ as follows:
\beq \tp = \begin{bmatrix} 1 & 0 \\ 1- \tp_{22} & \tp_{22}  \end{bmatrix} , \;  \belief_0 = \begin{bmatrix} 0 \\ 1 \end{bmatrix} , \quad
\tau^0 = \inf\{ k:  \state_k = 1\}. \label{eq:tpqdp} \eeq
The system starts  in state 2 and then jumps to  the absorbing state 1 at time $\tau^0$. Clearly $\tau^0$ is geometrically distributed
with mean $1/(1-\tp_{22})$.

The cost criterion in classical quickest detection is the {\em Kolmogorov--Shiryayev 
criterion} for detection of disorder \cite{Shi63}  \index{quickest detection! Kolmogorov--Shiryayev  criterion of disorder}
 \beq J_\policy(\belief) =   d\, \Ep\{(\tau - \tau^0)^+\} + \Pp(\tau < \tau^0) , \quad \belief_0 = \belief.
\label{eq:ksd} \eeq
where $\policy$ denotes the decision policy.
The first term is the delay penalty in making a decision at time $\tau > \tau^0$ and $d$ is a positive real number.
The second term is the false alarm penalty incurred in announcing a change at time $\tau< \tau^0$.

\noindent {\bf Stopping time POMDP}:
The quickest detection problem with penalty (\ref{eq:ksd}) is a stopping time POMDP with 
$\actionspace = \{1 \text{ (announce change and stop)},2  \text{ (continue)} \}$,  $\statespace=\{1,2\}$,
transition matrix in (\ref{eq:tpqdp}), arbitrary observation probabilities $\oprob_{\state\obs}$,
 cost vectors  $c_1 = [ 0 ,\; 1 ]^\p$, $c_2 = [d,\; 0 ]^\p$ and discount factor $ \discount = 1$.

Theorem \ref{thm:pomdpconvex} then  implies the following 
 structural result.
 \begin{corollary} \label{cor:qdclassical}
 The optimal policy $\optpolicy$ for classical quickest detection has a {\em threshold} structure:
There exists a threshold point $\belief^* \in [0,1]$ such that  
\beq u_k = \optpolicy(\belief_k) = \begin{cases} 2 \text{ (continue) } & \text{ if }
\belief_k(2) \in [ \belief^*,1] \\   1 \text{ (stop and announce change)  } &  \text{ if } \belief_k(2) \in [0, \belief^*).
\end{cases} \label{eq:onedim}
\eeq  \end{corollary}
\begin{proof} Since $\statedim=2$, $\Belief$ is  the interval $[0,1]$, and   $\belief(2) \in [0,1]$ is the belief state.
Theorem \ref{thm:pomdpconvex} implies that the stopping set $\region_1$ is convex. In one dimension this implies  that 
$\region_1$ is an interval of the form $[a^*,\pi^*)$ for $0 \leq a< \pi^*\leq 1$. Since state 1 is absorbing,
 Bellman's equation (\ref{eq:bellmanstop}) with $\discount=1$ applied at $\belief = e_1$  implies
$$\optpolicy(e_1) = \argmin_u\{\underbrace{\cost(1,u=1)}_{0},\;\; d ( 1 - \belief(2)) + V(e_1)\} = 1.$$ 
So $ e_1$ or equivalently $\belief(2) = 0$ belongs to $\region_1$. Therefore,
$\region_1$ is an interval of the form $[0,\belief^*)$.
Hence  the optimal policy is of the form  (\ref{eq:onedim}). \end{proof}

 Theorem \ref{thm:pomdpconvex} says that   for quickest
 detection of a multi-state Markov chain,
 the stopping set $\region_1$  is  convex. (Recall $\region_1$ is the set of beliefs where $u=1=\text{stop}$ is optimal.) What about the continuing
 set $\region_2 = \Belief - \region_1$ where action $u=2$ is optimal? For $X = 2$,  using Corollary \ref{cor:qdclassical},
 $\region_2 = [\pi^*,1]$ and is therefore convex.
 However,  for $\statedim > 2$, Theorem \ref{thm:pomdpconvex}
does not say anything about the structure of $\region_2$;  indeed,
 $\region_2$ could be a disconnected set. In \secn \ref{sec:monotonestop}, we will use more powerful POMDP structural results
 to give sufficient conditions for
 both $\region_1$ and $\region_2$ to be connected sets. \index{quickest detection! classical|)}
%
%

\index{stopping time POMDP! convexity of stopping set|)}

\section{Monotone Optimal Policy for Stopping Time POMDP} \label{sec:monotonestop}
\index{monotone policy! stopping time POMDP|(}
\index{stopping time POMDP! threshold optimal policy|(}
We now consider the next major structural result:  sufficient conditions to  ensure that a stopping time POMDP has a monotone optimal policy.

Consider a stopping time POMDP with  state space and action space
$$\statespace=\{1,\ldots,\statedim\} ,\quad \actionspace=\{1 \text{ (stop)} ,2 \text{ (continue)}\}$$
Action 2 implies continue with transition matrix $\tp$, observation
distribution $\oprob$ and cost $\Cost(\belief,2)$, while  action 1 denotes stop with stopping cost 
$\Cost(\belief,1)$.  So the model is almost identical to the previous section except that the costs $\Cost(\belief,u)$ 
are in general nonlinear functions of the belief. Recall such nonlinear costs were 
motivated by controlled sensing applications in  Chapter \ref{sec:nonlinearpomdpmotivation}.

In terms of the belief state $\belief$, Bellman's equation reads
\begin{align} 
\valueaction(\belief,\action=1) &= \Cost(\belief,1), \;\valueaction(\belief,\action=2) =  \Cost(\belief,2)+\discount\, \sum_\obs \valuef(\filter(\belief,\obs,2))
\filterd(\belief,\obs,2),  \nonumber \\
\valuef(\belief) &= \min_{\action\in \{1,2\} } \valueaction(\belief,\action)
,\quad 
\optpolicy(\belief) = \argmin_{\action \in \{1,2\}} \valueaction(\belief,\action).  \label{eq:bellmanstop2}
\end{align}
\subsection{Objective}
One possible objective would be
to give 
sufficient conditions on a stopping time  POMDP so that 
the optimal policy $\optpolicy(\belief)$ is MLR increasing on $\Belief$. That is,
\beq \belief_1,\belief_2 \in
\Belief, \quad 
 \belief_1 \gr \belief_2 \implies \optpolicy(\belief_1) \geq \optpolicy(\belief_2).  \label{eq:submodsimplex} \eeq
  However,  because
$\Belief$ is only partially orderable with respect to the MLR order,  it is 
difficult to exploit (\ref{eq:submodsimplex}) for devising useful algorithms.
Instead, in this section, our aim is to give (less
  restrictive)  conditions that lead to
  \beq  \label{eq:submodline}
  \belief_1,\belief_2 \in \l(e_{i},\bp),   \quad 
 \belief_1 \gr \belief_2 \implies \optpolicy(\belief_1) \geq \optpolicy(\belief_2), \quad  i \in \{1,\statedim\}. \eeq
Here  $\l(e_{i},\bp)$ denotes any  line segment in $\Belief$ which starts at $e_1$ and ends at any belief
$\bp $ in the subsimplex
$\{e_2,,\ldots,e_\statedim\}$; or any line segment which starts at $e_\statedim$ and ends at any belief $\bp $ in the subsimplex
$\{e_1,\ldots,e_{\statedim-1}\}$. 
(These line segments are the dashed lines in  Figure  \ref{fig:pomdpstopstructure}(b).)
So instead of proving $\optpolicy(\belief)$ is MLR increasing for any two beliefs in the belief space,
we will prove that  $\optpolicy(\belief)$ is MLR increasing for any two beliefs on these special line segments $\l(e_{i},\bp)$. The main
reason is that the MLR order is a total order on such lines (not just a partial order) meaning that any two beliefs on  $\l(e_{i},\bp)$ are
MLR orderable. Proving (\ref{eq:submodline}) yields in turn two  very useful results: 
\begin{compactenum}
\item The optimal policy $\optpolicy(\belief)$ of a stopping time POMDP is characterized by  switching curve $\Gamma$; see Theorem \ref{thm:pomdpstop1}. This is illustrated in Figure  \ref{fig:pomdpstopstructure}(b).
\item The optimal   linear approximation to  switching curve $\Gamma$
that  preserves  (\ref{eq:submodline})   can be estimated via a simulation based
stochastic approximation algorithm thereby facilitating a simple controller; see \secn \ref{sec:linear}.
\end{compactenum}

\subsection{MLR Dominance on Lines} Since our plan is to prove (\ref{eq:submodline}) 
on  line segments in the belief space,  we formally define these line segments.
Define the 
sub-simplices, $\Hyperplane_1$ and $\Hyperplane_\statedim$:
\beq \label{eq:hi} \Hyperplane_1 =  \{\belief \in \Belief:  \belief(1) = 0 \}, \;
 \Hyperplane_\statedim =  \{\belief \in \Belief:  \belief(\statedim) = 0 \}
.\eeq
Denote a generic belief state that lies in either $\Hyperplane_1$ or $\Hyperplane_\statedim$ by $\bp$.
 For each such $\bp\in \Hyperplane_i$, $i\in \{1,\statedim\}$, construct the line segment $\l(e_{i},\bp)$ that connects $\bp$ to $e_{i}$. 
Thus each line segment
$\l(e_{i},\bp)$ comprises of belief states 
$\belief$ of the form:
\beq \l(e_{i},\bp) = \{\belief \in \Belief: \belief = (1-\epsilon) \bp + \epsilon e_{i}, \;
0 \leq \epsilon \leq 1 \} ,
 \bp \in \Hyperplane_i.  \label{eq:lines}
 \eeq

To visualize (\ref{eq:hi}) and (\ref{eq:lines}),
Figure \ref{fig:linesexample} illustrates the setup for $\statedim=3$. The   sub-simplex $\Hyperplane_1$ is simply
the line segment $(e_2,e_3)$;  and $\Hyperplane_3$ is the line segment $\{e_1,e_2\}$.
 Also shown are examples
of line segments $\l(e_{1},\bp)$ and $\l(e_{3},\bp)$ for arbitrary points $\bp_1$ and $\bp_2$ in $\Hyperplane_1$ and  $\Hyperplane_3$.

\begin{figure} \centering
\includegraphics[scale=0.31]{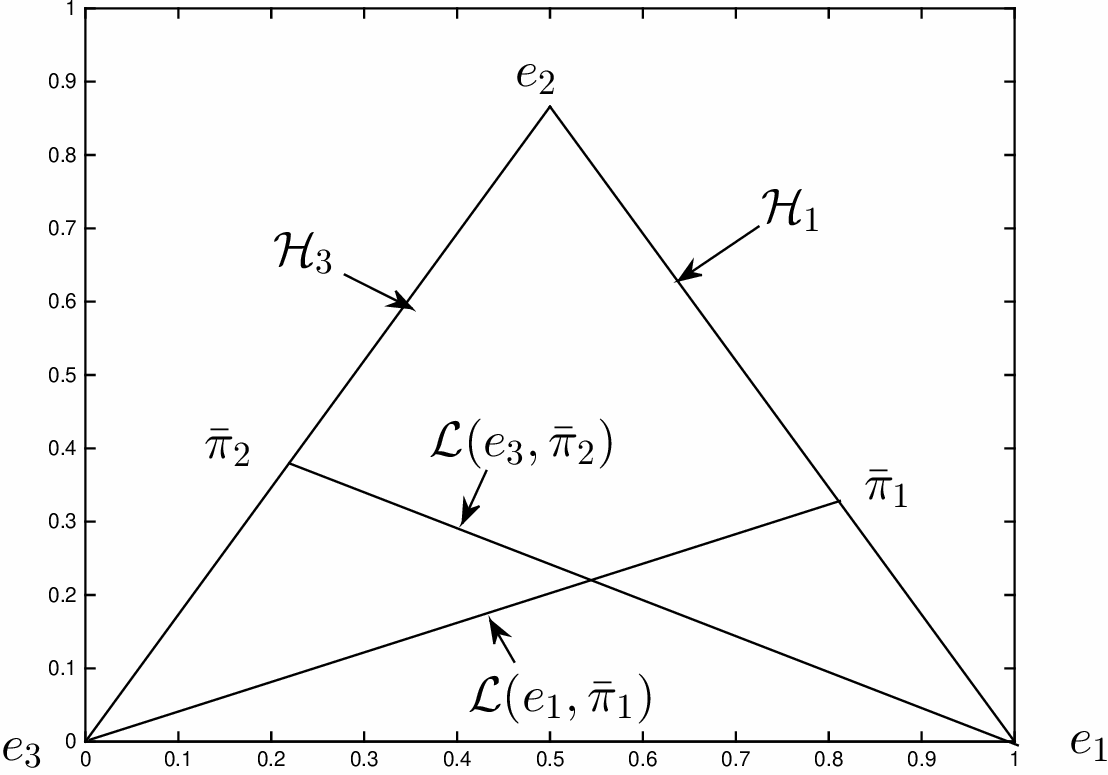}
\caption{Examples of  sub-simplices $\Hyperplane_1$ and $\Hyperplane_3$ and points $\bp_1 \in \Hyperplane_1$,
$\bp_2 \in \Hyperplane_3$.  Also shown are the lines $\l(e_1,\bp_1)$ and $\l(e_3,\bp_2)$ that connect these point to the vertices
$e_1$ and $e_3$.}
\label{fig:linesexample}
\end{figure}
 
 We now define the MLR order on such lines segments. Recall the definition of MLR order $\gr$ on the belief space $\Belief$ in
 Definition \vref{def:mlr}.

\begin{definition}[MLR ordering  ${\gl}$  on  lines]  \label{def:tp2l}
 $\belief_1$ is greater than $\belief_2$ with respect to the MLR ordering on
the line $\l(e_{i},\bp)$, $i \in \{1,\statedim\}$ -- denoted as $\belief_1\gl \belief_2$, if 
$\belief_1,\belief_2 \in \l(e_i,\bp)$ for  some $\bp \in \Hyperplane_i$,
and
$\belief_1 \gr \belief_2$. 
\end{definition}

Appendix \ref{app:mlrdom} shows that 
  the  partially ordered sets $[\l(e_1,\bp),\glX]$ and $[\l(e_\statedim,\bp),\glone]$ are chains, i.e., totally ordered sets. All elements
$\belief_1,\belief_2 \in \l(e_{X},\bp)$ are comparable, i.e., either $\belief_1\glX \belief_2$ or $\belief_2 \glX \belief_1$
(and similarly for  $\l(e_{1},\bp)$). The largest element (supremum) of $[\l(e_1,\bp),\glX]$ is $\bp$ and the smallest element (infimum) is $e_1$.

\subsection{Submodularity with MLR order} \index{submodularity! w.r.t.\ MLR partial order}

To prove the structural result (\ref{eq:submodline}), 
we will show that
 $Q(\belief,u)$  in   (\ref{eq:bellmanstop2}) is a submodular function on the  chains
 $[\l(e_1,\bp),\geq_{L_1}]$ and $[\l(e_1,\bp),\geq_{L_\statedim}]$.
This requires less restrictive conditions than submodularity on the entire simplex $\Belief$. 

\begin{definition}[Submodular function] \label{def:supermod} Suppose $i = 1$ or $X$. Then
 $f:\l(e_i,\bp)\times \actionspace \rightarrow \reals$  is  submodular 
if 
$f(\belief,u) - f(\belief,\bar{u}) \leq f(\tilde{\belief},u)-f(\tilde{\belief},\bar{u})$, for
$\bar{u} \leq u$, $\belief \gl \tilde{\belief}$.
\end{definition}

A more general definition of submodularity on a lattice is given in Appendix \vref{sec:appsupermod}.
Also Appendix \ref{app:mlrdom} contains additional properties that will be used in  proving the main theorem below.

The following key result says that for a submodular function  $Q(\belief,u)$, there exists a version of the optimal policy 
$\mu^*(\belief)=\argmin_u Q(\belief,u)$  that is MLR increasing on lines.

\begin{theorem}[Topkis Theorem] \label{res:monotone}
\label{res:supermod} Suppose $i=1$ or $X$. If $f:\l(e_i,\bp)\times \actionspace \rightarrow \reals$ is submodular, then
there exists a $\mu^*(\belief) = \argmin_{u\in \actionspace} f(\belief,u)$, that  is increasing on $[\l(e_i,\bp),\gl]$,
i.e., $\tbelief \gl {\belief} \implies \mu^*(\belief) \leq \mu^*(\tbelief)$.
\end{theorem}

\subsection{Assumptions and Main Result} 
 For convenience , we repeat \vref{A1} and \ref{A2}, \vref{A3}. The main additional assumption below  is the submodularity assumption (S).
\begin{myassumptions}
\item[\ref{A1}]   $\belief_1 \gs \belief_2$ implies 
$\Cost(\belief_1,\action) \leq \Cost(\belief_2,\action)$ for each $\action$.  \\ 
For linear  costs, the condition  is:  $\cost(\state,\action) $ is decreasing in $\state$ for each~$\action$.
\item[\ref{A2}]  $\oprob$ is    totally positive of order 2 (TP2).
\item[\ref{A3}]  $\tp$ is    totally positive of order 2 (TP2).  \index{TP2 matrix}
\item[\nl{S}{(S)}]   
 $\Cost(\belief,\action)$ is submodular on $ [\l(e_\statedim,\bp),\glX]$ and $[\l(e_1,\bp),\glone]$.
\\
For linear costs the condition is $\cost(\state,2) - \cost(\state,1) \geq \cost(\statedim,2) - \cost(\statedim,1) $ and
 $\cost(1,2) - \cost(1,1) \geq \cost(\state,2) - \cost(\state,1) $.
\end{myassumptions}

\begin{theorem}[Switching Curve Optimal Policy] \label{thm:pomdpstop1} 
Assume \ref{A1}, \ref{A2}, \ref{A3} and \ref{S} hold for a  stopping time POMDP.  Then:
\begin{compactenum}
\item There
  exists an optimal policy $\mu^*(\belief)$ that is $\glX$ increasing
on lines $\l(e_X,\bp)$ and $\glone$ increasing on lines $\l(e_1,\bp)$.
\item Hence there exists a 
 threshold switching curve $\Gamma$ 
that partitions 
 belief  space $\Belief$ into two individually connected\footnote{A set  is connected if it cannot be expressed as the union of                     disjoint nonempty closed sets \cite{Rud76}.}
regions $\region_1$, $\region_2$, such
that
the optimal   policy is
\beq  \mu^*(\belief) = \begin{cases} \text{continue} = 2 & \text{ if } \belief \in \region_2 \\
                                  \text{stop} =1 & \text{ if } \belief \in \region_1 \end{cases}
                    \label{eq:mustar}              \eeq
                                      The threshold curve $\Gamma$ intersects each line $\l(e_X,\bp)$ and $\l(e_1,\bp)$ at most once.
                \item
                  There exists
$i^* \in \{0,\ldots,X\}$, such that  $e_1,\ldots,e_{i^*} \in \region_1$ and $e_{i^*+1},\ldots,e_X \in \region_2$. 
\item
 For the case $\statedim=2$,  there exists a unique threshold point $\belief^*(2)$.
 \end{compactenum}
  \qed            
\end{theorem}

 Let us explain the  intuition behind the proof of the theorem.
 As shown in Theorem \vref{thm:pomdpmonotoneval}, 
\ref{A1},  \ref{A2}  and \ref{A3} are sufficient conditions  for the value function $\optvalue(\belief)$ to be MLR decreasing in $\belief$.

 \ref{S} is sufficient for the costs $\cost_\action^\p \belief$ to be submodular on lines $\l(e_\statedim,\bp)$ and $\l(e_1,\bp)$.
Finally \ref{A1},\ref{A2} and \ref{S} are sufficient for  $Q(\belief,u)$ to be  submodular
 on lines $\l(e_X,\bp)$ and $\l(e_1,\bp)$.   \index{submodularity! on line segments in unit simplex}

As a result, Topkis Theorem \ref{res:monotone} implies
that the optimal policy is monotone
 on each chain  $[\l(e_{X},\bp),\glX]$. 
So there exists
a threshold belief state on each line $\l(e_{X},\bp)$ where the optimal policy switches
from 1 to 2.  (A similar argument holds for lines  $[\l(e_{1},\bp),\glone]$). \\
The entire simplex $\Belief$ can
be covered by the union of lines $\l(e_X,\bp)$. The union of the resulting
threshold belief states yields the switching curve $\Gamma(\belief)$. This is illustrated in
Figure \vref{fig:structure}.

\begin{figure} \centering
\includegraphics[scale=0.35]{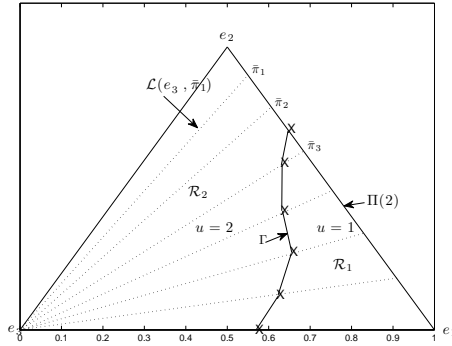}
\caption{Illustration of switching decision curve $
\Gamma$ for optimal policy of a stopping time POMDP. Here $X=3$ and hence  $\Belief$ is an equilateral
triangle. Theorem
\ref{thm:pomdpstop1} shows that for a nonlinear cost POMDP, the stopping region $\region_1$  is a connected set and $e_1\in \region_1$.
(Recall that for linear stopping cost, $\region_1$ is convex from Theorem \ref{thm:pomdpconvex}.)
 Also  
$\region_2$ is connected.
 The lines
segments $\l(e_X,\bp_1) $ connecting the sub-simplex $\Pi(2)$ to $e_3$ are
defined in (\ref{eq:lines}). Theorem \ref{thm:pomdpstop1} says that the threshold curve $\Gamma$ can intersect
each line $\l(e_X,\bp) $ only once. Similarly, $\Gamma$ can intersect each line $\l(e_1,\bp)$ only once (not shown).
} \label{fig:structure}
\end{figure}

\section{Characterization of Optimal Linear Decision Threshold for Stopping time POMDP} \label{sec:linear}
\index{threshold policy! stopping time POMDP|(}
\index{stopping time POMDP! optimal parametrized linear policy|(}
In this section, we assume \ref{A1}, \ref{A2}, \ref{A3} and \ref{S}  hold. Therefore 
 Theorem~\ref{thm:pomdpstop1} applies and 
 the optimal policy $\optpolicy(\belief)$ is characterized by a switching  curve $\Gamma$ as illustrated in Figure \ref{fig:structure}.
 
  How can the switching curve $\Gamma$
 be estimated (computed)?
 In general, any user-defined basis function approximation can be used to parametrize
this curve. However,  such parametrized policy needs to capture the essential feature
of Theorem \ref{thm:pomdpstop1}: it needs to be MLR increasing on lines.

\subsection{Linear Threshold Policies}
We derive  the optimal {\em linear} approximation
to the switching curve $\Gamma$
on 
simplex $\Belief$.
 Such a linear decision threshold has two attractive properties:
(i) Estimating it is  computationally  efficient.
(ii) We give  conditions on the  coefficients of the linear threshold that are necessary and
sufficient for the resulting  policy to be MLR increasing
on lines. Due to the necessity and sufficiency of the condition, optimizing
over the space of linear thresholds on $\Belief$   yields the ``optimal" linear
approximation to  switching curve $\Gamma$.

Since $\Belief$ is a subset of $\reals^{\statedim-1}$, a linear hyperplane on $\Belief$ is parametrized by $\statedim-1$ coefficients.
So, on  $\Belief$, define the linear threshold policy $\mu_{\theta}(\belief)$ as 
\beq \label{eq:linear}
\mu_{\theta}(\belief)  = \begin{cases}
\text{stop} = 1 & \text{ if }  \begin{bmatrix} 0 & 1 & \theta^\p\end{bmatrix}^\p \begin{bmatrix}\belief \\ -1 \end{bmatrix} < 0
\\
\text{continue} = 2 & \text{ otherwise }  
  \end{cases} \quad \belief \in \Belief.
 \eeq
 Here $\theta = (\theta(1),\ldots,\theta(\statedim-1))^\p \in \reals^{\statedim-1}$ denotes the parameter vector of the linear threshold policy.

  Theorem \ref{thm:dep} below characterizes the optimal  linear decision threshold approximation 
 to the threshold switching curve on $\Belief$.
 Assume conditions \ref{A1}, \ref{A2} and   \ref{S}  hold
 so that  Theorem \ref{thm:pomdpstop1} holds.
Also assuming that the stopping region $\region_1$ is non empty, then
 $e_1$ lies in the stopping set. This 
 implies that the $(\statedim-1)$-th component of $\theta$ satisfies
$\theta(\statedim-1) >0$.

\begin{theorem}[Optimal Linear Threshold Policy] 
\label{thm:dep}
For belief states $\belief \in \Belief$,
the   
linear threshold policy $\mu_{\theta}(\belief)$ defined
in (\ref{eq:linear}) is  \\
(i)  MLR increasing
on lines $\l(e_\statedim,\bp)$  iff   $\theta(\statedim-2) \geq 1 $ and $\theta(i) \leq \theta(\statedim-2)$ for $i< \statedim-2$. \\
(ii) MLR increasing
on lines  $\l(e_1,\bp)$  iff $\theta(i)\geq 0$,
for $i<\statedim-2$.\qed
\end{theorem}
\begin{proof}
Given any $\belief_1,\belief_2 \in \l(e_\statedim,\bp)$ with $\belief_2\glX \belief_1$, we need to prove:
$\mu_\theta(\belief_1) \leq \mu_\theta(\belief_2)$ iff  $\theta(X-2) \geq 1 $, $\theta(i) \leq \theta(X-2)$ for $i< X-2$.
But from the structure of (\ref{eq:linear}), obviously $\mu_\theta(\belief_1) \leq \mu_\theta(\belief_2)$ 
is equivalent to 
$ \begin{bmatrix} 0 & 1 & \theta^\p\end{bmatrix}^\p \begin{bmatrix}\belief_1 \\ -1 \end{bmatrix} 
\leq \begin{bmatrix} 0 & 1 & \theta^\p\end{bmatrix}^\p \begin{bmatrix}\belief_2 \\ -1 \end{bmatrix}$,
or equivalently, 
$$\begin{bmatrix} 0 & 1 & \theta(1) & \cdots & \theta(X-2)\end{bmatrix} (\belief_1 - \belief_2 ) \leq 0. $$
Now from Lemma \ref{lem:subquad}(i),  $\belief_2 \glX\belief_1 $ implies that
 $\belief_1 = \epsilon_1 e_{X} + (1-\epsilon_1) \bp$,
$\belief_2 = \epsilon_2 e_{X} + (1-\epsilon_2) \bp$ and $\epsilon_1 \leq \epsilon_2$.
Substituting these into the above expression, we need to prove
$$
( \epsilon_1 -  \epsilon_2)\bigl(\theta(X-2) -  \begin{bmatrix} 0 & 1 & \theta(1) & \cdots & \theta(X-2)\end{bmatrix}^\p \bp \bigr) \leq 0, \quad \forall \bp \in \Hyperplane_X$$  iff  $\theta(X-2) \geq 1 $, $\theta(i) \leq \theta(X-2)$, $i < X-2$. This is obviously true.

A similar proof shows that  on lines $\l(e_1,\bp)$ the linear
threshold policy satisfies
$\mu_\theta(\belief_1) \leq \mu_\theta(\belief_2)$ iff  $\theta(i) \geq 0$ for $i < X-2$.
\end{proof}

As a consequence of Theorem \ref{thm:dep}, the optimal
 linear threshold
approximation to  switching curve $\Gamma$ of Theorem \ref{thm:pomdpstop1} is 
the solution of the following constrained optimization problem:
\beq
\theta^* = \arg \min_{{\theta} \in \reals^{X}} J_{\mu_{\theta}}(\belief), \; \text{ subject to  $0 \leq \theta(i)\leq \theta(X-2)$, $\theta(X-2) \geq 1$ and $\theta(X-1)>0$
}
\label{eq:tc}\eeq
 where the cumulative cost $J_{\mu_{\theta}}(\belief)$ is obtained as in  (\ref{eq:discountedcost})
by applying threshold policy
$\mu_{\theta}$ in (\ref{eq:linear}). 
 
 \noindent {\em Remark}:
The constraints in (\ref{eq:tc}) are necessary {\em and} sufficient for the linear threshold policy
(\ref{eq:linear})
to be MLR increasing on lines $\l(e_X,\bp)$ an $\l(e_1,\bp)$. That is, (\ref{eq:tc}) 
defines the  set of all
MLR increasing linear threshold policies -- it does not leave out any MLR
increasing polices; nor does it  include
any non MLR increasing policies.
 Therefore optimizing  over the space of MLR increasing linear 
threshold policies yields the optimal linear  approximation to
 threshold curve~$\Gamma$.

\subsection{Algorithm to compute the optimal linear threshold policy} \label{sec:spsastoppomdp}
\index{monotone policy! stopping time POMDP! linear threshold}
In this section a stochastic approximation algorithm is presented to estimate the threshold vector $\theta^*$ in (\ref{eq:tc}).
  Because the  cumulative cost $J_{\mu_\theta}(\belief)$  in (\ref{eq:tc}) of the linear threshold policy $\mu_\theta$ cannot be computed in closed form, we resort to simulation
 based stochastic optimization.  Let 
$n=1,2\ldots,$ denote iterations of the algorithm.  The aim is to  solve the following linearly constrained stochastic optimization problem:
\begin{align} \text{  Compute } 
\theta^* &= \arg\min_{\theta \in \Theta} \E\{ {J}_n(\mu_{\theta})  \} \nn
\\
\text{ subject to } & 0 \leq \theta(i)\leq \theta(X-2), \theta(X-2) \geq 1 \text{ and } \theta(X-1)>0.
 \label{eq:stochopt2}\end{align}
Here the 
sample path  cumulative  cost   $ {J}_n(\mu_{\theta}) $ is evaluated as
\beq
{J}_n(\mu_{\theta}) =  \sum_{k=0}^\infty \discount^{k} C(\belief_k,u_k), \quad
\text{ where } u_k = \mu_\theta(\belief_k) \text{ is computed via (\ref{eq:linear}) }
\label{eq:jnmu}
\eeq
 with prior $\belief_0$ sampled uniformly from $\Belief$. 
 A convenient way of sampling uniformly from $\Belief$ is to use 
the Dirichlet distribution (i.e., $\belief_0(i) = x_i/\sum_i x_i$, where $x_i \sim $ unit exponential distribution).

The above constrained stochastic optimization problem can be solved by a variety of methods.
One method is to convert  it into an equivalent unconstrained
problem via the following parametrization: Let  $\phi = (\phi(1),\ldots\phi(X-1))^\p \in \reals^{X-1}$
and parametrize $\theta$ as
$ \theta^\phi = \begin{bmatrix} \theta^\phi(1),\ldots,\theta^\phi(X-1)
\end{bmatrix}^\p$ where 
\beq
\theta^\phi(i) =
\begin{cases}
\phi^2(X-1) & i =X-1 \\
 1 + \phi^2(X-2) & i = X-2\\
(1+ \phi^2(X-2)) \sin^2(\phi(i)) & i =1,\ldots,X-3
 \end{cases}
\label{eq:transformation}
\eeq
Then 
$\theta^\phi$ trivially  satisfies constraints in (\ref{eq:stochopt2}).
So
(\ref{eq:stochopt2}) is equivalent to the following unconstrained stochastic
optimization problem: 
\begin{align}  \text{  Compute } & \mu_{{\phi^*}}(\belief) \text{ where }
\phi^* = \arg\min_{\phi \in \reals^{X-1}}
 \E\{ {J}_n(\phi)  \} \text{ and }   \nonumber \\
  {J}_n(\phi)  & \text{ is computed using (\ref{eq:jnmu}) with policy 
  $\mu_{\theta^\phi}(\belief)$ in  (\ref{eq:transformation}). }
\label{eq:phi}
\end{align}

Algorithm \vref{alg:spsalinear}
uses the SPSA Algorithm to
generate a sequence of estimates $\hphi_{n}$,
 $n=1,2,\ldots,$ that
converges to a local minimum of  the optimal
linear threshold $\phi^*$ 
 with
policy $ \mu_{\phi^*}(\belief) $.  \index{policy gradient algorithm!  stopping time monotone POMDP}


\begin{algorithm}[h]
\caption{Policy Gradient SPSA Algorithm for computing optimal linear threshold policy} \label{alg:spsalinear}
Assume \ref{A1}, \ref{A2},  \ref{A3}, \ref{S} hold so that the optimal  policy is
characterized by a switching curve in Theorem \ref{thm:pomdpstop1}.\\
Step 1: Choose initial threshold coefficients  $\hphi_{0}$ and linear threshold policy
$\mu_{\hphi_{0}}$.
\\
Step 2: For iterations $n=0,1,2,\ldots$ 
\begin{itemize}
\item Evaluate sample  cumulative cost 
$
{J}_n({\hphi_n})$ using (\ref{eq:phi}). \\
\item   
 Update threshold coefficients  $\hphi_n$  via SPSA Algorithm as
\beq \label{eq:sa}
\hphi_{n+1} = \hphi_n - \epsilon_{n+1} \nablat  {J}_n({\hphi}_n) \eeq
  \end{itemize}
\end{algorithm}

The stochastic gradient algorithm
(\ref{eq:sa}) converges to a  local optimum. So  it is necessary
to try several initial conditions $\hphi_0$. The computational cost  at each
iteration is linear in the dimension of $\hphi$.
\index{stopping time POMDP! optimal parametrized linear policy|)}
\index{threshold policy! stopping time POMDP|)}

\section{Example. Machine Replacement POMDP} We continue here with the machine replacement example discussed in
\index{machine replacement! POMDP! structural result}
 \secn \ref{sec:mrfull} and \secn \ref{sec:mrpomdp}.  
We show that the conditions of Theorem \ref{thm:pomdpstop1} hold and so the optimal policy for machine replacement
 is characterized by a threshold switching curve.
 
 Recall the state space is $\statespace = \{1,2,\ldots,\statedim\}$ where
 state $\statedim$ denotes the best state (brand new machine) while state 1 denotes the worst state, and the action
 space is $\actionspace = \{1 \text{ (replace)}, 2 \text{ (continue)} \}$.
Consider an infinite horizon discounted cost version of the problem. Bellman's equation reads
\begin{align}
 \optpolicy(\belief) &= \argmin_{\action \in \actionspace} Q(\belief,\action),  \quad V(\belief) = \min_{\action \in \actionspace} Q(\belief,\action) \nn \\
 Q(\belief,1) &= R + \discount\, V(e_X) ,  \quad Q(\belief,2) = c_2^\p \belief + \discount\, \sum_{\obs \in \obspace} V(T(\belief,\obs))
 \, \filterd(\belief,\obs) . \label{eq:bellmanmr}\end{align}
Since every time action $\action=1$ (replace)  is chosen, the belief state switches to $e_X$, Bellman's equation 
(\ref{eq:bellmanmr})
is similar to that of a stopping POMDP.
The cost of operating the machine $\cost(\state,\action=2)$ is decreasing in state $x$ since the smaller $x$ is, the higher
the cost incurred due to 
loss of productivity. So \ref{A1} holds. The transition matrix $\tp(2)$ defined in  (\ref{eq:mrtp}) satisfies \ref{A3}. 
Assume that the
 observation matrix $\oprob$ satisfies \ref{A2}.  
 Finally, since $\cost(\state,2)$ is decreasing in $\state$ and $\cost(\state,1)= R$ is independent of $\state$, it follows that $\cost(\state,\action)$
 is submodular and so \ref{S} holds.
 Then from Theorem \ref{thm:pomdpstop1},
 the optimal policy is MLR increasing  and characterized by a threshold switching curve. 
  Also from Theorem \ref{thm:pomdpconvex}, the set of beliefs $\region_1$  where it is optimal to replace
 the machine is convex. Since the optimal policy is MLR increasing, if $\region_1$ is non-empty, then $e_1 \in \region_1$.
Algorithm \ref{alg:spsalinear} can be used to estimate the optimal linear parametrized policy that is MLR increasing.

\begin{subappendices}

\section{Lattices and Submodularity}  \label{sec:appsupermod}   \index{submodularity! definition on a lattice}
Definition \vref{def:supermod}  on submodularity suffices for our treatment of POMDPs. Here
we  outline a more abstract definition; see \cite{Top98} for  details.

{\bf (i) Poset:} Let $\Omega$ denote  a nonempty set and  $\preceq$ denote  a binary relation. Then
$(\Omega,\preceq)$ is a partially ordered set (poset) if 
for any elements $a, b ,c \in \Omega$, the following hold:
\begin{compactenum}
\item  $a \preceq a$ (reflexivity)
\item if $a\preceq b$ and $b \preceq a$, then $a = b$  (anti-symmetry)
\item if $a\preceq b$ and $b \preceq c$, then $a\preceq c$  (transitivity).
\end{compactenum}
For a POMDP, clearly  $(\Belief,\lr)$ is a poset, where $$\Belief = \left\{\belief \in \reals^{\statedim}: \one^{\p} \belief = 1,
\quad 
0 \leq \belief(i) \leq 1 \text{ for all } i \in \statespace \right\}$$
 is the belief space  and $\lr$  is  the MLR order defined
in (\ref{eq:mlrorder}).

{\bf (ii) Lattice:} 
A poset $(\Omega,\preceq)$  is called a lattice if the following property holds: $a, b \in \Omega$, then 
$ a \vee b \ole \max\{a,b\} \in \Omega$ and $ a \wedge b \ole \min\{a,b\} \in \Omega$. (Here $\min$ and $\max$ are with respect to
partial order $\preceq$.)

Clearly, $(\Belief,\lr)$ is a lattice. Indeed if two beliefs $\belief_1,\belief_2 \in \Belief$, then if $\belief_1 \lr \belief_2$, 
obviously  $\belief_1 \vee \belief_2 = \belief_2$ and $\belief_1 \wedge \belief_2 = \belief_1$ belong to $\Belief$. If $\belief_1$ and $\belief_2$
are not MLR comparable, then $\belief_1 \vee \belief_2 = e_\statedim$ and  $\belief_1 \wedge \belief_2 = e_1$, where $e_i$ is the unit vector
with 1 in the $i$-th  position.

Note that $\Omega = \{ e_1,e_2, e_3\}$ is not a lattice if one uses the natural element wise ordering. Clearly,
$e_1 \vee e_2 = (1,1,0)  \notin \Omega$ and $e_1 \wedge e_2 = (0,0,0) \notin \Omega$.

Finally, $\Belief \times \{1,2,\ldots,\actionspace\}$ is also a lattice. This is what we use in our POMDP structural results.

{\bf (iii) Submodular function:} 
Let  $(\Omega,\preceq)$ be a lattice and $f: \Omega \rightarrow \reals$. Then $f$ is  submodular if for all $a,b \in \Omega$,
\beq  f(a) + f(b) \geq f(a \vee b) + f(a \wedge b)  . \label{eq:submodapp} \eeq

For the two component case, namely $a= (\belief_1,\action+1)$, $b = (\belief_2,\action)$, with $\belief_1 \lr \belief_2$, clearly Definition (\ref{eq:submodapp})
is equivalent to   Definition \vref{def:supermod}; and this suffices for our purposes.
When each of $a$ and $b$ consist of more than two components, the proof needs more work;
see \cite{Top98}.

\section{MLR Dominance and Submodularity on Lines} 
 \index{submodularity! on line segments in unit simplex}
Recall the definition of $\l(e_{i},\bp) $ in (\ref{eq:lines}).
\label{app:mlrdom}
\begin{lemma} \label{lem:convex} The following properties hold on
 $[\Belief,\gr]$, $[\l(e_X,\bp),\glX]$.\\
(i) On  $[\Belief,\gr]$, $e_1$ is the least and  $e_X$ is the greatest element.
On $[\l(e_X,\bp),\glX]$, $\bp$ is the least  and $e_X$ is the greatest element and all points 
are MLR orderable.\\
(ii) Convex combinations of MLR comparable belief states form a chain. 
For any $\gamma \in [0,1]$,
$\belief \lr \tbelief \implies \belief \lr \gamma\belief + (1-\gamma) \tbelief \lr \tbelief $.
 \end{lemma}

\begin{lemma} \label{lem:subquad}
 (i) For $i\in \{1,\statedim\}$,
$\belief_1 \gl \belief_2$  is equivalent to  $\belief_j = (1-\epsilon_j) \bp + \epsilon_j e_X$ and $\epsilon_1 \geq \epsilon_2$ for $\bp \in \Hyperplane_i$ where $\Hyperplane_i$ is defined in (\ref{eq:hi}). \\ (ii) So submodularity on $\l(e_i,\bp)$, $i\in \{1,\statedim\}$,
is equivalent to showing  \beq \belief^\epsilon = (1-\epsilon) \bp + \epsilon e_i  \implies
C(\belief^\epsilon,2)-C(\belief^\epsilon,1) \text{ decreasing w.r.t.\  $\epsilon$}  \label{eq:sx}\eeq
\end{lemma}

The proof of Lemma \ref{lem:subquad}  follows
from Lemma \ref{lem:convex} and is omitted.

As an example motivated by controlled sensing,
consider costs that are quadratic in the belief. 
Suppose  $C(\belief,2)-C(\belief,1)$ is of the form $\phi^\p \belief+ \alpha (h^\p \belief)^2$. Then  from (\ref{eq:sx}),  sufficient conditions
for submodularity on  $\l(e_X,\bp)$ and  $\l(e_1,\bp)$ are for $\bp \in \Hyperplane_X$ and $\Hyperplane_1$, respectively,
\beq \phi_X - \phi^\p \bp + 2 \alpha h^\p \belief^\epsilon (h_X - h^\p \bp) \leq 0,\quad 
\phi_1 - \phi^\p \bp + 2 \alpha h^\p \belief^\epsilon (h_1 - h^\p \bp) \geq 0  \label{eq:subg}\eeq
If $h_i \geq 0$ and  increasing or decreasing in $i$, then (\ref{eq:subg}) is equivalent to
\beq
 \phi_X - \phi^\p \bp + 2 \alpha h_X  (h_X - h^\p \bp) \leq 0,\quad 
\phi_1 - \phi^\p \bp + 2 \alpha h_X (h_1 - h^\p \bp) \geq 0  \label{eq:subm}\eeq
where  $\bp \in \Hyperplane_X$ and $\bp \in \Hyperplane_1$, respectively.

\section{Proof of  Theorem \ref{thm:pomdpstop1}}

{\bf Part 1}:
Establishing that  $Q(\belief,u)$ is submodular, requires showing
that $Q(\belief,1) - Q(\belief,2)$ is $\gl$ on lines $\l(e_X,\bp)$ for $i=1$ and $X$. 
Theorem \ref{thm:pomdpmonotoneval}  shows by induction that for each $k$,  $V_k(\belief)$ is $\gr$ decreasing on $\Belief$
 if \ref{A1}, \ref{A2}, \ref{A3} hold. This implies that $V_k(\belief)$ is $\gl$ decreasing on lines $\l(e_X,\bp)$ and  $\l(e_1,\bp)$.
 So to prove $Q_k(\belief,u)$ in (\ref{eq:bellmanstop2}) is 
 submodular, we  only need to show that 
$C(\belief,1) - C(\belief,2)$ is  $\gl$ decreasing on $\l(e_i,\bp)$, $i=1,X$. But this is implied by  \ref{S}.
Since submodularity is closed under pointwise limits \cite[Lemma 2.6.1 and Corollary 2.6.1]{Top98}, it follows that
$Q(\belief,u)$ is submodular on $\gl$, $i=1,X$
Having established $Q(\belief,u)$ is submodular on $\gl$, $i=1,X$,
Theorem \ref{res:monotone}   implies
that the optimal policy $\mu^*(\belief)$ is $\glone$ and $\glX$ increasing on lines.

\noindent {\bf Part 2(a) Characterization of Switching Curve $\Gamma$.}
For each $\bp \in \Hyperplane_X$ (\ref{eq:lines}), construct the  line
segment  $\l(e_X,\bp)$ connecting $\Hyperplane_X$ to $e_X$ as in (\ref{eq:lines}).
By Lemma~\ref{lem:convex}, on the line
segment connecting $(1-\epsilon) \underline{\belief} + \epsilon e_{\statedim}$, all belief states are MLR orderable. 
 Since  $\mu^*(\belief)$ is monotone increasing
for $\belief \in \l(e_X,\bp)$,
moving along this line segment towards $e_{\statedim}$, 
pick the largest $\epsilon$ for which  $\mu^*(\belief) = 1$.  (Since $\mu^*(e_X) = 1$, such an $\epsilon$ always exists). The belief state corresponding
to this $\epsilon$ is the threshold belief state. Denote it by $\thr(\bp) = \belief^{\epsilon^*,\bp} \in \l(e_X,\bp)\text{ where
} \epsilon^* = \sup\{\epsilon \in [0,1] :
 \mu^*(\belief^{\epsilon,\bp})=1\} $.\\
 The above construction
 implies that on $\l(e_X,\bp)$,
there is a unique threshold point $\thr(\bp)$.
Note  that the entire simplex can be covered by considering
all pairs of lines $\l(e_X,\bp)$, for $\bp \in \Hyperplane_X$,
i.e.,
$ \Belief = \cup_{\bp \in \Hyperplane}  \l(e_X,\bp)  $.
Combining all  points $\thr(\bp)$ for all pairs of lines  $\l(e_X,\bp)$,
 $\bp \in \Hyperplane_X$,
yields a unique  threshold curve in $\Belief$ denoted
$
\thr =  \cup_{\bp \in \Hyperplane} \thr(\bp)$.

\noindent {\bf Part 2(b)} 
{\em Connectedness of $\region_1$}: 
Since  $e_1 \in \region_1$, 
call $\region_{1a}$ the subset of $\region_1$ that contains  $e_1$.
Suppose $\region_{1b}$ was a subset of  $\region_1$ that was disconnected from $R_{1a}$.
Recall that 
every point in $\Belief$ lies on a line
segment $\l(e_1,\bp)$ for some $\bp$.
Then such a line segment starting from $e_1 \in \region_{1a}$ would leave
the region $\region_{1a}$, pass through a region where action 2 was optimal, and then intersect the region $\region_{1b}$ where action 1 is optimal.
But this violates the requirement that $\mu(\belief)$ is increasing on $\l(e_1,\bp)$. 
Hence $\region_{1a}$ and $\region_{1b}$ have to be connected. 

{\em Connectedness of $\region_2$}: Assume $e_X \in \region_2$, otherwise $\region_2 = \emptyset $ and there is nothing
to prove. Call the region $\region_2$ that contains $e_X$ as $\region_{2a}$.
Suppose $\region_{2b}\subset \region_2$ is disconnected from $R_{2a}$.
Since 
every point in $\Belief$ can be joined by the line
segment $\l(e_X,\bp)$ to $e_X$. 
Then such a line segment starting from $e_X \in \region_{2a}$ would leave
the region $\region_{2a}$, pass through a region where action 1 was optimal, and then intersect the region $\region_{2b}$ (where action 2 is optimal).
This violates the property that $\mu(\belief)$ is increasing on $\l(e_X,\bp)$. 
Hence $\region_{2a}$ and $\region_{2b}$ are connected.

\noindent {\bf  Part 3}: Suppose $e_i \in \region_1$. Then considering lines $\l(e_i,\bp)$
and ordering $\gl$, it follows that $e_{i-1} \in \region_1$. 
Similarly if $e_i \in \region_2$, then considering lines $\l(e_{i+1},\bp)$ and ordering $\glp$, 
it follows that $e_{i+1} \in \region_2$.

\noindent
{\bf Part 4} follows trivially since for $X=2$, $\Belief$ is a one dimensional simplex.
\end{subappendices}


 \chapter{Quickest Change Detection} \label{chp:stopapply}

 \index{quickest detection|(}
 Chapter \ref{ch:pomdpstop} presented three structural results for stopping time POMDPs: convexity of the stopping region
 (for linear costs),  the existence of a threshold switching curve for the optimal policy (under suitable conditions) and characterization of the optimal
 linear threshold policy.
 This chapter discusses several examples of stopping time POMDPs in quickest change  detection. We will show that for these examples,
 convexity of the stopping set and threshold optimal policies arise naturally.
Therefore, the structural results of  Chapter \ref{ch:pomdpstop}
serve as a unifying theme and give insight into what might otherwise  be considered as a collection of sequential detection methods.
 
\section{Example 1. Quickest Detection with Phase-Distributed Change Time  and Variance Penalty}
\index{quickest detection! phase-distributed change time|(}
Here we formulate quickest  detection of a phase-distributed change time  as a stopping time POMDP and analyze the structure of the
optimal policy.  The reader should review \secn \ref{sec:classicalqd} where the classical quickest detection problem with geometric change times was discussed.
We will consider two generalizations of the classical quickest detection problem:   phase-type (PH) distributed change times   and variance stopping penalty.
 
  PH-distributions  include geometric distributions as a special case and  are used widely in modelling discrete event systems \cite{Neu89}.
The optimal detection of a PH-distributed change point is useful  since the family of  PH-distributions forms a dense subset for the set of all distributions.
As described in \cite{Neu89}, a PH-distributed change time can be modelled
as a multi-state Markov chain with an absorbing state. 
So the space of public belief states $\belief$ now is a multidimensional simplex of probability
mass functions. 
 We will formulate the problem as a stopping time POMDP and characterize the optimal decision policy.

The second generalization we consider is 
 a stopping penalty
comprising of the false alarm  and a  variance penalty. The variance
penalty is essential in stopping problems where one is interested in ultimately estimating the state $x$. It penalizes stopping too soon
if the uncertainty  of the state estimate is large.\footnote{In \cite{BB89}, a continuous time
stochastic control problem is formulated 
with a quadratic stopping cost, and the 
 existence of
the solution to the resulting quasi-variational inequality is proved.}
 The variance penalty results in a stopping cost that is quadratic in the belief state $\belief$.

\subsection{Formulation of Quickest Detection as a Stopping Time POMDP}
Below we formulate the quickest detection problem with PH-distributed change time and variance penalty as a stopping time POMDP.
We can then use the structural results of Chapter \ref{ch:pomdpstop} to characterize the optimal policy.

\subsubsection{Transition Matrix for PH distributed change time}
The change point  $\tau^0$ is modeled by a   
{\em  phase type (PH) distribution}. 
The family of all PH-distributions forms a dense subset for the set of all distributions
	\cite{Neu89} i.e., for any given distribution function $F$ such that $F(0) = 0$, one can find a sequence of PH-distributions	
$\{F_n , n	\geq	1\}$	 to		approximate	$F$	uniformly over $[0, \infty)$.
Thus PH-distributions  can be used to  approximate   change points with an arbitrary distribution. This is done by
	constructing a multi-state  Markov chain as follows:
Assume    state `1'  (corresponding to belief $e_1$) is an absorbing state
and denotes the state after the jump change.  The states $2,\ldots,X$ (corresponding to beliefs $e_2,\ldots,e_X$) can be viewed as  a single composite state that $x$ resides in before the jump. 
To avoid trivialities, assume that  the change occurs after at least one measurement. So the initial distribution $\belief_0$ satisfies $\belief_0(1) = 0$.
The 
transition probability matrix  is of the form
\beq \label{eq:phmatrix}
\tp = \begin{bmatrix}  1 & 0 \\ \underline{\tp}_{(X-1)\times 1} & \bar{\tp}_{(X-1)\times (X-1)} \end{bmatrix}.
\eeq
The ``change time" $\tau^0$ denotes the time at which $x_k$ enters the absorbing state 1:
 \beq \tau^0 = \min\{k: x_k = 1\}.  \label{eq:tau}  \eeq
The distribution of $\tau^0$ is determined by choosing the transition probabilities $\underline{P}, \bar{P}$ in 
 (\ref{eq:phmatrix}).  To ensure that $\tau^0$ is finite, assume states $2,3,\ldots X$ are transient.
This  is equivalent to $\bar{P}$ satisfying  $\sum_{n=1}^\infty \bar{P}^{n}_{ii} < \infty$ for $i=1,\ldots,X-1$ (where $\bar{P}^{n}_{ii}$ denotes the $(i,i)$ element of the $n$-th power
of matrix $\bar{P}$).
 The 
distribution of the absorption time to state 1 is 
\beq \label{eq:nu}
 \nu_0 = \belief_0(1), \quad \nu_k = \bar{\belief}_0^\p \bar{P}^{k-1} \underline{P}, \quad k\geq 1, \eeq
 where $\bar{\belief}_0 = [\belief_0(2),\ldots,\belief_0(X)]^\p$.
The key idea  is that by appropriately choosing the pair $(\belief_0,P)$ 
and the associated state space dimension $X$,
 one can approximate any given discrete distribution on $[0, \infty)$ by the distribution $\{\nu_k, k \geq 0\}$; see 
\cite[pp.240-243]{Neu89}.
 The event $\{x_k = 1\}$ means the change point has occurred at time $k$ according
 to PH-distribution (\ref{eq:nu}). In the special case when $x$ is a 2-state Markov chain,
 the change time $\tau^0$  is geometrically distributed.

\subsubsection{Observations} The  observation $\obs_k  \in \obspace$ given state $\state_k$ has
conditional probability pdf or pmf 
 \beq \label{eq:obs}
\oprob_{\state\obs} = \pdf(\obs_k = \obs| \state_k = \state) , \quad \state \in  \statespace, \obs \in \obspace. \eeq
where $\obspace\subset \reals$ (in which case $\oprob_{xy}$ is a pdf)
or $ \obspace = \{1,2,\ldots,\obsdim\}$ (in which case $\oprob_{xy}$ is a pmf).
In quickest detection,
 states $2, 3,.\ldots, X$ are fictitious and are  defined to generate  the PH-distributed change time $\tau^0$ in (\ref{eq:tau}).
So   states  $2, 3,.\ldots, X$ are indistinguishable in terms of the observation $y$.
That is, the observation probabilities $B$ in (\ref{eq:obs})  satisfy
 \beq B_{2y} =B_{3y} = \cdots = B_{Xy} \text{ for all } y\in \obspace.
\label{eq:obsin}
\eeq

\subsubsection{Actions}   
At each time  $k$, a decision $u_k$ is taken
where 
\beq \label{eq:actionpolicy}
 u_k = \mu(\belief_k)  \in \actionspace = \{1 \text{ (announce change and stop)} ,2 \text{ (continue) }\}. \eeq
In (\ref{eq:actionpolicy}),  the policy $\mu$  belongs to the class of stationary decision
policies.

\subsubsection{Stopping Cost}
If decision $u_k=1$  is chosen, then the decision maker announces that a change has occurred and the problem terminates. If  $u_k=1$ is chosen before  the change point $\tau^0$, then a false
alarm and variance penalty is paid.  If $u_k=1$ is chosen at or after the change point $\tau^0$, then only a variance penalty is paid.
Below these costs are formulated.

Let $\statelvl = (\statelvl_1,\ldots,\statelvl_X)^\p$ specify the  physical state levels  associated with states $1,2,\ldots, X$ of the Markov chain
 $x \in \{e_1,e_2,\ldots,e_\statedim\}$.
 The {\em variance  penalty}   is 
 \beq  \label{eq:varder}
\begin{split} &  \E\{ \|(x_k -  \belief_k)^\p \statelvl \|^2 \mid \info_k \} = \sqg^\p \belief_k(i) - (\statelvl^\p \belief_k)^2,   \\  \text { where } &
\sqg_i = \statelvl_i^2 \text{ and } \sqg=(\sqg_1,\sqg_2,\ldots,\sqg_X) , \\
& \info_k = (y_1,\ldots,y_k,u_0,\ldots,u_{k-1}).
\end{split} \eeq
This  conditional variance   penalizes choosing the stop action if the uncertainty in the state estimate is large.
Recall we discussed  POMDP with nonlinear costs in the context of controlled sensing in \secn \ref{sec:nonlinearpomdpmotivation}.

Next, the false alarm event  $\cup_{i\geq 2}   \{x_k =  e_i\} \cap \{u_k = 1\} = \{x_k\neq e_1\} \cap \{u_k = 1\}$ represents the event
that a change is announced before the change happens at time $\tau^0$.
To evaluate the {\em false alarm penalty}, 
 let  $f_i I(x_k=e_i,u_k=1)$ denote the cost of a false alarm in state $e_i$, $i \in \statespace$, where
$f_i \geq 0$. Of course,
$f_1 = 0$ since a false alarm is only incurred if the stop action is picked in states $2,\ldots, X$.  The expected false alarm penalty is 
\begin{multline}  \sum_{i \in \statespace} f_i \E\{I(x_k=e_i,u_k=1)|\info_k\} = \f^\p \belief_k  I(u_k=1), \\
\text{ where }
  \f = (f_1,\ldots,f_X)^\p, \; f_1 = 0. \label{eq:falsev}\end{multline}
The false alarm  vector $\f$ is chosen with increasing elements so that states  further
from  state~1 incur larger  penalties.

 Then with $\alpha, \beta$ denoting non-negative  constants that weight the relative importance of these costs, the 
 stopping cost expressed in terms of the belief state at time $k$ is
 \beq \label{eq:cp1}
\Cb(\belief_k,u_k=1) = \alpha (\sqg^\p \belief_k - (\statelvl^\p \belief_k)^2) + \beta \, \f^\p \belief_k . \eeq
 One can  view $\alpha$  as a Lagrange multiplier in a stopping time problem
 that seeks to minimize a cumulative cost subject to a variance stopping constraint.

\subsubsection{Delay cost of continuing}
We allow two possible choices for the delay costs for action $u_k=2$:\\
(a) {\em Predicted Delay}: If action $u_k=2$ is taken then 
 $ \{x_{k+1} = e_1, u_k = 2\}$ is the event that no change
is declared at time $k$ even though the state has changed at time $k+1$.
So with  $d$ denoting a non-negative constant,
 $d \,I(x_{k+1} = e_1, u_k=2)$ depicts a {\em delay cost}.
The expected delay cost  for decision $u_k=2$ is 
\beq \Cb(\belief_{k},u_k=2) = d \,\E\{I(x_{k+1}= e_1,u_k=2)|\F_{k}\} 
= d e_1^\p P^\p \belief_k .  \label{eq:exd} \eeq
The above cost is motivated by  applications (e.g., sensor networks) where  if the decision maker chooses $u_k=2$, then
it needs to gather observation $y_{k+1}$  thereby
incurring an additional operational cost denoted as $\cop$. Strictly speaking, $\Cb(\belief,2) = d e_1^\p P^\p \belief + \cop$.
Without loss of generality
 set the constant $\cop$ to zero, as it does not affect our structural results.
The penalty   $d \,I(x_{k+1} = e_1, u_k=2)$ gives incentive for
 the decision maker to predict the state $x_{k+1}$ accurately. \\
(b) {\em Classical Delay}:
Instead of (\ref{eq:exd}), a more `classical'  formulation is that a delay cost is incurred when the event  $ \{x_{k} = e_1, u_k = 2\}$ occurs.
The expected delay cost is \beq  \Cb(\belief_{k},u_k=2) = d\,\E\{I(x_{k} = e_1, u_k=2) | \info_k\}
= d e_1^\p \belief_k  \label{eq:exd2}. \eeq

\noindent  {\em Remark}:  Due to the variance penalty, the   cost $\Cb(\belief,1)$ in (\ref{eq:cp1}) is quadratic in the belief state $\belief$. Therefore,
the formulation cannot be reduced to a standard stopping problem
with linear costs in the belief state.

\noindent{\em Summary}:
 It is clear from the above formulation that quickest detection is simply a stopping-time POMDP of the form (\ref{eq:stopdiscountedcost}) with  transition matrix (\ref{eq:phmatrix}) and costs (\ref{eq:falsev}),
(\ref{eq:exd}) or (\ref{eq:exd2}). In the special case of geometric change time ($\statedim = 2$),  and  delay cost (\ref{eq:exd2}),
 the cost function assumes the classical {\em Kolmogorov--Shiryayev
criterion} for detection of disorder (\ref{eq:ksd}).

\subsection{Main Result. Threshold Optimal Policy for Quickest Detection}

Note first that for $\alpha = 0$ in (\ref{eq:cp1}), the stopping cost is linear. Then  Theorem \ref{thm:pomdpconvex} applies implying that the stopping set $\region_1$ is convex. Below we focus on establishing Theorem \ref{thm:pomdpstop1} to show the existence of a threshold curve for the optimal policy. As discussed
at the end of 
\secn \ref{sec:someintuition}, such a result goes well beyond establishing convexity of the stopping region.

We consider the predicted cost and delay cost cases separately below:

\paragraph{Quickest Detection with Predicted Delay Penalty}
First consider the quickest detection problem with the predicted delay cost (\ref{eq:exd}).
For the stopping cost $\Cb(\belief,1)$ in  (\ref{eq:cp1}), choose    $\f = [0, 1, \cdots, 1]^\p = \one_X - e_1$.
This  weighs the states $2,\ldots,X$ equally in the false alarm penalty. 
With assumption (\ref{eq:obsin}), the variance penalty  (\ref{eq:varder}) becomes $\alpha(e_1^\p  \belief - (e_1^\p \belief)^2)$.
The delay cost $\Cb(\belief,2)$ is chosen as (\ref{eq:exd}).
   To summarize \beq  \label{eq:stylized}
\Cb(\belief,1) = \alpha\left(e_1^\p \belief - (e_1^\p \belief)^2\right) + \beta(1-e_1^\p \belief), \quad
\Cb(\belief,2)  = d e_1^\p P^\p \belief. \eeq


Theorem \ref{thm:pomdpstop1} is now illustrated for the costs (\ref{eq:stylized}).   Before proceeding there is one issue we need to fix.
Even for  linear  costs in (\ref{eq:stylized}) ($\alpha = 0$), denoting $\Cb(\belief,\action) = \cost_\action^\p \belief$, it is
seen that the elements of $\cost_1$ are increasing while the elements of $\cost_2$ are decreasing. So at first sight, it is not possible to apply Theorem \ref{thm:pomdpstop1} since 
assumption \ref{A1} of Theorem \ref{thm:pomdpstop1} requires that the elements of the cost vectors are decreasing. But  the following  novel transformation can be applied, which is nicely described in  \cite[pp.389--390]{HS84} (for
 fully observed MDPs).

Define 
\begin{align}
V(\belief) &= \Vb(\belief) -  (\alpha+\beta)  \falsealarm^\p \belief,
\quad C(\belief,1) = \alpha(\sqg^\p \belief - (\statelvl^\p \belief)^2 ) - \alpha \falsealarm^\p \belief  \nonumber \\
C(\belief,2) &= \Cb(\belief,2) - (\alpha+\beta) \f^\p \belief + \discount(\alpha+\beta) \falsealarm^\p P^\p \belief.
%
 \label{eq:costdef} \end{align}
Then clearly $V(\belief)$ satisfies Bellman's dynamic programming  equation 
\begin{align} \label{eq:dp_algt}
\mu^*(\belief)&= \arg\min_{\action \in \actionspace} Q(\belief,u) , \;J_{\mu^*}(\belief) = V(\belief) = \min_{u \in \{1,2\}} Q(\belief,u),\\
 \text{ where }  Q(\belief,2) &=  C(\belief,2)
+ \discount \sum_{y \in \obspace}  V\left( \filter(\belief ,y) \right) \filterd(\belief,y),\quad
Q(\belief,1) =   C(\belief,1)
  \nonumber
\end{align}
Even though the value function is now changed, 
the optimal policy $\mu^*(\belief)$ and hence stopping set $\region_1$ remain unchanged with this coordinate transformation.
The nice thing is that the new costs $C(\belief,\action)$  in (\ref{eq:costdef}) can be chosen to be   decreasing under suitable conditions. For example,
if $\alpha = 0$, then clearly $C(\belief,1) = 0$ (and so is decreasing by definition) and it is easily checked that if  $d \geq \discount \beta$ then
$C(\belief,2) $ is also decreasing.

With the above transformation, we are now ready to apply Theorem \ref{thm:pomdpstop1}.
Assumptions \ref{A1} and \ref{S}  of
Theorem \ref{thm:pomdpstop1} 
specialize
to the following assumptions on the transformed costs in (\ref{eq:costdef}):
\begin{myassumptions}
\item[(C-Ex1)] 
$ d \geq \discount(\alpha+\beta)$
\item[(S-Ex1)] 
$(d-\discount(\alpha+\beta))(1-P_{21}) \geq  \alpha -\beta $
 \end{myassumptions}

\begin{theorem} Under (C-Ex1), \ref{A2}, \ref{A3} and (S-Ex1),  Theorem \ref{thm:pomdpstop1} holds implying  that the optimal 
policy  for quickest
detection with PH-distributed change times has a threshold structure. Thus  Algorithm \ref{alg:spsalinear}  estimates
the optimal linear threshold. 
\end{theorem}  \index{structural result! variance penalized quickest detection}
 
\begin{proof}
The proof follows from Theorem \ref{thm:pomdpstop1}.  We only need to show that (C-Ex1) and (S-Ex1) hold.
These are specialized versions of conditions \ref{A1} and \ref{S} arising in Theorem \ref{thm:pomdpstop1}.
 
 First consider (C-Ex1):
 Consider Lemma \ref{lm:monotone_nl_cost} with the choice
of  $\phi=2 \alpha e_1$, $h  = e_1$ in  (\ref{eq:suffmon}).
This yields $2 \alpha \geq 0$ and $2 \alpha \geq 2 \alpha$ which always hold. So $C(\belief,1)$ is
$\gr$ decreasing in $\belief$ for any non-negative $\alpha$.
It is easily verified that (C-Ex1) is sufficient for the linear cost $C(\belief,2)$ to be decreasing.

Next consider (S-Ex1).  Set    $\phi_i =  ( d- \discount(\alpha+\beta))P^\p e_i + (\beta-\alpha) e_1$, $h = e_1$  in (\ref{eq:subm}).
The first inequality is equivalent to:
(i) $ (d-\discount(\alpha+\beta)) (P_{X1}-P_{i1}) \leq 0$ for $i \geq 2$
 and (ii) $(d - \discount(\alpha+\beta))(1-P_{X1}) \geq \alpha-\beta $.
Note that (i)  holds if $d \geq \discount(\alpha+\beta)$.
The second inequality in (\ref{eq:subm}) is equivalent to
$(d-\discount(\alpha+\beta)) (1 - P_{i1}) \geq \alpha - \beta$. 
Since $P$ is TP2, from Lemma \ref{lem:tp21} it follows that 
(S-Ex1) is sufficient for these inequalities to hold.
\end{proof}

\paragraph{Quickest Detection with Classical Delay Penalty}
Finally,  consider the `classical' delay cost $\Cb(\belief,2)$ in  (\ref{eq:exd2}) and stopping cost $\Cb(\belief,1)$ in (\ref{eq:cp1}) with $\statelvl$
in (\ref{eq:obsin}).
 Then
\beq  \Cb(\belief,1) = \alpha\left(e_1^\p \belief - (e_1^\p \belief)^2\right)  + \beta \, \f^\p \belief_k , \quad
\Cb(\belief,2) =  d e_1^\p\belief . \label{eq:modified2}\eeq
Assume that  the decision maker designs
the false alarm vector $\f$ to satisfy the following 
 linear constraints:
\begin{myassumptions}
\item[(AS-Ex1)]
 (i) $f_i \geq \max\{1, \discount \frac{\alpha+\beta}{\beta} \f^\p P^\p e_i + \frac{\alpha-d}{\beta}\}$, $i\geq 2$.\\
(ii) $f_j - f_i  \geq \discount \f^\p P^\p( e_j -e _i) $, $j \geq i, i \in \{2,\ldots,X-2\}$ \\
(iii) $f_X - f_i \geq \frac{\discount(\alpha+\beta)}{\beta} \f^\p P^\p (e_X-e_i)$, $i \in \{2,\ldots,X-1\}$.
\end{myassumptions}
Feasible choices of $\f$ are easily obtained by a linear programming solver.
\index{quickest detection! phase-distributed change time|)}

Then Theorem \ref{thm:pomdpstop1} continues to hold, under conditions  (AS-Ex1), \ref{A2},\ref{A3}.

{\em Summary}: We modeled  quickest detection  with PH-distributed change
time  as a multi-state POMDP. We then gave sufficient conditions for the optimal policy  to have a threshold structure. The optimal linear parametrized policy can be estimated via the policy gradient Algorithm
\ref{sec:spsastoppomdp}.

\section{Example 2: Risk Sensitive Quickest Detection with Exponential Delay Penalty} \label{sec:risk}
\index{quickest detection! risk sensitive}

In this example, we generalize the results of  \cite{Poo98}, which deals with exponential delay penalty
and geometric change times. We consider  exponential delay penalty with PH-distributed change time. Our formulation
involves risk sensitive partially observed stochastic control.
We first show  that the exponential penalty cost function 
in \cite{Poo98} is a special case of risk-sensitive stochastic control cost function when the state space dimension $X=2$.  We then use the
risk-sensitive stochastic control formulation to derive structural results for 
PH-distributed change time. In particular, the main result below (Theorem~\ref{thm:risk}) shows that 
the threshold switching curve still characterizes the optimal stopping region $\region_1$. The assumptions and main results are conceptually similar to Theorem \ref{thm:pomdpstop1}.

Risk sensitive stochastic control with exponential cost has been studied extensively  \cite{DR79,JBE94,BR13}.
 In simple terms,  quickest time detection seeks to optimize the  objective  $ \E\{J^0\}$
 where $J^0  $ is the accumulated  sample path cost until some stopping time $\tau$.
 In risk sensitive control, one seeks to optimize $J = \E\{\exp(\risk J^0\}$. For $\risk > 0$,  $J$ penalizes
 heavily large sample path costs due to the presence of second order moments.  This is termed a risk-averse control.
 Risk sensitive control provides a nice formalization of the exponential penalty delay cost  and allows us to generalize
 the results in  \cite{Poo98} to phase-distributed change times. 

 Below, we will use $c(e_i,u=1)$ to denote  false alarm costs and
$c(e_i,u=2)$ to denote delay costs, where $i \in \statespace$.
Risk sensitive control \cite{Ben92} considers the exponential  cumulative  cost function
\beq \label{eq:risk}
J_\mu(\belief) = \Ep\biggl\{ \exp \biggl( \risk \sum_{k=0}^{\tau-1} c(x_k,u_k=2) + \risk \,c(x_\tau, u_\tau = 1) \biggr)\biggr\} \eeq
where $\risk >  0$ is the risk sensitive parameter.

Let us first
show that the exponential penalty cost in \cite{Poo98} is  a special case of (\ref{eq:risk}) for consider the case $X=2$ (geometric distributed change time). For the state $x \in \{e_1,e_2\}$, choose
$c(x,u=1) = \beta I(x\neq e_1,u=1) = \beta (1 -e_1^\p x)$ (false alarm cost) , $c(x,u=2) = d\, I(x=e_1,u=2)
= d e_1^\p x
$ (delay cost).
Then it is easily seen that
$\sum_{k=0}^{\tau-1} c(x_k,u_k=2) +  c(x_\tau, u_\tau = 1) =  d \, |\tau - \tau^0|^+ + \beta I(\tau < \tau^0)$. Therefore   (recall $\tau^0$ is defined in (\ref{eq:tau}) and $\tau$ is defined
in (\ref{eq:tauu})),
\begin{align}
J_\mu(\belief) &= \Ep\bigl\{ \exp \bigl(\risk d \, |\tau - \tau^0|^+ + \risk \beta I(\tau < \tau^0) \bigr)\bigr\} \left[I(\tau <\tau^0) + I(\tau = \tau^0) + I(\tau > \tau^0) 
\right] \nonumber\\
&= \Ep\bigl\{\exp(\risk \beta) I(\tau < \tau^0) + \exp(\risk d\,  |\tau-\tau^0|^+) I(\tau > \tau^0) + 1\bigr\}\nonumber\\
&=\Ep\bigl\{ (e^{\risk \beta} - 1) I(\tau < \tau^0) + e^{\risk d |\tau-\tau^0|^+ }\bigr\} \nonumber\\
&= (e^{\risk \beta} - 1) \Pp(\tau < \tau^0) + \Ep \{e^{\risk d|\tau-\tau^0|^+ }\} \label{eq:poorexp}
\end{align}
which is identical to  exponential delay cost function in \cite[Eq.40]{Poo98}. Thus the Bayesian quickest time detection
with exponential delay penalty in  \cite{Poo98} is a special case of a risk sensitive stochastic control problem.

We consider the delay cost as in (\ref{eq:exd}); so
for state $x\in \{e_1,\ldots,e_X\}$,  $c(x,u_k=2) = d e_1^\p P^\p  x$. To get an intuitive
feel for this modified delay cost function, for the case  $X=2$,
$$\sum_{k=0}^{\tau-1}c(x_k,u_k=2) + c(x_\tau,u_\tau=1) = d|\tau - \tau^0|^+ + \beta I(\tau < \tau^0) 
+ dP_{21} (\tau^0-1) I(\tau^0 < \tau)$$
Therefore,
 for $X=2$, the exponential delay  cumulative cost function is
\beq
J_\mu(\belief)= (e^{\risk \beta} - 1) \Pp(\tau < \tau^0) +  \Ep\bigl\{ e^{\risk d \left[ |\tau-\tau^0|^+ + P_{21}(\tau_0-1)I(\tau_0<\tau) \right]}\bigr\}.
\eeq
This is similar to (\ref{eq:poorexp}) except for the additional term $P_{21}(\tau_0-1)I(\tau_0<\tau)$ in the exponential.

With the above motivation,  we consider risk sensitive quickest detection
for PH-distributed change time, i.e. $X \geq 2$.
Let $\belief$ denote the risk sensitive belief state, see \cite{EAM95,JBE94} for  extensive descriptions of the risk sensitive belief state and verification theorems 
for dynamic programming in  risk sensitive control.
 Bellman's equation reads
 \begin{align} 
&\Vb(\belief) = \min\{ \Cb(\belief,1) ,\sum_{y\in \obspace} \Vb(T(\belief,y)) \sigma(\belief,y) \}  \quad \text{ where } \label{eq:dprisk}  \\
& \Cb(\belief,1) = R_1^\p  \belief, \quad
T(\belief,y) = \frac{B_y P^\p \text{diag}(R_2) \belief}{\sigma(\belief,y)}, \;
\sigma(\belief,y) = \one^\p B_y P^\p \text{diag}(R_2) \belief  \nn \\
&R_1 = (1, e^{\risk \beta},\ldots, e^{\risk \beta})^\p, \; R_2 = (e^{\risk d}, e^{\risk d P_{21}}, \ldots, e^{\risk d P_{X1}})^\p,\;
B_y= \diag(\oprob_{1y},\ldots \oprob_{\statedim y} ). \nn
 \end{align}   

Similar to the transformation used in (\ref{eq:costdef}),
 define
$V(\belief) = \Vb(\belief) - \Cb(\belief,1)$. Then $V(\belief)$ satisfies Bellman's equation (\ref{eq:dp_algt}) with
\beq C(\belief,1) = 0,\quad 
C(\belief,2) = R_1^\p (P^\p  \text{diag}(R_2) - I)  \belief. 
\label{eq:riskcosts}
\eeq

Assume the following condition holds
\begin{myassumptions}
\item[(C-Ex3)] The elements of $R_1^\p (P^\p  \text{diag}(R_2) - I)$ are decreasing w.r.t.\  $i=1,2,\ldots,X$.
\end{myassumptions}
Evaluating $ C(\belief,2) = R_1^\p (P^\p  \text{diag}(R_2) - I)  \belief$, then (C-Ex3) 
 is equivalent to
$$ e^{\risk d}-1 \geq e^{\risk d P_{21}}(P_{21} + e ^{\risk \beta}(1 - P_{21}))-e ^{\risk \beta} \text{ and }
e^{\risk d P_{i1}}(P_{i1}+ e^{\risk \beta}(1 - P_{i1})) $$
 decreasing in $i\in \{2,\ldots,X\}$.
For example, if $d=\epsilon=1$, then for $\beta \geq 1$, the following are verified by elementary 
calculus:\\
(i)  (C-Ex3) always holds
for $\beta \geq 1$ when $X=2$ (geometric distributed change time). \\
(ii) For PH-distributed change time,
if \ref{A3} holds, then (C-Ex3) always holds
providing $P_{21} < 1/(e^\beta-1)$.

 \index{structural result! risk sensitive quickest detection}
\begin{theorem}\label{thm:risk}
The stopping region $\region_1$ is a convex subset of $\Belief$.
Under (C-Ex3), \ref{A2}, \ref{A3}, Theorem \ref{thm:pomdpstop1} holds. Thus  Algorithm \ref{alg:spsalinear}  estimates
the optimal linear threshold. \qed
\end{theorem}

\begin{proof}
The only difference compared to a standard stopping time POMDP  is the update of the belief state (\ref{eq:dprisk}) which
now includes the term $\text{diag}(R_u)$. The elements of $R_u$ are non-negative and 
functionally independent
of the observation $y$. Therefore the three main requirements
that $T(\belief,y)$ is MLR increasing in $\belief$, $T(\belief,y)$  is MLR increasing in $Y$,
and $\sigma(\belief,:)$ is $\gr$ increasing in $\belief$ continue to hold. Then the rest of the proof is identical
to Theorem \ref{thm:pomdpstop1}.
\end{proof}

\noindent{\em Remarks: (i) Delay Formulation in \cite{Poo98}}: Consider the formulation in \cite{Poo98} which is equivalent to  (\ref{eq:poorexp}). Then for 
the geometric distributed case $X=2$, the convexity of $\region_1$ holds using a similar proof to above. Since $\Belief$ is a 1-dimensional simplex
and $e_1 \in \region_1$, convexity implies there exists (a possible degenerate) threshold point $\belief^*$ that characterizes $\region_1$ such
that the optimal policy is of the form (\ref{eq:onedim}).
As a sanity check,  the 
analogous condition
to (C-Ex3) reads $e^{\risk d} -1 > P_{21}(1 - e^{\risk \beta})$. This always holds for $\risk \geq 0$. Therefore, assuming \ref{A2} holds,
the above theorem holds
for the  exponential delay penalty case under \ref{A2}.  (Recall \ref{A3} holds trivially when $X=2$).
Finally, for $X>2$, using a similar proof,  the
conclusions of Theorem \ref{thm:risk} hold.

\section{Example 3:  Multi-agent Social Learning} \label{sec:change}
\index{quickest detection! multi-agent|(}
This section  deals with a multi-agent 
Bayesian stopping time problem where
 agents perform greedy social learning and reveal their actions to subsequent agents. 
 Given such a protocol of local decisions, how can the multi-agent system  make a global decision when to stop?
 We show that 
the optimal decision policy of the stopping time problem has multiple thresholds.   The motivation for such problems arise
in automated decision systems (e.g., sensor networks) where agents make local
decisions and reveal these local decisions to subsequent agents. The multiple threshold behavior of
the optimal global decision shows that making global decisions
based on local decision involves non-monotone policies. 

\subsection{Motivation: Social Learning amongst myopic agents}\label{sec:herdaa}

We refer the reader to \cite{Cha04,KP14,KNH14} for details of social learning.
Consider a multi-agent system with agents indexed as $k=1,2,\ldots$ performing social learning  to estimate an underlying random state $x$
with prior $\belief_0$.
(We assume in this section that $x$ is a random variable and not a Markov chain.)
Let $y_k \in \obspace =  \{1,2,\ldots,Y\}$ denote the private observation of agent $k$ and $a_k \in \A =  \{1,2,,\ldots, A\}$ denote the local action agent $k$ takes. Define:
\begin{align}  \mathcal{H}_k &= 
(a_1,\ldots,a_{k-1},y_k),   \quad 
\mathcal{G}_k =  
 (a_1,\ldots,a_{k-1},a_k).   \label{eq:sigg} \end{align}

Let us highlight the key Bayesian
update equations in the social learning protocol; see \cite{Cha04,KP14,KH15}:\\
A time $k$,  based on its private observation $\obs_k$ and public belief $\belief_{k-1}$, agent $k$:\\
1. Updates its private belief $\priv_k = \E\{x|\mathcal{H}_k\} $
as
\beq
\priv_k =\frac{ \oprob_{\obs_k} \belief_{k-1}}{ \one^\p \oprob_{\obs_k} \belief_{k-1}}  \label{eq:mudef} \eeq
2. Takes local myopic  action $a_k = \arg\min_{a\in \A} \{c_a^\p\priv_k\}$ where $\A = \{1,2,,\ldots, A\}$ denotes the set of local actions.\\
3. Based on $a_k$, the public belief  $\belief_k = \E\{ \state_k | \mathcal{G}_k\} $ is updated (by subsequent agents) via the social learning filter
(initialized with $\belief_0$)
\beq \pi_k = \filter(\pi_{k-1},a_k), \text{ where } \filter(\pi,a) = 
 \frac{\Bs_a \pi}{\filterd(\pi,a)},
\; \filterd(\pi,a) = \mathbf{1}_X^\p \Bs_a \pi
\label{eq:piupdatesl} \eeq
In (\ref{eq:piupdatesl}), $\Bs_a  = \text{diag}(P(a|x=e_i,\pi),i\in \statespace ) $ with elements
\begin{align} \label{eq:aprob}
 P(a_k = a|x=e_i,\pi_{k-1}=\pi) = \sum_{y\in \obspace} P(a_k=a|y,\pi)P(y|x=e_i) \\
= \sum_{y\in \obspace} \prod _{\ta \in \A - \{a\}}I(c_a^\p B_{y} \pi < c_{\ta}^\p B_{y} \pi) P(y|x=e_i) \nn \end{align}
Here $I(\cdot)$  is  the indicator function and $\oprob_y = \diag(\prob(\obs|\state=e_1),\ldots, \prob(\obs|\state= e_\statedim)$.
The procedure then repeats at time $k+1$ and so on.

Recall that in classical social learning 
 after some finite time $\bar{k}$, all agents  choose  the same action and the public belief freezes resulting in 
 an information cascade; see \cite{Cha04,KP14} and \cite{KA12,KA13} for a financial application.

\subsection{Example 3: Stopping time POMDP with Social Learning: Interaction of Local and Global Decision Makers}  \label{sec:social1}

\index{structural result! social learning POMDP! stopping time}
\index{social learning! local and global decision makers}

Suppose a multi-agent system makes local decisions and performs social learning as above. 
 How can the multi-agent system  make a global decision when to stop?
Such problems are motivated
in decision systems where a global decision needs to be made based on local
decisions of agents. Figure \vref{fig:lddm} shows the setup with interacting local and global decision makers.

 We consider a Bayesian sequential detection problem
for state $x=e_1$. Our goal below is to derive structural results for the optimal stopping policy. The main result
below (Theorem \ref{thm:stopsocial}) is that the global decision of when to stop is a multi-threshold function of the belief state.

Consider  $\statespace = \obspace=\{1,2\}$ and $\A = \{1,2\}$ and
 the social learning model of \secn \ref{sec:herdaa}, where  the costs $c(e_i,a)$ satisfy
\beq c(e_1,1) < c(e_1,2) ,  \quad  c(e_2,2) < c(e_2,1) . \label{eq:costdom}\eeq
Otherwise one action will always dominate the other action and the problem is un-interesting.

Let $\tau$ denote a stopping time adapted to  $\mathcal{G}_{k}$, $k \geq 1$  (see  (\ref{eq:sigg})). 
In words, each agent has only the public belief obtained via social learning to make the global
decision of whether to continue or stop. The goal is to solve the following stopping time POMDP
to detect state $e_1$: Choose stopping time $\tau$ to minimize
\beq
  J_\mu(\belief)  =
   \Ep\{\sum_{k=0}^{\tau-1} \discount^{k}  \E\left\{ \left. 
  d I(x = e_1) \right\vert \mathcal{G}_{k}\right\}  + 
  \discount^{\tau}
  \beta \,\E\{ I(x \neq e_1) | \mathcal{G}_{\tau}\} \}  \label{eq:costsocial1}
\eeq
The first term is the delay cost and penalizes the decision of choosing
$u_k=2$ (continue) when the state is $e_1$ by the non-negative constant $d$.   The second term is the
stopping cost incurred by choosing $u_\tau = 1$ (stop and declare state 1) at time  $k =\tau$.  It is 
the error probability of  declaring state $e_1$ when the actual state is $e_2$. $\beta$ is a positive scaling  constant.
  In terms of the public belief, (\ref{eq:costsocial1}) is
\begin{align}  \label{eq:stopsocial}
J_\mu(\belief) &= 
 \Ep \{\sum_{k=0}^{\tau-1} \discount^{k} \Cb(\pi_{k-1},u_k=2) + \discount^{\tau} \Cb(\pi_{\tau-1} , u_\tau=1) \}  \\
\Cb(\pi,2) &=  
   d e_1^\p \pi,  
   \quad 
   \Cb(\pi,1) = \beta
e_2^\p \pi.
 \nonumber  \end{align}

The global  decision $u_k = \mu(\pi_{k-1}) \in \{1 \text{ (stop) } ,2 \text{ (continue)}\} $ is a function of  the  public belief $\pi_{k-1}$ updated according to the social learning protocol (\ref{eq:mudef}), (\ref{eq:piupdatesl}).
The optimal policy $\mu^*(\pi)$ and value function $V(\pi)$ satisfy Bellman's equation (\ref{eq:dp_algt}) with
\begin{align} \label{eq:dpsocialstop}
Q(\pi,2) &= C(\pi,2) + \discount \sum_{a \in \actionspace}  V\left( \filter(\pi ,a) \right) \filterd(\pi,a) \text{ where } \\
C(\pi,2) &= \Cb(\pi,2) -  (1-\discount) \Cb(\pi,1), \quad Q(\pi,1) = C(\pi,1) = 0. \nonumber
\end{align}
Here $\filter(\pi,a)$ and $\filterd(\pi,a)$ are obtained from the social learning  filter (\ref{eq:piupdatesl}).
  The above stopping time problem  can be viewed as a macro-manager that operates on the public belief generated
  by  micro-manager decisions. 
   Clearly the micro and macro-managers interact -- the  local decisions $a_k$ taken by the micro-manager determine
   $\belief_k$ and hence determines decision $u_{k+1}$ of the  macro-manager. 
   
Since  $\statespace=\{1,2\}$, the public belief state $\pi = [1-\pi(2), \;\pi(2)]^\p$ is parametrized by the scalar $\pi(2) \in [0,1]$, and the belief
space 
 is the interval $[0,1]$. Define  the following
 intervals which form a partition of the interval  [0,1]:
\beq \label{eq:plregions}
\begin{split}
\mathcal{P}_l &= \{\pi(2):  \etaregion_l <\pi(2) \leq \etaregion_{l-1} \},  \quad l=1,\ldots, 4  \;\text { where } \\ \etaregion_0 &= 1, \;
\etaregion_1  = \frac{(c(e_1,2) - c(e_1,1)) B_{11}}{(c(e_1,2) - c(e_1,1)) B_{11} + (c(e_2,1) - c(e_2,2)) B_{21} } \\
\etaregion_2 &=  \frac{(c(e_1,2) - c(e_1,1))} {(c(e_1,2) - c(e_1,1))  + (c(e_2,1) - c(e_2,2))  } \\
\etaregion_3 & = \frac{(c(e_1,2) - c(e_1,1)) B_{12}}{(c(e_1,2) - c(e_1,1)) B_{12} + (c(e_2,1) - c(e_2,2)) B_{22} } , \quad \etaregion_4 = 0.\end{split}
\eeq
 $\etaregion_0$ corresponds to belief state $e_2$, and $\etaregion_4$ corresponds to belief state $e_1$.
 (See discussion
at the end of this section for more intuition about the intervals $\mathcal{P}_i$).

It is readily verified that if the observation matrix $B$ is TP2, then  $\etaregion_3 \leq \etaregion_2 \leq \etaregion_1$.
The following is the main result.  \index{TP2 matrix}

\begin{theorem} \label{thm:stopsocial} Consider the stopping time problem (\ref{eq:stopsocial}) where agents
perform social learning using the social learning Bayesian filter (\ref{eq:piupdatesl}). Assume (\ref{eq:costdom}) and $\oprob$
is symmetric and satisfies \ref{A2}.
Then
the optimal stopping policy $\mu^*(\pi)$  has the following structure: The stopping set $\region_1$ is the union of  at most three
intervals. That is  $\region_1 = \region_1^a \cup \region_1^b  \cup \region_1^c$ where $\region_1^a$, $\region_1^b$, $\region_1^c$ are possibly empty intervals. Here \\
(i)  The stopping interval $\region_1^a \subseteq  \mathcal{P}_1 \cup  \mathcal{P}_4$ and is characterized by a threshold point. That is, 
	if $ \mathcal{P}_1$ has a threshold point $\pi^*$,  then $\mu^*(\pi) = 1$ for all $\pi(2) \in \mathcal{P}_4$ and 
	\beq
\mu^*(\pi) = \begin{cases} 2  & \text{ if } \pi(2) \geq \pi^* \\
	1 & \text{ otherwise }  \end{cases}, \quad \pi(2) \in \mathcal{P}_1 .\eeq
	Similarly, if $\mathcal{P}_4$ has a threshold point $\pi_4^*$, then $\mu^*(\pi) =2 $ for all $\pi(2) \in \mathcal{P}_1$.
\\	
(ii)  The stopping intervals $\region_1^b \subseteq \mathcal{P}_2$ and     $\region_1^c \subseteq \mathcal{P}_3$ 
		\\
	(iii) The intervals $\mathcal{P}_1$ and $\mathcal{P}_4$ are regions of information cascades.  That is, if $\pi_k \in \mathcal{P}_1
\cup \mathcal{P}_4$, then social learning ceases and $\pi_{k+1} = \pi_k$ (see Theorem \ref{thm:herd} for definition
of information cascade). \qed
\end{theorem}

\index{structural result! social learning POMDP! non-convex stopping set}  \index{value function for POMDP! social learning! non-concavity}
\index{stopping time POMDP! social learning! non-convex stopping set}

The proof of Theorem \ref{thm:stopsocial}  is in \cite{Kri16}. The proof depends on properties of the social learning
filter and these are summarized in Lemma \ref{lem:social} in Appendix \ref{app:stopsocial}. The proof  is more complex than that of Theorem \ref{thm:pomdpconvex} since now $V(\pi)$ in 
is not necessarily
concave over $\Belief$, since $\filter(\cdot)$ and $\filterd(\cdot)$ are functions
of $\Bs_a$ (\ref{eq:aprob}) which itself is an explicit (and in general non-concave)
function of $\pi$.

{\em Example}: 
To illustrate the multiple  threshold structure of the above theorem, consider the 
stopping time problem (\ref{eq:stopsocial}) with the following parameters:
\beq
\discount = 0.9, \quad d = 1.8,\quad B = \begin{bmatrix} 0.9 & 0.1 \\ 0.1 & 0.9 \end{bmatrix} ,\;  c(e_i,a) = \begin{bmatrix}
4.57 & 5.57 \\ 2.57 &   0 \end{bmatrix}, \quad \beta =2.
\label{eq:socialexample}
\eeq
Figure \ref{fig:doublethreshold}(a) and  (b) show the optimal policy and value function. These
were computed by constructing a  grid of 500 values for $\I = [0,1]$.
The double threshold behavior  of the stopping time problem when agents
perform social learning is due to the discontinuous dynamics  of the  social learning
filter (\ref{eq:piupdatesl}). 

\begin{figure} \centering
\mbox{\subfigure[Policy $\mu^*(\pi)$]
{\includegraphics[scale=0.28]{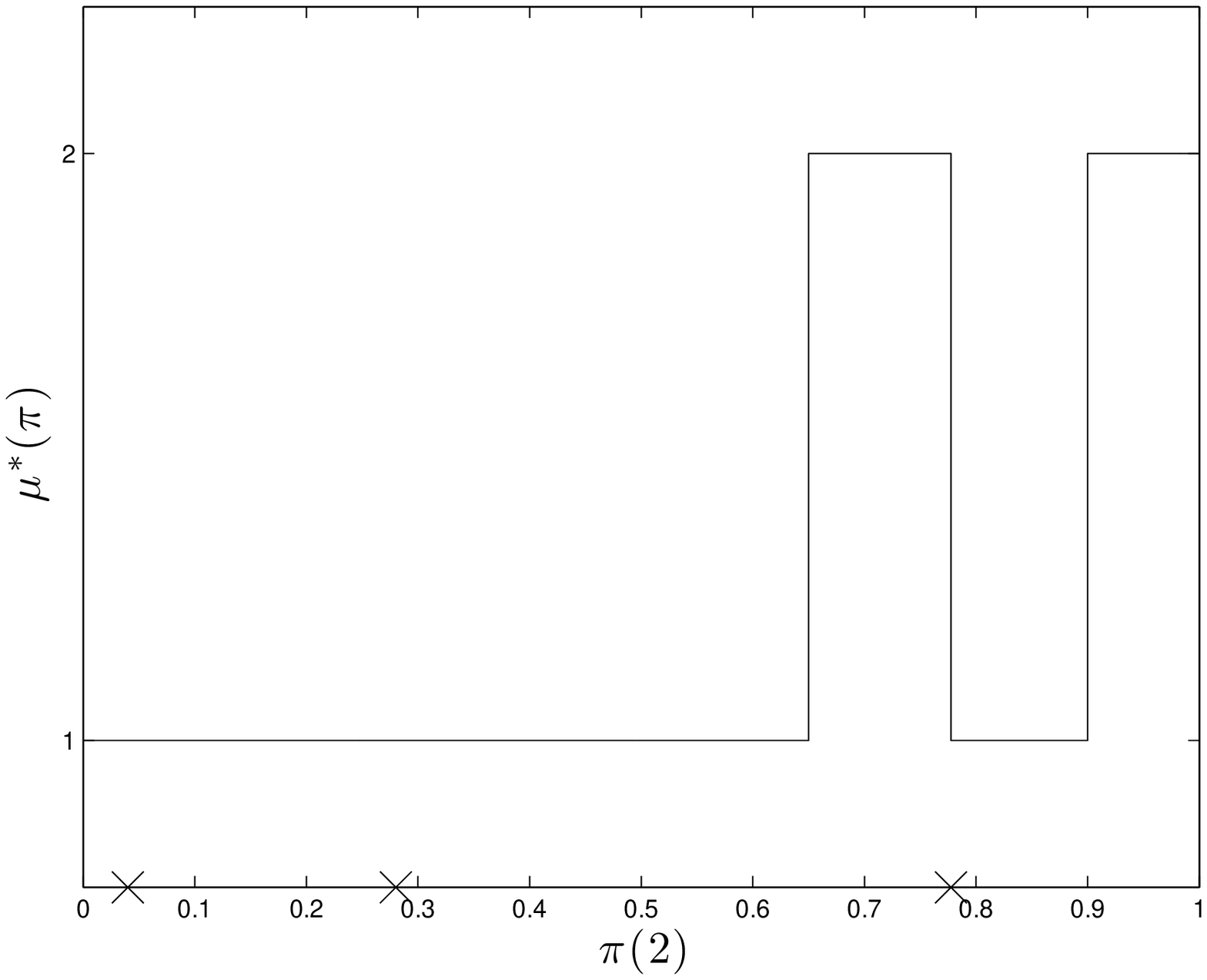}} \quad
\subfigure[Value function $V(\pi)$]{\includegraphics[scale=0.28]{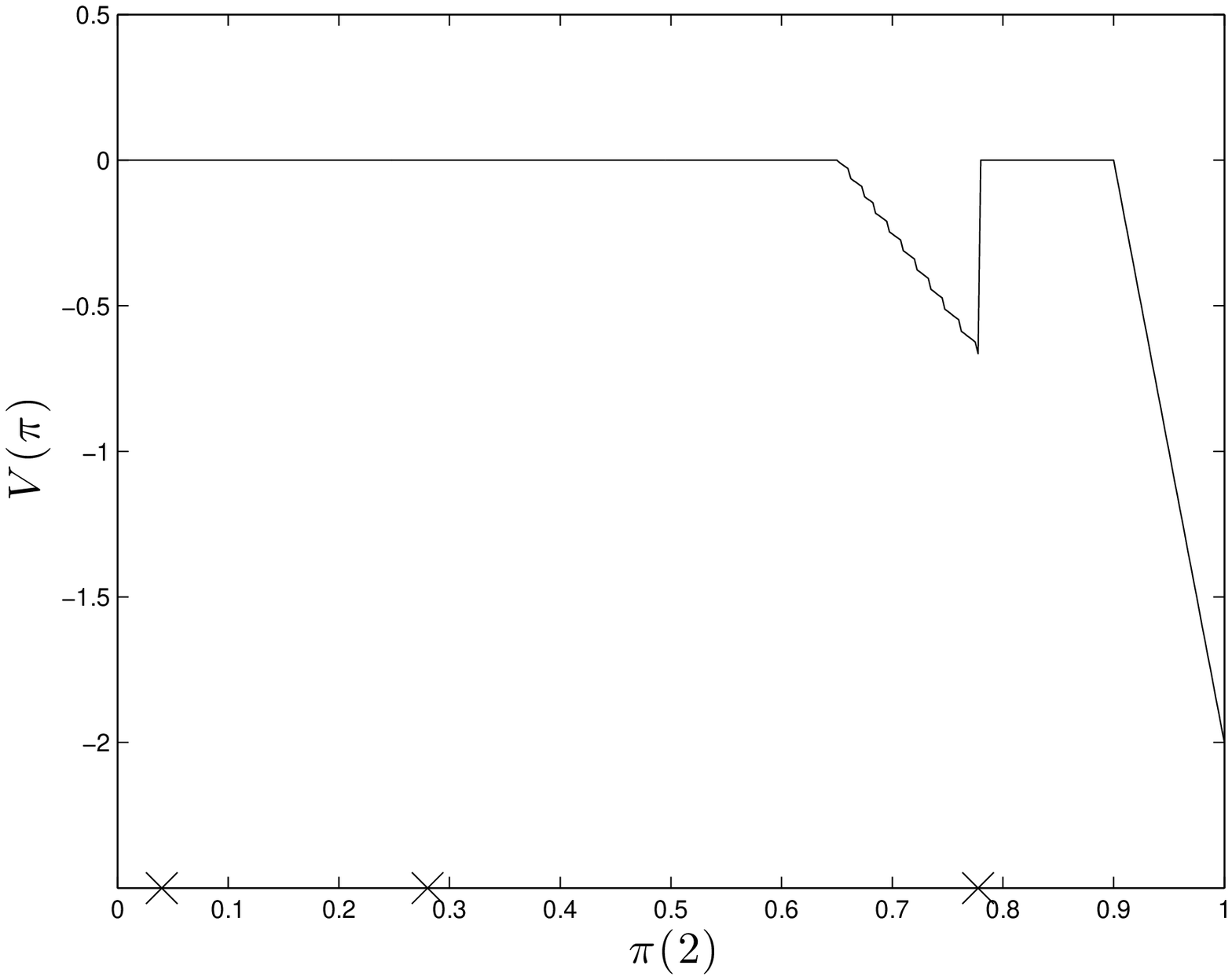}}}
\caption{Double Threshold Policy in Stopping Time Problem involving social learning. The parameters are specified
in (\ref{eq:socialexample}).}
\label{fig:doublethreshold}
\end{figure}

\subsubsection*{Discussion}
The multiple threshold behavior (nonconvex stopping set $\region_1$) of Theorem
\ref{thm:stopsocial}   is unusual.  One would have thought
that if it was optimal to `continue'  for a particular  
belief $\pi^*(2)$, then it should be optimal to continue  for all beliefs $\pi(2)$ larger than $\pi^*(2)$.
The multiple threshold optimal policy shows that this is not true. Figure \ref{fig:doublethreshold}(a) shows
that 
as the public belief $\pi(2) $ of state 2 decreases, the optimal decision switches from `continue' to `stop'
to `continue' and finally `stop'.
Thus 
  the global decision (stop or continue) is a non-monotone function
of public  beliefs obtained from local decisions. 

The main reason for this unusual  behavior is the  dependence of the  action likelihood  $\Bs_a$  on the belief state $\pi$.
This causes the social learning Bayesian filter to have a discontinuous update.
The value function is no longer concave on $\Belief$ and  the optimal policy is not necessarily monotone.
As shown in the proof of Theorem \ref{thm:stopsocial}, the value function $V(\pi) $ is concave on each of the 
intervals 
$\mathcal{P}_{l}$, $l=1,\ldots,4$.

\section{Example 4: Quickest Detection with Controlled Sampling} \label{chp:qdos}
\index{sensor scheduling! controlled sampling}
\index{controlled sampling}
\index{quickest detection! controlled sampling}

This section discusses quickest change detection when the decision maker controls how often it observes (samples) a noisy Markov chain.
The aim  is to detect when a nosily observed Markov chain hits a target state
by  minimizing a combination of  false alarm, delay cost and  measurement sampling cost.
There is an inherent trade-off between these costs: 
Taking more frequent measurements yields accurate estimates but incurs a higher measurement cost. Making an erroneous decision too soon incurs a false alarm penalty. Waiting too long to declare the target state
incurs a delay penalty. Since there are multiple ``continue" actions, the problem is not a standard stopping time POMDP.
For the 2-state case, we  show that under reasonable conditions,
the optimal policy has the following intuitive structure:
if the belief is away from the target state, look less frequently;  if the belief is close to the target state, look more frequently.

\subsection{Controlled Sampling Problem} \label{sec:qdmodel}
Let $\timet =0,1,\ldots$ denote discrete time and $\{\state_\timet\}$ denote a Markov chain on the finite state space
$\statespace = \{e_1,\ldots,e_\statedim\}$ with transition matrix $\tp$.

  Let  $\epoch_0,\epoch_1,\ldots,\epoch_{k-1}$ denote  discrete time instants at which measurement
samples  have been taken, where by convention $\epoch_0 = 0$. 
Let $\epoch_k$ denote the current time-instant at which a measurement is taken. The measurement sampling protocol proceeds according to the  following
steps:

\noindent
 {\em Step 1. Observation}: A noisy measurement $y_k \in \obspace$ at time $\timet=\epoch_k$   of the Markov chain is  obtained with conditional pdf
 or pmf $\oprob_{xy} = \pdf(\obs_k = y| \state_{\epoch_k}= x)$.

\noindent
{\em  Step 2. Sequential Decision Making}:  
Let $\info_k =  \{y_{1},\ldots,y_{k},u_0,u_1,\ldots,u_{{k-1}} \}$ denote the history of past decisions and available observations.
At times $\epoch_k$, an action $u_k$ is chosen according to the stationary policy $\policy$, where
\begin{align}  \label{eq:actionstrategya}
 u_{k} &= \policy(\info_k)  \in \actionspace = \Delayseti  . \end{align} 
 Here,  $ u_k = l $  denotes take the next measurement after  $ D_l $ time points, $l \in \delayseti$.
 The initial decision at time $\epoch_0=0$ is  $u_0 = \mu(\belief_0)$ where $\belief_0$ is the initial distribution.
Also, $D_1 < D_2 < \cdots < D_L$ are $L$ distinct positive integers that denote the set of possible sampling time intervals.
Thus the  decision $u_{k}$ specifies the next time $\epoch_{k+1}$ to make a measurement as follows:
\beq
\epoch_{k+1} = \epoch_k  + D_{u_{k}} , \quad u_k \in \delayseti, \quad \epoch_0 = 0. \eeq
{\em Step 3. Costs}: If decision $u_k \in \delayseti$ is chosen,
a  {\em decision cost}
$c(x_t,u_{k})$ is incurred by the decision-maker at each time $\timet \in [\epoch_k,\ldots,\epoch_{k+1}-1]$ until the next measurement is taken at time $\epoch_{k+1}$. Also at each time $\epoch_k$, $k=0,1,\ldots,\kstar-1$, the decision maker
pays a  non-negative {\em measurement  cost}
$\bmc(\state_{\epoch_k},\state_{\epoch_{k+1}},\obs_{k+1},u_k)$
to observe the noisy Markov chain at time $\epoch_{k+1} = \epoch_k + D_{u_k}$. 
  In terms of $\info_k$, this is equivalent to choosing the measurement cost as (see (\ref{eq:gencoste}))
\beq  \mc(\state_{\epoch_k}=e_i,u_k) =  \sum_{j} \tp^{D_{u_k}}\vert_{ij} \oprob_{jy} \,\bmc(\state_{\epoch_k}=e_i,\state_{\epoch_{k+1}}=e_j,\obs_{k+1}=\obs,u_k) 
\label{eq:mcosty}
\eeq
 where 
     $\tp^{D_u}\vert_{ij}$ denotes the $(i,j)$ element of matrix $\tp^{D_u}$. 

\noindent
{\em Step 4}: If at time $t=\epoch_{\kstar}$ the decision $u_{\kstar}  = 0$ is chosen, then
a terminal cost $c(x_{\epoch_{\kstar}},0)$ is incurred  and 
the problem terminates.
\\
If decision $u_k \in \delayseti$,
set $k$ to $k+1$ and go to Step~1.  \qed \\

In terms of the belief state, the objective to be minimized  can be expressed  as
\begin{align}
 J_\mu(\beliefzero) &= \Ep\left\{\sum_{k=0}^{\kstar-1}  C\(\pi_{k},u_{k}\) 
+C\(\pi_{{{\kstar}}}, u_{{\kstar}}=0\)
 \right\}   \\
\intertext{ where $ C(\pi,u) =  C_u ^\p \pi  \text{ for } u \in \actionspace$}  
  C_u =& \begin{cases}
\mc_u  +  (I+ \tp + \cdots + \tp^{D_u-1}) c_u & u \in \delayseti \\
c_0 & u = 0 \end{cases} \nonumber
\\ 
c_u = & \begin{bmatrix}c(e_1,u),\ldots,c(e_X,u)\end{bmatrix}^\p, \label{eq:cost} \quad
\mc_u =  \begin{bmatrix} \mc(e_1,u),\ldots, \mc(e_X,u) \end{bmatrix}^\p . 
 \end{align}

Define the stopping set $\region_1$ as
\beq
\region_1 =  \{\pi \in \Belief : \mu^*(\pi) = 0\}  = \{\pi \in \Belief:Q(\pi,0) \leq Q(\pi,u),
 \; u \in \delayseti \}.
 \label{eq:stopsetsam}
\eeq
Bellman's dynamic programming equation reads
\begin{align}
\mu^*(\pi)&= \arg\min_{u \in \actionspace} Q(\pi,u) , \;J_{\mu^*}(\pi) = V(\pi) = \min_{u \in \actionspace} Q(\pi,u), \nonumber \\
   Q(\pi,u) &=  C(\pi,u)
+ \sum_{y \in \Y}  V\left( \filter(\pi ,y,u) \right) \filterd(\pi,y,u),\; u =1 ,\ldots,L,  \nonumber\\ 
Q(\pi,0) &=  C(\pi,0).
  \label{eq:dp_delay}
\end{align}

\subsection{Example: Quickest Change Detection with Optimal Sampling} \label{sec:exqd}
\index{optimal sampling! quickest detection}
We now formulate the quickest detection problem with optimal sampling which  serves as an example to illustrate
 the above  model.
Recall  that  decisions (whether to stop, or continue and take next observation sample
after $D_l$ time points) are made at times $\epoch_1,\epoch_2, \ldots$. In contrast,  the state of the Markov chain  (which models the change
we want to detect) can change
at any time $t$. We need to construct the  delay  and false alarm penalties 
to   take this into account. 
\\
1.  {\em Phase-Distributed (PH) Change time}:  As in Example 1, in quickest detection, the target state (labelled as state 1) is  absorbing. The transition matrix $\tp$ is specified in (\ref{eq:phmatrix}).
 Denote the time at which the Markov chain hits the target state as
\beq \changetime = \min\{t:  \state_t = 1\} .  \label{eq:changetime} \eeq
\\
2. {\em Observations}:  As in (\ref{eq:obsin}) of Example 1, $\oprob_{2y} =\oprob_{3y} = \cdots = \oprob_{Xy} $. \\
3. {\em  Costs}: Associated with the quickest detection problem are the following costs. \\
{\em (i) False Alarm}:
Let $\epoch_{\kstar}$  denote the time  at which  decision $u_{\kstar}=0$ (stop and announce target state) is chosen, so that  the problem terminates.
If the decision to stop is made before the Markov chain reaches the target state 1, i.e., $\epoch_{\kstar} < 
\changetime$, then a unit false
alarm penalty  is paid.  Choosing $\f = \one - e_1$ in  (\ref{eq:falsev}), in terms of the belief state,
the  false alarm penalty at epoch $k=\kstar$ is
\beq  \sum_{i \neq 1}  \E\{I(x_{\epoch_k}=e_i,u_{k}=0)|\info_k\} =  (\one - e_1)^\p \pi_k  I(u_k=0).\eeq
(ii) {\em Delay cost of continuing}: Suppose decision $u_{k} \in \delayseti$ is taken at time $\epoch_k$. So the next sampling
time  is $\epoch_{k+1} = \epoch_k + D_{u_k}$.
Then for any time $\timet \in  [\epoch_k, \epoch_{k+1}-1]$,  the event $ \{x_{\timet} = e_1, u_{k}\}$ signifies that a change
has occurred but not been announced by the decision maker. Since the decision maker can make the next decision (to stop
or continue) at $\epoch_{k+1}$, the
 delay cost incurred  in the time  interval  $[\epoch_k, \epoch_{k+1}-1]$ is
  $d \sum_{t=\epoch_k}^{\epoch_{k+1}-1} I(x_{t} = e_1, u_{k} )$ 
  where   $d$ is a non-negative constant.
For $u_k \in \delayseti$, the expected delay cost in interval  $[\epoch_k, \epoch_{k+1}-1] = [\epoch_k, \epoch_k+D_{u_k}-1] $ is
$$  d  \sum_{t=\epoch_k}^{\epoch_{k+1}-1} \E\{I(x_{t}= e_1,u_k )|\F_{k}\}  
= d e_1^\p (I + \tp + \cdots + \tp^{D_{u_k}-1})^\p \pi_k  . $$
%
(iii) {\em Measurement Sampling Cost}: Suppose decision $u_{k} \in \delayseti$ is taken at time $\epoch_k$.
As in (\ref{eq:cost}) let   $\mc_{u_{k}} = (\mc(x_{\epoch_k}=e_i,u_{k}), i\in \statespace)$ denote the non-negative measurement  cost  vector for choosing to take a measurement. 
Next, since in quickest detection, states $2,\ldots,X$ are fictitious states that are indistinguishable in terms of cost, choose $\mc(e_2,u) =\ldots = \mc(e_X,u)$.\\  
 Choosing a constant measurement cost at each time (i.e., $\mc(e_i,u)$  independent of state $i$ and action $u$), still
results in non-trivial global costs for the decision maker. This is because
choosing a  smaller sampling interval will result in 
 more measurements until the final decision to stop, thereby incurring a higher
total measurement cost for the global decision maker.

\noindent
{\bf Remarks}:
(i) {\em Quickest State Estimation}: 
The setup is identical to  above, except that unlike (\ref{eq:phmatrix}), the transition matrix $\tp$ no longer has an absorbing target state.
Therefore the Markov chain can jump in and out of the target state. To avoid pathological cases,   we assume $\tp$ is irreducible. 
Also there is no requirement for the observation probabilities to satisfy $B_{2y} =B_{3y} = \cdots = B_{Xy}$.
\\
(ii) {\em Summary}:  In the notation of (\ref{eq:cost}), the costs   for quickest detection/estimation  optimal sampling are $C(\pi,u)= C_u ^\p \pi$  
where $C_0 = c_0= \one - e_1$ and 
\begin{align} 
 C_u = \mc_u +   (I+\tp+ \cdots + \tp^{D_u-1}) c_u, \quad c_u = d e_1,  \quad  u \in \delayseti.
\label{eq:qdcosts} \end{align}
(iii) {\em Structural Results}: As mentioned earlier, since there are multiple ``continue" actions $u \in \delayseti$, the problem is not a standard stopping time POMDP.
Of course, Theorem \ref{thm:pomdpconvex}  applies and so the optimal policy for the ``stop'' action, i.e.,  stopping set $\region_1$  (\ref{eq:stopsetsam}), is a convex set.  Characterizing the structure
of the policy  for the 
actions $\delayseti$ is more difficult.
For the 2-state case, we obtain structural results  in \secn \ref{sec:threshold2sam}.
For the multi-state case, we will develop results in Chapter \ref{chp:myopicul}.

\subsection{Threshold Optimal Policy for Quickest Detection with  Sampling} \label{sec:threshold2sam}
\index{structural result! quickest detection with controlled sampling}
 Consider   quickest detection with optimal sampling for geometric distributed change time.
The transition matrix is 
$ \tp = \begin{bmatrix} 1 & 0 \\ 1-\tp_{22} & \tp_{22} \end{bmatrix}$ and expected change time is $\E\{\changetime\} = \frac{1}{1-\tp_{22}}$
where $\changetime$ is defined in (\ref{eq:changetime}).
For a 2-state Markov chain  since $\belief(1)+\belief(2) = 1$, it suffices to represent $\belief$ by its second element
 $\belief(2) \in [0,1]$. 
That is, the belief space $\Belief$ is the interval $[0,1]$.

\begin{theorem}\label{cor:qd}
Consider the  quickest detection  optimal sampling problem of  \secn \ref{sec:exqd}
with geometric-distributed change time  and costs (\ref{eq:qdcosts}).
 Assume the measurement cost $\mc(e_i,u)$ satisfies \ref{A1},  \ref{S} and  
 the observation distribution satisfies \ref{A3}. Then there exists an optimal policy $\mu^*(\belief)$ with the following monotone structure:
There exist up to $L$ thresholds denoted $\belief_1^*, \ldots, \belief_L^*$ with
$0=\belief_0^*  \leq \belief_1^* \leq \belief_{L}^* \leq  \belief_{L+1}^*=1$ such that, for $\belief(2) \in [0,1]$,
 \beq \label{eq:thresqd}
 \mu^*(\belief) = l  \; \text{ if }\;   \belief(2) \in [\belief_l^*, \belief_{l+1}^*) , \quad l = 0,1,\ldots, L.\eeq
Here the sampling intervals are ordered as $D_1 < D_2 < \ldots < D_L$.
 So the optimal sampling policy (\ref{eq:thresqd}) makes measurements  less 
frequently when the posterior $\belief(2)$ is away from the target state and more  frequently when
closer to the target state.  (Recall the target state is $\belief(2) = 0$.) \end{theorem}

 The proof follows  from that of Theorem \vref{thm:pomdp2state}.
There are two main conclusions regarding Theorem \ref{cor:qd}. First, for constant measurement cost, \ref{A1} and \ref{S}  hold trivially.
 For the  general measurement cost  $\bmc(\state_{\epoch_{k+1}}=e_j,\obs_{k+1},u_k)$ (see (\ref{eq:mcosty})) that depends on the state at epoch $k+1$, then   $\mc(e_i,u)$ in 
 (\ref{eq:mcosty})
 automatically satisfies \ref{S}  if $\tp$ satisfies \ref{A2} and $\bmc$ is decreasing in $j$.
 Second, the optimal policy 
 $\mu^*(\belief)$ is monotone  in posterior $\belief(1)$ and therefore has a finite dimensional characterization. To determine the optimal policy, one only
 needs to compute the values of the $L$ thresholds 
 $\belief_1^*, \ldots, \belief_L^*$. These can be estimated via a simulation-based stochastic optimization algorithm.

\section{\pwe} 
The book \cite{PH08} is devoted to quickest detection problems and contains numerous references.
 PH-distributed change times are used widely to model discrete event systems \cite{Neu89}
and  are a natural candidate for modeling arrival/demand processes for services that have an expiration date  \cite{DG09}.
It would be useful to do a performance analysis of the various optimal detectors proposed in this chapter --
see \cite{TV05,VV09} and references therein.
 \cite{BV12}  considers a measurement control problem 
 for geometric-distributed change times (2-state Markov chain with an absorbing state).  \cite{Kri13} considers joint
 change detection and measurement control POMDPs
 with more than 2 states.
 
 In \cite{Kri08,Kri09} similar structural results are developed for one-shot Bayesian games to characterize the Nash equilibrium.
 \index{quickest detection|)}


\chapter{Myopic Policy Bounds for POMDPs
and Sensitivity} \label{chp:myopicul}

\index{myopic policy! optimality for POMDP|(}

Chapter \ref{ch:pomdpstop}
 discussed   stopping time POMDPs and gave sufficient conditions for the optimal policy to have a  monotone structure.
In this chapter we consider more general POMDPs (not necessarily with a stopping action) 
and present the following structural results:
\begin{compactenum}
\item {\em Upper and Lower Myopic Policy Bounds using Copositivity Dominance}:  For general POMDPs it is difficult to provide sufficient conditions for monotone policies.
Instead, we provide 
sufficient conditions so that  the optimal policy  can be upper and lower bounded  by judiciously chosen myopic policies. 
These sufficient conditions involve the {\em copositive ordering} described  in Chapter \ref{chp:filterstructure}.
The myopic policy bounds are constructed to maximize the volume of belief states where they coincide with the optimal policy. 
Numerical examples illustrate these myopic policies for  continuous and discrete valued  observations.

\item {\em Upper Myopic Policy Bounds using Blackwell Dominance}:  Suppose  the observation probabilities  for actions 1 and 2
can be related via the following factorization: $\oprob(1)  = \oprob(2)\, \aB$ where $\aB$ is a stochastic matrix.
We then say that  $\oprob(2)$ {\em Blackwell} dominates $\oprob(1)$.
If this Blackwell dominance holds,  we will show that 
a myopic policy  coincides with the  optimal policy for all belief states where
  choosing action 2 yields a smaller instantaneous cost  than  choosing action 1. Thus, the myopic policy
forms an  upper bound to the optimal policy.  We provide two examples: scheduling an optimal filter versus an optimal predictor, and  scheduling with
ultrametric observation matrices.

\item {\em Sensitivity to POMDP parameters}: The final  result considered in this chapter is: How does the optimal  cumulative cost of POMDP
depend on the transition and observation probabilities? The ordinal results
use the copositive ordering  of transition matrices and Blackwell dominance of observation matrices that yield an ordering of the  achievable 
optimal  costs of a POMDP.
\end{compactenum}

\section {The Partially Observed Markov Decision Process} \label{sec:pomdp}
Throughout this chapter we will consider discounted cost infinite horizon POMDPs discussed in \secn \ref{chp:discpomdp}.
Let us briefly review this model.
 A   discrete time Markov chain  evolves on the  state space $\statespace = \{e_1,e_2,\ldots, e_\statedim\}$ where $e_i$ denotes
 the unit $\statedim$-dimensional vector with $1$ in the $i$-th position. Denote the
action space  as $\actionspace = \{1,2,\ldots,\actiondim\}$ and observation space as $\obspace$. For discrete-valued observations $\obspace = \{1,2,\ldots,\obsdim\}$ and for continuous observations $\obspace \subset \reals$.

Let
$\Belief = \big\{\belief: \belief(i) \in [0,1], \sum_{i=1}^\statedim \belief(i) = 1 \big\}$ denote the belief space of $\statedim$-dimensional probability vectors.  For stationary policy  $\policy: \Belief \rightarrow \actionspace$,
 initial belief  $\belief_0\in \Belief$,  discount factor $\discount \in [0,1)$, define the  discounted cost:
\begin{align}\label{eq:discountedcostor}
J_{\policy}(\belief_0) = \E\left\{\sum_{\time=0}^{\infty} \discount ^{\time} \cost_{\policy(\belief_\time)}^\p\belief_\time\right\}.
\end{align}
Here $\cost_\action = [\cost(1,\action),\ldots,\cost(\statedim,\action)]^\p$, $\action\in \actionspace$ is the cost vector for each action, and the belief state evolves as
$\belief_{k} = \filter(\belief_{k-1},\obs_k,\action_k)$ where
\begin{align}  \filter\left(\belief,\obs,\action\right) &= \cfrac{\oprob_{\obs}(\action) \tp'(\action)\belief}{\filternorm\left(\belief,\obs,\action\right)} , \quad
\filternorm\left(\belief,\obs,\action\right) = \one^\p \oprob_{\obs}(\action) \tp'(\action)\belief, \nonumber \\
\oprob_{\obs}(\action) &= \diag\big(\oprob_{1,\obs}(\action),\cdots,\oprob_{\statedim,\obs}(\action) \big). \label{eq:information_stateu}
\end{align}
Recall  $\one$ represents a $\statedim$-dimensional vector of ones,
$ \tp(\action) = \left[\tp_{ij}(\action)\right]_{\statedim\times\statedim}$
$ \tp_{ij}(\action) = \prob(\state_{\time+1} = e_j | \state_\time = e_i, \action_\time =\action)$ denote the transition probabilities,
 $\oprob_{\state\obs}(\action) = \prob(\obs_{\time+1} = \obs| \state_{\time+1} = e_\state, \action_{\time} = \action)$ when $\obspace$ is finite,
 or  $\oprob_{\state\obs}(\action)$ is the conditional probability density function when $\obspace \subset \reals$.

The aim is to compute the optimal  stationary policy $\optpolicy:\Belief \rightarrow \actionspace$ such that
$J_{\optpolicy}(\belief_0) \leq J_{\policy}(\belief_0)$ for all $\belief_0 \in \Belief$.
Obtaining the optimal policy  $\optpolicy$ is equivalent to solving
 Bellman's  dynamic programming equation:
$ \optpolicy(\belief) =  \underset{\action \in \actionspace}\argmin ~\valueaction(\belief,\action)$, $J_{\optpolicy}(\belief_0) = \optvalue(\belief_0)$, where
\begin{equation}
\optvalue(\belief)  = \underset{\action \in \actionspace}\min ~\valueaction(\belief,\action), \quad
  \valueaction(\belief,\action) =  ~\cost_\action^\prime\belief + \discount\sum_{\obs \in \obsdim} \optvalue\left(\filter\left(\belief,\obs,\action\right)\right)\filternorm \left(\belief,\obs,\action\right). \label{eq:bellman}
\end{equation}

Since  $\Belief$ is continuum, Bellman's equation \eqref{eq:bellman} does not translate into practical solution methodologies.   This motivates the construction of  judicious myopic policies
that upper and lower bound $\optpolicy(\belief)$. 

\section{Myopic Policies using Copositive Dominance: Insight} 

For stopping time POMDPs, in Chapter \ref{ch:pomdpstop} we gave sufficient conditions for $\valueaction(\belief,\action)$ in Bellman's equation
 to be submodular,
i.e., 
 $\valueaction(\belief,\action+1 ) - \valueaction(\belief,\action)$ is decreasing in $\belief $ with respect to the monotone likelihood ratio order. This implied
 that the  optimal policy  $\optpolicy(\belief)$ was MLR increasing in belief $\belief$ and had a threshold structure.
 
Unfortunately, for a general POMDP, giving sufficient conditions for   $\valueaction(\belief,\action)$ to be submodular is still an open problem.\footnote{For the two-state case, conditions for submodularity are given in \chp  \ref{sec:2statepomdp}, but these do not generalize to more than two states.}
Instead of showing submodularity, in this chapter we will  give sufficient conditions for  $\valueaction(\belief,\action)$ to satisfy
\beq  \valueaction(\belief,\action+1 ) - \valueaction(\belief,\action) \leq \Cost(\belief,\action+1) - \Cost(\belief,\action) \label{eq:qbound} \eeq
where $\Cost(\belief,\action)$ is a cleverly chosen  instantaneous cost in terms of the belief state. A nice consequence of (\ref{eq:qbound}) is the following:
Let $\policyu(\belief) = \argmin_\action \Cost(\belief,\action)$ denote the myopic policy that minimizes the instantaneous cost. Then  (\ref{eq:qbound}) implies that
the optimal policy $ \optpolicy(\belief) $ satisfies
$$ \optpolicy(\belief) \leq \policyu(\belief), \quad \text{  for all }   \belief \in \Belief. $$
In words: if (\ref{eq:qbound}) holds, then the myopic policy $\policyu(\belief)$ is provably an upper bound to the optimal policy $\optpolicy(\belief)$.  Since the myopic policy is trivially
computed, this is a useful result. But there is more! As will be described below, for discounted cost POMDPs, the optimal policy remains unchanged for a family
of costs $\Cost(\belief,\action)$. So by judiciously choosing these costs we can also construct myopic policies $\policyl(\belief)$  that  {\em lower bound} the optimal 
policy. To summarize, for any belief state  $\belief$, we will present sufficient conditions under which the optimal policy $\optpolicy(\belief)$ of a POMDP can be upper and lower bounded by myopic policies
 denoted by $\policyu(\belief)$ and $\policyl(\belief)$, respectively, i.e., (see Figure \vref{fig:upperlowerboundpolicy} for a visual display)
\beq \policyl(\belief) \le \optpolicy(\belief) \le \policyu(\belief) \quad \text{  for all }  \belief \in \Belief. 
\label{eq:upperlowerboundpolicy}
\eeq
  Clearly, for belief states $\belief$ where ${\policyu(\belief)}={\policyl(\belief)}$, the optimal policy $\optpolicy(\belief)$ is completely determined.
  
\begin{figure}[h] \centering
\includegraphics[scale=0.45]{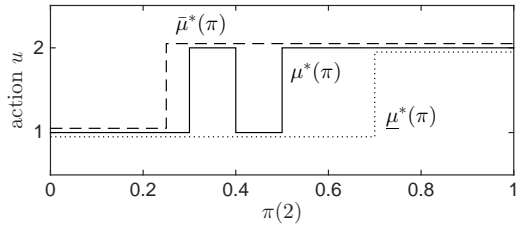}
\caption{Illustration of main result of this chapter. 
The aim is to construct an upper bound $\bar{\policy}$ (dashed line) and lower bound $\underline{\policy}$ (dotted line), to the optimal policy 
 $\optpolicy$ (solid line) such that 
 (\ref{eq:upperlowerboundpolicy}) holds for each belief state $\belief$.
Thus the optimal policy is sandwiched between the judiciously chosen myopic policies  $\underline{\policy}$ and  $\bar{\policy}$ over the entire
belief space $\Belief$.
Note for $\belief$ where $\policyu(\belief) = \policyl(\belief)$, they coincide with the optimal policy $\optpolicy(\belief)$.
Maximizing the volume of beliefs where $\policyl(\belief)
 = \policyu(\belief) $ is achieved by solving a linear programming problem as described in \secn \ref{sec:optimize_bounds}.}
\label{fig:upperlowerboundpolicy}
\end{figure}

Interestingly, these judiciously constructed myopic policies are independent of the actual values of the observation probabilities (providing they satisfy a sufficient condition) which makes the structural results  applicable to both discrete and continuous observations.
Finally, we  will construct the myopic policies, $\policyu(\belief)$ and $\policyl(\belief)$, to maximize the volume of the belief space where they coincide with the optimal policy $\mu^*(\pi)$.

As an extension of the above results, motivated by examples in controlled sensing \cite{KD07,WFW07,BB89}, one can show that similar myopic bounds  hold for POMDPs with quadratic costs in the belief state. 

Numerical examples are presented to illustrate the performance of  these myopic policies.
To quantify how well the myopic policies perform we use two parameters:  the volume of the belief space where the myopic policies coincide with the optimal policy, and
an upper bound to the average percentage loss in optimality due to following this optimized myopic policy.

\paragraph{Context.}
The  papers \cite{Lov87,Rie91,RZ94} give sufficient conditions for (\ref{eq:qbound})   so that the optimal policy of a POMDP can be upper bounded\footnote{Since \cite{Lov87} deals with maximization rather than minimization, the myopic policy constructed in \cite{Lov87} forms a lower bound} by a myopic policy. Unfortunately, despite the  enormous usefulness of such
a result, the sufficient conditions given in \cite{Lov87} and \cite{Rie91}  are not useful - it is impossible to generate non-trivial examples that satisfy the conditions (c), (e), (f) of \cite[Proposition 2]{Lov87} and condition (i) of \cite[Theorem 5.6]{Rie91}. In this chapter, we provide a fix to these sufficient conditions so that  the results of \cite{Lov87,Rie91} hold for constructing a myopic policy that upper bounds the optimal policy. It turns out  that
Assumptions \ref{A4} and \ref{A5} described in Chapter \ref{chp:filterstructure} are precisely the fix we need.
We also show how this idea of constructing a upper bound myopic policy can be extended to constructing a {\em lower} bound myopic policy.

\section {Constructing Myopic Policy Bounds for Optimal  Policy using Copositive Dominance} \label{sec:exist_bounds}
\index{structural result! myopic policy bounds|(}

\index{myopic policy! optimality for POMDP! copositive dominance|(}

With the above motivation, we are now ready to construct myopic policies that provably sandwich the optimal policy for a POMDP.
\subsection*{Assumptions}
\ref{A4} and \ref{A5} below are the main  copositivity assumptions.

\begin{myassumptions}
\item  
\item[\nl{it:increasing_cost}{(C1)}]  There exists a vector $\of \in \reals^{\statedim}$ such that the
$\statedim$-dimensional vector
$\ucost_\action \equiv \cost_{\action} + \left(\eye - \discount\tp(\action)\right)\of$ is strictly increasing elementwise
for each action $\action \in \actionspace$.
\item[\nl{it:decreasing_cost}{(C2)}] There exists a vector  $\uf \in \reals^{\statedim}$ such that 
the
$\statedim$-dimensional vector
$\lcost_\action \equiv \cost_{\action} + \left(\eye - \discount\tp(\action)\right)\uf$ is strictly decreasing 
 elementwise
for each action $\action \in \actionspace$.
\item[{\bf \ref{A2}}]  
$\oprob(\action)$, $\action \in \actionspace$ is  totally positive of order 2 (TP2). That is, all second-order minors and nonnegative.
\item[{\bf \ref{A3}}] $\tp(\action)$, $\action \in \actionspace$ is  totally positive of order 2 (TP2).
\item[{\bf \ref{A4}}] 
$ \gamma^{j,\action,\obs}_{mn} + \gamma^{j,\action,\obs}_{nm} \ge 0~\forall m,n,j,\action,\obs$ where
\begin{multline}  \gamma^{j,\action,\obs}_{mn} =  \oprob_{j,\obs}(\action)\oprob_{j+1,\obs}(\action+1)\tp_{m,j}(\action)\tp_{n,j+1}(\action+1) \\ - \oprob_{j+1,\obs}(\action)\oprob_{j,\obs}(\action+1)\tp_{m,j+1}(\action)
\tp_{n,j}(\action+1) \end{multline}
\item[{\bf \ref{A5}}] 
$\sum_{\obs \le \bar{\obs}}\sum_{j=1}^\statedim\left[\tp_{i,j}(\action)\oprob_{j,\obs}(\action)
 - \tp_{i,j}(\action+1)\oprob_{j,\obs}(\action+1)\right] \le 0 $ for  $ i \in \{1,2,\ldots,\statedim\}$ and $ \bar{\obs} \in \obspace$.
\end{myassumptions}

\subsubsection*{Discussion} Recall \ref{A2}, \ref{A3}, \ref{A4}, \ref{A5}  were discussed in Chapter \ref{chp:filterstructure}.
 As described in Chapter \ref{chp:filterstructure},
\ref{A4} and \ref{A5} are a relaxed version of Assumptions (c), (e), (f) of \cite[Proposition 2]{Lov87} and Assumption (i) of \cite[Theorem 5.6]{Rie91}.
  In particular, the assumptions (c), (e), (f) of \cite{Lov87} require that $\tp(\action+1)$ $ \gtp \tp(\action)$ and $\oprob(\action+1) \gtp \oprob(\action)$, where $\gtp$ (TP2 stochastic ordering) is defined in \cite{MS02}, which is impossible for stochastic matrices, unless $\tp(\action) =  \tp(\action+1)$, 
  $\oprob(\action) =\oprob(\action+1)$ or the matrices $\tp(\action), \oprob(\action)$ are rank 1 for all $\action$ meaning that the observations are non-informative. \index{TP2 stochastic order}
  
  Let us now discuss {\ref{it:increasing_cost}} and {\ref{it:decreasing_cost}}.
If  the elements of $\cost_\action$ are strictly increasing then {\ref{it:increasing_cost}} holds trivially. Similarly, if the elements of $\cost_\action$ are strictly decreasing then {\ref{it:decreasing_cost}} holds; indeed then  {\ref{it:decreasing_cost}} is equivalent to \vref{A1}.

{\ref{it:increasing_cost}} and {\ref{it:decreasing_cost}} are easily verified by checking the feasibility of the following linear programs:
    \begin{align}\label{eq:lp_a1_a2}
    LP1:&  \underset{\of \in \poly_{\of}}{\min}~ \one_{\statedim}^\prime\of,~LP2: \underset{\uf \in \poly_{\uf}}{\min} ~\one_{\statedim}^\prime\uf. \\
\label{eq:polytopes_u}
    \poly_{\of} &= \left\{ \of:  \ucost_\action^\prime\basisvec_i \le \ucost_\action^\prime\basisvec_{i+1} ~~\forall \action \in \actionspace, i \in \{1,2,\ldots,\statedim\} \right \}  \\
  \label{eq:polytopes_l}
\poly_{\uf} &= \Big\{ \uf:  \lcost_\action^\prime\basisvec_i \ge \lcost_\action^\prime\basisvec_{i+1} ~~\forall \action \in \actionspace, i \in \{1,2,\ldots,\statedim\} \Big \}
\end{align}
where $\basisvec_i$ is the unit  $\statedim$-dimensional vector with 1 at the $i$th position.

\subsection{Construction of Myopic Upper and Lower Bounds}\label{subsec:subsec_asmp_result}
We are interested in   myopic policies of the form $\underset{\action \in \actionspace}{\argmin}~ C_\action^\prime\pi$ where cost vectors $C_\action$
are constructed so that when applied to Bellman's equation \eqref{eq:bellman}, they leave the optimal policy $\mu^*(\pi)$ unchanged. This is for several reasons: First, similar to \cite{Lov87}, \cite{Rie91} it allows us to construct useful myopic policies that provide provable upper and lower bounds to the optimal policy. Second, these myopic policies can be straightforwardly extended to 2-stage or multi-stage myopic costs. Third, such a choice  precludes choosing useless myopic bounds such
as $\policyu(\belief) = \actiondim$ for all $ \belief \in \Belief$.

Accordingly, for any two vectors $\of$ and $\uf \in \reals^\statedim$, define the myopic policies associated with the transformed costs $\ucost_{\action}$ and $\lcost_{\action}$ as follows:
\begin{align}\label{eq:new_cost_u}
\policyu(\belief) & \equiv \underset{\action \in \actionspace}\argmin ~\ucost_{\action}'\belief,\quad \text{ where } ~\ucost_{\action} = \cost_\action + \left( \eye - \discount \tp(\action) \right)\of \\
\label{eq:new_cost_l}
\policyl(\belief)& \equiv \underset{\action \in \actionspace}\argmin ~\lcost_{\action}'\belief,\quad \text{ where } ~\lcost_{\action} = \cost_\action + \left( \eye - \discount \tp(\action) \right)\uf.
\end{align}
It is easily seen that Bellman's equation \eqref{eq:bellman} applied to optimize the objective \eqref{eq:discountedcostor} with transformed costs $\ucost_\action$ and $\lcost_\action$ yields the same optimal strategy $\optpolicy(\belief)$ as the Bellman's equation with original costs $\cost_\action$. The corresponding value functions are $\bvalue(\belief) \equiv \optvalue(\belief) + \of^\prime\belief$ and $\uvalue(\belief) \equiv \optvalue(\belief) + \uf^\prime\belief$.

\begin{theorem}
\label{th:theorem1udit}
Consider a POMDP $\left(\statespace, \actionspace, \obspace, \tp(\action), \oprob(\action), \cost, \discount\right)$ and assume {\ref{it:increasing_cost}},
 \ref{A2}, \ref{A3}, \ref{A4}, \ref{A5}
 holds. Then the myopic policies, $\policyu(\belief)$ and $\policyl(\belief)$, defined in \eqref{eq:new_cost_u}, \eqref{eq:new_cost_l} satisfy:
$\policyl(\belief) \le \optpolicy(\belief) \le \policyu(\belief)$ for all  $\belief \in \Belief$.
\end{theorem}

The above  result where the optimal policy $ \optpolicy(\belief)$ is sandwiched between $\policyl(\belief)$ and $\policyu(\belief)$ is
illustrated in Figure \ref{fig:upperlowerboundpolicy} for $\statedim=2$.

\begin{proof}
 We show that
under {\ref{it:increasing_cost}},\ref{A2}, \ref{A3}, \ref{A4} and \ref{A5},
$\optpolicy(\belief) \le \policyu(\belief)~\forall \belief \in \Belief$.
Let $\bvalue$ and $\bvalueaction$ denote the variables in Bellman's equation (\ref{eq:bellman}) when using costs $ \ucost_{\action}$ defined in (\ref{eq:new_cost_u}).
Then from Theorem \ref{thm:pomdpmonotoneval} in Chapter \ref{ch:pomdpstop},
$\bvalue(\filter(\belief,\obs,\action))$ is increasing in $\obs$. From Theorem \ref{thm:filterstructure}(5), 
under \ref{A5}, $\filternorm(\belief,\action+1) \gs \filternorm(\belief,\action)$. Therefore, 
\begin{align}\label{eq:proof_1}
\begin{aligned}
\sum_{\obs \in \obspace}\bvalue(\filter(\belief,\obs,\action))\filternorm(\belief,\obs,\action)
\underset{(a)}{\le} &\sum_{\obs \in \obspace}\bvalue(\filter(\belief,\obs,\action))\filternorm(\belief,\obs,\action+1) \\
\underset{(b)}{\le} &\sum_{\obs \in \obspace}\bvalue(\filter(\belief,\obs,\action+1))\filternorm(\belief,\obs,\action+1)
\end{aligned}
\end{align}
 Inequality (b) holds since from Theorem \ref{thm:filterstructure}(4) and Theorem \ref{thm:pomdpmonotoneval},
  $$\bvalue(\filter(\belief,\obs,\action+1)) \ge \bvalue(\filter(\belief,\obs,\action)) \quad \forall \obs \in \obspace. $$ 
 Equation \eqref{eq:proof_1} implies that $\sum_{\obs \in \obspace}\bvalue(\filter(\belief,\obs,\action))\filternorm(\belief,\obs,\action)$ is increasing w.r.t.\  $\action$ or equivalently,
\beq \label{eq:proof_2}
\bvalueaction(\belief,\action) - \ucost_{\action}'\belief \le  \bvalueaction(\belief,\action+1) - \ucost_{\action+1}'\belief .\eeq
%
It therefore follows that 
$$ \{\belief: \ucost_{\action'}'\belief \geq \ucost_{\action}'\belief \} \subseteq 
 \{\belief: \bvalueaction(\belief,\action') \geq  \bvalueaction(\belief,\action) \}, \;  \action' > \action$$ which implies that
 $  \policyu(\belief) \leq u \implies  \optpolicy(\belief)  \leq u$. 
 %
%
%
The proof that $\optpolicy(\belief) \ge \policyl(\belief)$ is similar and omitted. (See \cite[Lemma 1]{Lov87b} for a more general statement.)
\end{proof}
\index{myopic policy! optimality for POMDP! copositive dominance|)}

\section {Optimizing the Myopic Policy Bounds to Match the Optimal Policy} \label{sec:optimize_bounds}
The aim of this section is to determine the vectors $\of$ and $\uf$, in \eqref{eq:polytopes_u} and \eqref{eq:polytopes_l}, that maximize the volume of the simplex where the myopic upper and lower policy bounds, specified by \eqref{eq:new_cost_u} and \eqref{eq:new_cost_l}, coincide with the optimal policy.
That is, we wish to maximize the volume of the `overlapping region'
\begin{align}\label{eq:def_overlap}
\begin{aligned}
\text{vol}\left(\setoverlap\right),~\text{where}~ \setoverlap \equiv \{\belief: \policyu(\belief) = \policyl(\belief) = \optpolicy(\belief)\}.
\end{aligned}
\end{align}
Notice that the myopic policies $\policyu$ and $ \policyl$ defined in (\ref{eq:new_cost_u}), (\ref{eq:new_cost_l})  do not depend on the observation probabilities $\oprob_\action$ and so
neither does $\text{vol}\left(\setoverlap\right)$.  So  $\policyu$ and  $ \policyl$ can be chosen to maximize $\text{vol}\left(\setoverlap\right)$ independent
of $\oprob(\action)$ and therefore work for  discrete and continuous observation spaces.  Of course, the proof of Theorem  \ref{th:theorem1udit} requires conditions on $\oprob(\action)$.

\subsection*{Optimized Myopic Policy for Two Actions} 
For a two action POMDP, obviously for a belief $\belief$, if $\policyu(\belief)=1$ then $\optpolicy(\belief) = 1$. Similarly, if  $\policyl(\belief)=2$, then $ \optpolicy(\belief) = 2$.
 Denote the set of beliefs (convex polytopes) where $\policyu(\belief) = \optpolicy(\belief) = 1$ and $\policyl(\belief)= \optpolicy(\belief) = 2$ as
\begin{align}\label{eq:overlapping_regions}
\begin{aligned}
\Belief^{\of}_{1} &= \left\{ \belief:  \ucost_1^\prime\belief \le \ucost_2'\belief \right\}= \left\{ \belief:  (\cost_1-\cost_2 - \discount(\tp(1) - \tp(2))\of)'\belief  \le 0\right\},\\
\Belief^{\uf}_{2} &= \left\{ \belief:  \lcost_2'\belief \le \lcost_1'\belief \right\}= \left\{ \belief:  (\cost_1-\cost_2 - \discount(\tp(1) - \tp(2))\uf)'\belief  \ge 0\right\}.
\end{aligned}
\end{align}
Clearly  $\setoverlap = \Belief^{\of}_{1} \cup \Belief^{\uf}_{2}$. Our goal is to find $\of^* \in \poly_{\of}$ and $\uf^* \in \poly_{\uf}$ such that $\text{vol}\left(\setoverlap\right)$ is maximized.
\begin{theorem}
\label{th:theorem_opt_bounds}
Assume that there exists two fixed $\statedim$-dimensional vectors $\of^*~\text{and}~\uf^*$ such that
\begin{align}\label{eq:optimal_overlap}
\begin{aligned}
&(\tp(2) - \tp(1))\of^* \lc (\tp(2) - \tp(1))\of , \; \forall \of \in \poly_{\of}\\
&(\tp(1) - \tp(2))\uf^* \lc (\tp(1) - \tp(2))\uf, \; \;\forall \uf \in \poly_{\uf}
\end{aligned}
\end{align}
where for $\statedim$-dimensional vectors $a$ and $b$, $a \lc b \Rightarrow \left[a_1 \le b_1,\cdots, a_{\statedim} \le b_{\statedim}\right]$.
If  the myopic policies $\policyu$ and $\policyl$ are constructed using $\of^*$ and $\uf^*$, then $\text{vol}(\setoverlap)$ is maximized.
\end{theorem}
\begin{proof}
The sufficient conditions in \eqref{eq:optimal_overlap} ensure that $\Belief^{\of^*}_{1} \supseteq \Belief^{\of}_{1} ~\forall \of \in \poly_{\of}$ and $\Belief^{\uf^*}_{2} \supseteq \Belief^{\uf}_{2} ~\forall \uf \in \poly_{\uf}$.
Indeed, to establish that $\text{vol}\left(\Belief^{\of^*}_{1}\right)$ $\ge$ $\text{vol}\Big(\Belief^{\of}_{1}\Big)$ $\forall \of \in \poly_{\of}$:
\begin{align}\label{eq:proof_th_opt_bounds_1}
\begin{aligned}
&\left(\tp(1) - \tp(2)\right)\of^* \gc \left(\tp(1) - \tp(2)\right)\of \quad \forall \of \in \poly_{\of}\\
\Rightarrow~& c_1 - c_2 - \discount\left(\tp(1) - \tp(2)\right)\of^* \lc  c_1 - c_2 - \discount\left(\tp(1) - \tp(2)\right)\of ~~\forall \of \in \poly_{\of}\\
\Rightarrow~& \Belief^{\of^*}_{1} \supseteq \Belief^{\of}_{1} ~\forall \of \in \poly_{\of}
\Rightarrow \text{vol}\left(\Belief^{\of^*}_{1}\right) \ge \text{vol}\Big(\Belief^{\of}_{1}\Big) \forall \of \in \poly_{\of}
\end{aligned}
\end{align}
So $\text{vol}\left(\Belief^{\of^*}_{1}\right) \ge \text{vol}\Big(\Belief^{\of}_{1}\Big) \forall \of \in \poly_{\of}$ and $\text{vol}\left(\Belief^{\uf^*}_{2}\right) \ge \text{vol}\Big(\Belief^{\uf}_{2}\Big) \forall \of \in \poly_{\of}$. Since $\setoverlap = \Belief^{\of^*}_{1} \cup \Belief^{\uf^*}_{2}$, the proof is complete.
\end{proof}

Theorem  \ref{th:theorem_opt_bounds} asserts that  myopic policies $\policyu$ and $\policyl$ characterized by two
fixed vectors $\of^*~\text{and}~\uf^*$
maximize  $\text{vol}(\setoverlap)$ over the entire belief space $\Belief$.
The existence and computation of these policies characterized by
 $\of^* \in \poly_{\of}$ and $\uf^* \in \poly_\uf$ are determined by  Algorithm~\ref{alg1udit}.
Algorithm \ref{alg1udit} solves $\statedim$ linear programs to obtain $\of^*$. If no $\of^* \in \poly_{\of}$ satisfying \eqref{eq:optimal_overlap} exists, then Algorithm \ref{alg1udit} will terminate with no solution. The procedure for computing $\uf^*$ is similar.  \index{linear programming! myopic policy of POMDP}

\begin{algorithm} 
\caption{Compute $\of^*$}  \label{alg1udit}
\begin{algorithmic}[1]  
 \FORALL{$i \in \statedim$}
  \STATE \label{op0}$\alpha_i \gets \underset{\of \in \poly_{\of}}\min ~\basisvec_i'(\tp(2) - \tp(1))\of$
 \ENDFOR
 \STATE $\of^* \in \poly_{\of^*}, \poly_{\of^*} \equiv \left\{\of^*:\of^*\in \poly_{\of}, \basisvec_i'(\tp(2) - \tp(1))\of^* = \alpha_i, i = 1,\cdots,\statedim\right\}  $\label{op1}
 \STATE $\policyu(\belief) = \underset{\action \in \{1,2\}}{\argmin}~\belief'\ucost^{*}_{\action} ~\forall \belief \in \Belief$, where $\ucost^{*}_{\action} = \cost_\action + \left( \eye - \discount \tp(\action) \right)\of^*$
 \STATE $\policyu(\belief) = \optpolicy(\belief) = 1, \forall \belief \in \Belief^{\of^*}_{1}$.
\end{algorithmic}
\end{algorithm}

\subsection*{Optimizing Myopic Policies for more than 2 actions} Unlike Theorem \ref{th:theorem_opt_bounds},
for the case  $\actiondim > 2$, we are unable to show that a single fixed choice of $\policyu$ and $\policyl$  maximizes $\text{vol}(\setoverlap)$. Instead at each time $\time$,  $\policyu$ and $\policyl$ are optimized
depending on the belief state $\belief_\time$.
Suppose at time $\time$, given observation $\obs_\time$, the belief state, $\belief_\time$, is computed by using  \eqref{eq:information_stateu}. For this belief state $\belief_\time$, the aim is to compute $\of^* \in \poly_{\of}$ \eqref{eq:polytopes_u} and $\uf^* \in \poly_{\uf}$ \eqref{eq:polytopes_l} such that the difference between myopic policy bounds, $\policyu(\belief_\time) - \policyl(\belief_\time)$, is minimized. That is,
\begin{align}\label{eq:LP_min_dist}
\begin{aligned}
    \left(\of^*,\uf^*\right) = \underset{\of \in \poly_{\of},~\uf \in \poly_\uf}{\argmin} ~\policyu(\belief_\time) - \policyl(\belief_\time).
\end{aligned}
\end{align}
\eqref{eq:LP_min_dist} can be decomposed into following two optimization problems,
\begin{align}\label{eq:LP_min_ub}
\begin{aligned}
    \of^* = \underset{\of \in \poly_{\of}}{\argmin}~ \policyu(\belief_\time), ~\uf^* = \underset{\uf \in \poly_{\uf}}{\argmax}~ \policyl(\belief_\time).
\end{aligned}
\end{align}
If assumptions {\ref{it:increasing_cost}} and {\ref{it:decreasing_cost}} hold, then the optimizations in \eqref{eq:LP_min_ub} are feasible.
Then $\policyu(\belief_\time)$ in \eqref{eq:new_cost_u} and $\of^*$, in \eqref{eq:LP_min_ub} can be computed as follows:  Starting with $\policyu(\belief_\time) = 1$, successively solve a maximum of $\actiondim$ feasibility LPs, where the $i$th LP searches for a feasible $\of \in \poly_\of$ in \eqref{eq:polytopes_u} so that the myopic upper bound yields action $i$, i.e. $\policyu(\belief_\time) = i$. The $i$th feasibility LP can be written as
 \begin{align}\label{eq:LP_UB_iterative}
\begin{aligned}
    &~\underset{\of \in \poly_\of}{\min}~\one_\statedim^\prime\of\\
    &s.t.,~~~\ucost_i^\prime\belief_\time \le \ucost_\action^\prime\belief_\time~\forall \action \in \actionspace, \action \ne i
\end{aligned}
\end{align}
The smallest $i$, for which \eqref{eq:LP_UB_iterative} is feasible, yields the solution $\left(\of^*, \policyu(\belief_\time)=i\right)$ of the optimization in \eqref{eq:LP_min_ub}. The above procedure is straightforwardly modified to obtain $\uf^*$ and the lower bound $\policyl(\belief_\time)$ \eqref{eq:new_cost_l}.

\section {Numerical Examples} \label{sec:numerical_examples}
Recall that on the set $\setoverlap$ \eqref{eq:def_overlap}, the upper and lower myopic bounds coincide with the optimal policy $\optpolicy(\belief)$. What is the performance loss outside the set $\setoverlap$? To quantify this, define the  policy
\begin{align}\nonumber
\begin{aligned}
\approxpolicy(\belief) = \begin{cases} \optpolicy(\belief) &\forall \belief \in \setoverlap\\
\text{arbitrary action (e.g. 1)} &\forall \belief \not\in \setoverlap\end{cases}
\end{aligned}
\end{align}
Let  $J_{\approxpolicy}(\belief_0)$ denote the discounted cost associated with $\approxpolicy(\belief_0)$. Also denote
\begin{align*}
\tilde{J}_{\optpolicy}(\belief_0) &= \E\left\{\sum_{k=0}^{\infty} \discount^{k} \tilde{\cost}_{\optpolicy(\belief_\time)}^\p \belief_k\right\}, \end{align*}
where
\begin{align*} \tilde{\cost}_{\optpolicy(\belief)} &= \begin{cases} {\cost}_{\optpolicy(\belief)}& \belief \in \setoverlap\\
\left[\underset{\action \in \actionspace}{\min}~\cost(1,\action),\cdots,\underset{\action \in \actionspace}{\min}~\cost(\statedim,\action)\right]^\p& \belief \not\in \setoverlap\end{cases}
\end{align*}
Clearly an upper bound for the  percentage loss in optimality due to using  policy $\approxpolicy$ instead of  optimal policy $\optpolicy$ is
\begin{align}\label{eq:percentloss}
\begin{aligned}
\percentloss = \cfrac{J_{\approxpolicy}(\belief_0) - \tilde{J}_{\optpolicy}(\belief_0)}{\tilde{J}_{\optpolicy}(\belief_0)}.
\end{aligned}
\end{align}
In the numerical examples below,  to evaluate $\percentloss$, 1000 Monte-Carlo simulations were run to estimate the discounted costs $J_{\approxpolicy}(\belief_0)$ and $\tilde{J}_{\optpolicy}(\belief_0)$ over a horizon of 100 time units.
The parameters $\percentloss$ and $\text{vol}\left(\setoverlap\right)$ are used to evaluate the performance of the optimized myopic policy bounds constructed according to \secn  \ref{sec:optimize_bounds}. Note that  $\percentloss$ depends on the choice of observation distribution $\oprob$, unlike $\text{vol}\left(\setoverlap\right)$, see discussion below (\ref{eq:def_overlap}) and also Example 2 below.

\textit{Example 1. Sampling and Measurement Control with Two Actions}: In this problem discussed in \secn \ref{chp:qdos}, at every decision epoch, the decision maker has the option of either recording a noisy observation (of a Markov chain) instantly (action $\action = 2$) or waiting for one time unit and then recording an  observation using a better sensor (action $\action = 1$). Should one record observations more frequently and less accurately or more accurately but less frequently?

We chose
 $\statedim = 3$, $\actiondim = 2$ and $\obsdim = 3$. Both  transition and observation probabilities are action dependent (parameters specified in the Appendix). The percentage loss in optimality is evaluated by simulation for different values of the discount factor $\discount$. Table \ref{tb:table1a} displays $\text{vol}\left(\setoverlap\right)$, $\percentloss_1$ and $\percentloss_2$. For each $\discount$, $\percentloss_1$ is obtained by assuming $\belief_0 = \basisvec_3$ (myopic bounds overlap at $\basisvec_3$) and $\percentloss_2$ is obtained by uniformly sampling $\belief_0 \notin \setoverlap$. Observe that $\text{vol}\left(\setoverlap\right)$ is large and $\percentloss_1$, $\percentloss_2$ are  small,  which indicates the usefulness of the proposed myopic policies.

\begin{table}
\caption{Performance of optimized myopic policies versus discount factor $\discount$ for five numerical examples.  The performance metrics $\text{vol}\left(\setoverlap\right)$ 
and $\percentloss$ are 
defined in (\ref{eq:def_overlap}) and (\ref{eq:percentloss}).}
\centering
\subtable[Example 1]{
\centering
\begin{tabular}{c|ccc}\hline
$\discount$ & \text{vol}$\left(\setoverlap\right)$ & $\percentloss_1$ & $\percentloss_2$\\ \hline
0.4& $95.3\%$  &  $0.30\% $ & $16.6\%$  \\
0.5& $94.2\%$  &  $0.61\%$ & $13.9\%$ \\
0.6& $92.4\%$  &  $1.56\%$ & $11.8\%$ \\
0.7& $90.2\%$  &  $1.63\%$ & $9.1\%$ \\
0.8& $87.4\%$  &  $1.44\%$ & $6.3\%$ \\
0.9& $84.1\%$  &  $1.00\%   $ & $3.2\%$    \\ \hline
\end{tabular}
\label{tb:table1a}
}
\subtable[Example 2]{
\centering
\begin{tabular}{ccccc}\hline
\text{vol}$\left(\setoverlap\right)$ & $\percentloss^d_1$&$\percentloss^d_2$&$\percentloss^c_1$&$\percentloss^c_2$  \\ \hline
$64.27\%$ & 7.73\%   & 12.88\% & 6.92\%   & 454.31\%  \\
$55.27\%$ & 8.58\%   & 12.36\% & 8.99\%   & 298.51\%  \\
$46.97\%$ & 8.97\%   & 11.91\% & 12.4\%   & 205.50\%  \\
$39.87\%$ & 8.93\%   & 11.26\% & 14.4\%  & 136.31\%  \\
$34.51\%$ & 10.9\%  & 12.49\% & 17.7\%  & 88.19\%  \\
$29.62\%$ & 11.2\%  & 12.24\% & 20.5\%  & 52.16\%  \\ \hline
\end{tabular}
\label{tb:table1c}
}
\subtable[Example 3]{
\centering
\begin{tabular}{ccc}\hline
\text{vol}$\left(\setoverlap\right)$ & $\percentloss_1$ & $\percentloss_2$\\ \hline
 $61.4\%$ & 2.5\% & 10.1\%	\\
 $56.2\%$ & 2.3\% & 6.9\%	\\
 $47.8\%$ & 1.7\% & 4.9\%	\\
 $40.7\%$ & 1.4\% & 3.5\%	\\
 $34.7\%$ & 1.1\% & 2.3\%	\\
 $31.8\%$ & 0.7\% & 1.4\%	\\ \hline
\end{tabular}
\label{tb:table1d}
}
\subtable[Example 4]{
\centering
\begin{tabular}{cccccc}\hline
$\overline{\text{vol}}\left(\setoverlap\right)$ & $\underline{\text{vol}}\left(\setoverlap\right)$ & $\overline{\percentloss}_1$ & $\underline{\percentloss}_1$ & $\overline{\percentloss}_2$ & $\underline{\percentloss}_2$\\ \hline
$98.9\%$ & $84.5\%$ & 0.10\% & 6.17\%	& 1.45\% & 1.71\%\\
$98.6\%$ & $80.0\%$ & 0.18\% & 7.75\%	& 1.22\% & 1.50\%\\
$98.4\%$ & $75.0\%$ & 0.23\% & 11.62\%	& 1.00\% & 1.31\%\\
$98.1\%$ & $68.9\%$ & 0.26\% & 14.82\%	& 0.75\% & 1.10\%\\
$97.8\%$ & $61.5\%$ & 0.27\% & 19.74\%	& 0.51\% & 0.89\%\\
$97.6\%$ & $52.8\%$ & 0.25\% & 24.08\%	& 0.26\% & 0.61\%\\ \hline
\end{tabular}
\label{tb:table1f}
}
\label{tb:table1}
\end{table}

\textit{Example 2. 10-state POMDP}: Consider a POMDP with $\statedim=10$, $\actiondim = 2$. Consider two sub-examples: the first with discrete observations  $\obsdim = 10$ (parameters  in Appendix), the second with  continuous  observations obtained using the additive Gaussian noise model, i.e. $\obs_\time = \state_\time + \noise_\time$ where $\noise_\time \sim \normal(0,1)$. The  percentage loss in optimality  is evaluated by simulation for these two sub examples and denoted by  $\percentloss^d_1, \percentloss^d_2$ (discrete observations) and $\percentloss^c_1, \percentloss^c_2$ (Gaussian observations)
in Table \ref{tb:table1c}.

$\percentloss^d_1$ and $\percentloss^c_1$ are obtained by assuming $\belief_0 = \basisvec_5$ (myopic bounds overlap at $\basisvec_5$). $\percentloss^d_2$ and $\percentloss^c_2$ are obtained by sampling $\belief_0 \notin \setoverlap$. Observe  from Table \ref{tb:table1c} that $\text{vol}\left(\setoverlap\right)$  decreases with $\discount$. 

\textit{Example 3. 8-state and 8-action POMDP}: Consider a  POMDP with $\statedim=8$, $\actiondim = 8$ and $\obsdim = 8$ (parameters in  Appendix). Table \ref{tb:table1d} displays $\text{vol}\left(\setoverlap\right)$, $\percentloss_1$ and $\percentloss_2$. For each $\discount$, $\percentloss_1$ is obtained by assuming $\belief_0 = \basisvec_1$ (myopic bounds overlap at $\basisvec_1$) and $\percentloss_2$ is obtained by uniformly sampling $\belief_0 \notin \setoverlap$. The results indicate that the myopic policy bounds are still useful for some values of $\discount$.

\textit{Example 4.  Myopic Bounds versus Transition Matrix}:  The aim here is to illustrate the performance of the optimized myopic bounds over a range of transition probabilities.
 Consider a POMDP  with $\statedim=3$, $\actiondim = 2$, additive Gaussian noise model of Example 2, and   transition  matrices 
\begin{align}\nonumber
\tp(2) =\begin{bmatrix}
1 & 0 & 0\\
1-2\theta_1 & \theta_1 & \theta_1\\
1-2\theta_2 & \theta_2 & \theta_2
\end{bmatrix},\; \tp(1) = \tp^2 (2)
\end{align}
It is straightforward to show that $\forall~\theta_1, \theta_2$ such that $\theta_1 + \theta_2 \le 1, \theta_2 \ge \theta_1$, $\tp(1)$ and $\tp(2)$ satisfy  \ref{A3} and \ref{A4}. The costs are  $\cost_1 = \left[1, 1.1, 1.2\right]^\prime$ and $\cost_2 = \left[1.2, 1.1, 1.1\right]^\prime$. 
Table \ref{tb:table1f} displays the worst case and best case values for performance metrics $(\text{vol}\left(\setoverlap\right),\percentloss_1,\percentloss_2)$ versus discount factor $\discount$
by sweeping over the entire range of $(\theta_1,\theta_2)$. The worst case performance is  denoted by  $\underline{\text{vol}}\left(\setoverlap\right)$,  $\underline{\percentloss}_1$,
 $\underline{\percentloss}_2$ and the best case  
by $\overline{\text{vol}}\left(\setoverlap\right)$, $\overline{\percentloss}_1$,
 $\overline{\percentloss}_2$.

%

\section{Blackwell Dominance of Observation Distributions and Optimality of Myopic Policies} \label{sec:blackwelldom}
\index{myopic policy! optimality for POMDP! Blackwell dominance|(}
\index{Blackwell dominance|(}
\index{structural result! Blackwell dominance|(}
In previous sections of this chapter, we used copositive dominance to construct upper and lower myopic bounds to the optimal policy of a POMDP.
In this section we will use another concept, called Blackwell dominance, to construct lower myopic bounds to the optimal policy for a POMDP.

\subsection{Myopic Policy Bound to Optimal Decision Policy} \label{sec:myopic}

Motivated by active sensing applications, consider the following POMDPs where
based on the current
belief state $\belief_{k-1}$, agent $k$ chooses sensing mode $$u_k \in \{1 \text{ (low resolution sensor) }, 2 \text{ (high resolution sensor)}\}.$$
Depending on its mode $u_k$, the sensor  views the world according to this mode  -- that is, 
it  obtains  observation  
from a distribution that depends on $u_k$. 
Assume that for mode $u\in \{1,2\}$, the observation $y^{(u)} \in \obspace^{(u)} = \{1,\ldots,Y^{(u)}\}$ is obtained
from  the matrix of conditional probabilities
\begin{multline*} \oprob(\action) = \big(B_{iy^{(u)}}(\action) , i \in \{1,2,\ldots,\statedim\},  y^{(u)} \in \obspace^{(u)} \big) \\ \text{ where }
B_{iy^{(u)}}(\action) = \prob(y^{(u)}|x=e_i,u) .\end{multline*}
The notation $\obspace^{(u)}$ allows for mode dependent observation spaces.
In  sensor scheduling \cite{Kri02}, the tradeoff is as follows: Mode
$u= 2$ yields more accurate observations of the state than mode $u=1$, but the cost of choosing mode  $u=2$ is higher
than mode $u=1$.
Thus
there is an tradeoff between the cost of acquiring information and the value of the information.

The assumption that mode $u=2$ yields more accurate observations than mode $u=1$ is modeled as follows: We say mode 2 {\em Blackwell
dominates} mode 1, denoted as \index{Blackwell dominance}
\beq \label{eq:bd}
  \oprob(2) \bd \oprob(1)\quad  \text{ if }  \quad  B(1)  = B(2)\, \aB . \eeq
   Here $\aB$ is a $Y^{(2)} \times Y^{(1)}$ stochastic matrix. $\aB$ can be viewed as a {\em confusion matrix} that maps $\obspace^{(2)}$ probabilistically to $\obspace^{(1)}$.
   (In a communications context, one can view $\aB$ as a noisy discrete memoryless channel with input $y^{(2)}$ and output $y^{(1)}$).
Intuitively (\ref{eq:bd})  means that
  $B^{(2)} $ is more accurate than   $B^{(1)} $.

The goal is to compute the optimal policy $\mu^*(\belief) \in \{1,2\}$ to minimize the  expected cumulative  cost incurred by  all the agents
\beq
J_\mu(\belief) = \Ep \{ \sum_{k=0}^\infty \discount^{k} C(\belief_{k},u_k) \} .
\eeq
where  $\discount \in [0,1)$ is  the discount factor.
Even though solving the above POMDP 
is computationally intractable in general,
 using Blackwell dominance, we show below that a myopic policy forms a lower
 bound for the optimal policy.

The
value function $V(\belief)$ and optimal policy $\mu^*(\belief)$  satisfy Bellman's equation
\beq  \label{eq:dp_algmove}\begin{split}
V(\belief) &= \min_{u \in \actionspace} Q(\belief,u), \quad
\mu^*(\belief)= \arg\min_{u \in \actionspace} Q(\belief,u) ,\; J_{\mu^*}(\belief) = V(\belief) \\
 Q(\belief,u) &=  C(\belief,u) 
+ \discount \sum_{y^{(u)} \in \obspace^{(u)}}  V\left( T(\belief ,y, {u}) \right) \sigma(\belief,y, {u}), \\
T(\belief,y, {u}) &= \frac{B_{y^{(u)}} (\action) P^\p \belief}{\sigma(\belief,\obs,\action)},
\; \filterd(\belief,y, \action) = \mathbf{1}_X^\p B_{y^{(u)}}(\action) P^\p \belief .
\end{split} \eeq

We now present the structural result.
Let $\Pi^s \subset \I$ denote the set of belief states for which
$C(\belief,2) < C(\belief,1)$. 
Define the  myopic policy
$$\policyl(\belief) = \begin{cases} 2 & \belief \in \Pi^s \\
 								1 & \text{ otherwise } \end{cases}$$

\begin{theorem}\label{thm:compare2}  Assume that $\Cost(\belief,\action)$ is concave with respect to $\belief \in \Belief$
for each action $\action$.
Suppose  $\oprob(2) \bd \oprob(1)  $, i.e., $    B(1)  = B(2) \aB $ holds where $\aB$ is a stochastic matrix.
Then the myopic policy
 $\policyl(\belief)$ is a lower    bound to the optimal
policy $\mu^*(\belief)$, i.e.,  $\mu^*(\belief) \geq \policyl(\belief)$ for
all $\belief \in \I$. 
In particular, 
for $\belief \in \Pi^s$,  $\mu^*(\belief) = \policyl(\belief)$,
i.e., it is optimal to choose action 2 when the  belief is in $\Pi^s$.
\qed
\end{theorem}

{\em Remark}: If  $\oprob(1) \bd \oprob(2)  $, then the myopic policy constitutes an upper bound to the optimal policy.

Theorem \ref{thm:compare2} is proved below.
 The proof exploits the fact that the value function is concave  and uses Jensen's inequality.
The usefulness of Theorem \ref{thm:compare2} stems from the fact that $\policyl(\belief)$ is trivial to compute. It  forms a provable
lower bound to the computationally intractable optimal policy $\mu^*(\belief)$.  
Since $\policyl$ is sub-optimal, it incurs a higher   cumulative cost. This  cumulative cost can be evaluated via simulation and is
an upper bound to the achievable  optimal cost.

Theorem \ref{thm:compare2} is  non-trivial.
The instantaneous costs
satisfying $C(\belief,2) < C(\belief,1)$,  does not trivially  imply that the myopic policy 
$\policyl(\belief)$ coincides with the optimal policy $\mu^*(\belief)$, since the optimal policy applies to a cumulative cost function involving
an infinite horizon
 trajectory of the dynamical system.

It is instructive to compare Theorem \ref{thm:compare2} with Theorem \vref{th:theorem1udit}.
Theorem \ref{th:theorem1udit}  used copositive dominance (with several assumptions on the POMDP model) to construct both upper and lower bounds to the optimal policy.
In comparison, (Theorem \ref{thm:compare2})  needs no assumptions on the POMDP model apart from the  Blackwell dominance condition $B(1)  = B(2)\, \aB$ and concavity of costs with respect to the belief; but only yields an upper bound. 


\subsection{Example.  Optimal Filter  vs Predictor Scheduling} \index{sensor scheduling! filter vs predictor} 
\index{Blackwell dominance! filter vs predictor scheduling}
Suppose $\action=2$ is an active sensor  (filter) which obtains measurements of the underlying Markov chain and uses the optimal HMM  filter
on these measurements to compute the belief and therefore
the state estimate.
 So the usage cost of sensor 2 is high (since obtaining observations is expensive and can also result in increased threat of being discovered), but its performance cost is
low (performance quality is high). 

Suppose  sensor $\action=1$ is a predictor which needs no measurement. So its usage cost  is low (no measurement is required).
However its performance cost is high since it is more inaccurate compared to sensor 2.

Since the predictor has non-informative observation probabilities, its observation probability matrix is $\oprob(1) = 
\frac{1}{\obsdim}\ones_{\statedim \times \obsdim}$. So clearly
$\oprob(1) = \oprob(2)  \oprob(1) $
meaning that the filter (sensor 2)  Blackwell dominates the predictor (sensor 1)
Theorem \ref{thm:compare2} then says that if the current belief is $\belief_k$, then if  $\Cost(\belief_k,2) < \Cost(\belief_k,1)$, it is always
optimal to deploy the filter (sensor 2).

\subsection{Proof of Theorem  \ref{thm:compare2}}

$\Cost(\belief,\action)$ concave implies that   $V(\belief)$ is concave on $\Belief$. 
We then use the Blackwell dominance condition (\ref{eq:bd}). In particular,
\begin{align*} \filter(\belief,\yi,1) &=   \sum_{\yii \in \obspace^{(2)}} \filter(\belief,\yii,2) \frac{\filterd(\belief,\yii,2)}{\sigs(\belief,\yi,1)} P(\yi|\yii)  \\
 \sigs(\belief,\yi,1) &= \sum_{\yii \in \obspace^{(2)}} \filterd(\belief,\yii,2) P(\yi|\yii). \end{align*}
Therefore $\frac{\filterd(\belief,\yii,2)}{\sigs(\belief,\yi,1)} P(\yi|\yii) $ is a probability measure w.r.t.\  $\yii$ (since the denominator is the sum
of the numerator over all $\yii$).
Since $V(\cdot)$ is concave, using Jensen's inequality it follows that
\begin{align}
&V(\filter(\belief,\yi,1) )  = V \left(\sum_{\yii \in \obspace^{(2)}} \filter(\belief,\yii,2) \frac{\filterd(\belief,\yii,2)}{\sigs(\belief,\yi,1)} P(\yi|\yii) \right)\nn  \\
&\geq \sum_{\yii \in \obspace^{(2)}}  V (\filter(\belief,\yii,2)) \frac{\filterd(\belief,\yii,2)}{\sigs(\belief,\yi,1)} P(\yi|\yii) \nn \\
&\implies   \sum_{\yi}  V(\filter(\belief,\yi,1) ) \sigs(\belief,\yi,1) \geq
\sum_{\yii} V(\filter(\belief,\yii,2)\filterd(\belief,\yii,2). \label{eq:bdproof1}
\end{align}
Therefore for $\belief \in \Pi^s$, 
$$ C(\belief,2) + \discount\sum_{\yii} V(\filter(\belief,\yii,2)\filterd(\belief,\yii,2) \leq 
C(\belief,1) + \discount \sum_{\yi}  V(\filter(\belief,\yi),1 ) \sigs(\belief,\yi,1) . $$
So for $\belief \in \Pi^s$, the optimal policy $\mu^*(\belief) = \arg\min_{u \in \actionspace}Q(\belief,u) = 2$.
So $\policyl(\belief) = \mu^*(\belief)=2$ for $\belief \in \Pi^s$ and $\bar{\mu}(\belief)=1$ otherwise, implying that
$\bar{\mu}(\belief)$ is a lower  bound for $\mu^*(\belief)$.

\index{Blackwell dominance|)}
\index{structural result! Blackwell dominance|)}
\index{myopic policy! optimality for POMDP! Blackwell dominance|)}

\section{How does optimal POMDP cost vary with state and observation dynamics?} \index{sensitivity of POMDP! ordinal}
\index{structural result! POMDP sensitivity|(}
\label{sec:order}

\newcommand{\Ab}{$\overline{\text{A1}}$}

This and the next section focus on
{\em achievable costs} attained by the optimal policy.
This section presents gives bounds on the achievable performance of the optimal policies by the decision maker.
This is done by introducing a partial ordering of the transition and observation probabilities -- the larger these parameters with respect to this order, the larger the 
optimal  cumulative cost incurred. 


How does the  optimal expected   cumulative cost $J_{\mu^*}$ of a POMDP vary
  with
 transition matrix $\tp$ and observation distribution $\oprob$? Can the  transition
 matrices and observation distributions be ordered
so that the larger they are, the larger the optimal  cumulative  cost?  Such a result  is very useful -- it allows us to compare the optimal
 performance of different POMDP models, even though computing these is intractable. 
 Recall that the transition matrix specifies the mobility of the state and the observation matrix specifies
 the noise distribution; so understanding how these affect the  achievable  optimal   cost is important.
 
 Consider two distinct POMDPs with transition matrices  $\model = \tp$ and $\bmodel = \btp$, respectively.
Alternatively, consider two distinct POMDPs with observation distributions  $\model = \oprob$ and $\bmodel = \boprob$, respectively.
Assume that the instantaneous costs   $C(\belief,u)$  and discount factors $\discount $ for both POMDPs are identical.

Let $\mu^*(\model)$ and $\mu^*({\bmodel})$ denote, respectively, the optimal policies for the two POMDPs.
 Let $J_{\mu^*(\model)}(\belief;\model) =  V(\belief;\model)$ and 
$J_{\mu^*(\bmodel)}(\belief;\bmodel) =V(\belief;\bmodel)$ denote the   optimal value functions corresponding to applying the respective optimal policies.

Consider two arbitrary transition matrices $\tpone$ and $\tptwo$.   Recalling Definition \ref{def:lR} for $\succeq$ and \ref{A4p}, 
  assume the copositive ordering  \beq \tp \succeq \tptwo.  \label{eq:mor} \eeq
Recall the   Blackwell dominance of observation distributions:  $\boprob$  Blackwell dominates $\oprob$ denoted \index{Blackwell dominance}
as 
\beq \boprob \bd  \oprob \text{ if }  \oprob = \boprob \aB  \label{eq:blackwell}\eeq
where $\aB= (\aB_{lm})$ is a stochastic kernel, i.e., $\sum_m\aB_{lm} = 1$.

The question we pose is:  How does the optimal  cumulative cost
$J_{\mu^*(\model)}(\belief;\model)$ vary with transition matrix $\tp$ or  observation distribution $\oprob$? For example,
 in quickest change detection,
do certain phase-type distributions for the change time result in larger  optimal  cumulative cost compared to other phase-type
distributions? 
In controlled sensing, do certain noise distributions incur a larger optimal  cumulative cost than other noise distributions?

\begin{theorem} \label{thm:tmove}
\begin{compactenum}
\item  Consider two distinct POMDP problems with transition matrices $\tp$ and $\btp$, respectively,  where ${\tp} \succeq \btp$
 with respect to copositive ordering (\ref{eq:mor}).
 If 
  \ref{A1}, \ref{A2}, \ref{A3} hold, then   the optimal cumulative  costs   satisfy $$J_{\mu^*(\tp)}(\belief;\tp) \leq J_{\mu^*(\btp)}(\belief;\btp). $$
\item  Consider two distinct POMDP problems with observation distributions $\oprob$ and $\boprob$, respectively,  where  $\boprob \bd \oprob$ with respect to Blackwell ordering (\ref{eq:blackwell}).  
Then $$J_{\mu^*(\oprob)}(\belief;\oprob) \geq J_{\mu^*(\boprob)}(\belief;\boprob).$$
\end{compactenum}
 \end{theorem}

The proof is in Appendix  \cite{Kri16}.
Computing the optimal policy and associated  cumulative   cost of a POMDP 
 is  intractable. Yet, the above
theorem facilitates comparison of these optimal  costs for different transition and observation  probabilities. 

It is instructive to compare  Theorem \ref{thm:tmove}(1)  with  Theorem \ref{thm:tmdpmove} of \secn  \ref{sec:mdptpvar},
which dealt with the optimal  costs of two  {\em fully} observed MDPs.
Comparing the assumptions of Theorem \ref{thm:tmdpmove} with Theorem \ref{thm:tmove}(1), we see that the assumption on the costs
(A1) is identical to \ref{A1}. Assumption  (A2) in  Theorem \ref{thm:tmdpmove} is replaced by
\ref{A2}, \ref{A3} which are conditions on the transition and observation probabilities.
 The first order dominance condition  (A2) in  Theorem \ref{thm:tmdpmove} is a weaker condition than the TP2 condition \ref{A2}. In particular, \ref{A2} implies
 (A2).
Finally, (A5) in  Theorem \ref{thm:tmdpmove} is replaced by the stronger assumption of copositivity \ref{A4p}. Indeed, \ref{A4p} implies (A5).

\noindent {\em Remark}: An obvious consequence of Theorem \ref{thm:tmove}(1) is that  a Markov chain 
with  transition probabilities $\tp_{i\statedim}=1$ for each state $i$ incurs the lowest cumulative cost. After one transition such a Markov chain always remains in state
$\statedim$. Since the instantaneous costs are decreasing with state \ref{A1}, clearly, such a transition matrix incurs the lowest  cumulative cost. Similarly if $\tp_{i1}=1$ for each state $i$,
then the highest cumulative cost is incurred.
A   consequence of Theorem \ref{thm:tmove}(2) is that  the optimal  cumulative  cost incurred with perfect measurements is smaller than
that with noisy measurements. 


\section{\pwe}
This chapter is based on \cite{KP15} and  extends the structural results of \cite{Lov87,Rie91}.  
Constructing myopic policies using Blackwell dominance goes back to \cite{WD80}.
 \cite{Rie91}
 uses the multivariate TP2 order  for  POMDPs with multivariate observations.
 \cite{HL11} shows the elegant result that the $p$-th root of a stochastic matrix $\tp$ is a stochastic matrix providing $\tp^{-1}$ is an M-matrix.
The structural results can also be developed for stochastic control of continuous time HMMs. By using a robust formulation of the continuous-time HMM filter \cite{JKL96}, time  discretization
yields exactly the HMM filter. One can then consider the resulting discretized Hamilton Jacobi-Bellman equation and obtain structural results.

Although, not discussed in this article, the parameters of the underlying HMM can be estimated recursively via online estimators such as those in \cite{KY02} or offline via EM type
algorithms. 

\begin{subappendices}
\section{POMDP Numerical Examples}
{\footnotesize
{\textit{Parameters of Example 1}}: For  the first example the parameters are defined as,
\begin{align} \nonumber
\begin{aligned}
c &=
 \begin{pmatrix}
    1.0000 &   1.5045   & 1.8341\\
    1.5002 &   1.0000   & 1.0000
 \end{pmatrix}^\prime,~
\tp({2}) =
 \begin{pmatrix}
1.0000  &  0.0000 &   0.0000\\
0.4677  &  0.4149 &   0.1174\\
0.3302  &  0.5220 &   0.1478
 \end{pmatrix}
,~\tp(1) = \tp^2(2)\\
\oprob(1) &=
 \begin{pmatrix}
 0.6373 &   0.3405 &   0.0222\\
 0.3118 &   0.6399 &   0.0483\\
 0.0422 &   0.8844 &   0.0734
 \end{pmatrix}
,~
 \oprob(2) =
 \begin{pmatrix}
 0.5927  &  0.3829 &   0.0244\\
0.4986  &  0.4625 &   0.0389\\
0.1395  &  0.79   &   0.0705
 \end{pmatrix}.
\end{aligned}
\end{align}


{\textit{Parameters of Example 2}: For discrete observations $\oprob(\action) = \oprob ~\forall \action \in \actionspace$,}
{\begin{align} \nonumber
    \begin{aligned}
    &\oprob =
     \begin{pmatrix}
0.0297 &   0.1334 &   0.1731 &   0.0482 &   0.1329 &   0.1095 &   0.0926 &   0.0348 &   0.1067 &   0.1391\\
0.0030 &   0.0271 &   0.0558 &   0.0228 &   0.0845 &   0.0923 &   0.1029 &   0.0511 &   0.2001 &   0.3604\\
0.0003 &   0.0054 &   0.0169 &   0.0094 &   0.0444 &   0.0599 &   0.0812 &   0.0487 &   0.2263 &   0.5075\\
     0 &   0.0011 &   0.0051 &   0.0038 &   0.0225 &   0.0368 &   0.0593 &   0.0418 &   0.2250 &   0.6046\\
     0 &   0.0002 &   0.0015 &   0.0015 &   0.0113 &   0.0223 &   0.0423 &   0.0345 &   0.2133 &   0.6731\\
     0 &        0 &   0.0005 &   0.0006 &   0.0056 &   0.0134 &   0.0298 &   0.0281 &   0.1977 &   0.7243\\
     0 &        0 &   0.0001 &   0.0002 &   0.0028 &   0.0081 &   0.0210 &   0.0227 &   0.1813 &   0.7638\\
     0 &        0 &        0 &   0.0001 &   0.0014 &   0.0048 &   0.0147 &   0.0183 &   0.1651 &   0.7956\\
     0 &        0 &        0 &        0 &   0.0007 &   0.0029 &   0.0103 &   0.0147 &   0.1497 &   0.8217\\
     0 &        0 &        0 &        0 &   0.0004 &   0.0017 &   0.0072 &   0.0118 &   0.1355 &   0.8434
     \end{pmatrix}
    \end{aligned}
\end{align}
\begin{align} \nonumber
    \begin{aligned}
    &\tp({1}) =
     \begin{pmatrix}
0.9496  &  0.0056 &   0.0056 &   0.0056&    0.0056 &   0.0056 &   0.0056 &   0.0056  &  0.0056 &   0.0056\\
0.9023  &  0.0081 &   0.0112 &   0.0112&    0.0112 &   0.0112 &   0.0112 &   0.0112  &  0.0112 &   0.0112\\
0.8574  &  0.0097 &   0.0166 &   0.0166&    0.0166 &   0.0166 &   0.0166 &   0.0166  &  0.0166 &   0.0167\\
0.8145  &  0.0109 &   0.0218 &   0.0218&    0.0218 &   0.0218 &   0.0218 &   0.0218  &  0.0218 &   0.0220\\
0.7737  &  0.0119 &   0.0268 &   0.0268&    0.0268 &   0.0268 &   0.0268 &   0.0268  &  0.0268 &   0.0268\\
0.7351  &  0.0126 &   0.0315 &   0.0315&    0.0315 &   0.0315 &   0.0315 &   0.0315  &  0.0315 &   0.0318\\
0.6981  &  0.0131 &   0.0361 &   0.0361&    0.0361 &   0.0361 &   0.0361 &   0.0361  &  0.0361 &   0.0361\\
0.6632  &  0.0136 &   0.0404 &   0.0404&    0.0404 &   0.0404 &   0.0404 &   0.0404  &  0.0404 &   0.0404\\
0.6301  &  0.0139 &   0.0445 &   0.0445&    0.0445 &   0.0445 &   0.0445 &   0.0445  &  0.0445 &   0.0445\\
0.5987  &  0.0141 &   0.0484 &   0.0484&    0.0484 &   0.0484 &   0.0484 &   0.0484  &  0.0484 &   0.0484
     \end{pmatrix}
    \end{aligned}
\end{align}
\begin{align} \nonumber
    \begin{aligned}
      \tp({2}) &=
          \begin{pmatrix}
0.5688 &   0.0143  &  0.0521  &  0.0521  &  0.0521 &   0.0521  &  0.0521 &   0.0521  &  0.0521 &   0.0522\\
0.5400 &   0.0144  &  0.0557  &  0.0557  &  0.0557 &   0.0557  &  0.0557 &   0.0557  &  0.0557 &   0.0557\\
0.5133 &   0.0145  &  0.0590  &  0.0590  &  0.0590 &   0.0590  &  0.0590 &   0.0590  &  0.0590 &   0.0592\\
0.4877 &   0.0145  &  0.0622  &  0.0622  &  0.0622 &   0.0622  &  0.0622 &   0.0622  &  0.0622 &   0.0624\\
0.4631 &   0.0145  &  0.0653  &  0.0653  &  0.0653 &   0.0653  &  0.0653 &   0.0653  &  0.0653 &   0.0653\\
0.4400 &   0.0144  &  0.0682  &  0.0682  &  0.0682 &   0.0682  &  0.0682 &   0.0682  &  0.0682 &   0.0682\\
0.4181 &   0.0144  &  0.0709  &  0.0709  &  0.0709 &   0.0709  &  0.0709 &   0.0709  &  0.0709 &   0.0712\\
0.3969 &   0.0143  &  0.0736  &  0.0736  &  0.0736 &   0.0736  &  0.0736 &   0.0736  &  0.0736 &   0.0736\\
0.3771 &   0.0141  &  0.0761  &  0.0761  &  0.0761 &   0.0761  &  0.0761 &   0.0761  &  0.0761 &   0.0761\\
0.3585 &   0.0140  &  0.0784  &  0.0784  &  0.0784 &   0.0784  &  0.0784 &   0.0784  &  0.0784 &   0.0787
        \end{pmatrix}
\\
    \end{aligned}
    \end{align}}
  {\begin{align} \nonumber
    \begin{aligned}
        \cost &=
          \begin{pmatrix}
    0.5986 &   0.5810  &  0.6116 &   0.6762  &  0.5664  &  0.6188  &  0.7107 &   0.4520 &   0.5986 &   0.7714\\
    0.6986 &   0.6727  &  0.7017 &   0.7649  &  0.6536  &  0.6005  &  0.6924 &   0.4324 &   0.5790 &   0.6714
          \end{pmatrix}^\prime
          \end{aligned}
    \end{align}}

{\textit{Parameters of Example 3}: $\oprob(\action) = \trid_{0.7} ~\forall \action \in \actionspace$},
$\text{where}~\trid_\varepsilon~\text{is a tridiagonal matrix defined as}$
    \begin{align} \nonumber
    \begin{aligned}
        ~\trid_\varepsilon = \left[\varepsilon_{ij}\right]_{\statedim \times \statedim}, \varepsilon_{ij} =\begin{cases} \varepsilon &i=j\\
        1-\varepsilon &(i,j) = (1,2),(\statedim-1,\statedim)\\
        \cfrac{1-\varepsilon}{2} &(i,j)=(i,i+1),(i,i-1),i\ne 1,\statedim\\
        0 &\text{otherwise}\end{cases}
          \end{aligned}
    \end{align}
\begin{align} \nonumber
    \begin{aligned}
    &\tp({1}) =
     \begin{pmatrix}
   0.1851 &   0.1692 &   0.1630 &   0.1546 &   0.1324  &  0.0889 &   0.0546 &   0.0522\\
     0.1538 &   0.1531 &   0.1601 &   0.1580 &   0.1395  &  0.0994 &   0.0667 &   0.0694\\
     0.1307 &   0.1378 &   0.1489 &   0.1595 &   0.1472  &  0.1143 &   0.0769 &   0.0847\\
     0.1157 &   0.1307 &   0.1437 &   0.1591 &   0.1496  &  0.1199 &   0.0840 &   0.0973\\
     0.1053 &   0.1196 &   0.1388 &   0.1579 &   0.1520  &  0.1248 &   0.0888 &   0.1128\\
     0.0850 &   0.1056 &   0.1326 &   0.1618 &   0.1585  &  0.1348 &   0.0977 &   0.1240\\
     0.0707 &   0.0906 &   0.1217 &   0.1578 &   0.1629  &  0.1447 &   0.1078 &   0.1438\\
     0.0549 &   0.0757 &   0.1095 &   0.1502 &   0.1666  &  0.1576 &   0.1189 &   0.1666
     \end{pmatrix}
    \end{aligned}
\end{align}
\begin{align} \nonumber
    \begin{aligned}
      \tp({2}) &=
          \begin{pmatrix}
    0.0488 &   0.0696 &   0.1016 &   0.1413 &   0.1599 &   0.1614 &   0.1270  &  0.1904\\
     0.0413 &   0.0604 &   0.0882 &   0.1292 &   0.1503 &   0.1661 &   0.1425  &  0.2220\\
     0.0329 &   0.0482 &   0.0752 &   0.1195 &   0.1525 &   0.1694 &   0.1519  &  0.2504\\
     0.0248 &   0.0388 &   0.0649 &   0.1097 &   0.1503 &   0.1732 &   0.1643  &  0.2740\\
     0.0196 &   0.0309 &   0.0566 &   0.0985 &   0.1429 &   0.1805 &   0.1745  &  0.2965\\
     0.0158 &   0.0258 &   0.0517 &   0.0934 &   0.1392 &   0.1785 &   0.1794  &  0.3162\\
     0.0134 &   0.0221 &   0.0463 &   0.0844 &   0.1335 &   0.1714 &   0.1822  &  0.3467\\
     0.0110 &   0.0186 &   0.0406 &   0.0783 &   0.1246 &   0.1679 &   0.1899  &  0.3691
          \end{pmatrix}
    \end{aligned}
    \end{align}
    \begin{align} \nonumber
    \begin{aligned}
    &\tp({3}) =
     \begin{pmatrix}
   0.0077 &   0.0140 &   0.0337 &   0.0704  &  0.1178  &  0.1632  &  0.1983  &  0.3949\\
     0.0058 &   0.0117 &   0.0297 &   0.0659  &  0.1122  &  0.1568  &  0.1954  &  0.4225\\
     0.0041 &   0.0090 &   0.0244 &   0.0581  &  0.1011  &  0.1494  &  0.2013  &  0.4526\\
     0.0032 &   0.0076 &   0.0210 &   0.0515  &  0.0941  &  0.1400  &  0.2023  &  0.4803\\
     0.0022 &   0.0055 &   0.0165 &   0.0439  &  0.0865  &  0.1328  &  0.2006  &  0.5120\\
     0.0017 &   0.0044 &   0.0132 &   0.0362  &  0.0751  &  0.1264  &  0.2046  &  0.5384\\
     0.0012 &   0.0033 &   0.0106 &   0.0317  &  0.0702  &  0.1211  &  0.1977  &  0.5642\\
     0.0009 &   0.0025 &   0.0091 &   0.0273  &  0.0638  &  0.1134  &  0.2004  &  0.5826
     \end{pmatrix}
    \end{aligned}
\end{align}
\begin{align} \nonumber
    \begin{aligned}
      \tp({4} )&=
          \begin{pmatrix}
    0.0007 &   0.0020 &   0.0075 &   0.0244  &  0.0609  &  0.1104  &  0.2013  &  0.5928\\
     0.0005 &   0.0016 &   0.0063 &   0.0208  &  0.0527  &  0.1001  &  0.1991  &  0.6189\\
     0.0004 &   0.0013 &   0.0049 &   0.0177  &  0.0468  &  0.0923  &  0.1981  &  0.6385\\
     0.0003 &   0.0009 &   0.0038 &   0.0149  &  0.0407  &  0.0854  &  0.2010  &  0.6530\\
     0.0002 &   0.0007 &   0.0031 &   0.0123  &  0.0346  &  0.0781  &  0.2022  &  0.6688\\
     0.0001 &   0.0005 &   0.0023 &   0.0100  &  0.0303  &  0.0713  &  0.1980  &  0.6875\\
     0.0001 &   0.0004 &   0.0019 &   0.0083  &  0.0266  &  0.0683  &  0.1935  &  0.7009\\
     0.0001 &   0.0003 &   0.0014 &   0.0069  &  0.0240  &  0.0651  &  0.1878  &  0.7144
          \end{pmatrix}
    \end{aligned}
    \end{align}
    \begin{align} \nonumber
    \begin{aligned}
    &\tp({5}) =
     \begin{pmatrix}
        0.0000 &   0.0002 &   0.0010 &   0.0054  &  0.0204  &  0.0590  &  0.1772  &  0.7368\\
     0.0000 &   0.0001 &   0.0008 &   0.0041  &  0.0168  &  0.0515  &  0.1663  &  0.7604\\
     0.0000 &   0.0001 &   0.0006 &   0.0038  &  0.0156  &  0.0480  &  0.1596  &  0.7723\\
     0.0000 &   0.0001 &   0.0005 &   0.0032  &  0.0139  &  0.0450  &  0.1603  &  0.777\\
     0.0000 &   0.0001 &   0.0004 &   0.0028  &  0.0124  &  0.0418  &  0.1590  &  0.7835\\
     0.0000 &   0.0001 &   0.0003 &   0.0023  &  0.0106  &  0.0389  &  0.1547  &  0.7931\\
     0.0000 &   0.0000 &   0.0003 &   0.0018  &  0.0090  &  0.0351  &  0.1450  &  0.8088\\
     0.0000 &   0.0000 &   0.0002 &   0.0015  &  0.0080  &  0.0325  &  0.1386  &  0.8192
     \end{pmatrix}
    \end{aligned}
\end{align}
\begin{align} \nonumber
    \begin{aligned}
      \tp({6}) &=
          \begin{pmatrix}
    0.0000 &   0.0000 &   0.0001 &   0.0012  &  0.0067  &  0.0296  &  0.1331  &  0.8293\\
     0.0000 &   0.0000 &   0.0001 &   0.0010  &  0.0059  &  0.0275  &  0.1238  &  0.8417\\
     0.0000 &   0.0000 &   0.0001 &   0.0009  &  0.0056  &  0.0272  &  0.1238  &  0.8424\\
     0.0000 &   0.0000 &   0.0001 &   0.0009  &  0.0053  &  0.0269  &  0.1234  &  0.8434\\
     0.0000 &   0.0000 &   0.0001 &   0.0006  &  0.0043  &  0.0237  &  0.1189  &  0.8524\\
     0.0000 &   0.0000 &   0.0001 &   0.0005  &  0.0038  &  0.0215  &  0.1129  &  0.8612\\
     0.0000 &   0.0000 &   0.0000 &   0.0004  &  0.0032  &  0.0191  &  0.1094  &  0.8679\\
     0.0000 &   0.0000 &   0.0000 &   0.0003  &  0.0025  &  0.0161  &  0.1011  &  0.8800
          \end{pmatrix}
    \end{aligned}
    \end{align}
    \begin{align} \nonumber
    \begin{aligned}
    &\tp({7}) =
     \begin{pmatrix}
   0.0000 &   0.0000 &   0.0000 &   0.0003  &  0.0022  &  0.0143  &  0.0938  &  0.8894\\
     0.0000 &   0.0000 &   0.0000 &   0.0002  &  0.0019  &  0.0136  &  0.0901  &  0.8942\\
     0.0000 &   0.0000 &   0.0000 &   0.0002  &  0.0017  &  0.0126  &  0.0849  &  0.9006\\
     0.0000 &   0.0000 &   0.0000 &   0.0002  &  0.0015  &  0.0118  &  0.0819  &  0.9046\\
     0.0000 &   0.0000 &   0.0000 &   0.0001  &  0.0013  &  0.0108  &  0.0754  &  0.9124\\
     0.0000 &   0.0000 &   0.0000 &   0.0001  &  0.0011  &  0.0098  &  0.0714  &  0.9176\\
     0.0000 &   0.0000 &   0.0000 &   0.0001  &  0.0010  &  0.0090  &  0.0713  &  0.9186\\
     0.0000 &   0.0000 &   0.0000 &   0.0001  &  0.0009  &  0.0084  &  0.0675  &  0.9231
     \end{pmatrix}
    \end{aligned}
\end{align}
\begin{align} \nonumber
    \begin{aligned}
      \tp({8}) &=
          \begin{pmatrix}
    0.0000 &   0.0000 &   0.0000 &   0.0001  &  0.0008  &  0.0078  &  0.0665  &  0.9248\\
     0.0000 &   0.0000 &   0.0000 &   0.0000  &  0.0007  &  0.0068  &  0.0626  &  0.9299\\
     0.0000 &   0.0000 &   0.0000 &   0.0000  &  0.0006  &  0.0061  &  0.0581  &  0.9352\\
     0.0000 &   0.0000 &   0.0000 &   0.0000  &  0.0005  &  0.0057  &  0.0561  &  0.9377\\
     0.0000 &   0.0000 &   0.0000 &   0.0000  &  0.0005  &  0.0053  &  0.0558  &  0.9384\\
     0.0000 &   0.0000 &   0.0000 &   0.0000  &  0.0004  &  0.0051  &  0.0558  &  0.9387\\
     0.0000 &   0.0000 &   0.0000 &   0.0000  &  0.0004  &  0.0045  &  0.0522  &  0.9429\\
     0.0000 &   0.0000 &   0.0000 &   0.0000  &  0.0003  &  0.0040  &  0.0505  &  0.9452
          \end{pmatrix}
    \end{aligned}
    \end{align}
 \begin{align} \nonumber
    \begin{aligned}
      \cost &=
          \begin{pmatrix}
     1.0000 &   2.2486 &   4.1862 &   6.9509 &  11.2709 &  15.9589 &  21.4617 &  27.6965 \\
    31.3230 &   8.8185 &   9.6669 &  11.4094 &  14.2352 &  17.8532 &  22.3155 &  27.5353 \\
    50.0039 &  26.3162 &  14.6326 &  15.3534 &  17.1427 &  19.7455 &  23.1064 &  27.3025 \\
    65.0359 &  40.2025 &  27.5380 &  19.5840 &  20.3017 &  21.8682 &  24.2022 &  27.4108 \\
    79.1544 &  53.1922 &  39.5408 &  30.5670 &  23.3697 &  23.9185 &  25.1941 &  27.4021 \\
    90.7494 &  63.6983 &  48.6593 &  38.6848 &  30.4868 &  25.7601 &  26.0012 &  27.1867 \\
    99.1985 &  71.1173 &  55.0183 &  44.0069 &  34.7860 &  29.0205 &  26.9721 &  27.1546 \\
  106.3851  & 77.2019  & 60.0885  & 47.8917  & 37.6330  & 30.8279  & 27.7274  & 26.4338
          \end{pmatrix}
    \end{aligned}
    \end{align}
}

\end{subappendices}

\index{myopic policy! optimality for POMDP|)}

\bibliographystyle{plain}

\bibliography{$HOME/styles/bib/vkm}

\end{document}